\documentclass[12pt,leqno]{amsart}
\usepackage[latin1]{inputenc}
\usepackage{color}
\usepackage{amssymb}
\usepackage{enumerate}
\usepackage{enumitem}
\usepackage[all]{xy}
\usepackage{hyperref,cleveref,graphics,mathrsfs}
\newtheorem{thm}{Theorem}[section]
\newtheorem{cor}[thm]{Corollary}
\newtheorem{lem}[thm]{Lemma}
\newtheorem{prop}[thm]{Proposition}
\theoremstyle{definition}
\newtheorem{defn}[thm]{Definition}
\newtheorem{question}[thm]{Question}
\newtheorem{rem}[thm]{Remark}

\newtheorem{example}[thm]{Example}
\numberwithin{equation}{section}
\newcommand{\norm}[1]{\left\Vert#1\right\Vert}

\newcommand{\abs}[1]{\left\vert#1\right\vert}

\newcommand{\Real}{\mathbb R}
\newcommand{\eps}{\varepsilon}

\newcommand{\nat}{\mathbb{N}}

\newcommand{\Nat}{\mathbb{N}}

\def\epsilon{\varepsilon}

\newcommand{\erre}{\mathbb R}

\usepackage{bbm}
\newcommand{\vertiii}[1]{{\left\vert\kern-0.25ex\left\vert\kern-0.25ex\left\vert #1 
    \right\vert\kern-0.25ex\right\vert\kern-0.25ex\right\vert}}

\newcommand{\triple}[1]{{\left\vert\kern-0.25ex\left\vert\kern-0.25ex\left\vert #1 
    \right\vert\kern-0.25ex\right\vert\kern-0.25ex\right\vert}}

\newcommand{\bigabs}[1]{\bigl\lvert#1\bigr\rvert}
\newcommand{\bignorm}[1]{\bigl\lVert#1\bigr\rVert}
\newcommand{\Bigabs}[1]{\Bigl\lvert#1\Bigr\rvert}

\newcommand{\Bignorm}[1]{\Bigl\lVert#1\Bigr\rVert}
\newcommand{\biggnorm}[1]{\biggl\lVert#1\biggr\rVert}

\newcommand{\olx}{{\overline{x}}}

\newcommand{\ola}{{\overline{x}[\alpha]}}
\newcommand{\rad}{\mathrm{Rad}}
\newcommand{\spn}{\mathrm{span}}
\newcommand{\fbl}{\mathrm{FBL}}
\newcommand{\fbp}{{\mathrm{FBL}}^{(p)}}
\newcommand{\FBLp}{{\mathrm{FBL}}^{(p)}}

\DeclareMathOperator{\FVL}{FVL}
\DeclareMathOperator{\FBLi}{FBL^{(\infty)}}
\DeclareMathOperator{\cone}{cone}
\DeclareMathOperator{\Range}{Range}

\usepackage{bbm}
\def\one{\mathbbm 1}

\DeclareMathOperator{\ran}{Range}

\title{Free  Banach lattices}

\author[Oikhberg]{T.~Oikhberg}
\address{Dept. of Mathematics, University of Illinois, Urbana IL 61801, USA}
\email{oikhberg@illinois.edu}

\author[Taylor]{M.A.~Taylor}
\address{Department of Mathematics, University of California, Berkeley, CA, 94720, USA}
\email{mitchelltaylor@berkeley.edu}
\author[Tradacete]{P.~Tradacete}
\address{Instituto de Ciencias Matem\'aticas (CSIC-UAM-UC3M-UCM)\\
Consejo Superior de Investigaciones Cient\'ificas\\
C/ Nicol\'as Cabrera, 13--15, Campus de Cantoblanco UAM\\
28049 Madrid, Spain.}
\email{pedro.tradacete@icmat.es}
\author[Troitsky]{V.G.~Troitsky}
\address{Department of Mathematical and Statistical Sciences, University of Alberta, Edmonton, Alberta T6G 2G1, Canada}
\email{troitsky@ualberta.ca}
\date{\today}



\thanks{T.~Oikhberg was supported by the NSF award 1912897. \\ 
P.~Tradacete was partially supported by grants PID2020-116398GB-I00 and CEX2019-000904-S funded by MCIN/AEI/ 10.13039/501100011033.\\ V.~Troitsky was supported by NSERC grant RGPIN-2020-04855.}

\begin{document}

\begin{abstract}
We investigate the structure of the free $p$-convex Banach lattice $\fbp[E]$ over a Banach space $E$. After recalling why such a free lattice exists, and giving a convenient functional representation of it, we focus our study on how properties of an operator $T:E\rightarrow F$ between Banach spaces  transfer to the associated lattice homomorphism $\overline{T}:\fbp[E]\rightarrow \fbp[F]$. Particular consideration is devoted to the case when the operator $T$ is an isomorphic embedding, which leads us to examine extension properties of operators into $\ell_p$, and several classical Banach space properties such as being a G.T. space. A detailed investigation of basic sequences and sublattices of free Banach lattices is provided.  In addition, we begin to build a dictionary between Banach space properties of $E$ and Banach lattice properties of  $\fbp[E]$. In particular, we characterize the existence of lattice copies of $\ell_1$ in $\fbp[E]$ and show that $\fbl[E]$ has an upper $p$-estimate if and only if $id_{E^*}$ is $(q,1)$-summing  ($\frac{1}{p}+\frac{1}{q}=1$). We also highlight the significant differences between $\fbp$-spaces depending on whether $p$ is finite or infinite. For example, we show that $\fbl^{(\infty)}[E]$ is lattice isometric to $\fbl^{(\infty)}[F]$ whenever $E$ and $F$ have monotone finite dimensional decompositions, while, on the other hand, when $p<\infty$ and $E^*$ is smooth, $\fbp[E]$ determines $E$ isometrically. 
\end{abstract}

\date{\today}
\maketitle

\tableofcontents

\section{Introduction and preliminaries}

The goal of this article is to investigate free Banach lattices generated by Banach spaces. The history of this notion is quite recent: While free vector lattices were already present in the literature in the 1960's \cite{Baker, Bleier}, the corresponding normed version had been completely overlooked until B. de Pagter and A. Wickstead in \cite{dePW} first considered the free Banach lattice generated by a set. This can be considered as a natural precursor of the construction, due to A. Avil\'es, J. Rodr\'iguez and P. Tradacete in \cite{ART}, of the free Banach lattice generated by a Banach space. 
\\

Given a Banach space $E$, the free Banach lattice generated by $E$ is a Banach lattice $\fbl [E]$ together with a linear isometric embedding $\phi_E:E\rightarrow \fbl[E]$ such that for every bounded linear operator $T:E\to X$ into a Banach lattice $X$, there is a unique lattice homomorphism $\widehat T:\fbl[E]\rightarrow X$ such that $\widehat T\circ\phi_E=T$ and $\|\widehat T\|= \|T\|$. From a categorical point of view, this can be seen as a functor from the category of Banach spaces and bounded linear operators into the subcategory of Banach lattices and lattice homomorphisms. It is, in a certain sense, analogous to well studied functors in analysis, such as the one from compact Hausdorff spaces $K$ into spaces of continuous functions $C(K)$, or the functor creating the Lipschitz free space generated by a (pointed) metric space (cf.~\cite{GK}). 
\\

It will soon become clear that understanding the correspondence $E\mapsto \fbl[E]$ is a key to properly understand the interplay between Banach space and Banach lattice properties, a goal that has been pursued ever since the first developments of these theories (see, e.g.,~ \cite{JMST, Kalton_memoir}). In particular, our investigation will be far from categorical, focusing mainly on the fine structure of $\fbl[E]$ and the correspondence $E\mapsto \fbl[E]$.
\\

For several reasons, it will be convenient to  also work with free Banach lattices satisfying some convexity conditions, as considered in \cite{JLTTT}. For a Banach space $E$ and $p\in [1,\infty]$, we define the free $p$-convex Banach lattice over $E$ as follows: $\fbp [E]$ is a $p$-convex Banach lattice with $p$-convexity constant $1$ together with a linear isometric embedding $\phi_E:E\rightarrow \fbp[E]$ such that for every bounded linear operator $T:E\to X$ into a $p$-convex Banach lattice $X$, there is a unique lattice homomorphism $\widehat T:\fbp[E]\rightarrow X$ such that $\widehat T\circ\phi_E=T$, and  $\|\widehat T\|\leq M^{(p)}(X) \|T\|$. Here, $M^{(p)}(X)$ denotes the $p$-convexity constant of $X$. It is clear that $\fbl^{(1)}[E]$ coincides with $\fbl[E]$ (and we will stick to the latter notation). 
\\

Much of our investigation relies on the following explicit functional representation of $\fbp[E]$, first established in \cite{JLTTT}. For this, denote by $H[E]$ the linear subspace of $\erre^{E^*}$ consisting of all positively homogeneous functions $f:E^\ast \to \mathbb{R}$.
For $1\leq p<\infty$ and $f\in H[E]$ we define
\begin{equation} \label{eq:ART}
 	\|f\|_{\fbp[E]}=
	\sup\left\{\left(\sum_{k=1}^n |f(x_k^\ast)|^p\right)^{1/p}: \, n\in\mathbb N, \, x_1^*,\dots,x_n^*\in E^*,
        \,  \sup_{x\in B_E} \sum_{k=1}^n |x_k^\ast(x)|^p\leq 1\right\}.
\end{equation}
Note that, for $(x_k^*)_{k=1}^n\subseteq E^*$, by considering the operator $T:E\rightarrow \ell_p^n$ given by $Tx=(x_k^*(x))_{k=1}^n$ and using the fact that $\|T\|=\|T^{**}\|$, it follows that 
\begin{equation} \label{eq:ARTbidual}
 	\sup_{x\in B_E} \sum_{k=1}^n |x_k^\ast(x)|^p= \sup_{x^{**}\in B_{E^{**}}} \sum_{k=1}^n |x^{**}(x_k^\ast)|^p.
\end{equation}

Given any $x\in E$, let $\delta_x\in H[E]$ be defined by
$$
	\delta_x(x^\ast):= x^\ast(x) 	\quad\mbox{for all }x^*\in E^*.
$$
Then, $\fbp[E]$ coincides with the closed sublattice of $H[E]$ generated by $\{\delta_x:x\in E\}$ with respect to $\|\cdot\|_{\fbp[E]}$, together with the map $x\mapsto \delta_x$. As mentioned, this explicit representation of $\fbp[E]$ was originally proven in \cite{JLTTT}; the proof will be recalled below.  As we will see, it is the interplay between the universal property of $\fbp[E]$ and the explicit functional representation that allows us to discern the fine structure of these spaces. 
\\

When $p=\infty$, $\mathrm{FBL}^{(\infty)}[E]$ is nothing but the closed sublattice of $C(B_{E^*})$ generated by the point evaluations.  Here, $C(B_{E^*})$ denotes the space of continuous functions on the dual ball of $E$, which is equipped with the relative $w^*$-topology. In particular, we have
\begin{equation}\label{eq:infinite_case}
\|f\|_{\fbl^{(\infty)}[E]}=
	\sup\left\{|f(x^\ast)| : \, x^*\in E^*, \, \|x^*\| \leq 1 \right\}.
\end{equation}
As we will show in \Cref{p:structure of fbl infty}, the closure of the point evaluations in $C(B_{E^*})$ coincides with the space of \emph{all} positively homogeneous weak$^*$ continuous functions on $B_{E^*}$. This gives a very concrete description of $\fbl^{(\infty)}[E]$ - this space often behaves  differently  from $\fbp[E]$ when $1\leq p<\infty.$ 
\\

\subsection{A word on free objects}

The universal property defining $\fbl[E]$ (or analogously, $\fbp[E]$) can be visualized by means of the following diagram:
$$
\xymatrix{\fbl[E]\ar[rd]^{\widehat{T}}&\\
     E\ar[r]^{T}\ar^{\phi_E}[u]& X}
$$
meaning that for every \textit{object} $X$ (a Banach lattice) and every \textit{linear operator} $T:E\rightarrow X$ there is a unique \textit{morphism} $\widehat T$ (lattice homomorphism) making the diagram commutative. 
\\

The idea of the free object in a certain category (Banach lattices with lattice homomorphisms) generated by an object in a supercategory (Banach spaces with bounded linear operators) is certainly not new, and has been central in several developments in algebra, topology and analysis. We will not attempt here to address the fruitful developments of this idea in universal algebra, but let us just recall that these include many well-known concepts such as free groups, free modules, free algebras or free lattices. 
\\

The study of free objects in Banach space theory can be considered as a more recent development. However, some classical facts can be reworded in this language too. Consider, for instance, the subcategory of \textit{dual} Banach spaces together with \textit{dual operators} (equivalently, weak$^*$ continuous linear maps). Given a Banach space $E$, let $J_E:E\rightarrow E^{**}$ denote the canonical embedding; it is clear that every bounded linear operator $T:E\rightarrow X^*$ can be uniquely extended to a dual operator $\overset{*}T:E^{**}\rightarrow X^*$, given by $\overset{*}T=(T^*\circ J_X)^*$, in such a way that the following diagram commutes:
$$
\xymatrix{E^{**}\ar[rd]^{\overset{*}T}&\\
     E\ar[r]^{T}\ar^{J_E}[u]& X^*}
$$
Thus, we can consider $E^{**}$ as the \textit{free dual Banach space generated by} $E$.
\\

Lipschitz free spaces (also known as Arens-Eells or transportation cost spaces) have recently attracted considerable attention from researchers interested in Banach space theory and metric geometry (see, for instance, the survey paper \cite{Godefroy}). These spaces can be defined as follows: Given a metric space $M$ with a distinguished point $0$, $\mathcal F(M)$ is a Banach space equipped with an isometric map $\delta:M\rightarrow \mathcal F(M)$ with the property that for every Banach space $X$ and every Lipschitz map $f:M\rightarrow X$ with $f(0)=0$, there is a unique linear operator $\widehat f:\mathcal F(M)\rightarrow X$ making  the following diagram commute:
$$
\xymatrix{\mathcal F(M)\ar[rd]^{\widehat f}&\\
     M\ar[r]^{f}\ar^{\delta}[u]& X}
$$
A very fruitful line of research is devoted to analyzing the interplay between Banach space properties of $\mathcal F(M)$ versus metric properties of $M$. Results from this line have deeply inspired our research on free Banach lattices.
\\

Free objects also arise in the theory of group $C^*$-algebras. Suppose, for simplicity, that $G$ is a discrete group. We say that a map $\pi : G \to A$ ($A$ is a unital $C^*$-algebra) is a \emph{unitary representation} if it takes $G$ into the unitary group of $A$, and, for any $g, h \in G$, we have $\pi(gh) = \pi(g) \pi(h)$ (which implies $\pi(g^{-1}) = \pi(g)^*$). Then one defines the \emph{full} (or \emph{universal}) $C^*$-algebra $C^*(G)$ over $G$ as follows:   $C^*(G)$ is a $C^*$-algebra, together  with a unitary representation $\psi : G \to C^*(G)$,  with the property that for every $C^*$-algebra $A$, and every unitary representation $\pi : G \to A$, there exists a unique $*$-representation $\widehat{\pi} : C^*(G) \to A$, making  the following diagram commute:
$$
\xymatrix{C^*(G) \ar[rd]^{\widehat \pi}&\\
     G\ar[r]^{\pi}\ar^{\psi}[u]& A}
$$
For the construction and basic properties of $C^*(G)$, see \cite[Section 2.5]{BrOz} or \cite[Chapter 3]{Pisier-TP}. One can also consider full $C^*$-algebras of a more general class of locally compact groups; in this case, certain continuity properties of representations need to be assumed. The reader is referred to  \cite[Chapter VII]{Davidson} for details. 
\\

The investigation of free $C^*$-algebras was motivated by two related questions.
\\

1) Finding connections between properties of a group $G$ and those of $C^*(G)$ (with the reduced group $C^*$-algebra $C^*_r(G)$ often also added to the mix). A sample result is \cite[Theorem 2.6.8]{BrOz}: a discrete group $G$ is amenable if and only if $C^*(G)$ is nuclear.
\\

2) The famous Kirchberg's QWEP conjecture is equivalent to $C^*({\mathbb{F}})$ having the Weak Expectation Property (a relaxation of injectivity) for any free group ${\mathbb{F}}$ \cite[Chapter 13]{Pisier-TP}. By \cite[Chapter 14]{Pisier-TP}, this is also equivalent to Connes' Embedding Problem, which has recently been resolved in the negative  in \cite{MIP-RE}.
\\

We refer the reader to \cite{P93} and references therein for several other constructions of free topological objects, as well as their connections to various universal constructions, such as free and tensor products \cite{vanWaaij:13}.  We also note that a more ``axiomatic'' approach to freeness has been pursued by A.~Ya.~Helemskii in, e.g.,~\cite{Hel-projectivity}, \cite{Hel13}, \cite{Hel14}, \cite{Hel20}, and \cite{HO} (see also \cite{Arbib_Manes} for a different take on the same approach). Specifically, suppose ${\mathcal{K}}$ and ${\mathcal{L}}$ are categories, and $\square$ is a faithful covariant functor ${\mathcal{K}} \to {\mathcal{L}}$ (usually, a ``forgetful functor'').  ${\mathcal{K}}$ is called a \emph{rigged category}, and, with respect to appropriate rigs, the  works cited above construct free objects in various situations. This includes quantum spaces \cite{Hel-projectivity}, normed operator modules \cite{Hel14}, normed modules over sequence algebras \cite{Hel13},  matricially normed spaces \cite{Hel20}, as well as  multinormed spaces and their generalizations \cite{HO}. This allows one to construct projective objects in these categories as well. A similar approach has been recently used in \cite{AMR3}.

\subsection{Historical perspectives}\label{Section history}
Although free Banach lattices were not introduced until 2015, their inception triggered  a rapid development of the theory. Here, we briefly summarize the literature. With one exception, the articles below focus on $\fbl$; however, we work with the full scale of spaces $\fbp$, $1\leq p\leq \infty$. That being said, most of the results in this article are either new for $\fbl$, or require significantly different proofs in order to generalize to $\fbp$. 
\\

The theory of free Banach lattices began with \cite{dePW}, which introduced the concept of a free Banach lattice over a set $A$. In our terminology, this is simply the space $\fbl[\ell_1(A)]$. In \cite{dePW}, the authors proved several structural results, and showed that this new class of spaces  differs significantly from the classical Banach lattices. After this, free Banach lattices over general Banach spaces were introduced in \cite{ART}. Among other things, this allowed the authors of \cite{ART} to answer some questions left open in \cite{dePW}, as well as a question of J.~Diestel on weakly compactly generated Banach lattices, and opened the door for a deeper study of the relationship between Banach spaces and Banach lattices.
\\

After the above two seminal works, the theory expanded in several directions. One interesting direction - that we will not pursue here - centers around the free Banach lattice generated by a lattice. Recall that a Banach lattice combines two distinct structures: The Banach space structure, and a lattice structure. In analogy with the free Banach lattice $\fbl[E]$ generated by a Banach space $E$, one can consider the free Banach lattice $\fbl\langle\mathbb{L}\rangle$ generated by a lattice $\mathbb{L}$. This latter construction is also quite rich, and the two theories parallel each other to some extent. However, there are also some interesting differences. For example, $\fbl\langle\mathbb{L}\rangle$ is always lattice isomorphic to an AM-space \cite{AMRR2}, whereas $\fbl[E]$ is lattice isomorphic to an AM-space if and only if $E$ is finite dimensional. We refer the interested reader to \cite{AMRR,AMRR2,AR1,AR} for literature on $\fbl\langle\mathbb{L}\rangle$.  For literature on free lattices (without involving norms), we mention  \cite{Macula:92} for free $\alpha$-order complete vector lattices, \cite{Abbadini} for free $\sigma$-order complete truncated vector lattices, \cite{K77} for free lattice-ordered Lie algebras, \cite{T65} for projective vector lattices,  \cite{KM94} for free lattice-ordered groups and free products of lattice-ordered groups, \cite{F04} for free products of Boolean algebras and measure algebras with applications to tensor products of universally complete vector lattices, and references therein.
\\

As a second extension of the concept of free Banach lattices, free Banach lattices satisfying convexity conditions were constructed in \cite{JLTTT}. Recall that the defining property of $\fbl[E]$ is that any linear operator from $E$ to a Banach lattice $X$ extends uniquely to $\fbl[E]$ as a lattice homomorphism of the same norm. If instead of looking at \emph{all} Banach lattices $X$, one only requires that the above property hold for $p$-convex spaces, then one can construct a free object $\fbp[E]$ that is $p$-convex in its own right. This is of interest because $p$-convexity is a classical Banach lattice property (see, e.g.,~\cite{LT2}), and because having a scale of spaces $\fbp[E]$, $1\leq p\leq \infty$, adds an additional dimension to the theory, similar to how the $L_p$ scale enriches the study of $L_1$. Moreover, the paper \cite{JLTTT} constructs various free lattices satisfying convexity conditions $\mathcal{D}$, by placing a maximal $\mathcal{D}$-convex norm on the free vector lattice, and then completing the resulting space. Such a construction builds on ideas from \cite{Tro}.
\\

There have also been several papers focusing on $\fbl[E]$, and its applications. For example, \cite{Norm-attaining, DMRR} focus on the isometric theory of free Banach lattices. As a brief overview, \cite{DMRR} studies, among other things, when the norm of $\fbl[E]$ is octahedral; \cite{Norm-attaining} is able to use free Banach lattices to produce the first example of a lattice homomorphism that does not attain its norm. In a different direction, \cite{ATV} studies free Banach lattices generated by the classical sequence spaces $\ell_p(\Gamma)$. In particular, the authors of \cite{ATV} are able to precisely describe the moduli of the canonical unit vector bases, and  when these spaces are weakly compactly generated. \cite{AMRT} studies when a Banach lattice $X$ is lattice isomorphic to a lattice-complemented sublattice of $\fbl[X]$. As it turns out, any Banach lattice $X$ ordered by a $1$-unconditional basis has this property.
\\

Pure vector lattice properties of $\fbl[E]$ have also been studied. For example, \cite{APR} is able to prove that $\fbl[E]$ satisfies the countable chain condition, i.e., that any collection of pairwise disjoint vectors in $\fbl[E]$ must be countable. Finally, there are many interesting applications of free Banach lattices. For example, \cite{AMZT} is able to classify the separable Banach lattices $X$ such that whenever a Banach lattice $Y$ contains a subspace isomorphic to $X$ then it also contains a sublattice isomorphic to $X$. It is classical that $X=c_0$ has this property. However, it is shown in \cite{AMZT} that a separable Banach lattice  has this property if and only if it lattice embeds into $C[0,1].$ In particular, whenever $C[0,1]$ embeds linearly into a Banach lattice, it also embeds as a sublattice. This feature  is not shared by $C(\Delta)$, with $\Delta$ the Cantor set, even though $C[0,1]$ and $C(\Delta)$ are isomorphic as Banach spaces. As noted in \cite{AMZT}, it is not known if every Banach lattice for which linear embeddability implies lattice embeddability is necessarily separable,  but free Banach lattices put several constraints on how such a supposed space would look. 
\\

Free Banach lattices also play an important role in the study of projective Banach lattices. Projectivity for Banach lattices was also first considered by B. de Pagter and A. Wickstead in \cite{dePW}. Informally, a Banach lattice $P$ is projective if every lattice homomorphism from $P$ into the quotient of a Banach lattice $X$ lifts to a lattice homomorphism into $X$, with control of the norm. As a consequence of the fact that $\ell_1(A)$ is a projective Banach space for any nonempty set $A$, it easily follows that $\fbl[\ell_1(A)]$ is a projective Banach lattice. This is the first connection between projectivity and $\fbl$, but the topics interlace in much deeper ways. We refer the reader to   \cite{AMR, AMR2} for some of the connections between projectivity and free Banach lattices; however, plenty more results are scattered throughout the literature we have cited in this subsection. Free Banach lattices also have interesting applications to amalgamation and injectivity; in particular, they can be used to define push-outs and thus play a role in the construction of Banach lattices of universal disposition and separably injective Banach lattices,  see \cite{AT}.
\\

\subsection{A brief outline of the results}

We begin this paper by giving, in \Cref{construction}, a ``natural'' functional representation of free Banach lattices (\Cref{t:fblbp}, \Cref{p:structure of fbl infty}). \Cref{t:fblbp} is essentially taken from \cite{JLTTT}, but the identification $\fbl^{(\infty)}[E]=C_{ph}(B_{E^*})$ given in \Cref{p:structure of fbl infty} is new. \Cref{construction}  also contains various comments on the functional representation.
 The most important of these is \Cref{t:FVL order dense}, where we show that the free vector lattice over $E$ is not only norm dense, but order dense, in $\fbp[E]$. 
\\

\Cref{Section Tbar} studies the relationship between an operator $T:F\to E$, and its induced operator $\overline{T}:\fbp[F]\to\fbp[E]$. In \Cref{Injec-Surjec} we show that several properties -- injectivity, surjectivity, density of the range, etc., -- pass freely between $T$ and $\overline{T}$.  
We then look at the way $\fbp[F]$ sits inside of $\fbp[E]$ when $F$ is a subspace of $E$. \Cref{regularity}   shows that, if $\iota\colon F\hookrightarrow E$ is the inclusion map, then
  $\bar\iota\colon\FBLp[F]\to\FBLp[E]$ is order continuous -- in other words,
  $\FBLp[F]$ is a regular sublattice of $\FBLp[E]$.
Examples from \Cref{s:examples_of_lattice_structure} (built on ``low-tech'' Banach lattice techniques) show that, in the above setting, $\fbp[F]$ need not be closed in $\fbp[E]$ -- that is, $\overline{\iota}$ need not be an isomorphic embedding. This leads us to study the ``subspace problem'': under what conditions does the embedding $\iota : F \hookrightarrow E$ induce a lattice isomorphic embedding $\overline{\iota} : \fbp[F] \to \fbp[E]$?
\\

In \Cref{s:subspace_problem} we reduce this question to certain extension properties of (pairs of) Banach spaces. 
More specifically, in \Cref{p:POE1_vs_extension} we establish that $\overline{\iota}$ is bounded below if, and only if, any operator $T : F \to L_p$ extends to $E$.
In this case, we say that the pair $(F,E)$ has the POE-$p$ (Property of operator extension into $L_p$ -- \Cref{d:POEp}). Although the POE-$p$ is defined by extension  properties of a family of operators, the fact that it is equivalent to  the single operator $\overline{\iota}$ being an embedding gives some stability. More precisely, if, for every $\varepsilon>0$, $(F,E)$ has the POE-$p$ with constant $C+\varepsilon$, then \Cref{p:POEp_C_versus_C+} shows that $(F,E)$ has the POE-$p$ with constant $C$. 
\Cref{s:subspace_problem} finishes with a discussion of when $\overline{\iota}(\fbp[F])$ is complemented in $\fbp[E]$.
\\

Returning to the question of when $\overline{\iota}$ is an embedding, in \Cref{ss:general_facts_POEp}, we gather  general facts about the POE-$p$. 
This includes a reformulation in terms of $\ell_\infty$-factorable operators, various criteria in terms of $2$-summing operators, and a relation with $\mathcal{L}_p$-spaces.
\Cref{ss:POEp_duality} explores the connections between the POE-$p$, passing to the double dual, and taking ultrapowers. This allows us to give several examples of spaces having, or failing, the POE-$p$.
\\

\Cref{Further POEp} discusses several more properties of the POE-$p$. In particular, a push-out argument shows that one does not need to require a uniform constant independent of embeddings; it comes for free by \Cref{p:POEp}. Similarly, to check whether a Banach space $F$ has the POE-$p$, it suffices to only consider embeddings into spaces of the same density character, see \Cref{control density}.  In \Cref{p:relations_between_POEp} we use our results on the POE-$p$ to examine connections between the POE-$p$ for different values of $p$. In particular, we prove that $\ell_1$ has the POE-$p$  if and only if $2 \leq p \leq \infty$ (cf.~\Cref{p:extension_from_L1} for a more general result on $\mathcal{L}_{1,\mu}$-spaces). On the other hand, a space with a normalized unconditional basis has the POE-$1$ if and only if that basis is equivalent to the $c_0$ basis (\Cref{p:uncond_basis_POE1}).
 We finish by showing, in \Cref{p:hilbert_in_L1_p1}, that $(F,L_1)$ can never have the POE-$p$ when $F \subseteq L_1$ is an infinite dimensional Hilbertian subspace.  
 \\

In \Cref{s:basic_seq}, we investigate the properties of the sequence $\big( \big| \delta_{x_k} \big| \big)_k \subseteq \fbp[E]$, where $(x_k)$ is a basic (and often, unconditional) sequence in $E$. We show that every weakly null semi-normalized sequence $(x_n)$ in a Banach space $E$ has a subsequence so that $(|\delta_{x_{n_k}}|)$ is basic in $\fbp[E]$ (Proposition \ref{p:basicsubsequence}).
\Cref{lower 2 gives 1} shows that, if a normalized basis $(x_k)$ satisfies a lower $2$-estimate, then $\big( \big| \delta_{x_k} \big| \big)_k$ is equivalent to the $\ell_1$ basis.
By \Cref{p:l1_vs_l2}, the converse is true for unconditional bases.
We also examine whether $\big( \big| \delta_{x_k} \big| \big)_k$ is necessarily unconditional (\Cref{dual to the summing}, \Cref{p:summing_basis_conditional}). In \Cref{subsec:Haar}, we compute the moduli of branches of the Haar in $\fbl[L_1]$.
\\

Part of the motivation to study the sequence $\big( \big| \delta_{x_k} \big| \big)_k \subseteq \fbp[E]$ comes from the universal property of free Banach lattices. Suppose $(x_k)\subseteq E$ is as above, and $X$ is a $p$-convex Banach lattice. Then any operator $T : E \to X$ extends canonically to a lattice homomorphism $\widehat{T} : \fbp[E] \to X$, with $\widehat{T} |\delta_{x_k}| = |T x_k|$. Consequently, the sequence $(|\delta_{x_k}|)$ ``dominates'' $(|Tx_k|)$. In particular, if $(|\delta_{x_k}|)$ is weakly null, then so is $(|Tx_k|)$; see \Cref{weakly null}, and the subsequent discussion.
\\

We continue our work on $\big( \big| \delta_{x_k} \big| \big)_k$ in \Cref{s:endpoints}.
In particular, in \Cref{t:norms_versus_summing} and \Cref{c:1-summing} we express, for $a_1, \ldots, a_n \geq 0$, the norm
$\big\| \sum_{k=1}^n a_k  \big| \delta_{x_k} \big| \big\|_{\fbl[E]}$ as a $1$-summing norm of a certain operator. This is very useful for computations, and, in particular, allows us to recover some of the main results of \cite{ATV}.
\\

Armed with this knowledge, we attempt to reconstruct properties of $(x_k)$ from those of $\big( \big| \delta_{x_k} \big| \big)_k \subseteq \fbp[E]$.
Our first task is to describe sequences $(x_k) \subseteq E$ which are equivalent to $\big( \big| \delta_{x_k} \big| \big)_k \subseteq \fbp[E]$, $p<\infty$. It turns out that, if this holds, and $(x_k)$ is a normalized basis of $E$, then it has to be equivalent to the $\ell_1$ basis (\Cref{h}). 
However, in general, a normalized basic sequence $(x_k)$ may be equivalent to $\big( \big| \delta_{x_k} \big| \big)$, but not to the $\ell_1$ basis. Indeed, \Cref{p:l2_in_CK} shows that, if $(x_k) \subseteq C(\Omega)$ is a sequence equivalent to the $\ell_2$ basis, then $\big( \bigabs{ \delta_{x_k} } \big) \subseteq \fbp[C(\Omega)]$ ($1 \leq p \leq \infty$) is equivalent to the same basis. Moreover, in $\fbl^{(\infty)}[E]$, every unconditional basic sequence $(x_k)$ is equivalent to $\big( \big| \delta_{x_k} \big| \big),$ in stark contrast to the case $p<\infty$. 
\\

In \Cref{NOT comp}--\Cref{p:moduli in fbl infty} we characterize the normalized unconditional bases $(x_k)$ of $E$ for which $\overline{\text{span}}[|\delta_{x_k}| : k\in \mathbb{N}]$ is complemented in $\fbp[E]$. For $p=1$, this happens only for the $\ell_1$ basis, for $p\in (1,\infty)$ this  never happens, and for $p=\infty$ this  happens only for the $c_0$ basis. In \Cref{c:unconditional_moduli} and \Cref{c:indep_rv}, we give examples of sequences $(x_k) \subseteq L_1$ for which $\big( \big| \delta_{x_k} \big| \big)_k \subseteq \fbl[L_1]$ is equivalent to the $\ell_1$ basis.
\\

We finish \Cref{s:endpoints} by proving the following rigidity result: If $(x_k)$ is an unconditional basis of $E$, and $\big( \bigabs{ \delta_{x_k} } \big) \subseteq \fbp[E]$ is equivalent to the $\ell_2$ basis, then $(x_k)$ must be equivalent to the $c_0$ basis (\Cref{t:l2_created_by_c0}).
\\

In \Cref{ss:without_bibasic}, we use free Banach lattices to construct the first example of a subspace of a Banach lattice without a bibasic sequence (\Cref{t:c0_no_bibasis}, \Cref{r:l_q_no_bibabsis}). This answers a question from \cite{TT}. In \Cref{MM}, we discuss connections with majorizing maps, and prove some results akin to the Bibasis  \Cref{bibasis theorem}. In particular, in \Cref{Connection to Pedro}, we show that the class of sequentially uniformly continuous operators (defined originally in \cite{TT}) coincides with the class of $(\infty,\infty)$-regular operators (as defined in \cite{SanPe-Tra}). Moreover, in \Cref{absolute bases} we provide a converse to \cite[Proposition 7.8]{TT}: The $\ell_1$ basis is the only normalized basis that is absolute in any Banach lattice  where it linearly embeds.
\\

\Cref{s:sublattices of FBL} examines ($p$-convex) Banach lattices $E$ which embed into $\fbp[E]$ as a sublattice.
\Cref{p:sublattice} shows that, if the order on $E$ is determined by its $1$-unconditional basis, then $E$ embeds into $\fbp[E]$ as a sublattice, complemented by a contractive lattice homomorphic projection. These results partially overlap with those in \cite{AMRT}, though the proofs are very different.
\\

In \Cref{Dictionary} we develop a dictionary between Banach space properties of $E$ and Banach lattice properties of $\fbp[E]$. To begin, we prove that $\fbp[E]$ has a strong unit if and only if $E$ is finite dimensional (\Cref{No unit}), and $\fbp[E]$ has a quasi-interior point if and only if $E$ is separable (\Cref{QIP}). We further elaborate on this topic in \Cref{generators}, by showing (\Cref{number of generators}) that $E$ is finite dimensional if and only if $\fbp[E]$ is finitely generated (and, in this case, $\dim E$ equals the smallest number of generators).
\\

\Cref{WCGsection} considers the connection between $E$ being weakly compactly generated and $\fbp[E]$ being lattice weakly compactly generated. This is a topic that has been explored before, and the implications $E$ WCG $\Rightarrow$ $\fbl[E]$ LWCG $\not\Rightarrow$ $\fbl[E]$ WCG were used to solve a problem which was raised by J. Diestel in a conference in La Manga (Spain) in 2011. Our main contribution is to prove that if $\fbp[E]$ is LWCG then $E$ is a subspace of a WCG space. This makes significant progress towards the conjecture that $\fbl[E]$ is LWCG if and only if $E$ is WCG.
\\

In \Cref{complem l1} we consider the existence of complemented copies of $\ell_1$. \Cref{comp ell_1} shows that $E$ contains a complemented copy of  $\ell_1$ if and only if $\fbl[E]$ contains a lattice complemented  sublattice isomorphic to $\ell_1$ if and only if $\fbl[E]$ contains a complemented  copy of $\ell_1$ (a few other equivalent conditions on $\fbl[E]$ are also given).
\\

In \Cref{s:local theory} we characterize when $\fbl[E]$ satisfies an upper $p$-estimate, and deduce various corollaries. The main result is \Cref{p:upper est} which shows that $id_{E^*}$ is $(q,1)$-summing if and only if $\fbl[E]$ satisfies an upper $p$-estimate ($\frac{1}{p}+\frac{1}{q}=1$). In particular, this shows that $\fbl[E]$ can never be more than $2$-convex when $E$ has infinite dimension. \Cref{p:upper est} also leads to a ``local" version of \Cref{comp ell_1}: $E$ contains uniformly complemented copies of $\ell_1^n$ if and only if $\fbl[E]$ contains uniformly lattice complemented sublattice copies of $\ell_1^n$ (\Cref{c:criteria for l1}). Further, it allows us to generalize some classical theorems on $p$-convex Banach lattices to Banach lattices with an upper $p$-estimate. Specifically, \Cref{upper p version} shows that if a Banach lattice $F$ embeds POE-$1$ into a Banach lattice $E$ with an upper $p$-estimate ($1<p<2$), then $F$ must also have an upper $p$-estimate. This result with $p$-convexity in place of an upper $p$-estimate and POE-$1$ replaced by complementation is classical, see \cite[Theorem 1.d.7]{LT2}.
Finally, we ask if $\fbp[E]$ can be $q$-convex for some $q>p$. This leads to an interesting dichotomy at $p=2$, and the (sharp) estimate $q\leq \max\{2,p\}$ (\Cref{p:cant be worse than 2-convex}). In particular, although there are examples of infinite dimensional $E$ such that  $\fbl[E]$ is 2-convex, it is impossible for $\fbl^{(2)}[E]$ to be more than $2$-convex, unless $E$ is finite dimensional, in which case it is $\infty$-convex.
\\

In \Cref{Convexity} we further pursue the automatic convexity of free Banach lattices, and use this to study the connections between $\fbp[E]$ and $\fbl^{(q)}[E]$ for various values of $p$ and $q$. In the previous section, a characterization of when $\fbl[E]$ satisfies an upper $p$-estimate was given, and in this section a $p$-convex variant is proven. Specifically, \Cref{P_p} shows that $\fbp[E]$ and $\fbl^{(q)}[E]$ are lattice isomorphic if and only if every operator $T:E\to L_q$ factors strongly through $L_p$. This, of course, connects deeply with the Maurey-Nikishin factorization theory, and allows us to give new perspectives on this classical topic. One corollary (\Cref{no factor}) is that when $q\geq 1$ and $p>\max\{2,q\}$ every infinite dimensional Banach space $E$ admits an operator $T:E\to L_q$ which does not strongly factor through $L_p$. Moreover, we are able to prove the extrapolation \Cref{Extrapolation}: If $\fbp[E]$ has convexity $q>p$,  then $\fbp[E]$ is lattice isomorphic to $\fbl[E]$. This complements the characterization that  $\fbl[E]$ has non-trivial convexity if and only if $E^*$ has non-trivial cotype given in \Cref{c:criteria for l1}. It also allows us to present various situations where $\fbp[E]$ and $\fbl^{(q)}[E]$ are lattice isomorphic, and gives us the ability to distinguish $\fbp[E]$ from the $p$-convexification of $\fbl[E]$.
\\

In \Cref{Convexity} we also elaborate on our study of upper $p$-estimates. One interesting fact about the free $p$-convex Banach lattice is that $L_p$ is sufficient to witness its universal property, i.e., uniform extension of operators into $L_p$ implies uniform extension of operators into an arbitrary $p$-convex Banach lattice.  We prove a similar upper $p$-estimate version of this theorem. Specifically, we show in 
\Cref{p:isomorphic characterization of fbl upper p} that weak-$L_p$ is sufficient to verify the universal property of being the free Banach lattice satisfying an upper $p$-estimate. Morally, this means that if a Banach lattice $Z$ contains $E$ as a generating set and allows uniform lattice homomorphic extension of maps from $E$ into $L_{p,\infty}$, then the same is true with $L_{p,\infty}$ replaced by an arbitrary Banach lattice with an upper $p$-estimate. This then allows us to characterize  the class of $(p,\infty)$-convex operators in \Cref{Factoring}: An operator is $(p,\infty)$-convex if and only if it strongly factors through a Banach lattice with an upper $p$-estimate.
\\

\Cref{s:isomorphism} is devoted to determining whether $\fbp[E]$ and $\fbp[F]$ can be lattice isomorphic, even when the underlying spaces $E$ and $F$ are not. 
We begin, in \Cref{ss:lattice homs}, by representing lattice homomorphisms between free lattices as composition operators (\Cref{l:latticehomocomp}).
For $p = \infty$, we show that the lattices $\fbl^{(\infty)}[E]$ are lattice isometric to each other, for a wide class of spaces $E$ (\Cref{t:monot_FDD}).
On the other hand, for $p < \infty$, we show that $\fbp[E]$ will not be lattice isomorphic to a lattice quotient of $\fbp[F]$, provided $E$ and $F$ are ``sufficiently different'' (\Cref{p:domination}). Moreover, under fairly general conditions \Cref{p:isometric smooth} shows that a lattice isometry between $\fbp[E]$ and $\fbp[F]$ ($p<\infty$) descends to an isometry between $E$ and $F$. Along the way, we discover various properties of lattice homomorphisms between free Banach lattices.
\\

\subsection{Conventions}\label{convention}
We use the standard Banach space and Banach lattice notation throughout the paper. 
The reader can consult \cite{alb-kal} and \cite{LT1} for Banach spaces, \cite{AB}, \cite{LT2} and \cite{M-N} for Banach lattices. We work with real spaces, though we refer the reader to \cite{dHT} for information on free complex Banach lattices. The closed unit ball of a normed space $E$ shall be denoted by $B_E$. We assume, without mention, that all measures involved are $\sigma$-finite. In particular, this convention is in place when we state that $L_\infty(\mu)$ is injective. We use the shorthand ``$L_p$-space'' for $L_p(\mu)$.  When speaking of bases, $(e_k)$ will be our notation for the standard unit vector basis of $\ell_r$ or $c_0$, and $(x_k)$ will denote a generic basic sequence. From now on, ``subspace" will be synonymous with ``closed non-zero subspace", unless mentioned otherwise. 
\\

Throughout, operators are assumed to be linear and bounded. For extensions of operators, we adopt the following convention.
For an operator $T$, we denote by $\widehat T$ its lattice homomorphic extension.
Extensions which are merely linear and bounded are denoted by $\widetilde T$. We write $\overline{T}:\fbp[F]\to \fbp[E]$ for the canonical extension of $T:F\to E$; that is, $\overline{T}=\widehat{\phi_E\circ T}$. Also, when the Banach space $E$ is unambiguous, we will  write $\phi$ instead of $\phi_E$ for the canonical inclusion.
\\

We shall use the term ``lattice isomorphism'' to mean ``lattice homomorphic isomorphism''; ``lattice isometry'' is defined in a similar way. Further, we use the shorthand ``lattice projection'' to mean ``idempotent lattice homomorphism.'' If there is a lattice projection from $X$ to its sublattice $Y$, we say that $Y$ is ``lattice-complemented'' in $X$.

\section{Construction of $\fbp[E]$ and basic properties}\label{construction}

In this section, for the convenience of the reader, we recall the explicit construction of $\fbl^{(p)}[E]$, and some  of its basic properties. We first do the case $p<\infty$, and then provide a new,  concrete description of $\fbl^{(\infty)}[E]$.
\\

Recall that a Banach lattice $X$ is $p$-convex for $1\leq p\leq\infty$ if there is a constant $M\geq1$ such that for every choice of $(x_k)_{k=1}^n\subseteq X$ we have
$$
\bigg\|\bigg(\sum_{k=1}^n |x_k|^p\bigg)^{\frac{1}{p}}\bigg\|\leq M \bigg(\sum_{k=1}^n \|x_k\|^p\bigg)^{\frac{1}{p}},
$$
if $p<\infty$, or 
$$
\bigg\|\bigvee_{k=1}^n |x_k|\bigg\|\leq M \max_{1\leq k\leq n} \|x_k\|,
$$
if $p=\infty$. Let $M^{(p)}(X)$ denote the $p$-convexity constant of $X$; that is, the smallest possible value of $M$ in the inequalities above. Note in particular that every Banach lattice $X$ is $1$-convex with $M^{(1)}(X)=1$. We refer to \cite[Section 1.d]{LT2} for general background on $p$-convexity.
\\

Let $H[E]$ denote the linear subspace of $\erre^{E^*}$ consisting of all positively homogeneous functions $f:E^\ast \to \mathbb{R}$; i.e., functions satisfying $f(\lambda x^*)=\lambda f(x^*)$ for $\lambda\geq0$ and $x^*\in E^*$. Given $f\in H[E]$,  set
$$
\|f\|_{\fbp[E]}=
	\sup\left\{\left(\sum_{k=1}^n |f(x_k^\ast)|^p\right)^{1/p}: \, n\in\mathbb N, \, x_1^*,\dots,x_n^*\in E^*,
        \,  \sup_{x\in B_E} \sum_{k=1}^n |x_k^\ast(x)|^p\leq 1\right\}.
$$
It is easy  to see that 
\begin{displaymath}
  H_p[E]:=\bigl\{f\in H[E] :  \norm{f}_{\fbp[E]} <\infty\bigr\}
\end{displaymath}
is a sublattice of~$H[E]$ and that~$\norm{\,\cdot\,}_{\fbp[E]}$ defines a
complete $p$-convex lattice norm on~$H_p[E]$ with
$p$\nobreakdash-con\-vex\-ity constant one. Moreover, for $x\in E$, we define $\delta_x\in H[E]$  by $\delta_x(x^*)=x^*(x)$ for $x^*\in E^*$. Note that $\norm{\delta_x}_{\fbp[E]}=\norm{x}$ for every $x\in E$. 
\\

Let $\FVL[E]$ denote the sublattice generated by $\{\delta_x\}_{x\in E}$ in $H[E]$.  $\FVL[E]$  consists of all possible expressions which can be written with finitely many elements of the form $\delta_{x}$ and finitely many linear and lattice operations. In fact, by \cite[p.~204, Exercise 8(b)]{AB}  the sublattice generated by a subset~$W$ of a vector lattice is given by
\begin{equation}\label{eq:ABsublattice}
\biggl\{ \bigvee_{k=1}^n u_k - \bigvee_{k=1}^n w_k : n\in\mathbb{N},\, u_1,\ldots,u_n,w_1,\ldots,w_n\in \operatorname{span}W\biggr\}.
\end{equation}
As we will show in the proof of \Cref{t:fblbp} below, $\FVL[E]$ has the universal property of the free vector lattice over $E$. More specifically, every linear map $T:E\to X$ into an (Archimedean) vector lattice $X$ uniquely extends to $\FVL[E]$ as a lattice homomorphism. This justifies our notation, $\FVL[E]$, for this space. We define $\fbp[E]$ as the closure of~$\FVL[E]$ in~$H_p[E]$, and note that the map
$\phi_E\colon E\to\fbp[E]$ given by $\phi_E(x)=\delta_x$ is a
linear isometry. The goal now  is to show that this space satisfies the universal property of the free $p$-convex Banach lattice:

\begin{thm}\label{t:fblbp}
  Let $X$ be a $p$-convex Banach lattice $(1\leq p<\infty)$ and $T\colon E\to X$  an operator. There is a unique lattice homomorphism
  $\widehat{T}\colon \fbp[E]\to X$ such that
  $\widehat{T}\circ\phi_E=T,$ and
  $\|\widehat{T}\|\le M^{(p)}(X)\,\norm{T},$ where $M^{(p)}(X)$ denotes
  the $p$-convexity constant of $X$.
\end{thm}

\begin{proof}
  We first want to show that there is a unique
  lattice homo\-mor\-phism \mbox{$\widehat{T}\colon \FVL[E]\to X$} such that
  $\widehat{T}\delta_x = Tx$ for every $x\in E$. To those familiar with the construction of the free vector lattice, this should be relatively clear, but we provide an explicit construction nonetheless.
  \\
  
Let $f\in \FVL[E]$. By definition, $f$ is a lattice-linear combination of $\delta_{x_1},\dots, \delta_{x_n}$ for some $x_1,\dots, x_n\in E$.  We define $\widehat{T}f$ to be the same lattice-linear combination of $Tx_1,\dots,Tx_n$ in $X$. That is, suppose that $f=F(\delta_{x_1},\dots, \delta_{x_n})$ for some lattice-linear expression $F(t_1,\dots,t_n)$; we then define $\widehat{T}f=F(Tx_1,\dots,Tx_n)$. To show that $\widehat{T}$ is well-defined, suppose $f=G(\delta_{y_1},\dots,\delta_{y_m})$ where $G(t_1,\dots,t_m)$ is another  lattice-linear expression. Choose a maximal linearly independent subset of $x_1,\dots, x_n,y_1,\dots,y_m$; denote these variables by $z_1,\dots, z_k$. Then write $F(\delta_{x_1},\dots, \delta_{x_n})=\tilde{F}(\delta_{z_1},\dots, \delta_{z_k})$ and $G(\delta_{y_1},\dots,\delta_{y_m})=\tilde{G}(\delta_{z_1},\dots,\delta_{z_k})$, by replacing those elements of 
$\{x_1,\dots,x_n,y_1,\dots,y_m\}\setminus \{z_1,\dots,z_k\}$ by their representation as a linear combination of $z_1,\dots,z_k$. Since $\FVL[E]$ is a sublattice of $H[E]$, the lattice operations are pointwise, hence 
$f(x^*)=\tilde{F}(\delta_{z_1}(x^*),\dots,\delta_{z_k}(x^*))=\tilde{F}(x^*(z_1),\dots,x^*(z_k))$ in $\mathbb{R}$ for each $x^*\in E^*$. Similarly, $f(x^*)=\tilde{G}(x^*(z_1),\dots, x^*(z_k))$. Since  $z_1,\dots, z_k$ are linearly independent, by picking an appropriate $x^*$ we deduce that $\tilde{F}(t_1,\dots,t_k)=\tilde{G}(t_1,\dots,t_k)$ for all $t_1,\dots, t_k\in \mathbb{R}$. Now by lattice-linear function calculus (cf.~\cite[1.d]{LT2}) we have that $\tilde{F}(Tz_1,\dots,Tz_k)=\tilde{G}(Tz_1,\dots,Tz_k)$ in $X$, and by linearity of $T$ it follows that $F(Tx_1,\dots,Tx_n)=G(Ty_1,\dots, Ty_m)$. Hence, $\widehat{T}$ is well-defined.  Moreover, it is clear that $\widehat{T}$ is the unique lattice homomorphism extending $T$ in the sense that $\widehat{T}\delta_x=Tx$ for every $x\in E$. 
\\

Our next objective is
  to show that
  \begin{equation}\label{t:fblb:Eq2}
    \norm{\widehat T f}_{X}\le M^{(p)}(X)\,\norm{T}\,\norm{f}_{\fbp[E]}
  \end{equation}
  for every $f\in \FVL[E]$, as this will ensure that~$\widehat{T}$ extends
  uniquely to a lattice homomorphism defined on all of~$\fbp[E]$,
  and the extension has norm at most~$M^{(p)}(X)\,\norm{T}$. Without loss of generality, $\norm{T}=1$. We split the proof of the inequality~\eqref{t:fblb:Eq2} in two
  parts: First we establish it in the special case where $X=L_p(\mu)$
  for some $\sigma$-finite measure space $(\Omega,\Sigma,\mu)$, and then we show how
  to deduce the general version from the special case.
\\

Thus, suppose first that $X=L_p(\mu)$ for some $\sigma$-finite measure space   $(\Omega,\Sigma,\mu)$; one could even  assume that $\mu$ is a probability measure.   Let $f\in \FVL[E]$. As explained above, $f$ can be written as a lattice-linear
expression $f=F(\delta_{x_1},\dots,\delta_{x_m})$ for some
$x_1,\dots,x_m\in E$ and $\widehat{T}f=F(Tx_1,\dots,Tx_m)$ in
$L_p(\mu)$. Let $\varepsilon>0$ and fix $\delta>0$ (to be determined later). For each $i=1,\dots,m$, find
a simple function $y_i$ such that $\norm{Tx_i-y_i}<\delta$. Let
$\mathcal G$ be the (finite) sub-$\sigma$-algebra generated by
$y_1,\dots,y_m$. Let $P\colon L_p(\mu)\to L_p(\mathcal G,\mu)$ be the
conditional expectation. Consider the lattice homomorphism
$\widehat{PT}\colon\FVL[E]\to L_p(\mathcal G,\mu)$. It follows from
$Py_i=y_i$ that
  \begin{displaymath}
    \norm{PTx_i-Tx_i}\le\norm{PTx_i-Py_i}+\norm{y_i-Tx_i}<2\delta
  \end{displaymath}
  for every $i=1,\dots,m$. Since function
  calculus is norm continuous,
  \begin{displaymath}
    \bignorm{\widehat{T}f-\widehat{PT}f}
    =\Bignorm{F(Tx_1,\dots,Tx_m)-F(PTx_1,\dots,PTx_m)}<\varepsilon
  \end{displaymath}
  provided that $\delta$ is sufficiently small.  It follows that
  $\norm{\widehat{T}f}\le\norm{\widehat{PT}f}+\varepsilon$.  Now note that
  $L_p(\mathcal G,\mu)$ is lattice isometric to $\ell_p^n$ for some
  $n$; let $U\colon L_p(\mathcal G,\mu)\to\ell_p^n$ be a lattice
  isometry. Let $R=UPT$; then $\widehat{R}=U\widehat{PT}$ because both
  are lattice homomorphisms agreeing on the generators.
\\

  Being an operator into $\ell_p^n$, $R$ can be represented as
  \begin{math}
    Rx=\sum_{k=1}^nx_k^*(x)e_k.
  \end{math}
  for some $x_1^*,\dots,x_n^*$ in $E^*$. We then have
  \begin{displaymath}
    \sup_{x\in B_E}\Bigl(\sum_{k=1}^n\abs{x_k^*(x)}^p\Bigr)^{\frac1p}
    =\norm{R}\le 1. 
  \end{displaymath}
  It follows from
  \begin{displaymath}
    \norm{\widehat{PT}f}
    =\norm{U\widehat{PT}f}
    =\norm{\widehat{R}f}
    =\Bigl(\sum_{k=1}^n\abs{f(x_k^*)}^p\Bigr)^{\frac1p}
    \leq\norm{f}_{\fbp[E]}
  \end{displaymath}
  that
  $\norm{\widehat{T}f}\leq\norm{f}_{\fbp[E]}+\varepsilon$. Since
  $\varepsilon$ was arbitrary, we get
  $\norm{\widehat{T}f}\le\norm{f}_{\fbp[E]}$.
  \\

  We are now ready to tackle the general case where~$X$ is an
  arbitrary $p$-convex Banach lattice. Given $f\in \FVL[E]$, choose
  $x^*\in X^*_+$ with $\norm{x^*}=1$ and
  $x^*\bigl(|\widehat{T}f|\bigr)=\|\widehat{T}f\|_X$. Let
  $N_{x^*}$ denote the null ideal generated by~$x^*$, that is,
  $N_{x^*}=\bigl\{x\in X : x^*\bigl(\abs{x}\bigr)=0\bigr\}$, and
  let~$Y$ be the completion of the quotient lattice~$X/N_{x^*}$ with
  respect to the norm
  $\norm{x+N_{x^*}}:=x^*\bigl(\abs{x}\bigr)$. Since this is an
  abstract $L_1$-norm, $Y$ is lattice isometric
  to~$L_1(\Omega,\Sigma,\mu)$ for some measure space
  $(\Omega,\Sigma,\mu)$ (see, e.g., \cite[Theorem~1.b.2]{LT2}). The
  canonical quotient map of $X$ onto $X/N_{x^*}$ induces a lattice
  homomorphism $Q\colon X\rightarrow L_1(\Omega,\Sigma,\mu)$ with
  $\norm{Q}=1$. We may without loss of generality
  assume that $(\Omega,\Sigma,\mu)$ is $\sigma$-finite, passing for
  instance to the band generated by $Q(\widehat{T} f)$.
\\

  Since $Q$ is a lattice homomorphism and $X$ is $p$-convex,  we have
  \begin{displaymath}
    \biggnorm{\Bigl(\sum_{k=1}^n\bigabs{Q(x_k)}^p\Bigr)^{\frac{1}{p}}}_{L_1(\mu)}
    \le\biggnorm{\Bigl(\sum_{k=1}^n\abs{x_k}^p\Bigr)^{\frac{1}{p}}}_X
    \le M^{(p)}(X)\,\Bigl(\sum_{k=1}^n\norm{x_k}_X^p\Bigr)^{\frac{1}{p}}
  \end{displaymath}
  for every $n\in\mathbb N$ and $x_1,\dots, x_n\in X$.  Hence the
  Maurey--Nikishin Factorization Theorem (see, e.g,
  \cite[Theorem~7.1.2.]{alb-kal}, and recall that $p<\infty$) yields a
  positive function $h\in L_1(\Omega,\Sigma,\mu)$ with
  $\int_\Omega h\,d\mu=1$ such that $Q$ is bounded if we regard it as
  an operator into~$L_p(h\,d\mu)$. More precisely, we have a
  factorization diagram
  \begin{displaymath}
    \xymatrix{X\ar[d]_S\ar[rr]^{Q}&&L_1(\mu)\\
     L_p(h\,d\mu)\ar@{^{(}->}[rr]&& L_1(h\,d\mu),\ar_{j_h}[u]}
  \end{displaymath}
  where $Sx=h^{-1}Qx$ satisfies $\norm{S}\le M^{(p)}(X)$ and $j_h(g)=gh$
  is an isometric embedding. Note in particular that $S$ is also a
  lattice homomorphism.
\\

  Let us now consider the composite operator
  $R=S\circ T\colon E\to L_p(h\,d\mu)$. By the first part of the
  proof, we know that there is a unique lattice homomorphism
  $\widehat{R}\colon\fbp[E]\to L_p(h\,d\mu)$ such that
  $\widehat{R}(\delta_x) = Rx$ for every $x\in E$, and
  $\norm{\widehat{R}}=\norm{R}\le M^{(p)}(X)$. Since
  $S\circ \widehat{T}$ and $\widehat{R}$ are lattice homomorphisms
  which agree on the set $\{\delta_x : x\in E\}$, it follows that
  $S\circ \widehat{T}|_{\FVL[E]}=\widehat{R}|_{\FVL[E]}$. Hence we have
  \begin{multline*}
    \norm{\widehat{T} f}_X
    =x^*\bigl(\abs{\widehat{T} f}\bigr)
    =\bignorm{Q(\widehat{T}f)}_{L_1(\mu)}
    \le\bignorm{S(\widehat{T}f)}_{L_p(hd\mu)}\\
    =\norm{\widehat{R} f}_{L_p(h\,d\mu)}
    \le M^{(p)}(X)\,\norm{f}_{\fbp[E]},
  \end{multline*}
  as desired.
\end{proof}

We now consider the case $p=\infty$. By \cite[Lemma 3]{CL}, an  $\infty$-convex Banach lattice $X$ admits an equivalent norm (with equivalence constant equal to $M^{(\infty)}(X)$) under which it becomes an AM-space. In \cite{JLTTT}, it was shown that $\fbl^{(\infty)}[E]$ coincides with the closed sublattice generated by the point evaluations $\{\delta_x:x\in E\}$ in $C(B_{E^*})$. Here $C(B_{E^*})$ denotes the space of continuous functions on the dual ball of $E$, which is equipped with the relative $w^*$-topology. In particular, we have, per \eqref{eq:infinite_case} above,
$$ \|f\|_{\fbl^{(\infty)}[E]}=
	\sup\left\{|f(x^\ast)| : \, x^*\in E^*, \ \|x^*\| \leq 1 \right\}.
$$ 
In the case that $E$ is finite dimensional, therefore, one can identify $\fbl^{(\infty)}[E]$ with either the space $C(S_{E^*})$ of continuous functions on the unit sphere of $E^*$, or the space $C_{ph}(B_{E^*})$ of continuous positively homogeneous functions on $B_{E^*}$. We now give an explicit description of $\fbl^{(\infty)}[E]$ - for general $E$ - by showing that every positively homogeneous weak$^*$ continuous function on $B_{E^*}$  lies in $\fbl^{(\infty)}[E]$:

\begin{prop}\label{p:structure of fbl infty}
 Suppose $E$ is a Banach space. Then $\fbl^{(\infty)}[E]$ coincides with the lattice $C_{ph}(B_{E^*})$ of positively homogeneous weak$^*$ continuous functions on $B_{E^*}$.
\end{prop}

\begin{proof}
We begin by reviewing the aforementioned identification of $\fbl^{(\infty)}[E]$ as a lattice of weak$^*$ continuous positively homogeneous functions on $B_{E^*}$, with the norm being the $\sup$ norm on the unit ball of $E^*$. 
\\

For this, recall that $\FVL[E]$ denotes the sublattice generated by $\{\delta_x\}_{x\in E}$ in $H[E]$. Since all the functions in this sublattice are positively homogeneous, we can, by restriction, identify this space with the sublattice of  $\mathbb{R}^{B_{E^*}}$ generated by $\{\delta_x\}_{x\in E}$. It is clear that, under this identification, $\overline{\FVL[E]}^{\|\cdot\|_\infty}\subseteq C(B_{E^*})$, when $B_{E^*}$ is equipped with the $w^*$-topology. We claim that the closed sublattice of $C(B_{E^*})$ generated by $\{\delta_{x}: x\in E\}$  (or, more specifically, $\{\|x\|_E\delta_{\frac{x}{\|x\|_E}}: x\in E\setminus \{0\}\}\cup \{0\}$, which will be our canonical copy of $E$) satisfies the universal property of $\fbl^{(\infty)}[E]$.
\\

Indeed, let $T: E \to X$ be a bounded linear operator into an AM-space $X$, and assume without loss of generality that $\|T\|=1$. We may view $T$ as a map into $X^{**}$, and, since $X^{**}$ is the dual of an AL-space, we can identify it lattice isometrically with $C(K)$ for some compact Hausdorff space $K$. As in the proof of \Cref{t:fblbp}, we can extend $T$ to $\widehat{T}: \FVL[E] \to X^{**}=C(K)$ in a unique manner. It is clear that the range of $\widehat{T}$ is contained in $X$. 
\\

Fix $t_0\in K$, let $\phi_{t_0}$ be the evaluation functional at $t_0$, and define $x^*=\phi_{t_0}\circ T$. Since $\|T\|=1$, $x^*\in B_{E^*}$.
\\

Let $f\in \FVL[E]$. Then $f=h(\delta_{x_1},\dots, \delta_{x_m})$ for some $x_1,\dots, x_m\in E$ and some lattice-linear function $h$. By definition of the extension, $\widehat{T}f=h(Tx_1,\dots,Tx_m)$, which we can  evaluate point-wise  in $C(K)$ to get 
\begin{align*}
|(\widehat{T}f)(t_0)| & =|h(Tx_1(t_0),\dots, Tx_m(t_0))|=|h(x^*(x_1),\dots, x^*(x_m))| \\ & =|h(\delta_{x_1},\dots, \delta_{x_m})(x^*)| \leq \|h(\delta_{x_1},\dots,\delta_{x_m})\|_\infty.
\end{align*}
Since $t_0$ was arbitrary, $\|\widehat{T}f\|_X=\|\widehat{T}f\|_{C(K)}\leq  \|f\|_\infty,$ so $\|\widehat{T}\|\leq1$. Hence, $T$ extends uniquely to a norm one lattice homomorphism on $\overline{\FVL[E]}^{\|\cdot\|_\infty}$. This verifies the universal property of the free AM-space.
\\

We will now show that $\overline{\FVL[E]}^{\|\cdot\|_\infty}=C_{ph}(B_{E^*})$. Let $\mathfrak{M}$ be the set of all triples $(x^*,y^*,\lambda)$ where $x^*,y^*\in B_{E^*}$ and $0\le\lambda\le 1$ are such that $f(x^*)=\lambda f(y^*)$ for all $f\in\overline{\FVL[E]}^{\|\cdot\|_\infty}$. By \cite[Theorem 3]{Ka}, $\overline{\FVL[E]}^{\|\cdot\|_\infty}=C(B_{E^*};\mathfrak{M})$, where
$C(B_{E^*};\mathfrak{M})$ consists of all functions $f$ in $C(B_{E^*})$ such that $f(x^*)=\lambda f(y^*)$ whenever $(x^*,y^*,\lambda)\in\mathfrak{M}$. If $(x^*,y^*,\lambda)\in\mathfrak{M}$ then $\delta_x(x^*)=\lambda\delta_x(y^*)$ for every $x\in E$, that is, $x^*(x)=\lambda y^*(x)$ and, therefore, $x^*=\lambda y^*$. Since $\overline{\FVL[E]}^{\|\cdot\|_\infty}$ consists of positively homogeneous functions, we have $(x^*,y^*,\lambda)\in\mathfrak{M}$ if and only if $x^*=\lambda y^*$. It follows that $C(B_{E^*};\mathfrak{M})=C_{ph}(B_{E^*})$.
\end{proof}

\begin{rem}\label{free unital AM-space}
Along similar lines, it is easy to check that $C(B_{E^*})$ together with the map $\phi_E(x)=\delta_x$ define the free $C(K)$-space (or free unital AM-space) generated by $E$ (see \cite[Theorem 5.4]{JLTTT}).
\end{rem}

\begin{rem}\label{Extend universal property}
The universal property of $\fbp[E]$ can, in a sense, be extended. Indeed, recall that an operator $T:E\to X$ from a Banach space $E$ to a Banach lattice $X$ is called $p$-convex ($1\leq p\leq \infty$) if there is a constant $M$ such that 
\begin{equation}
    \bigg\|\left(\sum_{k=1}^n|Tx_k|^p\right)^\frac{1}{p}\bigg\|\leq M\left(\sum_{k=1}^n\|x_k\|^p\right)^\frac{1}{p}
\end{equation}
for every choice of vectors $(x_k)_{k=1}^n$ in $E$. In \cite[Theorem 3]{RT} it was shown that an operator $T:E\to X$ is $p$-convex if and only if it strongly factors through a $p$-convex Banach lattice, i.e., there exists a $p$-convex Banach lattice $Z$, a lattice homomorphism $\varphi: Z\to X$ and a linear operator $R:E\to Z$ such that $T=\varphi R.$ Using this fact, we see that an operator $T:E\to X$ is $p$-convex if and only if it strongly factors through $\fbp[E]$. In this case, one can choose the first operator $E\to \fbp[E]$ in the factorization to be the canonical embedding, and then the induced lattice homomorphism $\fbp[E]\to X$ is unique as $\phi_E(E)$ generates $\fbp[E]$ as a Banach lattice.
\end{rem}

\begin{rem}\label{continuity}
The above results allow one to identify elements of $\fbp[E]$ as functions on $E^*$. By positive homogeneity, such functions will be continuous on bounded subsets of $E^*$, in the weak$^*$ topology. Moreover, if $E$ is finite dimensional, then  norm and weak$^*$ convergence coincide, so every weak$^*$ convergent net in $E^*$ is eventually bounded. This implies that elements of $\fbp[E]$ are weak$^*$ continuous on the whole of $E^*$. 
\\

Although the elements in $\FVL[E]$ are weak$^*$ continuous on the whole of $E^*$, by contrast, if $E$ is infinite dimensional, then there exist $\phi \in \fbl[E]$ (hence also $\phi \in \fbp[E]$ for any $p$) which are not weak$^*$ continuous on $E^*$. To this end, find a normalized basic sequence $(x_k) \subseteq E$, and let $\phi = \sum_{k=1}^\infty 2^{-k} |\delta_{x_k}|$. We will construct  an unbounded net $(z_\alpha^*) \subseteq E^*$ so that $\operatorname{w*-lim}_\alpha z_\alpha^* = 0$, yet $\inf_\alpha|\phi(z_\alpha^*)| \geq 1$. Indeed, let ${\mathcal{A}}$ be the set of all finite subsets of $E$, ordered by inclusion. For each $\alpha = (a_1, \ldots, a_n) \in {\mathcal{A}}$, let  $k_\alpha$ be the smallest $k$ for which $x_k \notin \spn[a_1, \ldots, a_n]$. Find $z_\alpha^* \in E^*$ so that $z_\alpha^*(a_j) = 0$ for $1 \leq j \leq n$, and $z_\alpha^*(x_{k_\alpha}) = 2^{k_\alpha}$. The net $(z_\alpha^*)$ has the desired properties. 
\\

An issue similar to the above occurs in \cite[Lemma 5.1 and Example 5.2]{dePW}. Our explicit function space representation allows us to mostly bypass this technicality. On the other hand, we note that there \textit{is} a topology on the whole of $E^*$ that encodes the continuity of elements of $\fbp[E]$. More precisely, we leave it as an exercise to show that if $f$ is in $\fbp[E]$, then it is continuous as a map $f:(E^*,bw^*)\to\mathbb{R}$, where $bw^*$ is the bounded weak$^*$-topology. Here, the bounded weak$^*$-topology is the topology on $E^*$ for which a set $C$ is closed if and only if $C\cap A$ is $w^*$-closed in $A$ whenever $A$ is a norm bounded subset of $E^*$ (by \cite[p.~49]{Day}, it suffices to take $A$ to be the closed unit ball). The $bw^*$-topology is in many ways similar to the $w^*$-topology, but it is also more subtle. See \cite[Chapter 2]{Day} for a study of this topology.

\end{rem}

Next, we mention two other features of the above construction of $\fbp[E]$. The first notes that one cannot restrict to extreme points of the ball to evaluate the $\fbl^{(\infty)}$-norm. The latter notes that the $\fbp$ norms are ``nested" on $\FVL[E]$; the ability to compare these norms will be useful in various circumstances.
\begin{rem}\label{r:use_extreme_points}
 In \eqref{eq:infinite_case}, we can restrict the supremum of $x^*\in B_{E^*}$ to $x^*$ in the unit sphere (this is due to homogeneity). However, we cannot restrict our attention to extreme points of the unit ball.
 For instance, let $E$ be the space $c_0$, equipped with the equivalent norm
 $\|(x_1, x_2, \ldots)\| = \max_n \big\{ |x_{2n-1}| + |x_{2n}| \big\}$.
 In other words, we have $E = c_0(\ell_1^2)$, which implies that $E^* = \ell_1(\ell_\infty^2)$, and hence the  extreme points of the unit ball of $E^*$ are of the form $(0, \ldots, 0, \pm 1, \pm 1 , 0, \ldots)$. Here, the sequence starts with $2n$ zeros, $n\in \mathbb{N}\cup \{0\}$.
 Now denote by $(e_k)$ the canonical basis of $E$ (or $c_0$). Let $f = \big| \delta_{e_1} \big| - \big| \delta_{e_2} \big|$. Clearly, $\|f\|_{\fbl^{(\infty)}[E]} = 1$. However, if $x^*$ is an extreme point of the unit ball of $E^*$, then $f(x^*) = 0$.
\end{rem}

\begin{rem}\label{General comparison}
On $\FVL[E]$ all the $\fbp[E]$-norms can be evaluated. It is easy to see that the $\fbl$-norm is the greatest, and the $\fbl^{(\infty)}$-norm is the smallest (to confirm this, note that, for $p < q$, $\fbl^{(q)}[E]$ is $p$-convex, hence the canonical embedding $\phi^{(q)} : E \to \fbl^{(q)}[E]$ extends to a contractive lattice homomorphism $\widehat{\phi^{(q)}} : \fbp[E] \to \fbl^{(q)}[E]$). This observation will be useful for describing the behaviour of the moduli of sequences in the various free spaces.
Indeed, suppose $(x_k)$ is a sequence in $E$, so that $(|\delta_{x_k}|)$ is equivalent to the unit vector basis of $\ell_1$ when viewed in $\fbl^{(\infty)}[E]$. Then $(|\delta_{x_k}|)$ is equivalent to the unit vector basis of $\ell_1$ no matter which $\fbp[E]$ we view it in. 
\end{rem}

We will see next that $\FVL[E]$ is always order dense in $\fbp[E]$. As was essentially shown in the first part of the proof of Theorem \ref{t:fblbp}, $\FVL[E]$  has the universal property of being the free (Archimedean)  vector lattice generated by the vector space $E$. Namely, every linear map $T:E\to X$ to an (Archimedean) vector lattice $X$  extends uniquely to a lattice homomorphism $\widehat{T}:\FVL[E]\to X$ such that $\widehat{T}\delta_x=Tx$ for all $x\in E$. In other words, $\fbp[E]$ is simply the completion of the free vector lattice $\FVL[E]$ over $E$, under the maximal lattice norm with $p$-convexity constant $1$, which agrees with the norm of $E$ on the span of the generators; see \cite{JLTTT}. 
Note that in this construction we are viewing $E$ as a vector space. If $A$ is a Hamel basis of $E$, then $\FVL[E]$ can be identified with $\FVL(A)$ (the free vector lattice over the set $A$, as constructed in \cite[Section 3]{dePW}).

\begin{thm}\label{t:FVL order dense}
  $\FVL[E]$ is order dense in $\FBLp[E]$.
\end{thm}

Recall that a sublattice $A$ is \emph{order dense} in a vector lattice $Z$ if for any $z \in Z_+ \backslash \{0\}$ there exists $a \in A \backslash \{0\}$ so that $0 \leq a \leq z$. As explained in \cite[Section 5.3]{Locally solid}, a normed lattice is order dense in its norm completion if and only if it is regular in this completion. Moreover, this property admits an intrinsic characterization known as the \emph{pseudo $\sigma$-Lebesgue property}. In \cite{KV}, an order complete normed lattice $X$ was constructed in such a way that its norm completion $\widehat{X}$ fails to be $\sigma$-order complete. It follows from \cite[Theorem 5.32]{Locally solid} that the inclusion  $X\subseteq \widehat{X}$ cannot be order dense. 
\\

The proof of \Cref{t:FVL order dense} requires the following:

\begin{lem}\label{l:vanish outside cone}
 If $F$ is finite-dimensional, then for any open cone $C \subseteq F^*$, any $y_0^* \in C$, and any $\varepsilon > 0$, there exists $g\in \FVL[F]_+$ such that $g(y_0^*)>0$, $g \leq \varepsilon$ on  $B_{F^*}\cap C$, and $g$ vanishes outside $C$.
\end{lem}

\begin{proof}
By renorming, we can, and do, assume that $F=\ell_1^n$. We represent elements of $\FVL[F]$ as piecewise affine functions on $F^* = \ell_\infty^n$; further, it suffices to consider the restrictions of such functions on the unit sphere.
\\

Note that the function  $s(t_1,\dots,t_n)=\abs{t_1}\vee\dots\vee\abs{t_n}$ is in $\FVL[F]$  and its restriction  the unit sphere $S_{\ell_\infty^n}$ is $\one$. Without loss of generality,  $y_0^*\in S_{\ell_\infty^n}$. Restricting functions in $C_{ph}(B_{\ell_\infty^n})$  to $S_{\ell_\infty^n}$, we may identify $C_{ph}(B_{\ell_\infty^n})$ with
  $C(S_{\ell_\infty^n})$; $\FVL[F]$ then becomes a dense sublattice of
  $C(S_{\ell_\infty^n})$ containing $\one$. Note that $C\cap
  S_{\ell_\infty^n}$ is an open subset of $S_{\ell_\infty^n}$ containing
  $y_0^*$. By Urysohn's Lemma, we can find $v\in C(S_{\ell_\infty^n})$ such
  that $0\le v\le \one$, $v(y_0^*)=1$, and $v$ vanishes outside $C\cap
  S_{\ell_\infty^n}$. Since $\FVL[F]$ is dense in $C(S_{\ell_\infty^n})$, there
  exists $u\in\FVL[F]$ such that
  $\norm{v-u}_{C(S_{\ell_\infty^n})}<\frac13$. Put $w=(u-\frac13\one)^+$;
  then $w\in\FVL[F]$, $w$ vanishes outside $C\cap
  S_{\ell_\infty^n}$, and $w(y^*_0)\ge\frac13\ne 0$. Now put $g=\varepsilon w$
  and extend $g$ to $B_{\ell_\infty^n}$ by homogeneity; it is clear that $g$
  satisfies the required conditions.
\end{proof}
\begin{proof}[Proof of \Cref{t:FVL order dense}]
  Since $\FVL[E]\subseteq\FBLp[E]\subseteq\FBLi[E]$, it suffices to
  prove the theorem for $p=\infty$; recall $\FBLi[E]=C_{ph}(B_{E^*})$ (by \Cref{p:structure of fbl infty}). Take a non-zero $f\in\FBLi[E]_+$; our goal is to show the existence of $h \in \FVL[E] \backslash \{0\}$ with $0 \leq h \leq f$.
  \\
  
  Since $f\ne 0$, there exists $0\ne x_0^*\in B_{E^*}$ with $f(x_0^*)>0$. Hence there exists  $\varepsilon>0$ and a weak* open neighbourhood $U$ of $x_0^*$ in  $E^*$ such that $f$ is greater than $\varepsilon$ on $U\cap B_{E^*}$.
  Furthermore, we may assume that there exist $x_1,\dots,x_n\in E$ such that $x^*\in U$ if and only if   $\bigabs{x^*(x_i)-x_0^*(x_i)}<1$ for all $i=1,\dots,n$. Since  $x_0^*\ne 0$, by adding an extra point, if necessary, we may also
  assume that $x_0^*(x_i)\ne 0$ for some $i$.
\\

  Let $F$ be the subspace of $E$ spanned by $x_1,\dots,x_n$; let
  $\iota\colon F\hookrightarrow E$ be the inclusion map. Put
  $y_0^*=\iota^*x_0^*\in B_{F^*}$ and $V=\iota^*(U)$. Note that
  $y^*\in V$ if and only if $\bigabs{y^*(x_i)-y_0^*(x_i)}<1$ for all
  $i=1,\dots,n$. Hence, $V$ is a (weak*) open neighbourhood of $y^*_0$
  in $F^*$. We write $\cone(U)=\bigcup_{\lambda>0}\lambda U$. Clearly,
  $\cone(V)=\iota^*\bigl(\cone(U)\bigr)$. 
\\

  By \Cref{l:vanish outside cone}, there exists $g\in\FVL[F]_+$ which vanishes on the complement of $\cone(V)$, and satisfies $0 \leq g \leq \varepsilon$ on $B_{F^*} \cap \cone(V)$, as well as $g(y_0^*) > 0$.
  This $g$ may be written as a lattice-linear expression
  of $\delta_{x_1},\dots,\delta_{x_n}$. Let $h$ be the same
  lattice-linear expression of $\delta_{x_1},\dots,\delta_{x_n}$ in
  $\FVL[E]$. Then $h(x^*)=g(\iota^*x^*)$ for every $x^*\in E^*$. It
  follows that $h\ge 0$ and $h(x^*_0)=g(y^*_0)>0$, hence $h\ne 0$.
\\

  We claim that $h\le f$. Fix $x^*\in B_{E^*}$; we need to show that
  $h(x^*)\le f(x^*)$. If $x^*\in U$ then
  $\iota^*x^*\in V\subseteq\cone(V)$, hence
  $h(x^*)=g(\iota^*x^*)\leq\varepsilon<f(x^*)$. Since both $h$ and $f$
  are positively homogeneous, it follows that $h(x^*)\le f(x^*)$
  whenever $x^*\in\cone(U)$. On the other hand, if $x^*\notin\cone(U)$
  then $\iota^*x^*\notin\cone(V)$ and, therefore,
  $h(x^*)=g(\iota^*x^*)=0$. In either case, $h(x^*)\le f(x^*)$.
\end{proof}
\Cref{t:FVL order dense} allows us to recover an important result from \cite{APR}:
\begin{cor}
 For $1\leq p\leq \infty$, every disjoint collection of elements of $\fbp[E]$ is at most countable.
\end{cor}
\begin{proof}
 Let $(x_\alpha)$ be a collection of pairwise disjoint elements of $\fbp[E]$. Without loss of generality, all the elements $x_\alpha$ are  positive and non-zero. Use \Cref{t:FVL order dense} to find $0<y_\alpha\leq x_\alpha$, with $y_\alpha\in \FVL[E]$. Now, identify $\FVL[E]$ with $\FVL(A)$, where $A$ is a Hamel basis of $E$, and $\FVL(A)$ is the free vector lattice over the \emph{set} $A$. Finally, use the classical fact that pairwise disjoint collections in $\FVL(A)$ are at most countable. See, for example, \cite[Theorem 2.5]{Baker}.
\end{proof}

We conclude this section by noting some elementary facts about $\fbp[E]$. As mentioned, most of the literature on $\fbl[E]$ discussed in \Cref{Section history} generalizes to $\fbp[E]$ with relative ease, so we only collect here three of the most basic facts. Indeed, the following can be proved exactly as in \cite[Section 6]{dePW} by looking at the non-vanishing sets of weak$^*$ continuous functions on $B_{E^*}$: 

\begin{prop}\label{p:basicprop}
Let $E$ be a Banach space.
\begin{enumerate}
\item For every $x\in E$, $x\neq0$, $|\delta_x|$ is a weak order unit in $\fbp[E]$.
\item If $E$ has dimension strictly greater than one, then the only projection bands in $\fbp[E]$ are $\{0\}$ and $\fbp[E]$.
\item If $E$ has dimension strictly greater than one, then $\fbp[E]$ is not $\sigma$-order complete and contains no atoms.
\end{enumerate}
\end{prop}

\section{Properties of $\overline{T}$: Injectivity, surjectivity, regularity, and the subspace problem}\label{Section Tbar}
Recall that the universal property of $\fbp$ yields, in particular, that every bounded linear operator $T:F\rightarrow E$ between Banach spaces  extends uniquely to a lattice homomorphism $\overline{T}:\fbp[F]\rightarrow \fbp[E]$ making the following diagram commute:

 \begin{displaymath}
    \xymatrix{\fbp[F]\ar[rr]^{\overline{T}}&&\fbp[E]\\
     F\ar[rr]^{T}\ar_{\phi_F}[u]&& E\ar_{\phi_E}[u]}
  \end{displaymath}
Here $\phi_E$ and $\phi_F$ denote the canonical isometric embeddings, and $\|\overline{T}\|=\|T\|$ (simply consider the map $\phi_ET:F\rightarrow \fbp[E]$ and set $\overline{T}=\widehat{\phi_ET}$).
It is easy to check that given operators $S:F\rightarrow G$ and $T:E\rightarrow F$, we have $\overline{S\circ T}=\overline{S}\circ \overline{T}$. In particular, if $T$ is an isomorphism between Banach spaces $E$ and $F$, then $\overline{T}$ is a lattice isomorphism between $\fbp[E]$ and $\fbp[F]$. 
\\

The goal of this section is to relate properties of $T$ with properties of $\overline{T}$. More specifically, we will discover exactly when $\overline{T}$ is injective, surjective, a quotient map, etc. We will also study when $\overline{T}$ is an embedding, and when it is order continuous.
\\

The following observation will be useful for our purposes:

\begin{lem}\label{l:overlineT composition}
 Given $T:F\rightarrow E$, the extension $\overline{T}:\fbp[F]\rightarrow \fbp[E]$ is given, for $f\in \fbp[F]$, by
 $$T(f)=f\circ T^*.$$
\end{lem}

\begin{proof}
Consider the composition operator induced by $T^*$: $C_{T^*}f(x^*)=f(T^*x^*)$ for $f\in \fbp[F]$ and $x^*\in E^*$. It is straightforward to check that $C_{T^*}:\fbp[F]\rightarrow H_p[E]$ is a well-defined lattice homomorphism. Moreover, $C_{T^*} \delta_x=\delta_{Tx}$ for $x\in F$, which implies that the range of $C_{T^*}$ is actually contained in $\fbp[E]$. Because of uniqueness of extension, we must have $\overline T=C_{T^*}$.
\end{proof}

\subsection{Characterizations of injectivity, surjectivity and density of the range}\label{s:quotients}

We begin with some simple observations on injectivity and surjectivity of the extended operator $\overline{T}$. We thank A. Avil\'es for sharing with us an argument leading to the characterization of surjectivity.

\begin{prop}\label{Injec-Surjec}
Let $T:F\rightarrow E$ be a bounded linear operator and let $\overline{T}:\fbp[F]\rightarrow \fbp[E]$ be its unique extension to a lattice homomorphism given above. Then
\begin{enumerate}
    \item\label{IS-I} $T$ is injective if and only if $\overline{T}$ is injective.
    \item\label{IS-D} $T$ has dense range if and only if $\overline{T}$ has dense range.   
    \item\label{IS-S} $T$ is onto if and only if $\overline{T}$ is onto.
\end{enumerate}
\end{prop}

\begin{proof}
\eqref{IS-I} Suppose $T$ is injective. Let $T^*:E^*\rightarrow F^*$ be the adjoint operator; it is easy to check that its range $T^*(E^*)$ is weak$^*$ dense in $F^*$ (cf.~\cite[Theorem 3.18]{FHHMPZ}). Also, by \Cref{l:overlineT composition}, for $f\in \fbp[F]$, we can write $\overline{T}f=f\circ T^*$. Suppose $\overline{T}f=0$ for some $f\in\fbp[F]\setminus\{0\}$. Since $\overline{T}$ is a lattice homomorphism, we can suppose without loss of generality that $f\in\fbp[F]_+$. By \Cref{t:FVL order dense}, we can find $g\in\FVL[F]$ such that $0<g\leq f$, which by positivity also satisfies $\overline{T}g=0$. It follows that $g(T^*y^*)=0$ for every $y^*\in E^*$, and by weak$^*$ continuity of $g$ we must have $g=0$. Thus, $\overline{T}$ is injective. The converse is clear. 
\\

\eqref{IS-D} If $T$ has dense range then for every $y\in E$ and
$\varepsilon>0$ there exists $x\in F$ with
$\varepsilon>\norm{Tx-y}=\norm{\overline{T}\delta_x-\delta_y}$, hence
$\overline{\ran\overline{T}}$ contains
$\delta_y$. Since $\overline{T}$ is a lattice homomorphism,
$\overline{\ran\overline{T}}$ is a closed sublattice; it follows
that $\overline{\ran\overline{T}}=\fbp[E]$. Suppose now that
$\ran T$ is not dense. There exists $0\ne y^*\in E^*$ which
vanishes on it. Then the map $\widehat{y^*}\in\fbp[E]^*$
given by $\widehat{y^*}(g)=g(y^*)$ vanishes on $\overline{T}\delta_x$
for every $x\in F$. Since $\widehat{y^*}$ is a lattice homomorphism,
it vanishes on $\ran\overline{T}$; hence the range of $\overline{T}$ is not
dense.
\\

\eqref{IS-S} Suppose $T$ is onto. Let $Z=\fbp[F]/\ker \overline{T}$ and let $Q\colon\fbp[F]\to Z$ be the canonical quotient map. Since $\overline{T}$ is a lattice homomorphism, $\ker\overline{T}$ is an ideal, hence $Q$ is a lattice homomorphism, and, therefore, $Z$ is a $p$-convex Banach lattice. There exists an injective operator $S\colon Z\to\fbp[E]$ such that $\overline{T}=SQ$. Since $\overline{T}$ and $Q$ are lattice homomorphisms, so is $S$. Indeed, fix $z\in Z$. By the surjectivity of $Q$, we can find $x\in \fbp[F]$ such that $Qx=z$. Then $$S|z|=S|Qx|=SQ|x|=\overline{T}|x|=|\overline{T}x|=|SQx|=|Sz|.$$
Since $\ker T\subseteq\ker Q\phi_F$, there exists an operator $R\colon E\to Z$ such that $Q\phi_F=RT$. Let $\widehat R\colon\fbp[E]\to Z$ be the canonical extension of $R$. Let $y\in E$. Pick $x\in F$ such that $y=Tx$. Then  
\begin{displaymath}
    S\widehat{R}\phi_E y=SRy=SRTx=SQ\phi_Fx=\overline{T}\phi_Fx=\phi_Ey.
  \end{displaymath}
It follows that $S\widehat R$ is the identity on the range of $\phi_E$ and, therefore, on the sublattice generated by it. Since this sublattice is dense in $\fbp[E]$, $S\widehat R$ is the identity on $\fbp[E]$. It follows that $S$ is surjective and, therefore, so is $\overline{T}=SQ$.
\\

Conversely, suppose now that $\overline{T}$ is onto. Let $Q:F\rightarrow F/\ker T$ denote the canonical quotient map and let $S:F/\ker T\rightarrow E$ be the injective operator induced by $T$: $S(x+\ker T)=Tx$, for $x\in F$. Thus, we have $T=SQ$. Let us consider the corresponding lattice homomorphisms $\overline Q:\fbp[F]\rightarrow \fbp[F/\ker T]$, $\overline S:\fbp[F/\ker T]\rightarrow \fbp[E]$, which in particular satisfy $\overline{T}=\overline{S}\,\overline{Q}$. Note that since $\overline{T}$ is onto, so is $\overline{S}$. Moreover, as $S$ is injective, by part \eqref{IS-I} it follows that $\overline{S}$ is also injective. Hence, $\overline{S}$ is an isomorphism. In particular, it follows that $S$ is bounded below and has closed range. But, by part \eqref{IS-D}, it follows that $S$ has dense range, thus $S$ is onto. By construction of $S$, it follows that $T$ must be onto as well.
\end{proof}

In \Cref{s:subspace_problem}, we will study when $\overline{T}$ is an embedding. Unlike with injectivity, surjectivity, density of the range and being a quotient map, it is not true that $T$ is an embedding if and only if $\overline{T}$ is an embedding. In fact, $\overline{T}$ is an embedding if and only if $T$ is an embedding, and one can uniformly factor maps into $\ell_p^n$ through $T$. This will be made precise - and quantitative - in \Cref{prop:extendell1}.
\\

We also note the following form of ``restricted projectivity" for $\fbp[E]$. The proof is essentially as in \cite[Proposition 4.9]{JLTTT}, but on an operator-by-operator basis:
\begin{prop}\label{Restricted proj}
Let $E$ be a Banach space, $X$ a $p$-convex Banach lattice,  $J$ a closed ideal of $X$, $Q:X\to X/J$ the quotient map, and $T:E\to X/J$ an operator. Then
\begin{enumerate}
\item If $T:E\to X/J$ admits a lift to $\widetilde{T}: E\to X$ then $\widehat{\widetilde{T}}:\fbp[E]\to X$ is a lattice homomorphic lift of $\widehat{T}: \fbp[E]\to X/J$;
\item If $\widehat{T}: \fbp[E]\to X/J$ admits a linear lift $S:=\widetilde{\widehat{T}}:\fbp[E]\to X$ then $S\circ \phi_E:E\to X$ is a lifting of $T$;
\item If the identity $I:X/J\to X/J$ admits a linear lift $\widetilde{I}:X/J\to X$ then for any lattice homomorphism $S:\fbp[E]\to X/J$, the canonical extension of $\widetilde{I}\circ S\circ \phi_E: E\to X$ to $\fbp[E]$ is a lattice homomorphic lifting of $S$.
\end{enumerate}
\end{prop}
\begin{proof}
Argue by diagram chasing.
\end{proof}

\subsection{$\fbp[F]$ is a regular sublattice of $\fbp[E]$}\label{Regularity problem}
In this section, we prove that if $F$ is a subspace of $E$, the canonical inclusion $\overline{\iota}:\fbp[F]\to\fbp[E]$ is order continuous (that is, if $(f_\alpha)$ is a decreasing net in $\fbp[F]$, whose infimum is $0$, then the same is true for $(\overline{\iota} f_\alpha)$). This happens regardless of whether $\overline{\iota}$ is an embedding in its own right.
\\

To set notation, throughout this subsection we equip $B_{E^*}$ with its relative weak$^*$ topology. 
We let $F$ be a closed subspace of $E$ and $\iota\colon F\hookrightarrow E$ the canonical embedding. Then $\iota^*\colon E^*\to F^*$ is the restriction map: $\iota^*x^*=x^*_{|F}$, and $\bar\iota\colon\FBLp[F]\to\FBLp[E]$ is injective. 
\\

Recall from the construction of  $\fbp[E]$ that we defined $\FVL[E]$ to be the (non-closed) sublattice generated by $\{\delta_x\}_{x\in E}$ in $H[E]$. $\fbp[E]$ was then constructed as the closure of $\FVL[E]$ in $H_p[E]$. 

\begin{thm}\label{regularity}
  Let $F$ be a closed subspace of $E$; let
  $\iota\colon F\hookrightarrow E$ be the inclusion map. Then
  $\bar\iota\colon\FBLp[F]\to\FBLp[E]$ is order continuous. That is,
  $\FBLp[F]$ is a regular (not necessarily closed) sublattice of $\FBLp[E]$.
\end{thm}

\begin{proof}
  First we consider the case when $F$ is complemented in $E$. In this case, the argument is an adaptation of \cite[Proposition 5.9]{dePW}, which proves  that $\fbl[\ell_1(B)]$ is regularly embedded in $\fbl[\ell_1(A)]$ whenever $B\subseteq A$. 
Let $P:E\rightarrow F$ denote a projection (so that $P\iota=id_F$).
\\

To this end, let $f_\alpha\downarrow 0$ in $\fbp[F]$, and suppose $g\in \fbp[E]$ satisfies $0<g\leq \overline{\iota} f_\alpha$ for every $\alpha$. Let $x_0^*\in E^*$ be such that $g(x_0^*)>0$. It follows from $0<g(x_0^*)\le \overline{\iota}f_\alpha(x_0^*)=f_\alpha(\iota^*x_0^*)$
that $\iota^*x_0^*\neq0$. We may assume without loss of generality that $\norm{\iota^*x_0^*}=1$. Pick any $y_0\in F$ with $(\iota^*x_0^*)(y_0)=1$. Put $z_0^*=x_0^*-P^*\iota^*x_0^*$. Note that $\iota^* P^* = id_{F^*}$, hence $\iota^* z_0^* = 0$.
\\

Consider the operator $T:E\rightarrow \fbp[F]$ given by
$Tx=\delta_{Px}+z_0^*(x)\abs{\delta_{y_0}}$. Being a rank one
perturbation of $\phi_F\circ P$, $T$ is bounded, and, therefore,
extends to a lattice homomorphism $\widehat
T\colon\fbp[E]\to\fbp[F]$. Put $h=\widehat{T}g$.
\\

For every $y^*\in F^*$ and $x\in E$, we have
\begin{displaymath}
  (\widehat{y^*}\circ\widehat{T})(\delta_x)
  =(\widehat{T}\delta_x)(y^*)
  =y^*(Px)+z_0^*(x)\abs{y^*(y_0)}
  =\delta_x\bigl(\varphi(y^*)\bigr)
  =\widehat{\varphi(y^*)}(\delta_x),
\end{displaymath}
where $\varphi(y^*)=P^*y^*+\abs{y^*(y_0)}z^*_0$. The lattice homomorphisms $\widehat{y^*}\circ\widehat{T}$ and $\widehat{\varphi(y^*)}$ agree on every $\delta_x$, hence they are equal. It follows that
\begin{displaymath}
  h(y^*)=(\widehat{T}g)(y^*)
  =(\widehat{y^*}\circ\widehat{T})(g)
  =\widehat{\varphi(y^*)}(g)
  =g\bigl(\varphi(y^*)\bigr)
\end{displaymath}
for every $y^*\in F^*$. This yields
\begin{displaymath}
  h(y^*)
  \leq\overline\iota f_\alpha\bigl(\varphi(y^*)\bigr)
  =f_\alpha\bigl(\iota^*\varphi(y^*)\bigr)
  =f_\alpha\bigl(\iota^*P^*y^*+\abs{y^*(y_0)}\iota^*z_0^*\bigr)
  =f_\alpha(y^*)
\end{displaymath}
for every $y^*\in F^*$ and every $\alpha$ because $\iota^*P^*=id_{F^*}$.
Therefore, $0\leq h\leq f_\alpha$ for every $\alpha$, which yields $h=0$. It follows from $\varphi(\iota^*x_0^*)=x_0^*$ that
\begin{math}
  g(x_0^*)=h(\iota^*x_0^*)=0.
\end{math}
This contradiction proves the statement in the case when $F$ is a complemented subspace of $E$.
\\

We now proceed to the general case. For this, suppose that $f_\alpha\downarrow 0$ in $\FBLp[F]$ and there exists
  $g\in\FBLp[E]$ such that $0<g\le\bar\iota f_\alpha$ for every
  $\alpha$. Since $\FVL[E]$ is order dense in $\FBLp[E]$ by \Cref{t:FVL order dense}, we may
  assume without loss of generality that $g\in\FVL[E]$. Then $g$ is a lattice-linear
  combination of $\delta_{y_1},\dots,\delta_{y_n}$ for some
  $y_1,\dots,y_n$ in $E$.
  \\

  Let $G$ be the closed subspace of $E$ spanned by $F$ and
  $y_1,\dots,y_n$. Let $j$ and $k$ be the inclusion maps:
  $F\xrightarrow{j}G\xrightarrow{k}E$. Clearly, $\iota=k\circ j$. Let
  $h\in\FVL[G]$ be defined by the same lattice-linear combination of
  $\delta_{y_1},\dots,\delta_{y_n}$ as $g$ but viewed as an element of
  $\FVL[G]$. For every $z^*\in B_{G^*}$, we can extend it to some
  $y^*\in B_{E^*}$; it follows that
  \begin{displaymath}
    h(z^*)=g(y^*)\le\bar\iota f_\alpha(y^*)=f_\alpha(\iota^*y^*)
    =f_\alpha(j^*z^*)=\bar jf_\alpha(z^*)
  \end{displaymath}
  for every $\alpha$, where $\bar j\colon\FBLp[F]\to\FBLp[G]$ is the
  canonical inclusion induced by $j$. It follows that
  $0\le h\le \bar jf_\alpha$ in $\FBLp[G]$ for every $\alpha$. Since
  $F$ is complemented in $G$, the special case yields $h=0$. For every
  $x^*\in B_{E^*}$, we have $g(x^*)=h(k^*x^*)=0$, so $g=0$.
\end{proof}

\begin{rem}
In the above theorem we assumed that $F$ is a subspace of $E$, so that $\iota$ is an embedding. However, since $T:F\to E$ is injective if and only if $\overline{T}:\fbp[F]\to\fbp[E]$ is injective, to identify $\fbp[F]$ as a  vector sublattice of $\fbp[E]$ only requires the injectivity of $\iota$. Since regularity is a pure vector lattice property, one may think that injectivity of $\iota$ would be enough to ensure regularity of  the inclusion $\overline{\iota}:\fbp[F]\to \fbp[E]$. However, the above proof fails under this weaker assumption, and it remains an open problem to characterize those $T:F\to E$ such that $\overline{T}:\fbp[F]\to\fbp[E]$ is order continuous.
\end{rem}

\subsection{The embedding problem, and its connection to extensions of operators}\label{s:subspace_problem}

A direct consequence of Proposition \ref{Injec-Surjec} is that if $E$ is a Banach space quotient of $F$, then $\fbp[E]$ is a Banach lattice quotient of $\fbp[F]$. This partly motivates the question of whether the dual version of this fact also holds. To properly formulate this, note first that Proposition \ref{Injec-Surjec} also yields that if $F$ is a (closed non-zero) subspace of $E$, then the canonical embedding $\iota:F\hookrightarrow E$ induces an injective lattice homomorphism $\overline{\iota}:\fbp[F]\rightarrow \fbp[E]$ of norm 1. In this section, we consider the \textit{embedding problem}:  Suppose $F$ is a  subspace of $E$. Does the canonical embedding $\iota:F\hookrightarrow E$ induce a lattice embedding $\overline{\iota}:\fbp[F]\rightarrow \fbp[E]$? 
\\

For context, recall that an inclusion of metric spaces always induces an isometric embedding of the associated Lipschitz free spaces, cf.~\cite[Lemma 2.3]{GK}. As we will see, however, the situation for free Banach lattices is more subtle. Our main result is \Cref{prop:extendell1} which shows that $\overline{\iota}$ being a lattice embedding is equivalent to every operator $T : F \to L_p(\mu)$ having an extension to $E$. In particular, this reduces a problem about Banach lattices to a purely Banach space one. In the next section, this criterion will be combined with various Banach space techniques to provide several examples where $\overline\iota$ is an embedding, as well as several examples where it is not.
\\

To reiterate our goal, we aim to explore under which conditions the (injective) map $\overline{\iota}$ defines an isomorphic embedding, so that we can consider $\fbp[F]$ as a closed sublattice of $\fbp[E]$ in a natural way. Equivalently, we ask whether $\overline{\iota}$ is bounded below -- that is, whether there exists $C > 0$ so that any $f \in \fbp[F]$ satisfies $\| \overline{\iota} f \| \geq \|f\|/C$. Since $\overline{\iota}$ is norm one, this is equivalent to asking $\overline{\iota}$ to be a lattice $C$-isomorphic embedding.

\begin{rem}\label{More general c0}
As alluded to above, we only consider isometric embeddings $\iota:F\hookrightarrow E$ in this subsection. Nevertheless, the results add to \Cref{Injec-Surjec} a characterization of when an operator $T:F\to E$ induces a lattice isomorphic embedding $\overline{T}:\fbp[F]\to \fbp[E]$. Indeed, the restriction to isometric embeddings  presents little loss in generality, as given an operator $T:F\to E$, one can factor it as $T=j_2j_1$, where $j_1:F\to (T(F),\|\cdot\|_E)$, and $j_2:(T(F),\|\cdot\|_E)\to E$ is an isometric inclusion. If $\overline{T}$ is an embedding, then it is easy to see that $T$ is as well. On the other hand, if $T$ is an embedding then $\overline{T}=\overline{j_2}\circ\overline{j_1}$ is an embedding if and only if $\overline{j_2}$ is. Thus, $\overline{T}$ is an embedding if and only if both $T$ and $\overline{j_2}$ are. It therefore suffices to understand the map $\overline{j_2}$, which is the extension of the isometric mapping $j_2$.
\end{rem}

We now reduce the problem of whether an embedding $\iota : F \hookrightarrow E$ induces a lattice embedding $\overline{\iota} : \fbp[F] \to \fbp[E]$ to a certain Banach space question involving extensions of operators:

\begin{thm}\label{prop:extendell1}\label{c:extensions}\label{p:POE1_vs_extension}
  Let $\iota : F\hookrightarrow  E$ be an isometric embedding and $C>0$. The following are equivalent:
  \begin{enumerate}
      \item\label{POP-bar} $\overline{\iota} : \fbp[F] \to \fbp[E]$ is a lattice $C$-isomorphic embedding;
      \item\label{POP-Lp} For every $\sigma$-finite measure $\mu$, any $T\colon F \to L_p(\mu)$ extends to $\widetilde{T}\colon E \to L_p(\mu)$, with $\norm{\widetilde{T}} \leq C \norm{T}$;
      \item\label{POP-lpe} For every $n\in\mathbb N$ and $\varepsilon>0$, any $T\colon F \to\ell_p^n$ extends to $\widetilde{T}\colon E \to\ell_p^n$, with $\norm{\widetilde{T}} \leq C(1+\varepsilon)\norm{T}$.
  \end{enumerate}
\end{thm}

\begin{proof}
\eqref{POP-Lp}$\Rightarrow$\eqref{POP-lpe} is trivial.

\eqref{POP-lpe}$\Rightarrow$\eqref{POP-bar}: Fix $\varepsilon>0$. Since $\overline{\iota}$ is a lattice homomorphism with $\|\overline{\iota}\|=\|\iota\|=1$, we immediately get that
$$
\|\overline{\iota} f\|_{\fbp[E]}\leq \|f\|_{\fbp[F]},
$$
for all $f\in \fbp[F]$.
\\

Now, take $f$ in $\FVL[F]$. 
Given $x_1^*,\dots,x_n^*\in F^*$, we define $T:F\rightarrow \ell_p^n$ by $T(x)=(x_k^*(x))_{k=1}^n$. Recall that
$$
\|T\|=\sup_{x\in B_F}\left(\sum_{k=1}^n |x_k^\ast(x)|^p\right)^{\frac{1}{p}}.
$$
By hypothesis, there is an extension $\widetilde T:E\rightarrow \ell_p^n$ with $\|\widetilde T\|\leq C(1+\varepsilon)\|T\|$. Let $y_1^*,\ldots, y_n^*\in E^*$ be such that $\widetilde T(x)=(y_k^*(x))_{k=1}^n$ for each $x\in E$, so that $\iota^*y_k^*=x_k^*$.
It follows that $f(x_k^*)=f(\iota^*y_k^*)=\overline{\iota}f(y_k^*)$ for $k=1,\ldots, n$. Therefore, we have
$$
\left(\sum_{k=1}^n |f(x_k^\ast)|^p\right)^\frac{1}{p}=\left(\sum_{k=1}^n |\overline{\iota}f(y_k^\ast)|^p\right)^\frac{1}{p} \leq \|\overline{\iota}f\|_{\fbp[E]} \sup_{x\in B_E} \left(\sum_{k=1}^n |y_k^\ast(x)|^p\right)^\frac{1}{p}$$

$$\leq C(1+\varepsilon)\|\overline{\iota}f\|_{\fbp[E]} \sup_{x\in B_F} \left(\sum_{k=1}^n |x_k^\ast(x)|^p\right)^\frac{1}{p}.
$$

Taking supremum over $x_1^*,\dots,x_n^*\in F^*$, it follows that 
$$
\|f\|_{\fbp[F]}\leq C(1+\varepsilon) \|\overline{\iota}f\|_{\fbp[E]}.
$$
By density, this inequality holds for all $f\in \fbp[F]$. Now let $\varepsilon$ tend to zero.
\\

\eqref{POP-bar}$\Rightarrow$\eqref{POP-Lp}: The case of $p = \infty$ follows from the injectivity of $L_\infty$-spaces, so we restrict ourselves to $1 \leq p < \infty$. Let $T:F\rightarrow L_p(\mu)$. By the properties of a free Banach lattice $T$ extends to  a lattice homomorphism $\widehat{T} : \fbp[F] \to L_p(\mu)$, with $\|\widehat{T}\| = \|T\|$.
 Let $S$ be the inverse of $\overline{\iota}$, taking $\overline{\iota}(\fbp[F])$ back to $\fbp[F]$; clearly $S$ is a lattice isomorphism, with $\|S\| \leq C$.
 \\
 
 By \cite[Theorem 1.c.4]{LT2}, there exists a band projection from $L_p(\mu)^{**}$ onto $L_p(\mu)$.
 By \cite[Theorem 4]{Pisier_REG94} (for $p=1$, see also \cite{Lotz}), $\widehat{T} S$ extends to a regular operator $U : \fbp[E] \to L_p(\mu)$, with $\|U\| \leq \|\widehat{T} S\| \leq C \|T\|$.
 Now let $\widetilde{T} = U \phi_E$ (here, as before, $\phi_E : E \to \fbp[E]$ is the canonical embedding).
 Clearly $\|\widetilde{T}\| \leq \|U\|  \leq C \|T\|$.
 Moreover,
 $$
 \widetilde{T} \iota = U \phi_E \iota = U \overline{\iota} \phi_F = \widehat{T} S \overline{\iota} \phi_F ,
 $$
 and, since $S$ is the one-sided inverse of $\overline{\iota}$,
 $$
 \widetilde{T} \iota = \widehat{T} \phi_F = T.
 $$
In other words, $\widetilde{T}$ extends $T$.
\end{proof}

\Cref{prop:extendell1} motivates the following definition:

\begin{defn}\label{d:POEp}
Fix $p\in [1,\infty]$. We say that a pair $(F,E)$ with $F$ a subspace of $E$ has the \textit{POE-$p$} with constant $C$, or $C$-POE-$p$, if for every $n\in \mathbb{N}$, every operator $T:F\to \ell_p^n$ extends to $\widetilde{T}: E\to \ell_p^n$ with $\|\widetilde{T}\|\leq C\|T\|$.
 Here POE-$p$ stands for ``Property of operator extension into $L_p$".
A Banach space $F$ is said to have \textit{POE-$p$} with constant $C$ (or $C$-POE-$p$) if, for any space $E$ containing $F$, $(F,E)$ has the POE-$p$ with constant $C$. 
If $(F,E)$ (or $F$) has the POE-$p$ for some $C$, then we shall simply say that $(F,E)$ (resp.~$F$) has the POE-$p$.
\end{defn}

In these terms, \Cref{prop:extendell1} yields the following equivalent characterizations of the POE-$p$:

\begin{prop}\label{p:POEp_C_versus_C+}
For $C \geq 1$, $p \in [1,\infty]$, and a subspace $F$ of a Banach space $E$, the following are equivalent:
\begin{enumerate}
 \item $(F,E)$ has the $C$-POE-$p$;
 \item For any $\sigma$-finite measure $\mu$, any $T : F \to L_p(\mu)$ has an extension $\widetilde{T} : E \to L_p(\mu)$ with $\| \widetilde{T} \| \leq C \|T\|$;
 \item For any $n \in \Nat$ and $\varepsilon>0$, any $T : F \to \ell_p^n$ has an extension $\widetilde{T} : E \to \ell_p^n$ with $\| \widetilde{T} \| \leq C(1+\varepsilon) \|T\|$.
\end{enumerate}
\end{prop}

We note that the $1$-injectivity of $\ell_\infty^n$ implies:

\begin{prop}\label{AM works fine}
Any Banach space $F$ has the POE-$\infty$, with constant $1$. Consequently, 
if $\iota :F\hookrightarrow  E$ is an isometric embedding, then the map $\overline{\iota}:\mathrm{FBL}^{(\infty)}[F]\rightarrow \mathrm{FBL}^{(\infty)}[E]$ is a lattice isometric embedding. 
\end{prop}

The case of $1 \leq p < \infty$ is more interesting, and upcoming (sub)sections will discuss criteria for determining whether a pair $(F,E)$, or a space $F$, has the POE-$p$.
\begin{rem}\label{Converse complemented}
As was noted in \cite[Corollary 2.8]{ART}, if $F$ is a complemented subspace of $E$, then $\overline{\iota} : \fbl[F]\to \fbl[E]$ is a lattice isomorphic embedding. This, of course, also follows immediately from \Cref{prop:extendell1}. However, \cite[Corollary 2.8]{ART} (and slight modifications of its proof) show a lot more: If $\iota : F \hookrightarrow E$ is an embedding, and $P$ is a projection from $E$ onto $\iota(F)$, then $\overline{P}$ defines a lattice homomorphic projection from $\fbp[E]$ onto $\overline{\iota}(\fbp[F])$. A partial converse also holds.

\begin{prop}\label{complemented sublattice}
Suppose $F$ is isomorphic to a complemented subspace of a $p$-convex Banach lattice, and $\iota : F \hookrightarrow E$ is an embedding such that the induced map $\overline{\iota}: \fbp[F]\to \fbp[E]$ is an embedding, and there is a projection $P$ (which is not assumed to be a lattice projection) from $\fbp[E]$ onto $\overline{\iota}( \fbp[F])$. Then $F$ is complemented in $E$.
\end{prop}

\begin{proof}
 As $F$ is isomorphic to a complemented subspace of a $p$-convex Banach lattice, there is a projection $Q:\fbp[F]\to \phi_F(F)$ by \cite[Proposition 4.2]{JLTTT}. Diagram chasing shows that $V:=\phi_F^{-1}\circ Q\circ \overline{\iota}^{-1}\circ P\circ \phi_E: E\to F$ satisfies $I_F=V\circ \iota$.  In other words, $F$ is complemented in $E$.
\end{proof}

As an example, \Cref{prop:extendell1} (see \Cref{cor:c0} for additional details) shows that the inclusion $\iota : c_0\hookrightarrow \ell_\infty$ induces a lattice embedding of $\fbp[c_0]$ into $\fbp[\ell_\infty]$. However, $c_0$ is not complemented in $\ell_\infty$, hence $\fbp[c_0]$ cannot be complemented in $\fbp[\ell_\infty]$.
 \end{rem}

\subsection{Examples of lattice structures on a subspace spanned by Rademacher functions}\label{s:examples_of_lattice_structure}
In the previous subsection, we reduced the embedding problem for free Banach lattices to a pure Banach space problem involving extensions of operators into $L_p(\mu)$. This perspective on the embedding problem will be further expanded on in \Cref{s:POEp}. However, before that, we examine the embedding problem from a lattice point of view.
More specifically, here we consider an embedding $\iota : F \hookrightarrow E$, and explicitly calculate the norms of certain elements of $\overline{\iota}(\fbp[F]) \subseteq \fbp[E]$.
By discovering that, for certain $f \in \fbp[F]$, $\|f\|_{\fbp[F]}$ may be very different from $\|\overline{\iota} f\|_{\fbp[E]}$, we conclude that $\overline{\iota}$ is not bounded below.
\\ 

We denote by $\rad_q$ ($1 \leq q \leq \infty$) the span of independent Rademacher random variables in $L_q$; $R_q$ shall stand for the corresponding embedding.
Khintchine's inequality shows that for finite $q$, $\rad_q$ is isomorphic to $\ell_2$, and it is easy to verify that $\rad_\infty$ can be identified with $\ell_1$.
\\

It is well known that, for $1 < q < \infty$, $\rad_q$ is complemented in $L_q$, hence the pair $(\rad_q, L_q)$ has the POE-$p$ for any $p$. Below we examine the edge cases $q = 1, \infty$. In the next section, we will revisit this question from an extension of operators point of view and prove in \Cref{p:hilbert_in_L1_p1} that $(\rad_1,L_1)$
fails the POE-$p$, for any $1\leq p < \infty$, and in \Cref{p:relations_between_POEp}  that $(\rad_
\infty,L_\infty)$ has the POE-$p$ if and only if $2 \leq p \leq \infty$. However, this section presents a direct proof, in order to illustrate the  structure of free Banach lattices:

\begin{example}\label{prop:nonisomorphic}\label{ex:Rad1}
$\overline{R}_1 : \fbp[\rad_1] \to \fbp[L_1]$ is not a lattice isomorphic embedding for any $p\in [1,\infty)$.
\end{example}

\begin{proof}
Let $(e_k)$ denote the unit vector basis of $\ell_2$ and $(r_k)$ the sequence of Rademacher functions. Define 
$$R:\ell_2\rightarrow L_1[0,1] : \sum_k a_k e_k \mapsto \sum_k a_k r_k.
$$
As $\rad_1$ is canonically isomorphic to a Hilbert space, it suffices to show that $\overline{R}$ is not bounded below.
\\

Assume first that $p\in [1,2]$. Then for each $m\in\mathbb N$ we have that
$$
\Big\|\bigvee_{k=1}^m \delta_{e_k}\Big\|_{\fbp[\ell_2]}\geq \sqrt{m}.
$$
Indeed, let $I:\ell_2\rightarrow \ell_2$ be the identity map, and $\widehat I:\fbp[\ell_2]\rightarrow \ell_2$ the lattice homomorphism extending $I$, which exists because of the assumption that $p\leq 2$. It follows that
$$
\sqrt{m}=\Big\|\bigvee_{k=1}^m e_k\Big\|_{\ell_2}=\Big\|\bigvee_{k=1}^m \widehat I\delta_{e_k}\Big\|_{\ell_2}\leq\Big\|\bigvee_{k=1}^m \delta_{e_k}\Big\|_{\fbp[\ell_2]}.
$$
If instead $p\in (2,\infty)$, consider the inclusion $i: \ell_2 \hookrightarrow \ell_p$. Extend this to a contractive lattice homomorphism $\widehat{i}:\fbp[\ell_2]\to \ell_p$ to get 
$$m^\frac{1}{p}=\Big\|\bigvee_{k=1}^m e_k\Big\|_{\ell_p}\leq \Big\|\bigvee_{k=1}^m\delta_{e_k}\Big\|_{\fbp[\ell_2]}.$$

On the other hand, for every $m\in\mathbb N$ we have that
$$
\Big\|\bigvee_{k=1}^m \delta_{r_k}\Big\|_{\fbl[L_1]}=1.
$$
Indeed, note first that if $K$ is a compact Hausdorff space and $(f_j)_{j=1}^n\subseteq C(K)$, then as a consequence of the fact that the extreme points of the dual unit ball $B_{C(K)^*}$ are point measures of the form $\pm\delta_{k}$ for $k\in K$, we have that
$$
\sup_{x^*\in B_{C(K)^*}}\sum_{j=1}^n|x^*(f_j)|=\sup_{k\in K} \sum_{j=1}^n |f_j(k)|=\Big\|\sum_{j=1}^n |f_j|\Big\|_\infty.
$$
Combining this observation with \eqref{eq:ARTbidual} yields that
$$
\Big\|\bigvee_{k=1}^m \delta_{r_k}\Big\|_{\fbl[L_1]}=\sup\Big\{\sum_{j=1}^n\Big|\bigvee_{k=1}^m \int r_k f_j\Big|: n\in \mathbb{N}, \ f_1,\dots,f_n \in L_\infty, \,\Big\|\sum_{j=1}^n |f_j|\Big\|_\infty\leq1\Big\}.
$$
Since we have that
$$
\sum_{j=1}^n\Big|\bigvee_{k=1}^m \int r_k f_j\Big|\leq \sum_{j=1}^n \int |f_j|=\int\sum_{j=1}^n |f_j|\leq\Big\|\sum_{j=1}^n |f_j|\Big\|_\infty,
$$
it follows that
$$
\Big\|\bigvee_{k=1}^m \delta_{r_k}\Big\|_{\fbl[L_1]}\leq1.
$$
For the converse inequality,
$$
\Big\|\bigvee_{k=1}^m \delta_{r_k}\Big\|_{\fbl[L_1]}\geq \Big|\bigvee_{k=1}^m \delta_{r_k}(r_1)\Big|=1.
$$
Now, since $\bigvee_{k=1}^m\delta_{r_k}$ lies in $\FVL[L_1]$, 
all $\|\cdot\|_{\fbp[L_1]}$-norms can be evaluated on this element, and we have $$1=\Big\|\bigvee_{k=1}^m \delta_{r_k}\Big\|_{\fbl[L_1]}\geq \Big\|\bigvee_{k=1}^m \delta_{r_k}\Big\|_{\fbp[L_1]}=\Big\|\overline{R}\bigvee_{k=1}^m \delta_{e_k}\Big\|_{\fbp[L_1]}.$$
Thus, $\overline{R}$ is not bounded below.
\end{proof}

\begin{example}\label{ex:ell_1 different than c0}
The lattice homomorphism $\overline R_\infty : \fbp[\rad_\infty] \rightarrow \fbp[L_\infty[0,1]]$ is not an embedding for $p\in [1,2)$.
\end{example}

Here we provide a direct proof of this fact. Later, in \Cref{ell_1 different than c0}, we will use a different technique to show that $\overline R_\infty : \fbp[\rad_\infty] \rightarrow \fbp[L_\infty[0,1]]$ is an embedding if and only if $p\in [2,\infty]$.

 \begin{proof}
Consider the Rademacher isometry $R:=R_\infty : \ell_1 \to L_\infty : e_k \mapsto r_k$; here, $(e_k)$ form the canonical basis in $\ell_1$, while $(r_k)$ are independent Rademacher random variables.
As mentioned above, we shall show that $\overline{R}$ is not bounded below if $p\in [1,2)$.
\\

To this end, first note that
$\Big\|\bigvee_{k=1}^m \delta_{e_k}\Big\|_{\fbp[\ell_1]} = m^{1/p}$. 
Indeed, the upper estimate follows from the $p$-convexity of $\fbp[\ell_1]$:
$$
\Big\|\bigvee_{k=1}^m \delta_{e_k}\Big\|_{\fbp[\ell_1]}^p \leq
\Big\| \Big( \sum_{k=1}^m \bigabs{ \delta_{e_k} }^p \Big)^{1/p} \Big\|_{\fbp[\ell_1]}^p \leq 
\sum_{k=1}^m \bignorm{ \delta_{e_k} }^p = m. 
$$
For the opposite inequality, modify the arguments in \Cref{prop:nonisomorphic} (using the formal identity from $\ell_1$ to $\ell_p$).
\\

On the other hand, we shall show that $\Big\|\bigvee_{k=1}^m \delta_{r_k}\Big\|_{\fbp[L_\infty]} \sim \sqrt{m}$.
 By \eqref{eq:ART},
$\Big\|\bigvee_{k=1}^m \delta_{r_k}\Big\|_{\fbp[L_\infty]}$ is the supremum of $\left(\sum_{j=1}^n\Big| \bigvee_{k=1}^m \mu_j(r_k) \Big|^p\right)^{1/p}$, with the supremum taken over all $\mu_1, \ldots, \mu_n \in L_\infty^*$ with 
$$
\sup_{x\in B_{L_\infty}}\sum_{j=1}^n|\mu_j(x)|^p\leq~1.
$$
Now consider the contractive operator $u : L_\infty^* \to \ell_\infty^m : \mu \mapsto \big( \mu(r_k) \big)_k$.
Note that $\Big| \bigvee_{k=1}^m \mu(r_k) \Big| \leq \|u \mu\|$, hence
$\Big\|\bigvee_{k=1}^m \delta_{r_k}\Big\|_{\fbp[L_\infty]}$ is no greater than
$$
\sup \Big\{ \Big(\sum_{j=1}^n\|u \mu_j\|^p\Big)^{1/p} : n\in \mathbb{N}, \  \mu_1, \ldots, \mu_n \in L_\infty^*,  \sup_{x\in B_{L_\infty}}\sum_{j=1}^n|\mu_j(x)|^p\leq~1 \Big\} . 
$$
Arguing as in \eqref{eq:ARTbidual}, this last quantity equals $\pi_p(u)$, the $p$-summing norm of the operator $u$.
\\

By \cite[Theorem 2.8]{DJT}, $\pi_p(u)\leq \pi_1(u)$, so it suffices to bound $\pi_1(u)$. Denote by $i$ the formal identity from $\ell_\infty^m$ to $\ell_2^m$.
Note that $\|i^{-1}\| = 1$, and $\|i\| = \sqrt{m}$, hence $\|i \circ u\| \leq \sqrt{m}$.
By \cite[Theorem 3.1]{DJT}, $\pi_1(i \circ u) \leq K_G \|i \circ u\| \leq K_G \sqrt{m}$, where $K_G$ is Grothendieck's constant.
Thus, $\pi_1(u) = \pi_1 \big( i^{-1} \circ (i \circ u) \big) \leq \|i^{-1}\| \pi_1(i \circ u) \leq K_G \sqrt{m}$ by \cite[p.~37]{DJT}. 
Consequently, $\Big\|\bigvee_{k=1}^m \delta_{r_k}\Big\|_{\fbp[L_\infty]} \lesssim \sqrt{m}$.
\\

For the opposite inequality, recall that $\| \cdot \|_{\fbp[L_\infty]} \geq \| \cdot \|_{\fbl^{(2)}[L_\infty]}$. Therefore, it suffices to show that $\Big\|\bigvee_{k=1}^m \delta_{r_k}\Big\|_{\fbl^{(2)}[L_\infty]} \geq \sqrt{m}$.
Let $\mu_j = r_j \in L_1 \subseteq L_\infty^*$.
By Khintchine's inequality, the map $\ell_2 \to L_1 : e_j \mapsto \mu_j$ is contractive, where $(e_j)$ now stands for the canonical basis in $\ell_2$.
Therefore, $\sup_{x\in B_{L_\infty}}\sum_{j=1}^n|\mu_j(x)|^2\leq~1$.
However,  $\Big(\sum_{j=1}^m\Big| \bigvee_{k=1}^m \mu_j(r_k) \Big|^2\Big)^{1/2} = \sqrt{m}$.
\end{proof}

\section{Extensions of operators into $L_p$}\label{s:POEp}
In the previous section, we were able to reduce the embedding problem for $\fbp$ to the POE-$p$, or in other words, the study of extension properties of operators into $L_p$. We now embark on a detailed study of the POE-$p$. To begin, we provide several reformulations in terms of operator ideals and $\mathcal{L}_p$-spaces. We then study how the POE-$p$ behaves under duality, which  provides us with several examples of embeddings satisfying the POE-$p$; in particular, $(F,F^{**})$, $(F,F_{\mathcal{U}})$ and $(F,E)$ whenever $F$ is locally complemented or an ideal in $E$. We then show several stability properties of the POE-$p$, compare POE-$p$ with POE-$q$, and provide numerous (non-)examples.

\subsection{General facts about the POE-$p$}\label{ss:general_facts_POEp} We begin with several basic facts about the POE-$p$; namely, its relation to operator ideals, extensions into $\mathcal{L}_p$-spaces, and previous literature. Firstly, the universality of $\ell_\infty(I)$ spaces allows us to reformulate the definition of the POE-$p$ in terms of the ideal of $\ell_\infty$-factorable operators $(\Gamma_\infty, \gamma_\infty)$ (see \cite{DJT} for information about this and other operator ideals).

\begin{prop}\label{p:factorable}
Suppose $F$ is a Banach space, and $1 \leq p \leq \infty$. The following statements are equivalent:
\begin{enumerate}
 \item $F$ has the $C$-POE-$p$;
 \item For any operator $T : F \to \ell_p$, and any isometric embedding $F \hookrightarrow \ell_\infty(I)$, there exists an extension $\widetilde{T} : \ell_\infty(I) \to \ell_p$, with $\norm{\widetilde{T}}\le C\norm{T}$;
 \item For any operator $T : F \to \ell_p$, we have $\gamma_\infty(T) \leq C \|T\|$;
 \item For any compact operator $T : F \to \ell_p$, we have $\gamma_\infty(T) \leq C \|T\|$.
\end{enumerate}
In statements $(2)$, $(3)$, and $(4)$, $\ell_p$ can be replaced by any infinite dimensional $L_p$-space.
\end{prop}

\begin{proof}
 The implications $(1) \Rightarrow (2) \Rightarrow (3) \Rightarrow (4)$ are trivial.
 \\
  
  $(4) \Rightarrow (1)$: We suppose $F \hookrightarrow E$ and show that, for any $\varepsilon > 0$, any $T : F \to \ell_p^n$ has an extension $\widetilde{T} : E \to \ell_p^n$ with $\|\widetilde{T}\| \leq (C + \varepsilon) \|T\|$, so the conclusion will follow by  \Cref{p:POEp_C_versus_C+}.
  Find a factorization $T = u v$, with $v : F \to \ell_\infty(I)$ and $u : \ell_\infty(I) \to \ell_p^n$, with $\|u\| \|v\| \leq (C + \varepsilon) \|T\|$. 
  Extend $v$ to $\widetilde{v} : E \to \ell_\infty(I)$, with $\|v\| = \| \widetilde{v} \|$.
  Then $\widetilde{T} = u \widetilde{v}$ has the desired properties.
\end{proof}

In a similar fashion, we establish:

\begin{prop}\label{p:iteration}
 If $(F,F_1)$ has the $C_1$-POE-$p$, and $(F_1,F_2)$ has the $C_2$-POE-$p$, then $(F,F_2)$ has the $C_1 C_2$-POE-$p$. In particular, if $(F,E)$ has the $C_1$-POE-$p$, and $E$ is $C_2$-injective, then $F$ has the $C_1 C_2$-POE-$p$.
\end{prop}

 For $1\leq p\leq 2$, the POE-$p$ can be characterized in terms of $2$-summing operators as follows. First, recall that an operator  $T:F\rightarrow E$ between Banach spaces is \textit{$p$-summing} for $1\leq p<\infty$ ($T \in \Pi_p(F,E)$) if there is a constant $C$ such that for any finite collection $(x_k)_{k=1}^n\subseteq F$ we have
 $$ 
\left ( \sum _ {k =1 } ^ { n }  \left \| {Tx _ {k} } \right \|  ^ {p} \right ) ^ {1/p }
\leq  
C  \sup  \left \{ {\left ( \sum _ {k = 1 } ^ { n }   |  x^*(x_k)  |  ^ {p} \right ) ^ {1/p } } : {x^* \in F  ^*  , \left \| x^*\right \| \leq 1 } \right \}.
$$
The smallest possible $C$ appearing in this inequality is denoted $\pi_p(T)$ (cf.~\cite{DJT}).
\begin{prop}\label{p:POE1_criterion}
Let $F$ be a Banach space and $1\leq p \leq 2$. The following are equivalent:
\begin{enumerate}
\item $F$ has the POE-$p$;
\item There is a constant $C>0$ such that, for all $n$ and all $T:F\to \ell_p^n$, we have $\pi_2(T)\leq C\|T\|$;
\item $B(F,\ell_p)=\Pi_2(F,\ell_p)$;
\item $B(F,L_p)=\Pi_2(F,L_p)$.
\end{enumerate}
\end{prop}
\begin{proof}
(1)$\Rightarrow$(2): Consider an embedding $\iota : F\hookrightarrow C(K)$. By assumption any $T:F\to \ell_p^n$ has an extension $\widetilde{T}:C(K)\to \ell_p^n$ with $\|\widetilde{T}\|\leq C\|T\|$. By \cite[Theorem 3.5]{DJT}, $$\pi_2(T)\leq \pi_2(\widetilde{T})\leq K_G\|\widetilde{T}\|\leq K_G C\|T\| ,$$
where $K_G$ is the Grothendieck constant. \\

(2)$\Rightarrow$(1): By the $\Pi_2$-extension theorem \cite[Theorem 4.15]{DJT}, if $E,F,Y$  are Banach spaces with $F$ a subspace of $E$ then any $2$-summing operator $T:F\to Y$ has an extension $\widetilde{T}: E\to Y$ with $\pi_2(T)=\pi_2(\widetilde{T})$. 
\\

The equivalence of (2), (3), and (4) is a classical localization argument.
\end{proof}

\begin{rem}\label{r:POE_p_versus_Hilbert-Schmidt}
Using \Cref{p:POE1_criterion} we see that POE-$1$ and POE-$2$ are actually well-studied Banach space properties. Indeed, by \cite{Jarchow}, $F$ is POE-$2$ if and only if $F$ is a Hilbert-Schmidt space.
Moreover, by \cite[Proposition 6.2]{Pisier_FACT}, $F$ is POE-$1$ if and only if $F^*$ is a G.T.~space. 
\end{rem}

\begin{rem}\label{r:POEp in reverse}
 The POE-$p$ was also studied (under a different name) in \cite{CN03}. Indeed, \cite{CN03} investigates the spaces $E$ so that $(F,E)$ has the POE-$p$ for every $F \subseteq E$. 
 For instance, it is shown that, if $E$ is a Banach lattice with such property for some $p \in (2,\infty)$, then $E$ is weak Hilbert, and satisfies a lower $2$-estimate. If $E$ is a K\"othe function space on $(0,1)$, then it must be lattice isomorphic to $L_2(0,1)$. If $E$ is a space with a subsymmetric basis, then \cite[Proposition 12.4]{Pisier_Volume} can be used to show that this basis is equivalent to the $\ell_2$ basis. On the other hand, Maurey's Extension Theorem \cite[12.22]{DJT} yields:
\end{rem}

\begin{prop}\label{p:Maurey_ext}
Suppose $E$ has type $2$, 
$F$ is a subspace of $E$,
and $1 \leq p \leq 2$. 
 Then $(F,E)$ has the POE-$p$.
\end{prop}

The definition of the POE-$p$ involves extending operators into $L_p$-spaces. It turns out that we can extend operators into the wider class of ${\mathcal{L}}_p$-spaces.

 \begin{prop}\label{p:extend_into_script_Lp}
 Suppose $1 \leq p < \infty$, $F$ is a Banach space, and $X$ is an infinite dimensional ${\mathcal{L}}_p$-space. 
Consider the following statements:
  \begin{enumerate}
  \item $(F,E)$ has the POE-$p$;
  \item Any compact operator $T : F \to X$ has a bounded extension $\widetilde{T} : E \to X$;
  \item Any compact operator $T : F \to X$ has a compact extension $\widetilde{T} : E \to X$;
  \item Any bounded operator $T : F \to X$ has a bounded extension $\widetilde{T} : E \to X$.
 \end{enumerate}
 Then $(1) \Leftrightarrow (2) \Leftrightarrow (3)$. Moreover, if $X$ is complemented in $X^{**}$, then $(4)$ is equivalent to the three preceding statements.
 \end{prop}
 
 By \cite{lind_ros}, a ${\mathcal{L}}_p$-space $X$ is complemented in $X^{**}$ if and only if it embeds into $L_p$ as a complemented subspace.
 It is well known (see e.g.~\cite{JL}) that, for $1 < p < \infty$, any ${\mathcal{L}}_p$-space is reflexive, hence, in \Cref{p:extend_into_script_Lp}, $(1) \Leftrightarrow (2) \Leftrightarrow (3) \Leftrightarrow (4)$.
 For $p = 1$, \cite[Section 5]{lind_ros} provides an example of a ${\mathcal{L}}_1$-space which does not embed complementably into $L_1$.
 
 \begin{proof}
    Note that, if (4) holds, then there exists $C > 0$ so that any $T : F \to X$ has an extension $\widetilde{T} : E \to X$ with $\|\widetilde{T}\| \leq C \|T\|$.
  Indeed, (4) implies that the map $\Phi : B(E, X) \to B(F, X) : S \mapsto S|_F$ is surjective; thus, there exists $C > 0$ so that for any $T$ there exists $\widetilde{T}$ with $\|\widetilde{T}\| \leq C \|T\|$, and $\Phi(\widetilde{T}) = T$.
  We can reach similar conclusions in cases (2) and (3).
  \\
  
  By \cite[Theorem 1]{lind_ros}, $X$ contains a complemented copy of $\ell_p$. Consequently, either (2), (3), or (4) implies (1).
  Clearly $(3) \Rightarrow (2)$. The implications $(1) \Rightarrow (3)$ and (modulo complementability of $X$ in $X^{**}$) $(1) \Rightarrow (4)$ remain to be established.
  The proofs use \cite[Theorem 3]{lind_ros}: there exists a constant $\rho$ so that for every finite dimensional $Z \subseteq X$, we can find $Y$ so that $Z \subseteq Y \subseteq X$, $Y$ is $\rho$-isomorphic to $\ell_p^{{\textrm{dim} } Y}$, and $\rho$-complemented in $X$.
  Use this to find an increasing net of finite dimensional spaces $(Y_\alpha)_{\alpha \in {\mathcal{I}}}$, so that $X = \overline{ \cup_\alpha Y_\alpha}$ and, for any $\alpha$, denoting ${\textrm{dim }} Y_\alpha = n_\alpha$ we have $d(Y_\alpha,\ell_p^{n_\alpha}) \leq \rho$, and there exists a projection $P_\alpha : X \to Y_\alpha$ so that $\|P_\alpha\| \leq \rho$.
  \\
  
  For the remainder of the proof, we assume that $(F,E)$ has the POE-$p$ with constant $C$, $X$ is a ${\mathcal{L}}_p$-space, and $T : F \to X$ is a contraction. 
  \\
  
  $(1) \Rightarrow (3)$:
  Fix $\varepsilon > 0$. Assuming $T$ is compact, we shall find a compact extension $\widetilde{T} : E \to X$, with $\|\widetilde{T}\| \leq C \rho (1+3\varepsilon)$. We proceed recursively. Let $T_0 = 0$, $T_0' = T$. 
  Our first goal is to find $\alpha_1 \prec \alpha_2 \prec \ldots$, and operators $T_k, T_k' \in B(F,X)$ so that, for any $k$, we have
  $$
  T_{k-1}' = T_k + T_k' , \, \, T_k = P_{\alpha_k} T_{k-1}' , \, \, \|T_k'\| \leq \varepsilon 2^{-k} .
  $$
  Note that for any $k$ we have $T = T_1 + \cdots + T_k + T_k'$. By the triangle inequality, $\|T_k\| \leq \|T_{k-1}'\| + \|T_k'\| < \varepsilon 2^{2-k}$ for $k \geq 2$, and likewise, $\|T_1\| < 1 + \varepsilon$. Moreover, $\sum_{k=1}^\infty T_k$ converges to $T$.
  \\
  
  Suppose we have already found $\alpha_1 \prec \ldots \prec \alpha_n$, and the operators $T_0, \ldots, T_n, T_n'$ with the desired properties (if $n=0$, then we have taken $T_0 = 0$ and $T_0' = T$, and we ignore the condition about $\alpha_1, \ldots, \alpha_n$).
  Find $\alpha_{n+1} \succ \alpha_n$ so that
  $$
  \sup_{f \in F, \|f\| \leq 1} \inf_{y \in Y_{\alpha_{n+1}}} \|T_n' f - y\| < \varepsilon (\rho+1)^{-1} 2^{-n-1} .
  $$
  Let $T_{n+1} = P_{\alpha_{n+1}} T'_n$ and $T_{n+1}' = T_n' - T_{n+1}$, and note that $\|T_{n+1}'\| \leq \varepsilon 2^{-n-1}$. Indeed, fix $f \in B_F$, and find $y \in Y_{\alpha_{n+1}}$ so that $\|T_n' f - y\| < \varepsilon 2^{-n-1} (\rho+1)^{-1}$. Then
  $$
  \|T_{n+1}' f\| = \big\| (I - P_{\alpha_{n+1}}) (T_n' f - y) \big\| \leq \big( 1 + \|P_{\alpha_{n+1}}\| \big) \|T_{n}' f - y\| \leq \varepsilon 2^{-n-1} .
  $$
  So, $T_{n+1}$ and $T_{n+1}'$ have the desired properties. 
  \\
  
  Recall that $T_k (F) \subseteq Y_{\alpha_k}$, and the latter space is $\rho$-isomorphic to $\ell_p^{n_{\alpha_k}}$. Consequently, $T_k$ has an extension $\widetilde{T_k} : E \to Y_{\alpha_k}$, with $\|\widetilde{T_k}\| \leq C \rho \|T_k\|$. Recalling the estimates on the norms $\|T_k\|$ obtained above, we conclude that $\|\widetilde{T_1}\| \leq C \rho (1 + \varepsilon)$, and $\|\widetilde{T_k}\| \leq C \rho \varepsilon 2^{2-k}$ for $k \geq 2$. Then $\widetilde{T} = \sum_{k=1}^\infty \widetilde{T_k}$ extends $T$, and has norm not exceeding $C \rho(1 + 3\varepsilon)$. This proves (3).
  \\
  
  Denote by $Q$ a projection from $X^{**}$ onto $X$; we will show that $(1) \Rightarrow (4)$. For each $\alpha \in {\mathcal{I}}$, we find $\widetilde{T_\alpha} : E \to Y_\alpha \subseteq X$, which extends $P_\alpha T$, and has norm at most $C \rho^2$.
  It is well-known (see e.g.~\cite[p.~120]{DJT}) that $B(E,X^{**}) = \big( E \widehat{\otimes} X^* \big)^*$, where $\widehat{\otimes}$ denotes the projective tensor product. Hence, the net $(\widetilde{T_\alpha})$ has a subnet $\big(\widetilde{T_\beta})_{\beta \in {\mathcal{J}}}$ which converges to some $S : E \to X^{**}$ in the $\sigma \big( B(E,X^{**}) , E \widehat{\otimes} X^* \big)$ topology.
  Testing convergence on elementary tensor products $e \otimes x^*$ ($e \in E, x^* \in X^*$), we conclude that $\widetilde{T_\beta} \to S$ point-weak$^*$, hence $\|S\| \leq \limsup_\beta \|\widetilde{T_\beta}\| \leq C \rho^2$.
  \\
  
  Let $\widetilde{T} = Q S$. Then $\|\widetilde{T}\| \leq \|Q\| C \rho^2$. We claim that $\widetilde{T}$ extends $T$ -- that is, $\widetilde{T} f = Tf$ for any $f \in F$. We shall show that, in fact, $Sf = Tf$. Indeed, fix $\varepsilon > 0$, and find $\beta_0 \in {\mathcal{J}}$ so large that, for any $\beta \succ \beta_0$, we have
  $$
  \inf_{y \in Y_\beta} \big\| Tf - y \big\| < \varepsilon .
  $$
  As in the proof of $(1) \Rightarrow (3)$, show that $\|Tf - P_\beta T f\| = \|Tf - \widetilde{T_\beta} f\| < \varepsilon (\rho+1)$ when $\beta \succ \beta_0$. As $S f$ is the weak$^*$ limit of $(\widetilde{T_\beta} f)$, then  $\|Tf - S f\| \leq \varepsilon (\rho+1)$. To conclude that $Sf = Tf$, recall that $\varepsilon$ can be arbitrarily small.
 \end{proof}

\subsection{The POE-$p$: duality, local complementation and ultrapowers}\label{ss:POEp_duality}\label{ss:POEp}
We now explore the interplay between the POE-$p$ and duality. To fix the terminology below,  recall that an operator $Q : X \to Y$ between Banach spaces is said to be $\lambda$-surjective if for any $y \in Y$ with $\|y\| < 1$ there exists $x \in X$ with $Qx = y$, $\|x\| < \lambda$. A standard functional analysis result states that $Q$ is $\lambda$-surjective if and only if $Q^*$ is bounded below by $1/\lambda$ if and only if $Q^{**}$ is $\lambda$-surjective.
\\

For any Banach space $Z$, we can identify $B(Z,\ell_p^n)$ with $(Z^*)^n$, as a vector space. More precisely, any $T \in B(Z,\ell_p^n)$ can be written as $T = \sum_{k=1}^n z_k^* \otimes e_k$, where $e_1, \ldots, e_n$ form the canonical basis in $\ell_p^n$, and $z_1^*, \ldots, z_n^* \in Z^*$.
Then $T$, or $T^* \in B(\ell_{p'}^n, Z^*)$ ($1/p + 1/p' = 1$), can be identified with $(z_1^*, \ldots, z_n^*) \in (Z^*)^{n}$.
By Local Reflexivity (as laid out in \cite{dean}), $B(Z,\ell_p^n)^{**} = B(Z^{**},\ell_p^n)$.

\begin{prop}\label{p:POEp_passes_to_bidual}
 A pair $(F,E)$ has the $C$-POE-$p$ if and only if the same is true for $(F^{**},E^{**})$.
\end{prop}

\begin{proof}
Define, for any $n \in \Nat$, the operator $\Phi^{(n)}_{F,E} : B(E,\ell_p^n) \to B(F,\ell_p^n) : S \mapsto S|_F$. 
Fix $n$; we henceforth omit the upper index $(n)$. By the preceding paragraphs, $\Phi_{F,E}^{**}$ can be identified with $\Phi_{F^{**},E^{**}}$, hence one is $\lambda$-surjective if and only if the other is. In light of \Cref{p:POEp_C_versus_C+}, $(F,E)$ has the $C$-POE-$p$ if and only if $\Phi_{F,E}$ is $C$-surjective.
By the above, $(F,E)$ has the $C$-POE-$p$ if and only if $(F^{**},E^{**})$ does.
\end{proof}

For a Banach space $E$ and $n\in\mathbb N$, let $E^{(n)}$ denote its $n$-th dual.
The preceding result yields:

\begin{prop}\label{prop:duals}
Let $F$ be a closed subspace of $E$ and suppose that $F^{(2k)}$ is $C$-complemented in $E^{(2k)}$ for some $k\in\mathbb N$. Then $(F,E)$ has the $C$-POE-$p$.
\end{prop}
Similarly, since any operator $T : F \to \ell_p^n$ has an extension $T^{(2k)} : F^{(2k)} \to \ell_p^n$ ($k \in \Nat$) with the same norm, we see that:

\begin{prop}\label{cor:subspace}
For any Banach space $F$, $k \in \Nat$, and $p \in [1,\infty]$, $(F, F^{(2k)})$ has the $1$-POE-$p$.
\end{prop}

Using $\ell_\infty(I)$ spaces we can convert \Cref{p:POEp_passes_to_bidual} into a statement about Banach spaces with the POE-$p$ (rather than pairs of Banach spaces with the POE-$p$):

\begin{prop}\label{p:POEp_dual}
$F$ has the $C$-POE-$p$ if and only if $F^{**}$ does.
\end{prop}

\begin{proof}
Suppose $F$ has the $C$-POE-$p$. Embed $F$ isometrically into $\ell_\infty(I)$. By \Cref{p:factorable}, $F$ has the $C$-POE-$p$ if and only if $(F,\ell_\infty(I))$ does. By \Cref{p:POEp_passes_to_bidual}, this, in turn, is equivalent to $(F^{**},\ell_\infty(I)^{**})$  having the $C$-POE-$p$.
By \cite[Theorem 4.1]{Zippin_survey}, $\ell_\infty(I)^{**}$ is $1$-injective.
Thus, by \Cref{p:iteration}, if $(F^{**},\ell_\infty(I)^{**})$ has the $C$-POE-$p$, then so does $F^{**}$. 
\\

Conversely, suppose $F^{**}$ has the $C$-POE-$p$. Embed $F^{**}$ into $\ell_\infty(I)$, for a suitable index $I$. We have to show that $(F,\ell_\infty(I))$ has the $C$-POE-$p$. By \Cref{cor:subspace} $(F, F^{**})$ has the $1$-POE-$p$, hence \Cref{p:iteration} yields the desired result.
\end{proof}

We now give three examples where the above results apply. First, recall that by \cite[Theorem 4.2]{Zippin_survey},  $F^{**}$ is $C$-injective if and only if whenever $F$ is a closed subspace of $E$ and $Y$ is finite dimensional, every operator $T : F \to Y$ extends to $\widetilde{T} : E \to Y$ with $\|\widetilde{T}\| \leq  C\|T\|$. Hence, $\mathcal{L}_\infty$-spaces have POE-$p$ for all $p\in [1,\infty]$. To be more precise, by combining \cite[Theorem 3.3]{lind_memoir} with \cite[Theorem 4.2]{Zippin_survey}, we observe that, if $F$ is a ${\mathcal{L}}_{\infty,\mu}$-space for all $\mu>\lambda$, then $F^{**}$ is $\lambda$-injective.
This implies:

\begin{cor}\label{cor:c0}
If $F$ is a ${\mathcal{L}}_{\infty,\mu}$-space for all $\mu>\lambda$, then it has the $\lambda$-POE-$p$ for every $p\in[1,\infty]$. In particular, $c_0$ and $C(K)$ spaces have the $1$-POE-$p$.
\end{cor}

In a similar fashion, we apply \Cref{p:POEp_passes_to_bidual} and \Cref{prop:duals} to two well-studied classes of subspaces. Recall, following \cite{Kalton84}, that a closed subspace $F$ of $E$ is locally complemented in $E$ if there is $\lambda>0$ such that whenever $G$ is a finite dimensional subspace of $E$ and $\varepsilon>0$, there is a linear operator $T:G\rightarrow F$ such that $\|T\|\leq \lambda$ and $\|Tx-x\|\leq \varepsilon\|x\|$ for $x\in F\cap G$. 
It follows from \cite[Theorem 3.5]{Kalton84} that $F$ is locally complemented in $E$ if and only if $F^{**}$ is complemented in $E^{**}$ under the natural embedding.
Proposition \ref{prop:duals} thus implies:

\begin{cor}\label{c:locally_complemented}
 If $F$ is locally complemented in $E$, then $(F,E)$ has the POE-$p$ for every $p\in[1,\infty]$.
\end{cor}

On the other hand, recall that a subspace $F$ of a Banach space $E$ is called an ideal (cf.~\cite{GKS}) if $F^\perp=\{x^*\in E^*:x^*(y)=0 \text{ for }y\in F\}$ is the kernel of a contractive projection on $E^*$.
In this case $F^{**}$ is contractively complemented in $E^{**}$. Note that here, neither $E$ nor $F$ is assumed to have any order structure. One should distinguish between the ``Hahn-Banach ideals'' described above, and order ideals we are discussing in the context of Banach lattices. 
\\

For ideals (in the Banach space sense) \Cref{p:POEp_passes_to_bidual} implies:

\begin{cor}\label{c:ideals}\label{rem:ideals}
 If $F$ is an ideal in $E$, then $(F,E)$ has the $1$-POE-$p$ for any $p \in [1,\infty]$.
\end{cor}

Embeddings of Banach spaces into their ultrapowers behave in a fashion similar to embeddings into second duals.
Recall that given a Banach space $E$ and a free ultrafilter $\mathcal U$ on an infinite set $\Gamma$, the ultrapower of $E$ with respect to $\mathcal U$ is given by $E_{\mathcal U}=\ell_\infty(\Gamma,E)/N_{\mathcal U}$, where $N_{\mathcal U}$ is the subspace of elements in $\ell_\infty(\Gamma,E)$ which converge to zero along $\mathcal U$. 
A ``natural embedding'' of $E$ into $E_{\mathcal{U}}$ is determined by mapping $e$ to the equivalence class of $(e,e,\ldots)$.

\begin{prop}
For any $p \in [1,\infty]$ and any Banach space $F$, the pair $(F, F_{\mathcal{U}})$ has the $1$-POE-$p$.
\end{prop}

\begin{proof}
Given $T:F\rightarrow \ell_p^n$, let $T_{\mathcal U}:F_{\mathcal U}\rightarrow (\ell_p^n)_{\mathcal U}$ denote the natural extension (cf.~\cite[Theorem 1.64]{AA}) which satisfies $\|T_{\mathcal U}\|=\|T\|$. By compactness we have that $(\ell_p^n)_{\mathcal U}=\ell_p^n$.
\end{proof}

\subsection{Further characterizations of the POE-$p$}\label{Further POEp}
In the definition of POE-$p$ there is a uniform constant $C$ which is selected independently of the embedding $F \hookrightarrow E$. However, it is not necessary to require this:

\begin{prop}\label{p:POEp}
 For $1 \leq p \leq \infty$ and a Banach space $F$, the following are equivalent:
 \begin{enumerate}
  \item $F$ has the POE-$p$;
  \item For any Banach space $E$ containing $F$ there is a constant $C>0$ such that every operator $T:F\to \ell_p^n$ extends to $\widetilde{T}: E\to \ell_p^n$ with $\|\widetilde{T}\|\leq C\|T\|$. 
 \end{enumerate}
\end{prop}

\begin{proof}
 Clearly $(1) \Rightarrow (2)$. Now suppose $(1)$ fails; we will show that $(2)$ fails as well.
 \\
 
 For each $k \in \nat$, find an isometric embedding $j_k : F \hookrightarrow E_k$, and a contraction $T_k : F \to \ell_p^{n_k}$ so that any extension of $T_k$ to $E_k$ has norm at least $k$.
 We ``amalgamate'' the spaces $E_k$: let $E = (\sum_k E_k)_1/G$, where $G$ consists of all elements $(a_k j_k y)_k \in (\sum_k E_k)_1$ ($y \in F$) with $\sum_k a_k = 0$ (this sum is well defined, since the membership in $(\sum_k E_k)_1$ implies $\sum_k |a_k| < \infty$).
  Define $u_k : E_k \to E$ by taking $x \in E_k$ to the equivalence class of $x^{(k)} := (0, \ldots, 0, x, 0, \ldots)$ ($x$ is in the $k$-th position).
 Then $u_k$ is an isometry. Indeed, clearly this map is contractive. On the other hand, for any $x \in E_k$,
 \begin{align*}
 \|u_k x\| =
 &
 \inf_{g \in G} \big\| x^{(k)} + g \big\| = 
 \inf_{ y \in F, \sum_i a_i = 0} \Big\{ \big\| a_k j_k y + x \big\| + \sum_{i \neq k} \big\| a_i j_i y \big\| \Big\} 
 \\
 \geq 
 &
 \inf_{ y \in F, \sum_i a_i = 0} \big\{ \|x\| - |a_k| \|y\| + \sum_{i \neq k} |a_i| \|y\| \big\} = \|x\| .
 \end{align*}
 
 For $i, k \in \nat$ and $y \in F$, $[j_i y]^{(i)} - [j_k y]^{(k)} \in G$, hence $u_i j_i = u_k j_k$. Denote $u_k j_k$ (no matter what $k$ is -- all these maps coincide) by $j$; then $j : F \to E$ is an isometric embedding.
 Consider $T_k : F \to \ell_p^{n_k}$ as described above (that is, $T_k$ is an operator with no ``small norm'' extension to $E_k$).
 Suppose $S : E \to \ell_p^{n_k}$ extends $T_k$. Then $\|S\| \geq \big\| S|_{u_k(E_k)} \big\| \geq k$.
 As $k$ is arbitrary, we conclude that $(2)$ fails.
\end{proof}

We can also restrict to superspaces of the same density. For a Banach space $E$, let us denote by $dens(E)$ the density character of $E$ -- that is, the least cardinality of a dense subset.

\begin{prop}\label{control density}
 For $C \geq 1$, $1 \leq p \leq \infty$, and a Banach space $F$, the following are equivalent:
\begin{enumerate}
\item $F$ has the $C$-POE-$p$;
\item Whenever $F$ is a closed subspace of $E$ with $dens(E)=dens(F)$, every operator $T:F\to \ell_p^n$ extends to $\widetilde{T}: E\to \ell_p^n$ with $\|\widetilde{T}\|\leq C\|T\|$.
\end{enumerate}
\end{prop}

\begin{proof}
  Suppose $(2)$ holds and let $F$ be a closed subspace of an arbitrarily large $E$. By \cite{HeiMan} (see also \cite{SimsYost}), there exists a closed subspace $G$, such that $F\subseteq G\subseteq E$, $dens(G)=dens(F)$ and $G$ is an ideal in $E$. If $T:F\to \ell_p^n$, then by hypothesis we can find an extension $\widetilde{T}: G\to \ell_p^n$ with $\|\widetilde{T}\|\leq C\|T\|$. Since $G$ is an ideal in $E$, then we also have an extension $\widetilde{\widetilde{T}}: E\to \ell_p^n$ with $\|\widetilde{\widetilde{T}}\|\leq C\|T\|$ by \Cref{rem:ideals}. 
  \end{proof}

\subsection{The relations between POE-$p$ and POE-$q$, and several examples}\label{ss:criteria_forPOEp}
In this section we give several examples of pairs $(F,E)$ (or spaces $F$) which have the POE-$p$, and several which do not. We begin with $\mathcal{L}_1$-spaces:

\begin{prop}\label{p:extension_from_L1}
Suppose $2 \leq p \leq \infty$, and $F$ is a ${\mathcal{L}}_{1,\mu}$-space for all $\mu > \lambda$.
Then $F$ has the $\lambda$-POE-$p$.
\end{prop}

\begin{proof}
In light of \Cref{p:POEp_C_versus_C+}, it suffices to show that, for any embedding $\iota : F \hookrightarrow E$, and any $C > \lambda$, any operator $T : F \to \ell_p^n$ has an extension $\widetilde{T} : E \to \ell_p^n$, with $\|\widetilde{T}\| \leq C \|T\|$.
To this end, find $\mu > \lambda$ and $\varepsilon > 0$ so that $\mu (1 + \varepsilon) < C$.
Let $p'$ be the ``conjugate'' of $p$, so that $\frac{1}{p}+\frac{1}{p'}=1$.  As $1 \leq p' \leq 2$, there exists an isometric embedding $j : \ell_{p'}^n \to L_1$ (see e.g.~\cite[Section 4]{JL}).
 Then $j^* : L_\infty \to \ell_p^n$ is a quotient map.
 By \cite[Theorem 4.2]{lind_ros},
$T$ has a lifting $S: F \to L_\infty$, with $\|S\| \leq \mu (1+\varepsilon) \|T\|$, and $j^* S = T$.
 Find an extension $\widetilde{S} : E \to L_\infty$, with $\|\widetilde{S}\| = \|S\|$. Then $\widetilde{T} := j^* \widetilde{S}$ is the desired extension of $T$. 
\end{proof}

We next discuss the relations between the  POE-$p$, for different values of $p$. As a corollary, we deduce from \Cref{ell_1 different than c0} that \Cref{p:extension_from_L1} fails for $1 \leq p < 2$ and one cannot replace $c_0$ by $\ell_1$ in \Cref{cor:c0} when $p\in [1,2)$. 

\begin{prop}\label{p:relations_between_POEp}\label{ell_1 different than c0}
\begin{enumerate}
 \item  If $1 \leq p \leq q \leq 2$, and $F$ has the POE-$p$, then it has the POE-$q$;
 \item  If $2 \leq p < \infty$, and $F$ has the POE-$p$, then it has the POE-$2$;
 \item The space $\ell_1$ has the POE-$p$ if and only if $2 \leq p \leq \infty$.
\end{enumerate}
\end{prop}

\begin{proof}
 (1) By \cite[Section 4]{JL}, $\ell_q$ (or even $L_q$) embeds isometrically into $L_p$.
 If $F$ has the POE-$p$, then, by \Cref{p:POE1_criterion}, there exists a constant $C$ so that the inequality $\pi_2(T) \leq C \|T\|$ holds for any $T : F \to L_p$.
 The ideal of $2$-summing operators is injective, hence we have $\pi_2(T) \leq C \|T\|$ holds for any $T : F \to \ell_q$. Thus, $F$ has the POE-$q$.
 \\
 
 (2) Suppose $F$ has the POE-$p$ with constant $c$. We need to show that, for any $E \supseteq F$, any operator $T : F \to \ell_2^n$ has an extension $\widetilde{T}  : E \to \ell_2^n$, with $\|\widetilde{T}\| \leq C \|T\|$ ($C$ is a universal constant). Denote the canonical basis in $\ell_2^n$ by $(e_k)$, and let $j : \ell_2^n \to \ell_p^{2^n}$ be the ``Khintchine'' embedding -- that is, $j e_k = r_k$ for $1 \leq k \leq n$,
with $r_1, \ldots, r_n$ being Rademacher random variables realized in $\ell_p^{2^n}$.
Then $\|jx\| \geq \|x\|$ for any $x$. Further, there exists $\lambda = \lambda_p$ so that $\|j\| \leq \lambda$, and there exists a projection $P : \ell_p^{2^n} \to j(\ell_2^n)$ with $\|P\| \leq \lambda$. Consider $T : F \to \ell_2^n$. As $F$ has the POE-$p$, $j T : F \to \ell_p^{2^n}$ has an extension 
$S : E \to \ell_p^{2^n}$ with $\|S\| \leq c \|j\| \|T\| \leq c \lambda \|T\|$.
Then $\widetilde{T} = j^{-1} P S$ extends $T$, and $\|\widetilde{T}\| \leq c \lambda^2 \|T\|$.
\\

(3) The fact that $\ell_1$ has the POE-$p$ for $p \geq 2$ follows directly from \Cref{p:extension_from_L1}.
Now suppose $1 \leq p < 2$. By \Cref{p:POE1_criterion}, it suffices to show that for every $C > 0$ there exists a contractive $T : \ell_1 \to \ell_p^n$ so that $\pi_2(T) \geq C$.
Denote by $(e_k)$ the canonical bases in both $\ell_1$ and $\ell_p^n$, and set $T e_k = e_k$ if $1 \leq k \leq n$, $T e_k = 0$ otherwise.
By \cite[Theorem 9(v)]{Garling74}, $\pi_2(T) \sim n^{1/p - 1/2}$.
\end{proof}

The class of POE-$1$ spaces, albeit more restrictive than that of POE-$2$ spaces (by \Cref{p:relations_between_POEp}), is still fairly large. For instance, by \Cref{p:POE1_criterion} and \cite[Corollary III.I.13]{Wojt}, the disk algebra $A$ has the POE-$1$. However, for spaces with an unconditional basis, the POE-$1$ condition is very restrictive, as we will next see.

\begin{prop}\label{p:uncond_basis_POE1}
 A space $F$ with a normalized unconditional basis $(x_k)$ has the POE-$1$ if and only if $(x_k)$ is equivalent to the $c_0$ basis.
\end{prop}

\begin{proof}
Due to \Cref{cor:c0}, $c_0$ has the POE-$1$. Now suppose $F$ possesses a normalized unconditional basis $(x_k)$, and has the POE-$1$. 
 It is easy to see that the (semi-normalized) biorthogonal functionals $(x_k^*)$ form an unconditional sequence in $F^*$. 
 By \Cref{r:POE_p_versus_Hilbert-Schmidt}, $F^*$ is a G.T.~space.
 For $m \in \Nat$, denote by $P_m$ the canonical basis projection from $F$ onto $\spn[x_k : 1 \leq k \leq m]$. Then $P_m^*$ is a projection from $F^*$ onto $\spn[x_k^* : 1 \leq k\leq m]$.
 As $\sup_m \|P_m\| < \infty$, we conclude that there exists a constant $C$ so that, for any $m,n \in \Nat$, any operator $T : \spn[x_k^* : 1 \leq k \leq m] \to \ell_2^n$ satisfies $\pi_1(T) \leq C \|T\|$.
 The proof of \cite[Theorem 8.21]{Pisier_FACT} shows the existence of a constant $C'$ so that $(x_k^*)_{k=1}^m$ is $C'$-equivalent to the canonical basis of $\ell_1^m$. In other words, the inequality
 $$
 \frac1{C'} \sum_k |a_k| \leq \Big\| \sum_k a_k x_k^* \Big\| \leq \sum_k |a_k|
 $$
 holds for any finite sequence $(a_k)$.
 Thus, $(x_k)$ is equivalent to the $c_0$ basis.
\end{proof}

\begin{rem}\label{r:POE-1 unconditional}
An alternative proof for  \Cref{p:uncond_basis_POE1} can also be deduced from \cite{Rud}, where it is shown that a space with an unconditional basis has the POE-$2$ if and only if it is isomorphic to $\ell_1$, $c_0$, or $c_0 \oplus \ell_1$. Indeed, if $F$ has the POE-$1$, then by \Cref{p:relations_between_POEp}, it also has the POE-$2$, and since $\ell_1$ fails POE-$1$, the previous characterization yields that $F$ can only be isomorphic to $c_0$. 
\end{rem}

\begin{rem}\label{r:POEp_open_questions}
 The relations between the POE-$p$ for different values of $p$ remain unclear. For instance, we do not know whether POE-$p$ implies POE-$q$ in the following situations:
 \begin{enumerate}
  \item  $p \in [1,\infty)$, $q \in (2,\infty)$.
  \item  $1 \leq q < p < 2$.
 \end{enumerate}
Also, we do not have a characterization of POE-$p$ ($2 < p < \infty$) in terms of operator ideals, along the lines of \Cref{p:POE1_criterion}. Using \cite[Corollary 10.10]{DJT}, one can observe that, if $F$ has the POE-$p$ for $2 < p < \infty$, then $B(F,\ell_p) = \Pi_{p,s}(F,\ell_p) = \Pi_r(F,\ell_p)$ whenever $s < p < r$. However, this condition does not seem to be sufficient.
\\

Above, we have observed that the disk algebra $A$ has the POE-$p$ for $1 \leq p \leq 2$. We do not know whether $A$ has the POE-$p$ for $2 < p < \infty$. We know that, by \cite[Corollary 10.10]{DJT} and \cite[Corollary III.I.13]{Wojt}, $B(A,\ell_p) = \Pi_r(A,\ell_p)$ when $2 < p < r$. Also, by F.~and~M.~Riesz Theorem, $A^{**} = H_\infty \oplus_\infty M_s^*$, where $M_s$ is the set of measures singular with respect to the Lebesgue measure (cf.~\cite[p.~181]{Wojt}).
As the POE-$p$ passes to the double dual, we conclude that $H_\infty$ has the POE-$p$ for $1 \leq p \leq 2$. As with $A$, we do not know what the situation is for $2 < p < \infty$.
\end{rem}

Finally, we examine the POE-$p$ for some ``natural'' pairs $(F,E)$, where $F$ is a subspace of $L_p(\mu)$.

\begin{prop}\label{p:Hilbert_in_Lp}
Suppose $1 \leq p < \infty$, and $F$ is isomorphic to a complemented subspace of $L_p(\mu)$, for some measure $\mu$. If a Banach space $E$ contains $F$, then the following are equivalent:
\begin{enumerate}
 \item $F$ is complemented in $E$.
 \item $F$ is locally complemented in $E$.
 \item $(F,E)$ has the POE-$p$.
\end{enumerate}
\end{prop}

\begin{rem}\label{r:examples of complemented}
 \Cref{p:Hilbert_in_Lp} is applicable in the following situations:
 \begin{enumerate}
  \item
  $1 < p < \infty$, and $F$ is isomorphic to a Hilbert space. Indeed, suppose $I$ is an index set, and $\mu$ is the uniform probability measure on $\{0,1\}$. Then $\ell_2(I)$ is isomorphic to the span of independent Rademacher functions in $L_p \big( \mu^{\otimes I} \big)$; the latter space is complemented in $L_p \big( \mu^{\otimes I} \big)$.
  \item
  $1 < p < \infty$, and $F$ is a ${\mathcal{L}}_p$-space. Indeed, by \cite{lind_ros}, such an $F$ embeds complementably into an $L_p$-space.
 \end{enumerate}
\end{rem}

\begin{proof}[Proof of \Cref{p:Hilbert_in_Lp}]
 $(1) \Rightarrow (2)$ is easy, and $(2) \Rightarrow (3)$ follows from \Cref{c:locally_complemented}.
 \\
 
 $(3) \Rightarrow (1)$: 
 Let $F'$ be a complemented subspace of $L_p(\mu)$ isomorphic to $F$, $P : L_p(\mu) \to F'$ a projection, and $T : F \to F'$ an isomorphism.
 By \Cref{p:POEp_C_versus_C+}, $T$ has an extension $\widetilde{T}\colon E\to L_p(\mu)$. Then $Q:=T^{-1}|_{F'}P\widetilde{T}$ is a projection from $E$ onto $F$.
\end{proof}

Specializing to Hilbertian subspaces of $L_1$, we obtain:

\begin{prop}\label{p:hilbert_in_L1_p1}
 If $F$ is an infinite dimensional subspace of $L_1$, isomorphic to a Hilbert space, then $(F,L_1)$ fails the POE-$p$ for $1 \leq p < \infty$.
\end{prop}

\begin{proof}
If $1 < p < \infty$, the result follows from \Cref{r:examples of complemented}.
It remains to examine the case of $p=1$. Note that, if $F$ is Hilbertian, and $(F,E)$ has the POE-$p$, then $(F',E)$ has the POE-$p$ whenever $F'$ is a subspace of $F$.
Indeed, there exists a projection $P$ from $F$ onto $F'$. For any $T \in B(F',L_p)$, the operator $TP$ has an extension $S : E \to L_p$, which clearly also extends $T$. Therefore, it suffices to establish the failure of POE-$1$ for separable $F$. Also, we can restrict ourselves to $F \subseteq L_1(\mu)$, where $\mu$ is a probability measure.
\\

Suppose, for the sake of contradiction, that $(F,L_1(\mu))$ has the POE-$1$ with constant $C$. In particular, any operator $T : F \to F$ extends to $\widetilde{T} : L_1(\mu) \to L_1(\mu)$, with $\|\widetilde{T}\| \leq C \|T\|$. 
\\

Recall that a normalized (in $L_1$) Gaussian random variable $g$ can be realized on a measure space $(\Omega_0,\nu_0)$. Then independent normalized Gaussian variables $(g_i)_{i \in \Nat}$ can be realized in $L_1(\nu)$, with $\nu = \otimes_{i \in \Nat} \nu_0$; denote by $G$ the closure of their linear span. It is well-known that $G$ is Hilbertian. As operators on $F$ have an extension property described above, \cite{Rauch} shows that, for any isomorphism $J : F \to G$ there exists operators $u : L_1(\mu)  \to L_1(\nu)$ and $v : L_1(\nu)  \to L_1(\mu)$, extending $J$ and $J^{-1}$, respectively.
From this, we conclude that $(G, L_1(\nu))$ has the POE-$1$. Indeed, fix $J, u, v$ as in the preceding paragraph. For any $T \in B(G,L_1)$, the operator $S = TJ$ has an extension $\widetilde{S} : L_1(\mu) \to L_1$, with $\|\widetilde{S}\| \leq C \|J\| \|T\|$. Then $\widetilde{S} v : L_1(\nu) \to L_1$ extends $T$.  \\

Therefore, by \cite[Theorem 6.6, and the remark following it]{Pisier_FACT},
$L_1(\nu)^* = L_\infty(\nu)$ has cotype 2 (and satisfies Grothendieck's Theorem), which is clearly false. This is the desired contradiction.
\end{proof}

\begin{rem}\label{at least a copy}
Fix $p$, and suppose $F$ is a closed subspace of a Banach space $E$. We have found a characterization of when the canonical embedding extends to a lattice embedding of $\fbp[F]$ inside $\fbp[E]$. However, one might still wonder when $\fbp[E]$ at least contains some lattice isomorphic (or isometric) copy of $\fbp[F]$. Similarly, if $F$ is, moreover, a $p$-convex Banach lattice, when does $\fbp[E]$ contain a (nicely complemented) lattice copy of $F$? In general, these questions will have a negative answer: take for instance $E=C[0,1]$, and let $F$ be a subspace which is isomorphic to $\ell_1$. It will follow from \Cref{comp ell_1} that $\fbl[C[0,1]]$ never contains a sublattice isomorphic to $\ell_1$, so it fails to contain $\fbl[\ell_1]$ as a sublattice as well (this is due to \Cref{Finite ok}).
\end{rem}

\section{Basic sequences in $\fbp[E]$}\label{s:basic_seq}

In this section we study the structure of basic sequences in free Banach lattices. More specifically, we  begin with a basic sequence $(x_k)$  in a Banach space $E$, and try to understand the sequence of moduli $(|\delta_{x_k}|)$ it generates in $\fbp[E]$. This is important, since, due to the universal nature of free Banach lattices, the behaviour of the sequence $(|\delta_{x_k}|)$ reflects all possible embeddings of $E$ into arbitrary $p$-convex lattices. As an illustration of this, we note the following:

\begin{prop}\label{weakly null}
Suppose $(x_k)$ is a sequence in $E$, and $(|\delta_{x_k}|)$ is weakly null in $\fbp[E]$. Then, for any $p$-convex Banach lattice $X$, and any bounded map $T : E \to X$, the sequence $(|Tx_k|)$ is weakly null.
\end{prop}

\begin{proof}
 For $T$ as above, $\widehat{T}:\fbp[E]\to X$ is bounded, hence weak-to-weak continuous, hence if $(|\delta_{x_k}|)$ is weakly null then so is $(\widehat{T}|\delta_{x_k}|)=(|Tx_k|)$. 
\end{proof}

Taking into account the description of $(|\delta_{e_k}|) \subseteq \fbl[\ell_r]$ obtained in \cite{ATV}, we obtain:

\begin{cor}\label{l_r weakly null}
Suppose $2 < r \leq \infty$, and $(e_k)$ is the canonical basis in $\ell_r$ (if $r=\infty$, we take $c_0$ instead of $\ell_\infty$). If $X$ is a Banach lattice, and $T : \ell_r \to X$ is a bounded operator, then $(|Te_k|)$ is weakly null in $X$.
\end{cor}

The preceding result fails for $1 \leq r \leq 2$. Indeed, \cite[Theorem 5.4]{ART} (see also \Cref{lower 2 gives 1} below) shows that $(|\delta_{e_k}|)$ is equivalent to the $\ell_1$ basis for $1 \leq r \leq 2$. However, by \Cref{infty is different} $(|\delta_{e_k}|)$ is equivalent to the $\ell_r$ basis (hence weakly null, if $r > 1$) in $\fbl^{(\infty)}[\ell_r]$. These observations are consistent with the fact that the standard Rademacher random variables give a copy of $\ell_2$ in $L_r$ ($1\leq r<\infty$) but their moduli are not weakly null. Actually, the unit vector basis of $\ell_r$ ($1< r\leq 2)$ has weakly null moduli in $\fbp[\ell_r]$ if and only if $p=\infty$; the  Rademacher functions in $L_\infty$ shows that the moduli of $\ell_1$ in $\fbp[\ell_1]$ can never be weakly null. We will expand on these observations significantly in the results below.
\\

Next we generalize \cite[Proposition 1]{ATV}. Recall that a basic sequence $(x_k)$ is called $C$-suppression unconditional if for every choice of scalars $(a_k)$ and any set $A\subseteq \mathbb N$ we have $\|\sum_{k\in A}a_k x_k\|\leq C\|\sum_{k\in\mathbb N} a_k x_k\|$. It is standard to check that every $C$-suppression unconditional sequence is $2C$-unconditional.

\begin{prop}\label{Prop1}\label{l:unconditionality}
Let $(x_k)$ be a sequence in a Banach space $E$. Then, for the sequence $(|\delta_{x_k}|)$ in $\fbp[E]$, we have:
\begin{enumerate}
\item If $(x_k)$ is minimal then $(|\delta_{x_k}|)$ is minimal;
\item If $(x_k)$ is a basis then $(|\delta_{x_k}|)$ is basic;
\item  If $(x_k)$ is a $C$-suppression unconditional basis then $(|\delta_{x_k}|)$ is $C$-suppression unconditional, hence $2C$-unconditional;
\item  If $(x_k)$ is a symmetric basis then $(|\delta_{x_k}|)$ is symmetric.
\end{enumerate}
\end{prop}

Recall (see \cite[p.~54]{Singer}) that $(x_k)$ is called minimal if it admits a system of biorthogonal functionals. 

\begin{proof}
(1): Let $(x_k^*)$ be biorthogonal functionals for $(x_k)$, and extend them so that $x_k^*\in E^*$. Then $\widehat{x_k^*}\colon \fbp[E]\to\mathbb R$ is a lattice homomorphism for every $k$, so that
  \begin{math}
     \widehat{x_k^*}(\abs{\delta_{x_l}})
    =\bigabs{x_k^*(x_l)}=\delta_{k,l},
  \end{math}
showing that $\bigl(\widehat{x_k^*}\bigr)$ are biorthogonal functionals for $\bigl(\abs{\delta_{x_k}}\bigr)$. The proofs of statements (2)-(4) are similar to \cite[Proposition 1]{ATV}.
\end{proof}

\begin{rem}\label{r:unconditionality}
 In (3), $2C$-unconditionality cannot be replaced by $C$-unconditionality, even if $C=1$.
 Indeed, suppose $(e_k)$ is the canonical basis in $\ell_2$, and $p=1$.
We first show that $\big\| \big| \delta_{e_1} \big| + \big| \delta_{e_2} \big| \big\|_{\fbl[\ell_2]} = 2$.  Let $R :\ell_2\to L_2$ be the Rademacher mapping, which is known to be isometric. This lifts to a lattice homomorphism $\widehat{R}: \fbl[\ell_2]\to L_2$ of norm one. Hence,

$$2\geq \bigg\|\sum_{k=1}^2 |\delta_{e_k}|\bigg\|=\|R\|\bigg\|\sum_{k=1}^2 |\delta_{e_k}|\bigg\| \geq \bigg\|\widehat{R}\sum_{k=1}^2 |\delta_{e_k}|\bigg\|=\bigg\|\sum_{k=1}^2 |r_k|\bigg\|=2.$$
 
 To contrast this we shall show that
$\bignorm{\abs{\delta_{e_1}}-\abs{\delta_{e_2}}}=\sqrt{2}$.
It follows from
\begin{math}
  \bigabs{\abs{\delta_{e_1}}-\abs{\delta_{e_2}}}
  \le\abs{\delta_{e_1}-\delta_{e_2}}
\end{math}
that
\begin{displaymath}
  \bignorm{\abs{\delta_{e_1}}-\abs{\delta_{e_2}}}
  \le\norm{\delta_{e_1}-\delta_{e_2}}
  =\norm{e_1-e_2}=\sqrt{2}.
\end{displaymath}
For the converse inequality, let $T\colon\ell_2^2\to\ell_1^2$ be the
formal identity. Then
\begin{displaymath}
  \bignorm{\abs{\delta_{e_1}}-\abs{\delta_{e_2}}}
  \ge\frac{1}{\norm{T}}\bignorm{\abs{Te_1}-\abs{Te_2}}
  =\frac{1}{\sqrt{2}}\cdot 2=\sqrt{2}.
\end{displaymath}
\end{rem}

Note that in \Cref{Prop1} it was shown that if $(x_k)$ is a \emph{basis} of $E$, then $(|\delta_{x_k}|)$ is basic in $\fbp[E]$. However, the following question is open.

\begin{question}\label{Open problem on basic}
Suppose $(x_k)$ is a basic sequence in a Banach space $E$. Is the sequence $(|\delta_{x_k}|)$ basic in $\fbp[E]$? If $(x_k)$ is, further, unconditional, is $(|\delta_{x_k}|)$ unconditional as well?
\end{question}
We now present some partial progress on this question:

\begin{prop}\label{p:block basis}
If $E$ has a basis $(u_k)$, then for every block basic sequence $(x_k)$ of $(u_k)$ the sequence $(|\delta_{x_k}|)$ is basic in $\fbp[E]$.
\end{prop}

\begin{proof}
By a well-known result of Zippin (cf.~\cite[Lemma 9.5.5]{alb-kal}), a basis $(f_n)$ of $E$ can be constructed such that $(x_k)$ is a subbasis of $(f_n)$, say $f_{n_k}=x_k$. By Proposition \ref{Prop1}, we have that $(|\delta_{f_n}|)$ is a basic sequence in $\fbp[E]$. Hence, being a subsequence, so is $(|\delta_{x_k}|)$.
\end{proof}

\begin{rem}\label{p:block basis cor}
Using that $\||\delta_{x_k}|-|\delta_{y_k}|\|\leq \|x_k-y_k\|$, it follows from \Cref{p:block basis} and the principle of small perturbations that if $E$ has a basis $(u_k)$ then small perturbations of blocks of $(u_k)$ have moduli that are basic in $\fbp[E]$.
\end{rem}

A well-known result due to Bessaga and Pe\l czy\'nski allows one to extract a basic sequence from every semi-normalized weakly null sequence in a Banach space $E$. We will see next that this extraction can be made so that the corresponding sequence of moduli in $\fbp[E]$ is also basic:

\begin{prop}\label{p:basicsubsequence}
Let $E$ be a Banach space and $(x_n)$ a weakly null semi-normalized sequence in $E$. There is a subsequence such that $(|\delta_{x_{n_k}}|)$ is basic in $\fbp[E]$.
\end{prop}

\begin{proof}
Due to \cite{SimsYost} there is a separable subspace $F\subseteq E$ which is an ideal in $E$ and such that $(x_n)\subseteq F$. Hence, by Corollary \ref{c:ideals}, $\fbp[F]$ is an isometric sublattice of $\fbp[E]$. Therefore, for our purposes, we can assume without loss of generality that $E$ is actually separable. 
\\

Let us suppose that the set $S=\{|\delta_{x_n}|\}\subseteq \fbp[E]$ does not contain any basic sequence. Since $(x_n)$ is semi-normalized, $0\notin \overline{S}^{\|\cdot\|}$, hence, by \cite[Theorem 1.5.6]{alb-kal}, the weak-closure of $S$, $K=\overline{S}^w$ is a weakly compact set with $0\notin K$. 
\\

Now, let $\iota:E\rightarrow C[0,1]$ be an isometric embedding and $\overline \iota:\fbp[E]\rightarrow \fbp[C[0,1]]$ the induced lattice homomorphism. Since $\overline \iota$ is weak-weak continuous and injective, it follows that $\overline \iota(K)$ is a weakly-compact set and $0\notin\overline \iota(K)$. Thus, again by \cite[Theorem 1.5.6]{alb-kal}, $(|\delta_{\iota x_n}|)=(\overline\iota|\delta_{x_n}|)\subseteq \fbp[C[0,1]]$ does not contain a basic sequence. 
\\

However, $\iota(x_n)$ is a semi-normalized weakly null sequence in $C[0,1]$. Therefore, we can extract a subsequence $\iota(x_{n_k})$ which is a small perturbation of a block basic sequence of the monotone basis of $C[0,1]$. Hence, \Cref{p:block basis cor} yields that for some subsequence $(|\delta_{\iota(x_{n_m})}|)$ is a basic sequence in $\fbp[C[0,1]]$. This is a contradiction. We can thus assume that $S$ contains a basic sequence, which implies that we can extract an increasing sequence $(n_k)\subseteq \mathbb N$ such that $(|\delta_{x_{n_k}}|)$ is basic in $\fbp[E]$, as claimed.
\end{proof}

\begin{rem}
Of course, building on \Cref{l:unconditionality}, it is natural to study how other properties pass between the sequences $(x_k)$ and $(|\delta_{x_k}|)$ (e.g.~shrinking, boundedly complete, etc.) Although this will not be our focus, our results will indirectly shed partial light on such questions. In particular, we will discover some ``rigidity" results, i.e., properties of $(|\delta_{x_k}|)$ that force $(x_k)$ to take a particular form.
\end{rem}

\subsection{Lower 2-estimates, $\ell_1$, and $c_0$}
In this subsection we explicitly compute the moduli of certain bases, refining some results from \cite{ATV}. We begin by characterizing the behaviour of $c_0$ in $\fbp[c_0]$:

\begin{prop}\label{recall ATV}\label{Vladimir's prop on c0}
If $(e_k)$ is the canonical basis of $c_0$ then the sequence  $(|\delta_{e_k}|)$ in $\fbp[c_0]$ is equivalent to the canonical basis of $\ell_2$ for all $1\leq p<\infty$.
\end{prop}

\begin{proof}
By \Cref{l:unconditionality}, the sequence $\big( |\delta_{e_k}| \big)_k$ is an unconditional basic sequence. We shall show that, for finitely supported sequences $(a_k)$, 
$$\Big\|\sum_{k=1}^\infty a_k|\delta_{e_k}|\Big\|_{\fbp[c_0]}\sim \left(\sum_{k=1}^\infty |a_k|^2\right)^{1/2}$$
holds (with an equivalence constant depending on $p \in [1,\infty)$). Fix such $(a_k)$.
  Let
  \begin{displaymath}
    A_+=\{k : a_k\geq 0\} \quad\mbox{and}\quad
    A_-=\{k : a_k<0\}.
  \end{displaymath}
    Define an operator $T\colon c_0\to L_p[0,1]$
  via
  \begin{displaymath}
    Tx=\sum_{k\in A_+} a_ke^*_k(x)r_k.
  \end{displaymath}
  If $\norm{x}\le 1$ then Khintchine's inequality yields
  \begin{displaymath}
    \norm{Tx}
    \le B_p\Bigl(\sum_{k\in A_+}\bigabs{a_ke^*_k(x)}^2\Bigr)^{\frac12}
    \le B_p\Bigl(\sum_{k\in A_+} a_k^2\Bigr)^{\frac12}.
  \end{displaymath}
  It follows that
  \begin{math}
     \norm{T}\le B_p\Bigl(\sum_{k\in A_+} a_k^2\Bigr)^{\frac12}.
  \end{math}
  Note that $Te_k$ equals $a_kr_k$
  if $k\in A_+$ and zero otherwise.
\\

  Let $\widehat{T}\colon \fbp[c_0]\to L_p[0,1]$ be the canonical
  extension of $T$. Then $\|\widehat{T}\|\le B_p\Bigl(\sum_{k\in A_+} a_k^2\Bigr)^{\frac12}$ and
  \begin{displaymath}
    \widehat{T}\Bigl(\sum_{k=1}^\infty a_k\abs{\delta_{e_k}}\Bigr)
    =\sum_{k=1}^\infty a_k\abs{Te_k}
    =\sum_{k\in A_+}a_k\abs{a_kr_k}
    =\Bigl(\sum_{k\in A_+}a_k^2\Bigr)\mathbf{1}.
  \end{displaymath}
  It follows that
  \begin{displaymath}
    \sum_{k\in A_+}a_k^2
    =\Bignorm{\widehat{T}\Bigl(\sum_{k=1}^\infty a_k\abs{\delta_{e_k}}\Bigr)}
    \le B_p\Bigl(\sum_{k\in A_+} a_k^2\Bigr)^{\frac12}
        \Bignorm{\sum_{k=1}^\infty a_k\abs{\delta_{e_k}}},
  \end{displaymath}
  so that
  \begin{displaymath}
    B_p\Bignorm{\sum_{k=1}^\infty a_k\abs{\delta_{e_k}}}\ge
    \Bigl(\sum_{k\in A_+} a_k^2\Bigr)^{\frac12}.
  \end{displaymath}
  Similarly, we get
  \begin{displaymath}
    B_p\Bignorm{\sum_{k=1}^\infty a_k\abs{\delta_{e_k}}}\ge
    \Bigl(\sum_{k\in A_-} a_k^2\Bigr)^{\frac12}.
  \end{displaymath}
  Combining these estimates, we get
  \begin{displaymath}
    \sqrt{2}B_p\Bignorm{\sum_{k=1}^\infty a_k\abs{\delta_{e_k}}}_{\fbp[c_0]}
    \ge\Bigl(\sum_{k=1}^\infty\abs{a_k}^2\Bigr)^{\frac12}.
  \end{displaymath}
  
  Conversely, it was shown in  \cite{ATV} that $$\Big\|\sum_{k=1}^\infty a_k|\delta_{e_k}|\Big\|_{\fbl[c_0]}\sim \left(\sum_{k=1}^\infty |a_k|^2\right)^{1/2}.$$
    Hence, since the $\fbl$ norm is the largest of the $\fbp$ norms, 
    $$\Bignorm{\sum_{k=1}^\infty a_k\abs{\delta_{e_k}}}_{\fbp[c_0]}
    \leq \Bignorm{\sum_{k=1}^\infty a_k\abs{\delta_{e_k}}}_{\fbl[c_0]}\sim \left(\sum_{k=1}^\infty |a_k|^2\right)^{1/2}. \qedhere $$
\end{proof}

\begin{cor}\label{c:c_0 basic sequence}
Suppose a sequence $(e_k)$ in $E$ is equivalent to the canonical basis of $c_0$. Then the sequence  $(|\delta_{e_k}|)$ in $\fbp[E]$ is equivalent to the canonical basis of $\ell_2$ for all $1\leq p<\infty$.
\end{cor}

\begin{proof}
 Combine \Cref{Vladimir's prop on c0} with the fact that $c_0$ has the POE-$p$ (\Cref{cor:c0}).
\end{proof}

Combined with \Cref{recall ATV} the next result establishes a lower 2-estimate for the moduli of an arbitrary basic sequence:
\begin{prop}\label{prop:lowerell2}
 Let $(x_k)$ be a sequence in $E$, and assume that there are  biorthogonal functionals $(x_k^*)$ to $(x_k)$ such that $K:=\sup_k\|x_k^*\|<\infty$. Then for any finitely supported sequence of scalars $(a_k)$ we have 
$$
\Big\|\sum_{k=1}^\infty a_k|\delta_{e_k}|\Big\|_{\fbp[c_0]}\leq K\Big\|\sum_{k=1}^\infty a_k|\delta_{x_k}|\Big\|_{\fbp[E]},
$$
where $(e_k)$ denotes the unit vector basis of $c_0$. Consequently, for $1\leq p<\infty,$
$$
\left( \sum_{k=1}^\infty a_k^2 \right)^{1/2} \lesssim \Big\|\sum_{k=1}^\infty a_k|\delta_{x_k}|\Big\|_{\fbp[E]}.
$$
\end{prop}

\begin{proof}
The assumptions tell us that the operator $T:[x_k]\rightarrow \ell_\infty$ given by $Tx = (x_k^*(x))$ has norm at most $K$ and $Tx_k=e_k$. By injectivity of $\ell_\infty$ (or Hahn-Banach), we have an extension $\widetilde T:E\rightarrow \ell_\infty$ with $\|\widetilde T\|\leq K$. Let $\phi_{\ell_\infty}:\ell_\infty\rightarrow \fbp[\ell_\infty]$ denote the canonical isometric embedding and let $S=\phi_{\ell_\infty}\widetilde T:E\rightarrow \fbp[\ell_\infty]$. Let now $\widehat S:\fbp[E]\rightarrow \fbp[\ell_\infty]$ be the lattice homomorphism extending $S$, and note that $\|\widehat S\|\leq K$. It follows that for any finitely supported sequence of scalars $(a_k)$ we have 
$$
\Big\|\sum_{k=1}^\infty a_k |\delta_{e_k}|\Big\|_{\fbp[\ell_\infty]}=\Big\|\sum_{k=1}^\infty a_k |\widehat S \delta_{x_k}|\Big\|_{\fbp[\ell_\infty]}\leq K\Big\|\sum_{k=1}^\infty a_k|\delta_{x_k}|\Big\|_{\fbp[E]}.
$$
Using \Cref{cor:subspace} and the above estimate we get that
$$
 \Big\|\sum_{k=1}^\infty a_k |\delta_{e_k}|\Big\|_{\fbp[c_0]}= \Big\|\sum_{k=1}^\infty a_k |\delta_{e_k}|\Big\|_{\fbp[\ell_\infty]}\leq K\Big\|\sum_{k=1}^\infty a_k|\delta_{x_k}|\Big\|_{\fbp[E]}.
$$
Finally, the ``consequently'' statement follows from \Cref{Vladimir's prop on c0}.
\end{proof}

Statement (3) of \Cref{ell_1 different than c0} suggests that $\ell_1$ could be a counterexample to \Cref{Open problem on basic}. However, basic sequences equivalent to $\ell_1$ will always have moduli equivalent to $\ell_1$ in  free spaces. More generally, we have the following proposition:

\begin{prop}\label{p:dominations_abs_values}\label{Equiv of basic}\label{c:copies of l1}
 Suppose $(x_k)$ is a $C$-unconditional basic sequence in a Banach space $E$. Then, for $1 \leq p \leq \infty$, and for any $a_1 , \ldots , a_n \in \Real$,
 $$
 \bignorm{ \sum_{k=1}^n a_k \bigabs{ \delta_{x_k} } }_{\fbp[E]} \geq \frac1{2C} \bignorm{ \sum_{k=1}^n a_k x_k}.
 $$
\end{prop}
\begin{proof}
It suffices to consider $p=\infty$, as this is the weakest of the $\fbp$-norms. By \Cref{AM works fine}, we may assume that $(x_k)$ is a basis. The result then follows from statement (3) of \Cref{Prop1}:
  \begin{align*}
2C \bignorm{ \sum_{k=1}^n a_k \bigabs{ \delta_{x_k} } }_{\FBLi[E]}& \geq \bignorm{ \sum_{k=1}^n |a_k| \bigabs{ \delta_{x_k} } }_{\FBLi[E]} \\
&\geq \bignorm{ \sum_{k=1}^n a_k \delta_{x_k} }_{\FBLi[E]}= \bignorm{ \sum_{k=1}^n a_k x_k}.
\end{align*}
\end{proof}

\Cref{r:no domination} below will show that the unconditionality assumption in the preceding proposition is essential, in general, for the sequence of moduli to dominate the original sequence.
\\

We now look to characterize those bases $(x_k)$ of $E$ such that the sequence $(|\delta_{x_k}|)$ in $\fbp[E]$ is equivalent to the unit vector basis of $\ell_1$. We begin with a sufficient condition:

\begin{prop}\label{lower 2 gives 1}
Let $E$ be a Banach space with a normalized basis $(x_k)$ satisfying a lower $2$-estimate. Then for all $p\in [1,\infty)$ the sequence $(|\delta_{x_k}|)$ in $\fbp[E]$ is equivalent to the unit vector basis of $\ell_1$.
\end{prop}
\begin{proof}
We first prove this for the unit vector basis  $(e_k)$ of $\ell_2$. Let $R :\ell_2\to L_p$ be the Rademacher mapping.  This extends to a lattice homomorphism $\widehat{R}: \fbp[\ell_2]\to L_p$.
Now,
$$\|R\|\bigg\|\sum_{k=1}^n a_k|\delta_{e_k}|\bigg\|\gtrsim \|R\|\bigg\|\sum_{k=1}^n |a_k||\delta_{e_k}|\bigg\| \geq \bigg\|\widehat{R}\sum_{k=1}^n |a_k||\delta_{e_k}|\bigg\|=\bigg\|\sum_{k=1}^n |a_k||r_k|\bigg\|=\sum_{k=1}^n|a_k|,$$
where the first domination is by the unconditionality statement in \Cref{Prop1}.
\\

Now suppose that $(x_k)$ is normalized with a lower 2-estimate. Then the basis to basis map $T : E\to \ell_2$ is bounded, so we can extend it to a lattice homomorphism $\overline{T}: \fbp[E]\to \fbp[\ell_2]$. Note $\overline{T}(|\delta_{x_k}|)=|\delta_{e_k}|$. From this we get that 
$$\|T\|\bigg\|\sum_{k=1}^n a_k|\delta_{x_k}|\bigg\| \geq \bigg\|\sum_{k=1}^n a_k|\delta_{e_k}|\bigg\|\sim \sum_{k=1}^n |a_k|.$$
\end{proof}

\begin{example}\label{Walsh ex}
Let $1\leq r\leq 2$ and let $(e_k)$ be the canonical basis of $\ell_r$. Then the sequence $(|\delta_{e_k}|)$ in  $\fbp[\ell_r]$ is equivalent to the canonical basis of $\ell_1$, when $p\in [1,\infty)$. Similarly, the Walsh basis in $L_r[0,1]$ is normalized and satisfies a lower 2-estimate if $1 \leq r \leq 2$.
\end{example}

One cannot replace ``basis" with ``basic sequence" in \Cref{lower 2 gives 1}, see \Cref{p:l2_in_CK}. The dual to the summing basis in $c_0$ (see below) satisfies the conclusion but not the hypothesis of \Cref{lower 2 gives 1}:

\begin{example}\label{dual to the summing}
Let $(x_k)$ be the basis for $\ell_1$ such that $x_1=e_1$ and $x_k=e_k-e_{k-1}$ for $k\geq 2$. Then for all $p\in [1,\infty]$ the sequence $(|\delta_{x_k}|)$ in $\fbp[\ell_1]$ is equivalent to the unit vector basis of $\ell_1$.
\end{example}
\begin{proof}
By \Cref{General comparison} it suffices to work with $\fbl^{(\infty)}[\ell_1]$. Recall that the norm  $\left\|\sum_{k=1}^m a_k|\delta_{x_k}|\right\|$ is computed by taking $\sup_{f\in B_{\ell_\infty}}\left|\sum_{k=1}^m a_k|f(x_k)|\right|$.
 \\
 
Choosing $f=(1,-1,1,-1,1,-1,1,-1,\dots)$ we get 
 $$\left|\sum_{k=1}^m a_k|f(x_k)|\right|=|a_1+2a_2+2a_3+\cdots +2a_m|.$$
 This tells us $\big( \big| \delta_{x_k} \big| \big)_k$ is either conditional, or equivalent to the $\ell_1$ basis.
 \\
 
 Now, in general, to take care of signs, one picks $f$ to be a sequence of ones and negative ones, but now aligns them with the signs of the $a_k$ (we can safely ignore $x_1$).  This easily gives $$\bigg\|\sum_{k=2}^m a_k|\delta_{x_k}|\bigg\|_{\fbl^{(\infty)}[\ell_1]}\geq \max\left(\left|\sum_{k: a_k\geq 0} 2a_k\right|, \left|\sum_{k: a_k\leq 0} 2a_k\right|\right)\geq \sum_{k=2}^m|a_k|,$$
 and  proves the claim.
\end{proof}

 In contrast to \Cref{dual to the summing}, \Cref{lower 2 gives 1} is sharp for unconditional bases:

\begin{prop}\label{upper-root}\label{p:l1_vs_l2}
  Suppose that $(x_k)$ is a 1-unconditional basic sequence in $E$,
  $a_1,\dots,a_n\ge 0$, and $1\le p\le\infty$. Then
  \begin{equation}\label{Lower 2 technical}
    \Bignorm{\sum_{k=1}^na_k\abs{\delta_{x_k}}}_{\fbp[E]}
    \le K_G\Bigl(\sum_{k=1}^na_k\Bigr)^\frac12
    \Bignorm{\sum_{k=1}^n\sqrt{a_k}x_k}_{E}.
  \end{equation}  
 
\end{prop}
Here, $K_G$ denotes the universal Grothendieck constant. In particular, suppose $(x_k)$ fails to have a lower $2$-estimate -- that is, there exist $t_1, \ldots, t_n$ so that $\|\sum_{k=1}^n t_k x_k\| \ll (\sum_{k=1}^n t_k^2)^{1/2}$. Then by \eqref{Lower 2 technical}, $$\Bignorm{ \sum_{k=1}^n t_k^2 |\delta_{x_k}| }_{\fbp[E]} \ll \sum_{k=1}^n t_k^2,$$ implying that $(|\delta_{x_k}|)$ is not equivalent to the $\ell_1$ basis.
\begin{proof}
  Let $F= \overline{\spn}[x_k : k \in \nat]$. Let $T=\phi_E|_F\colon F\to\fbp[E]$ be the natural
  inclusion. In $\fbp[E]$, using Cauchy-Schwarz inequality we have
  \begin{displaymath}
    \sum_{k=1}^na_k\abs{\delta_{x_k}}
    \le\Bigl(\sum_{k=1}^na_k\Bigr)^\frac12
    \Bigl(\sum_{k=1}^na_k\abs{\delta_{x_k}}^2\Bigr)^\frac12
    =\Bigl(\sum_{k=1}^na_k\Bigr)^\frac12
      \Bigl(\sum_{k=1}^n\bigabs{T(\sqrt{a_k}x_k)}^2\Bigr)^\frac12.
  \end{displaymath}
  View $F$ as a Banach lattice under the order
  induced by $(x_k)$. Using Krivine's inequality \cite[Theorem 1.f.14.]{LT2}, and the
  fact that $(x_k)$ are disjoint in $F$, we get
  \begin{multline*}
    \Bignorm{\sum_{k=1}^na_k\abs{\delta_{x_k}}}_{\fbp[E]}
    \le\Bigl(\sum_{k=1}^na_k\Bigr)^\frac12
    \cdot K_G\norm{T}\Bignorm{\Bigl(\sum_{k=1}^n
      \bigabs{(\sqrt{a_k}x_k)}^2\Bigr)^\frac12}_{F}\\
    =K_G\Bigl(\sum_{k=1}^na_k\Bigr)^\frac12
    \Bignorm{\sum_{k=1}^n\sqrt{a_k}x_k}_{E}.
  \end{multline*}
\end{proof}

\bigskip

As already mentioned, if $(x_k)$ is an unconditional basis of $E$, then the sequence $(|\delta_{x_k}|)$ is unconditional in $\fbp[E]$. As we saw in \Cref{dual to the summing}, the modulus of a conditional basis need not be conditional - it also need not be unconditional.
For the sake of an example, consider the summing basis in $c_0$, consisting of vectors $s_k = (1, \ldots, 1, 0, \ldots)$ ($k$ $1$'s in a row).

\begin{prop}\label{p:summing_basis_conditional}\label{Some parts redundant}
For all $p\in [1,\infty]$ the basic sequence $\big( \bigabs{ \delta_{s_k} } \big)$ is conditional (even for constant coefficients) in  $\fbp[c_0]$.
 In fact, if $n$ is an even integer, then $\bignorm{ \sum_{k=1}^n \bigabs{ \delta_{s_k} } }_{\fbl^{(\infty)}[c_0]} = n$, but: 
 \begin{enumerate}
  \item  $\bignorm{ \sum_{k=1}^n (-1)^k \bigabs{ \delta_{s_k} } }_{\fbl^{(\infty)}[c_0]} = 1$.
  \item  
  $\kappa \sqrt{n} \leq \bignorm{ \sum_{k=1}^n (-1)^k \bigabs{ \delta_{s_k} } }_{\fbl[c_0]} \leq K_G \sqrt{n}$, where $K_G$ is Grothendieck's constant, and $\kappa > 0$ is a universal constant.
 \end{enumerate}
\end{prop}
\begin{proof}
(1) The case of $\fbl^{(\infty)}[c_0]$:
\\

The norm on $\fbl^{(\infty)}[c_0]$ arises from $\|f\| = \sup \big\{ |f(x^*)| : \|x^*\|_{\ell_1} \leq 1\}$. By the triangle inequality, $\big\|\sum_{k=1}^n \big| \delta_{s_k} \big| \big\| \leq \sum_{k=1}^n \|s_k\| = n$. For a lower estimate take $x^* = (1,0,0,\ldots) \in \ell_1$. Then
$$
\big\|\sum_{k=1}^n \big| \delta_{s_k} \big| \big\| \geq \sum_{k=1}^n \big| x^*(s_k) \big| = n .
$$

Now recall that 
$$
\big\|\sum_{k=1}^n (-1)^k \big| \delta_{s_k} \big| \big\| = \sup_{\|x^*\| \leq 1} \Big| \sum_{k=1}^n (-1)^k |x^*(s_k)| \Big| .
$$
Write $x^* = (a_1, a_2, \ldots)$. Then $x^*(s_k) = a_1 + \cdots + a_k$; hence
$$
\big| \sum_{k=1}^n (-1)^k |x^*(s_k)| \big| = \Big| \sum_{j=1}^{n/2} \Big( \big| \sum_{k=1}^{2j} a_k \big| - \big| \sum_{k=1}^{2j-1} a_k \big| \Big) \Big| \leq
\sum_{j=1}^{n/2} \Big| \big| \sum_{k=1}^{2j} a_k \big| - \big| \sum_{k=1}^{2j-1} a_k\big| \Big| .
$$
By the triangle inequality,
$\Big| \big| \sum_{k=1}^{2j} a_k \big| - \big| \sum_{k=1}^{2j-1} a_k \big| \Big| \leq |a_{2j}|$,
hence
$$
\big| \sum_{k=1}^n (-1)^k |x^*(s_k)| \big| \leq \sum_{j=1}^{n/2} \big| a_{2j} \big| \leq \|x^*\| \leq 1 ,
$$
which leads us to conclude that
$$
\big\|\sum_{k=1}^n (-1)^k \big| \delta_{s_k} \big| \big\| = \sup_{\|x^*\| \leq 1} \big| \sum_{k=1}^n (-1)^k |x^*(s_k)| \big| \leq 1 .
$$
On the other hand, testing on $x^* = (0,1,0, \ldots)$, we obtain $\big\|\sum_{k=1}^n (-1)^k \big| \delta_{s_k} \big| \big\| \geq 1$.
\\

(2) The case of $\fbl[c_0]$:
\\

The lower estimate for $\big\|\sum_{k=1}^n (-1)^k \big| \delta_{s_k} \big| \big\|$ follows from \Cref{prop:lowerell2} (the norms of the biorthogonal functionals $(0, \ldots, 0, 1,-1, 0, \ldots) \in \ell_1$ do not exceed $2$).
\\

For an upper estimate, we view $s_1, \ldots, s_n$ as living in $\ell_\infty^n$ (spanned by the first $n$ coordinates of $c_0$).
We need to prove that, if $x_1^*, \ldots, x_m^* \in \ell_1^n$ are such that $\max_{\pm}\bignorm{ \sum_{j=1}^m \pm x_j^* } \leq 1$, then
$$
\sum_{j=1}^m \Big| \sum_{k=1}^n (-1)^k \big| x_j^*(s_k) \big| \Big| \leq K_G \sqrt{n} .
$$

As noted in part (1),
$$
\Big| \sum_{k=1}^n (-1)^k \big| x_j^*(s_k) \big| \Big| \leq \|x_j^*\| .
$$
We have to therefore show that, for our sequence $(x_j^*)$, $\sum_j \|x_j^*\| \leq K_G \sqrt{n}$.
To this end, consider the operator $u : \ell_\infty^m \to \ell_1^n : e_j \mapsto x_j^*$
(here $(e_j)_{j=1}^m$ is the canonical basis of $\ell_\infty^m$). Note that $\sum_j \|x_j^*\| = \sum_j \|u e_j\|$, and $\| \sum_j \delta_j e_j \| = 1$ whenever $\delta_j = \pm 1$; thus, $\sum_j \|u e_j\| \leq \pi_1(u)$.
Therefore, it suffices to show that $\pi_1(u) \leq K_G \sqrt{n}$.
\\

Note that $\|u\| = \sup_{\delta_j = \pm 1} \bignorm{ \sum_{j=1}^m \delta_j x_j^* } \leq 1$, hence, by Grothendieck's Theorem (see e.g.~\cite[Theorem 3.5]{DJT}), $\pi_2(u) \leq K_G$.
Write $u = id \circ u$, where $id$ is the identity operator on $\ell_1^n$.
Then by \cite[Theorem 4.17]{DJT}, $\pi_2(id) = \sqrt{n}$, and therefore \cite[Theorem 2.22]{DJT} implies that $\pi_1(u) = \pi_1(id \circ u) \leq \pi_2(id) \pi_2(u) \leq K_G \sqrt{n}$.
\\

We can now interpolate these results to general $p$: We know that, for even $n$, \begin{displaymath}
\bignorm{ \sum_{k=1}^n \bigabs{ \delta_{s_k} } }_{\fbl^{(\infty)}[c_0]} = n \ \ \ \  \text{and} \ \ \ \ \bignorm{ \sum_{k=1}^n (-1)^k \bigabs{ \delta_{s_k} } }_{\fbl[c_0]}\lesssim n^{1/2}.
\end{displaymath} If $(|\delta_{s_k}|)$ were unconditional in $\fbp[c_0]$ then there would exist a $C$ such that 
$$\bignorm{ \sum_{k=1}^n \bigabs{ \delta_{s_k} } }_{\fbp[c_0]}\leq C\bignorm{ \sum_{k=1}^n (-1)^k \bigabs{ \delta_{s_k} } }_{\fbp[c_0]}.$$ But now using that the $\fbl^{(\infty)}$-norm is minimal and the $\fbl$-norm is maximal, we get
\begin{align*}
n&=\bignorm{ \sum_{k=1}^n \bigabs{ \delta_{s_k} } }_{\fbl^{(\infty)}[c_0]}\leq  \bignorm{ \sum_{k=1}^n \bigabs{ \delta_{s_k} } }_{\fbp[c_0]}\\
&\leq C\bignorm{ \sum_{k=1}^n (-1)^k \bigabs{ \delta_{s_k} } }_{\fbp[c_0]}\leq  C\bignorm{ \sum_{k=1}^n (-1)^k \bigabs{ \delta_{s_k} } }_{\fbl[c_0]}\lesssim n^{1/2},
\end{align*}
 a contradiction.
\end{proof}

\begin{rem}\label{r:no domination}
\Cref{c:copies of l1} states that, if $(x_k)$ is an unconditional basis in a Banach space $E$, then the inequality $\big\| \sum_k a_k \big| \delta_{x_k} \big| \big\| \geq c \big\| \sum_k a_k x_k \big\|$ holds, with a constant $c$ independent of $(a_k)$. However, this is false for conditional bases. Indeed, consider the ``alternating summing'' basis $s_k' = (-1)^k s_k$ in $c_0$. It is easy to see that, for any $n$, $\big\| \sum_{k=1}^n (-1)^k s_k' \big\| = \big\| \sum_{k=1}^n s_k \big\| = n$.
However, $\bignorm{ \sum_{k=1}^n (-1)^k \bigabs{ \delta_{s_k'} } }_{\fbl[c_0]} = \bignorm{ \sum_{k=1}^n (-1)^k \bigabs{ \delta_{s_k} } }_{\fbl[c_0]} \sim \sqrt{n}$, due to \Cref{p:summing_basis_conditional}.
\end{rem}

\subsection{Example: moduli of the Haar system in $L_1[0,1]$}\label{subsec:Haar}
In what follows, $\nat_0 = \nat \cup \{0\}$. Let $(h_{n,k})_{1\leq k \leq 2^n, n\in \nat_0}$ denote the normalized Haar basis in $L_1$. That is, 
$$
h_{n,k}=2^n \chi_{[\frac{k-1}{2^{n}},\frac{2k-1}{2^{n+1}}]}- 2^n\chi_{[\frac{2k-1}{2^{n+1}}, \frac{k}{2^{n}}]}.
$$
For more information about the Haar system the reader is referred to e.g.~\cite[Section 2.c]{LT2}. Clearly, for a fixed $n\in \nat_0$, $(h_{n,k})_{k=1}^{2^n}$ is 1-equivalent to the basis of $\ell_1^{2^n}$, and so is $(|\delta_{h_{n,k}}|)_{k=1}^{2^n}$ in $\fbl[L_1]$ as the following shows: 
\\
\begin{lem}
For every $n\in \mathbb N$, and scalars $(a_k)_{k=1}^{2^n}$ we have
$$
\Big\|\sum_{k=1}^{2^n}a_k |\delta_{h_{n,k}}|\Big\|_{\fbl[L_1]}=\sum_{k=1}^{2^n}|a_k|.
$$
\end{lem}

\begin{proof}
Since $\|h_{n,k}\|=1$, the triangle inequality  trivially implies that  $\|\sum_{k=1}^{2^n}a_k |\delta_{h_{n,k}}|\|\leq \sum_{k=1}^{2^n}|a_k|.$
\\

For the converse, let $T:L_1\rightarrow \ell_1^{2^n}$ be the operator given by $$Tf=\left(\int_\frac{k-1}{2^n}^\frac{2k-1}{2^{n+1}}fd\mu-\int_\frac{2k-1}{2^{n+1}}^\frac{k}{2^n}fd\mu\right)_{k=1}^{2^n}.$$
Clearly, $\|T\|=1$ and $Th_{n,k}=e_k$ for $1\leq k\leq 2^n$. Let $\widehat T:\fbl[L_1]\rightarrow \ell_1^{2^n}$ denote the lattice homomorphism extending $T$. Note $\widehat T |\delta_{h_{n,k}}|=|T h_{n,k}|=e_k$, which implies that
$$
\sum_{k=1}^{2^n}|a_k|=\Big\|\sum_{k=1}^{2^n}a_k e_k\Big\|=\Big\|\widehat{T}\Big(\sum_{k=1}^{2^n}a_k |\delta_{h_{n,k}}| \Big)\Big\|\leq \Big\|\sum_{k=1}^{2^n}a_k |\delta_{h_{n,k}}|\Big\|. \qedhere
$$
\end{proof}

By a branch of the Haar basis, we mean any sequence $(h_{n_j,k_j})_{j\in \nat}$ such that, for each $j\in \mathbb N$, the support of $h_{n_{j+1},k_{j+1}}$ is contained in that of $h_{n_j,k_j}$.

\begin{lem}
For every branch $(h_{n_j,k_j})_{j\in \mathbb N}$ of the Haar basis, we have that the sequence $(|\delta_{h_{n_j,k_j}}|)_{j\in \mathbb N}$ in $\fbl[L_1]$ is equivalent to the $\ell_1$ basis.
\end{lem}

\begin{proof}
We will do the computations for $(h_{n,1})_{n\in \nat_0}$, since this can be translated to any other branch. Let $T:L_1\rightarrow \ell_1(\nat_0)$ be the norm 1 operator defined by $Tf=(\int_{2^{-k-1}}^{2^{-k}}fd\mu)_{k\in \nat_0}$, and let $\widehat T:\fbl[L_1]\rightarrow \ell_1(\nat_0)$ denote the lattice homomorphism extending $T$. Note that 
$$
(T h_{n,1})_k=
\left\{
\begin{array}{ccc}
  0 &   &\text{ if }k<n,  \\
 -\frac12 &   & \text{ if } k=n, \\
 \frac{1}{2^{k-n+1}} &   & \text{ if } k>n.
\end{array}
\right.
$$

We claim that $(\widehat T|\delta_{h_{n,1}}|)_{n\in \nat_0}$ is equivalent to the $\ell_1$ basis. Indeed, $\widehat T|\delta_{h_{n,1}}|=|T h_{n,1}|$, which coincides with the sequence $(0,\ldots,0,\frac12,\frac14,\frac18,\ldots)$, starting with $n$ zeros. Note that if $S:\ell_1\rightarrow\ell_1$ denotes the right shift, and we set $R=\sum_{n\in\mathbb N_0}\frac{1}{2^{n+1}}S^n$, then this defines an invertible operator with $R^{-1}=2I-S$. Now, notice that for the unit vector basis $(e_n)$ in $\ell_1$ we have $R e_n=\widehat T|\delta_{h_{n,1}}|$ and this proves the claim. Therefore, there is $C>0$ such that for every sequence of scalars $(a_n)_{n\in\mathbb N}$ we have
$$
C \sum_{n\in\mathbb N}|a_n| \leq\Big\|\sum_{n\in\mathbb N} a_n \widehat T|\delta_{h_{n,1}}|\Big\|_{\ell_1}\leq \|T\| \Big\|\sum_{n\in\mathbb N} a_n |\delta_{h_{n,1}}|\Big\|_{\fbl[L_1]}\leq  \sum_{n\in\mathbb N}|a_n|.
$$
This finishes the proof.
\end{proof}

\begin{rem}
We do not know whether the double indexed sequence $(|\delta_{h_{n,k}}|)$ is equivalent to the $\ell_1$ basis. However, $(h_{n,k})$ is a monotone basis in $L_1(0,1)$ (see e.g. \cite[Section 2.c]{LT2}), hence, by \Cref{p:Markushevich_in_L1} below, 
$$
\big\| \sum_{n,k} a_{n,k} \big|\delta_{h_{n,k}}\big| \big\| \geq \frac12 \sum_{n,k} a_{n,k}
$$
whenever $(a_{n,k})$ are positive scalars.
\end{rem}

\section{Sequences with prescribed moduli}\label{s:endpoints}

In this section, we continue our investigation of connections between properties of the sequence $(x_k) \subseteq E$, and those of $\big(\big|\delta_{x_k}\big|\big) \subseteq \fbp[E]$.
In \Cref{ss:unconditional}, we show how $p$-summing norms can be used to compute the norm of certain expressions on $\fbp[E]$, which will be a helpful tool in the sequel. In \Cref{ss:equiv_moduli}, we show that, under fairly general conditions, if $(x_k)$ and $\big(\big|\delta_{x_k}\big|\big)$ are equivalent, then they both are equivalent to the $\ell_1$ basis. We also give examples showing that our conditions are necessary, and characterize when the span of $(|\delta_{x_k}|)$ is complemented in $\fbp[E]$.
\Cref{ss:basic in L1} is devoted to analysing $\big(\big|\delta_{x_k}\big|\big)$ for sequences  $(x_k) \subseteq L_1$.
Finally, in \Cref{s:copies_of_l2} we investigate $(x_k)$'s for which $\big(\big|\delta_{x_k}\big|\big)$ is equivalent to the $\ell_2$ basis.

\subsection{Using linear operators to compute non-linear expressions}\label{ss:unconditional}

Throughout this subsection, we fix $p \in [1,\infty]$. 
For $\olx = (x_1, \ldots, x_n) \in E^n$, we define the operator
$$
T_\olx: E^* \to \ell_p^n : x^* \mapsto (x^*(x_k))_{k=1}^n.
$$
Note $T_\olx=S_\olx^*$ for $S_\olx$ the operator given by
$$
S_\olx : \ell_q^n \to E : e_k \mapsto x_k,
$$
where $q$ is conjugate to $p$ ($\frac{1}{p}+\frac{1}{q} = 1$) and $(e_k)_{k=1}^n$ is the canonical basis of $\ell_q^n$.

\begin{prop}\label{t:norms_versus_summing}
In the above notation,
 \begin{equation}\label{reduce to summing}
 \Bignorm{ \Big( \sum_{k=1}^n \bigabs{ \delta_{x_k} }^p \Big)^{1/p} }_{\fbp[E]} = \pi_p(T_\olx) .
 \end{equation}
\end{prop}

\begin{proof}
   Put $T=T_\olx$ and
  \begin{math}
    f=\Bigl(\sum_{k=1}^n\abs{\delta_{x_k}}^p\Bigr)^{\frac1p}.
  \end{math}
  By the definition of $\pi_p$, we have
  \begin{equation*}
      \begin{split}
           \pi_p(T)&=\sup\Biggl\{\Bigl(\sum_{i=1}^m\norm{Tx_i^*}^p\Bigr)^{\frac1p} :
    x_1^*,\dots,x_m^*\in E^*,
    \sup\limits_{x^{**}\in B_{E^{**}}}\Bigl(\sum_{i=1}^m
    \bigabs{ x^{**}(x_i^*)}^p\Bigr)^{\frac1p}\leq 1\Biggr\}
    \\
    &=\sup\Biggl\{\Bigl(\sum_{i=1}^m\norm{Tx_i^*}^p\Bigr)^{\frac1p} :
    x_1^*,\dots,x_m^*\in E^*,
    \sup\limits_{x\in B_E}\Bigl(\sum_{i=1}^m
    \bigabs{ x_i^*(x)}^p\Bigr)^{\frac1p}\leq 1\Biggr\},
      \end{split}
  \end{equation*}
  by an argument similar to \eqref{eq:ARTbidual}. For each $i$, we have
  \begin{displaymath}
    \norm{Tx_i^*}
    =\Bigl(\sum_{k=1}^n\bigabs{x_i^*(x_k)}^p\Bigr)^{\frac1p}
   =f(x_i^*),
  \end{displaymath}
  hence
  \begin{displaymath}
    \pi_p(T)
   =\sup\Biggl\{\Bigl(\sum_{i=1}^m\abs{f(x_i^*)}^p\Bigr)^{\frac1p} :
   \sup\limits_{x\in B_E}\Bigl(\sum_{i=1}^m
   \bigabs{ x_i^*(x)}^p\Bigr)^{\frac1p}\leq 1\Biggr\}
   =\norm{f}_{\FBLp[E]}.
  \end{displaymath}
\end{proof}

We now mention some corollaries. Suppose $(x_k)$ is a $1$-unconditional  basis of $E$ and $\alpha = (a_1, \ldots, a_n)\in  \Real^n$. We use the notation $\ola = (a_1 x_1 , \ldots, a_n x_n) \in E^n$. The corresponding operator $T_\ola : E^* \to \ell_p^n : x^* \mapsto (a_k x^*(x_k))_{k=1}^n$ is \emph{diagonal} with respect to the $1$-unconditional bases given by the biorthogonal functionals in $E^*$ and $\ell_p^n$ respectively. The next result shows equivalence between the problem of computing the moduli of an unconditional basis in $\fbl[E]$, and the problem of computing a certain 1-summing norm:

\begin{cor}\label{c:1-summing}
Assume that $(x_k)$ is a $1$-unconditional basis of $E$, and let the notation be as above. For $\alpha = (a_1, \ldots, a_n) \in \Real^n$,

$$\Bignorm{ \sum_{k=1}^n a_k \bigabs{ \delta_{x_k} } }_{\fbl[E]} \leq \pi_1(T_\ola) \leq 2 \Bignorm{ \sum_{k=1}^n a_k \bigabs{ \delta_{x_k} } }_{\fbl[E]} .
 $$
\end{cor}

\begin{proof}
First  apply \Cref{t:norms_versus_summing} with $x_k$ replaced by $a_k x_k$ to obtain
$$
\Bignorm{ \sum_{k=1}^n |a_k| \bigabs{ \delta_{x_k} } }_{\fbl[E]} = \pi_1 \big(T_{\olx[\alpha]}\big).
$$
Now invoke \Cref{Prop1} to get that $(|\delta_{x_k}|)$ is 2-unconditional.
\end{proof}

In particular, for the canonical basis of $\ell_r$, setting $\frac1r + \frac1{r'} = 1$ we have
 \begin{equation}\label{Summing for ellp}
\Bignorm{ \sum_{k=1}^n a_k \bigabs{ \delta_{e_k} } }_{\fbl[\ell_r]} \sim \pi_1 ({\overline{\alpha}} ) ,
{\textrm{   with   }}
{\overline{\alpha}} = {\mathrm{diag}} \big( (a_k)_{k \in \nat} \big) : \ell_{r'} \to \ell_1 .
 \end{equation}
 Summing norms of diagonal operators between $\ell_p$-spaces
  have been investigated in \cite{Garling74} (specifically, Theorems 4 and 9 -- although the latter contains some typos). Combining these results with \eqref{Summing for ellp} gives an alternative proof of some of the results from \cite{ATV}. Generally, \Cref{c:1-summing} is a useful tool for computing the moduli of unconditional bases in $\fbl[E]$, as there is a large theory concerning 1-summing norms. However, due to the $p$-sum inside the norm of \eqref{reduce to summing}, the case when $p\in (1,\infty)$ is more difficult, and, in particular, we don't know the behaviour in $\fbp[\ell_r]$ of the moduli of the basis vectors from $\ell_r$  when both $p,r\in (2,\infty)$. The case when the basis is conditional is also more difficult, as \Cref{c:1-summing} only allows us to control positive scalars, but the moduli of a conditional basis need not be unconditional (see \Cref{p:summing_basis_conditional}).
 \\

\subsection{When are $(x_k)$ and $\big(\big|\delta_{x_k}\big|\big)$ equivalent?}
\label{ss:equiv_moduli}

Our next result can be considered as a converse of \Cref{Equiv of basic} when $p\in [1,\infty)$.
\begin{prop}\label{h}
Fix $p\in [1,\infty)$ and let $(x_k)$ be a  normalized basis of $E$ such that $(|\delta_{x_k}|)$  in $\fbp[E]$ is equivalent to $(x_k)$. Then $(x_k)$ must be equivalent to the unit vector basis of $\ell_1$.
\end{prop}

\begin{proof}
By  \Cref{recall ATV}, \Cref{prop:lowerell2} and the hypothesis, we have that
$$
\Big(\sum_k a_k^2\Big)^{\frac12}\lesssim \Big\|\sum_k a_k|\delta_{x_k}|\Big\|_{\fbp[E]}\lesssim \Big\|\sum_k a_k x_k\Big\|_E.
$$
Therefore, we have a bounded map $T:E\rightarrow \ell_2$ with $T(x_k)=e_k$, where $(e_k)$ is the unit vector basis of $\ell_2$. Let $\overline T:\fbp[E]\rightarrow \fbp[\ell_2]$ be the lattice homomorphism extending $T$. By \Cref{lower 2 gives 1} it follows that
\begin{align*}
\sum_k |a_k|&\lesssim \Big\|\sum_k a_k|\delta_{e_k}|\Big\|_{\fbp[\ell_2]}\leq \|\overline T\| \Big\|\sum_k a_k|\delta_{x_k}|\Big\|_{\fbp[E]}\\
&\lesssim \Big\|\sum_k a_k x_k\Big\|_E\leq \sum_k |a_k|. \qedhere
\end{align*}
\end{proof}

The case $p=\infty$ is  completely different. 

\begin{prop}\label{infty is different}
Let $(x_k)$ be an unconditional basic sequence in a Banach space $E$. Then $(\delta_{x_k})\sim (|\delta_{x_k}|)$ in $\fbl^{(\infty)}[E]$.
\end{prop}

\begin{proof}
  By Propositions~\ref{AM works fine} and~\ref{l:unconditionality}, we may assume that $(x_k)$ is a basis and that
  $\bigl(\abs{\delta_{x_k}}\bigr)$ is unconditional. For any
  $a_1,\dots,a_n$, we have
  \begin{displaymath}
    \Bignorm{\sum_{k=1}^na_kx_k}
    =\Bignorm{\sum_{k=1}^na_k\delta_{x_k}}
    \lesssim\Bignorm{\sum_{k=1}^n\abs{a_k}\abs{\delta_{x_k}}}
    \sim\Bignorm{\sum_{k=1}^na_k\abs{\delta_{x_k}}}.
  \end{displaymath}
Fix $x^*\in B_{E^*}$ and put
$\varepsilon_k=\operatorname{sign}x^*(x_k)$; 
we then have
\begin{displaymath}
  \Bigabs{\sum_{k=1}^n a_k\abs{\delta_{x_k}}}(x^*)
  =\Bigabs{\sum_{k=1}^n \varepsilon_ka_k x^*(x_k)}
  \le\Bignorm{\sum_{k=1}^n \varepsilon_ka_kx_k}
  \sim\Bignorm{\sum_{k=1}^na_kx_k}.
\end{displaymath}
Taking sup over $x^*\in B_{E^*}$, we get
\begin{math}
  \Bignorm{\sum_{k=1}^na_k\abs{\delta_{x_k}}}
  \lesssim\Bignorm{\sum_{k=1}^na_kx_k}.
\end{math}
\end{proof}

\Cref{dual to the summing} shows that the hypothesis of unconditionality in \Cref{infty is different} cannot be removed.  If ``basis" is replaced by ``basic sequence" in  \Cref{h}, the situation, again, is dramatically different:

\begin{prop}\label{p:l2_in_CK}
 Suppose $\Omega$ is a compact Hausdorff space and $(x_k)$ is a sequence in $C(\Omega)$  equivalent to the $\ell_2$ basis.
 Then for all $p\in [1,\infty]$ the sequence $\big( \bigabs{ \delta_{x_k} } \big)$ in $\fbp[C(\Omega)]$ is equivalent to the $\ell_2$ basis.
\end{prop}
\begin{proof}
We assume that $(x_k)$ is normalized and $C$-equivalent to the $\ell_2$ basis.
\\

Fix $a_1, \ldots, a_n \in \Real$. Using \Cref{infty is different} and the fact that the $\fbl^{(\infty)}$-norm is the weakest of the $\fbp$-norms, we get
$$
\big\| \sum_k a_k\bigabs{ \delta_{x_k} } \big\|_{\fbp[C(\Omega)]} \geq \big\|\sum_k a_k \bigabs{ \delta_{x_k} } \big\|_{\fbl^{(\infty)}[C(\Omega)]} \gtrsim \big( \sum_k|a_k|^2 \big)^{1/2} .
 $$
 
  To establish the converse, it suffices to work with the $\fbl$-norm as it is the strongest of the $\fbp$-norms. Note that 
 $$
\big| \sum_k a_k \bigabs{ \delta_{x_k} } \big| \leq \sum_k |a_k| \bigabs{ \delta_{x_k} } ,
$$
hence it suffices to prove that there exists a universal constant $K$ such that
$$
\big\| \sum_k a_k\bigabs{ \delta_{x_k} } \big\|_{\fbl[C(\Omega)]} \leq K  \big( \sum_k a_k^2 \big)^{1/2} 
$$
whenever $a_1, \ldots, a_n \geq 0$.
\\

By \Cref{t:norms_versus_summing},
$\big\| \sum_k a_k \bigabs{ \delta_{x_k} } \big\| = \pi_1(T)$,
where $T : C(\Omega)^* \to \ell_1^n$ takes $\mu$ to $(a_k \mu(x_k) )_{k=1}^n$. Write $T = (j T_0)^*$, where
$$ T_0 : \ell_\infty^n \to E = \spn[x_1, \ldots, x_n] : e_k\mapsto a_kx_k , $$
and $j$ is the embedding of $E$ into $C(\Omega)$. Then $\pi_1(T) \leq \|T_0\| \pi_1(j^*)$.
As $(x_k)$ is $C$-equivalent to the $\ell_2$ basis, we obtain $\|T_0\| \leq C \big( \sum_k a_k^2 \big)^{1/2}$.
Further, the domain of $j^*$ is an AL-space, hence, by Grothendieck's Theorem, $\pi_1(j^*) \leq K_G d(E, \ell_2^n) \leq K_G C^2$ (here $d( \cdot, \cdot)$ stands for the Banach-Mazur distance).
We conclude that $\pi_1(T) \leq K_G C^3 \big( \sum_k a_k^2 \big)^{1/2}$, which is what we need.
\end{proof}

\begin{cor}\label{c:dominated_l2_in_CK}
 Suppose $\Omega$ is a compact Hausdorff space and $(y_k)$ is a sequence in $C(\Omega)$, dominated by the $\ell_2$ basis, and admitting a bounded sequence of biorthogonal functionals.
 Then for all $p\in [1,\infty)$ the sequence $\big( \bigabs{ \delta_{y_k} } \big)$ in $\fbp[C(\Omega)]$ is equivalent to the $\ell_2$ basis.
\end{cor}

The above conditions are verified, for instance, if $(y_k)$ is a semi-normalized basic sequence, dominated by the $\ell_2$ basis.

\begin{proof}
 By \Cref{prop:lowerell2}, $\big( \bigabs{ \delta_{y_k} } \big)$  dominates the $\ell_2$ basis. To establish the converse, find a basic sequence $(x_k)$ in $C(\Omega')$, equivalent to the $\ell_2$ basis. It is well-known (see, e.g.,~\cite[Section 4]{Zippin_survey}) that $C(\Omega)^{**}$ is injective, hence there exists $T \in B(C(\Omega'), C(\Omega)^{**})$ so that $Tx_k = y_k$, for every $k$. This $T$ extends to a lattice homomorphism $\overline{T} : \fbp[C(\Omega')] \to \fbp[C(\Omega)^{**}]$. By \Cref{p:l2_in_CK}, $\ell_2$ dominates $\big( \bigabs{ \delta_{y_k} } \big) \subseteq \fbp[C(\Omega)^{**}]$. To complete the proof, invoke \Cref{cor:subspace}.
\end{proof}

\begin{rem}
 In \Cref{p:l2_in_CK}, one can replace $C(\Omega)$ with an arbitrary ${\mathcal{L}}_{\infty,\lambda}$ space (the equivalence constant will then depend not only on $C$ but also on $\lambda$). Likewise, \Cref{c:dominated_l2_in_CK} works for ${\mathcal{L}}_{\infty,\lambda}$ spaces as well (combine \cite[Theorem 4.1]{lind_ros} with \cite[Theorem 4.2]{Zippin_survey}).
\end{rem}

One can also describe sequences $(x_k)$ which are equivalent to $\big(\big|\delta_{x_k}\big|\big)$  via regular operators. 

\begin{prop}\label{unconditional seq}
Let $(x_k)$ be a 1-unconditional basic sequence in a Banach space $E$, and view $[x_k]$ as a Banach lattice with the coordinate order induced by the basis. Let $j:[x_k]\to \fbp[E]$ be the canonical inclusion $x_k\mapsto \delta_{x_k}$. 
Then the following are equivalent:
\begin{enumerate}[label=\roman*.]
\item $(|\delta_{x_k}|)\sim (x_k)$;
\item $j$ is regular;
\item $j$ is pre-regular.
\end{enumerate}
\end{prop}

Recall that an operator $T$ between Banach lattices is called \emph{pre-regular} if $T^*$ is regular; we set $\|T\|_{pre-reg} = \|T^*\|_r$ (the regular norm). For more information on this class of operators, the reader is referred to \cite[Section 4]{Dales:17}. This class also coincides with that of $(1,1)$-regular operators considered in \cite{SanPe-Tra}.

\begin{proof}
(i)$\Rightarrow$(ii): Consider the linear extension of the map $T_+(x_k)=(\delta_{x_k})_+$ and $T_-(x_k)=(\delta_{x_k})_-$. These maps have well-defined extensions to $[x_k]$ because
\begin{align*}
\bigg\|T_+\left(\sum_{k=1}^na_kx_k\right)\bigg\|_{\fbp[E]}
&
=
\bigg\|\sum_{k=1}^na_k(\delta_{x_k})_+\bigg\|_{\fbp[E]}\leq \bigg\|\sum_{k=1}^n|a_k\delta_{x_k}|\bigg\|\\
&
\sim
\bigg\|\sum_{k=1}^n|a_k|x_k\bigg\|\sim \bigg\|\sum_{k=1}^na_kx_k\bigg\|,
\end{align*}
and a similar estimate holds for $T_-$. Clearly these extensions are positive and $j=T_+-T_-$.
\\

(ii)$\Rightarrow$(iii) is trivial.
\\

(iii)$\Rightarrow$(i): By \Cref{p:dominations_abs_values}, $(|\delta_{x_k}|)$ dominates $(x_k)$. On the other hand,

$$\left|\sum_{k=1}^na_k|\delta_{x_k}|\right|\leq \sum_{k=1}^n|a_k\delta_{x_k}|=\sum_{k=1}^n|j(a_kx_k)|.$$
By \cite[Theorem 4.40]{Dales:17} we conclude that
$$\bigg\|\sum_{k=1}^na_k|\delta_{x_k}|\bigg\|\leq \|j\|_{pre-reg}\bigg\|\sum_{k=1}^n|a_kx_k|_{[x_k]}\bigg\|= \|j\|_{pre-reg}\bigg\|\sum_{k=1}^na_kx_k\bigg\|.$$

Here, $| \cdot |_{[x_k]}$ denotes the modulus in ${\overline{\mathrm{span}}}[x_k : k \in \Nat]$, arising from the order determined by the unconditional basis $(x_k)$.
\end{proof}

We now answer the question of when $\overline{\spn}[|\delta_{x_k}| : k \in \nat]$ is complemented in $\fbl[E]$:

\begin{cor}\label{NOT comp}
Let $E$ be a Banach space with a normalized unconditional basis $(x_k)$. The following statements are equivalent:
\begin{enumerate}
\item $\overline{\spn}[|\delta_{x_k}| : k \in \nat]$ is complemented in $\fbl[E]$;
\item There is an unconditional sequence of biorthogonal functionals $(u_k^*)$ to  $(\bigabs{ \delta_{x_k} })\subseteq \fbl[E]$ such that $\overline{\spn}[|\delta_{x_k}| : k \in \Nat]$ is normed by $\overline{\spn}[u_k^* : k \in \Nat]$;
\item  $(x_k)$ is equivalent to the $\ell_1$ basis. 
\end{enumerate} 
\end{cor}

Here, for a Banach space $E$, and subspaces $F\subseteq E$, $G\subseteq E^*$, we say that \textit{$F$ is normed by $G$} if there is $K>0$ such that for every $x\in F$, there is $x^*\in G$ with $\|x^*\|=1$ and $|x^*(x)|\geq \|x\|/K$. 

\begin{proof}
By \Cref{Prop1}, $\big( \big| \delta_{x_k} \big| \big)$ is an unconditional basic sequence in $\fbl[E]$. The implication $(1)\Rightarrow (2)$ is clear.
\\

$(2)\Rightarrow (3)$: If $(u_k^*)$ is an unconditional sequence of biorthogonal functionals to $(\bigabs{ \delta_{x_k} })\subseteq \fbl[E]$ such that $\overline{\spn}[|\delta_{x_k}| : k \in \Nat]$ is normed by $\overline{\spn}[u_k^* : k \in \Nat]$, then the argument in the proof of \cite[Theorem 1.d.6(ii)]{LT2} would go through and we would have that for any scalars $(a_k)_{k=1}^n$, 
$$\bignorm{ \sum_{k=1}^n a_k|\delta_{x_k}|} \sim \bignorm{ \big( \sum_{k=1}^n |a_k\delta_{x_k}|^2 \big)^{1/2} }.$$
On the other hand, applying \cite[Theorem 1.d.6]{LT2} (and its proof) to $(\delta_{x_k})$ yields 
$$ \bignorm{ \big( \sum_{k=1}^n |a_k\delta_{x_k}|^2 \big)^{1/2} } \lesssim\bignorm{ \sum_{k=1}^n a_k\delta_{x_k}}.$$ 
Thus, for any scalars $(a_k)_{k=1}^n$, 
\begin{align*}
    \bignorm{ \sum_{k=1}^n a_k\delta_{x_k}}&\leq \bignorm{ \sum_{k=1}^n |a_k| |\delta_{x_k}|}\sim \bignorm{ \sum_{k=1}^n a_k|\delta_{x_k}|}\\
    &\sim \bignorm{ \big( \sum_{k=1}^n |a_k\delta_{x_k}|^2 \big)^{1/2} } \lesssim\bignorm{ \sum_{k=1}^n a_k\delta_{x_k}}.
\end{align*}
Hence $(x_k)$ and $(|\delta_{x_k}|)$ are equivalent. By \Cref{h} $(x_k)$ is equivalent to the unit vector basis of $\ell_1$. 
\\

$(3)\Rightarrow (1)$: Suppose $(e_k)$ is the unit vector basis of $\ell_1$; we will show that $(|\delta_{e_k}|)$ is complemented in $\fbl[\ell_1]$. Denote by $(e_k^*)$ the canonical sequence in $\ell_\infty=\ell_1^*$ biorthogonal to $(e_k)$. Define the map $T:\fbl[\ell_1]\to \ell_1$ by $f\mapsto (f(e_j^*))$. From the definition of the $\fbl[\ell_1]$ norm, $T$ is contractive. Also, $T(|\delta_{e_k}|)=e_k.$ Since $(|\delta_{e_k}|)\sim (e_k)$ we can define $S:\phi(\ell_1)\to \overline{\spn}[|\delta_{e_k}| : k \in \Nat] : \delta_{e_k}\mapsto |\delta_{e_k}|$. Then $S\circ \phi\circ T :\fbl[\ell_1]\to \overline{\spn}[|\delta_{e_k}| : k \in \Nat]$ is our desired projection.
\end{proof}

For $p \in (1,\infty)$, the situation is considerably different. 
\begin{prop}\label{p:cannot be complemented}
  Let $E$ be a Banach space with an unconditional basis $(x_k)$. Then $\overline{\spn}[|\delta_{x_k}| : k \in \nat]$ is not complemented in $\fbp[E]$ for any $p\in(1,\infty)$.
\end{prop}

\begin{proof}
Without loss of generality, $(x_k)$ is normalized. Suppose, for the sake of contradiction, that $\overline{\spn}[|\delta_{x_k}| : k \in \Nat]$ is complemented in $\fbp[E]$.
As in the proof of \Cref{NOT comp}, we conclude that $(x_k)$ and $(|\delta_{x_k}|)$ should be equivalent to the $\ell_1$ basis. By \cite[Theorem 1.d.7 and the remark after]{LT2}, this is a contradiction with the $p$-convexity of $\fbp[E]$.
\end{proof}

Note that \Cref{p:cannot be complemented} fails for $p=\infty$, as the moduli of the $c_0$ basis will be complemented in $\fbl^{(\infty)}[c_0]$. Moreover, the converse holds as well:

\begin{prop}\label{p:moduli in fbl infty}
Let $(x_k)$ be a semi-normalized unconditional basis of a Banach space $E$. Then $(|\delta_{x_k}|)$ is complemented in $\fbl^{(\infty)}[E]$ if and only if $(x_k)\sim c_0$.
\end{prop}
\begin{proof}
Since $\fbl^{(\infty)}[E]$ is generated by $E$ as a lattice, it is separable. By \Cref{infty is different}, $(|\delta_{x_k}|)\sim (x_k)$. Hence, if $(x_k)\sim c_0$, then $(|\delta_{x_k}|)\sim c_0$, and Sobczyk's theorem applies.
\\

 For the converse, we note that by \cite[p.~74]{DJT} the only complemented semi-normalized unconditional basic sequences in $\mathcal{L}_\infty$-spaces are those equivalent to the unit vector basis of $c_0$. Since AM-spaces are $\mathcal{L}_\infty$, it follows that if $(|\delta_{x_k}|)$ is complemented in $\fbl^{(\infty)}[E]$ then $(|\delta_{x_k}|)\sim c_0$. Hence, if $(x_k)$ is not $c_0$, then by \Cref{infty is different}, $(|\delta_{x_k}|)$ is not $c_0$, so $(|\delta_{x_k}|)$ is not complemented in $\fbl^{(\infty)}[E]$.
\end{proof}
In the above three results (\Cref{NOT comp}-\Cref{p:moduli in fbl infty}) we assumed that $(x_k)$ was an unconditional basis; it is unclear how to characterize complementation of $(|\delta_{x_k}|)$ when the unconditionality assumption on $(x_k)$ is dropped.

\begin{rem}\label{r:positive parts vs moduli}
Throughout this and the previous section we have dealt with the sequence of moduli of a basis. As noted in \cite{ATV} one could consider other lattice expressions; for example, $((\delta_{x_k})_+)$ or $((\delta_{x_k})_-)$. Moreover, in \cite{ATV} it is shown that $((\delta_{x_k})_+)$ and $((\delta_{x_k})_-)$ are 1-equivalent to each other. For their relation to $(|\delta_{x_k}|)$, note that 
$$\bigg\|\sum_{k=1}^na_k|\delta_{x_k}|\bigg\|=\bigg\|\sum_{k=1}^na_k(\delta_{x_k})_++\sum_{k=1}^na_k(\delta_{x_k})_-\bigg\|\leq 2\bigg\|\sum_{k=1}^na_k(\delta_{x_k})_+\bigg\|.$$
This shows that $(|\delta_{x_k}|)\lesssim ((\delta_{x_k})_+).$ The converse domination is easily seen to be true for unconditional bases. 
\end{rem}

\subsection{Basic sequences in $\fbl[L_1]$}\label{ss:basic in L1}

Let us  say that a sequence $(x_k)$ in a Banach space $E$ is \emph{$C$-minimal} if it admits biorthogonal functionals of norm not exceeding $C$; $(x_k)$ is said to be \emph{uniformly minimal} if it is $C$-minimal for some $C$.

\begin{prop}\label{p:Markushevich_in_L1}
 Suppose $(x_k)$ is a normalized $C$-minimal sequence in $L_1$. Then for all $n\in \mathbb{N}$ and $a_1, \dots,a_n \geq 0$, we have
 $$
 \sum_{k=1}^n a_k \geq \Big\| \sum_{k=1}^n a_k \bigabs{ \delta_{x_k} } \Big\|_{\fbl[L_1]} \geq \frac1C \sum_{k=1}^n a_k .
 $$
\end{prop}

\begin{proof}
 The left hand side follows from the triangle inequality. \Cref{t:norms_versus_summing} states that
 $$
 \Big\| \sum_{k=1}^n a_k \bigabs{ \delta_{x_k} } \Big\| = \pi_1(T) ,
 {\textrm{   with   }}  T : L_\infty \to \ell_1^n: f \mapsto \sum_k a_k f(x_k) e_k
 $$
 (here $f(x_k)$ represents the dual action of $L_\infty$ on $L_1$, and $e_1, \ldots, e_n$ form the canonical basis of $\ell_1^n$).
 The domain of $T$ is $L_\infty$, hence $\pi_1(T) = \iota_1(T)$ (the integral norm, cf.~\cite[Corollary 5.8]{DJT}).
By the trace duality \cite[Theorem 6.16]{DJT},
$$
\iota_1(T) = \sup \big\{ {\mathrm{tr}}(T S) : S \in B(\ell_1^n, L_\infty) , \|S\| \leq 1 \big\} .
$$
By $C$-minimality, there exist $f_1, \ldots, f_n \in L_\infty$ so that $f_j(x_k) = \delta_{jk}$ (the Kronecker's delta), and $\max_j \|f_j\| \leq C$. Then $T f_j = a_j e_j \in \ell_1^n$. The operator $S : \ell_1^n \to L_\infty : e_j \mapsto C^{-1} f_j$ is contractive, and $T S e_j = C^{-1} a_j e_j$. Thus, $\pi_1(T) \geq {\mathrm{tr}}(T S) = C^{-1} \sum_j a_j$, as desired.
\end{proof}

We can now easily deduce the following:

\begin{cor}\label{c:unconditional_moduli}
Suppose $(x_k)$ is a normalized uniformly minimal sequence (in particular, a normalized basic sequence) in $L_1$.
 Then the sequence $\big( \bigabs{ \delta_{x_k} } \big)$ is unconditional in $\fbl[L_1]$ if and only if it is equivalent to the $\ell_1$ basis.
\end{cor}

\Cref{c:unconditional_moduli} may be useful for answering \Cref{Open problem on basic} in the special case that $E=L_1$ and $p=1$. Here is another particular case when the above corollary is applicable:

\begin{cor}\label{c:indep_rv}
 If $(x_k)$ is a sequence of normalized independent random variables in $L_1$, then $\big( \bigabs{ \delta_{x_k} } \big)$ is equivalent to the $\ell_1$ basis in $\fbl[L_1]$.
\end{cor}

\begin{proof}[Sketch of a proof]
 For $A \subseteq \nat$, denote by $\Sigma_A$ the $\sigma$-algebra generated by the random variables $\{x_k\}_{k \in A}$. Let $Q_A$ be the conditional expectation from $L_1$ onto $L_1(\Sigma_A)$.
 Clearly $Q_A$ is contractive, and $Q_A x_k = x_k$ if $k \in A$, $Q_A x_k = 0$ otherwise.
 Thus, $Q_A$ generates a contractive lattice homomorphism (denoted by $\overline{Q_A}$) on $\fbl[L_1]$, with $\overline{Q_A} \bigabs{ \delta_{x_k} } = \bigabs{ \delta_{x_k} }$ if $k \in A$, and $0$ otherwise -- that is, the sequence  $\big( \bigabs{ \delta_{x_k }} \big)$ is $1$-suppression unconditional.
\end{proof}

\begin{rem}\label{Still ell_1...}
In \cite{ATV} it was shown that the sequence $(|\delta_{e_k}|)$ in $\fbl[\ell_2]$ is equivalent to the standard $\ell_1$ basis. In contrast to \Cref{prop:nonisomorphic}, \Cref{c:indep_rv} shows that the sequence $\bigl(\abs{\delta_{r_k}}\bigr)$ ($r_k$ are independent Rademachers) in $\fbl[L_1[0,1]]$ is equivalent to the standard $\ell_1$ basis. 
\Cref{p:dominations_abs_values} implies that the sequence $(|\delta_{r_k}|)$ in $\fbp[L_\infty[0,1]]$ is equivalent to the standard $\ell_1$ basis for all $p\in [1,\infty]$, but we don't know how $(|\delta_{r_k}|)$ behaves in $\fbp[L_1[0,1]]$ for $p\in (1,\infty)$. For $p=\infty$ we obtain a copy of $\ell_2$, see \Cref{infty is different}.
\end{rem}

Normalized Haar functions form another notable sequence in $L_1$. These were investigated in \Cref{subsec:Haar}, above.

\subsection{Creating copies of $\ell_2$ in free Banach lattices}\label{s:copies_of_l2}

In earlier sections we proved several results in which we assumed knowledge about a basis $(x_k)$ of $E$ and analysed the behaviour of the basic sequence $(|\delta_{x_k}|)$ in $\fbp[E]$.
In particular, \Cref{c:c_0 basic sequence} shows that, if $(x_k) \subseteq E$ is equivalent to the $c_0$ basis, then $(|\delta_{x_k}|) \subseteq \fbp[E]$ is equivalent to the $\ell_2$ basis.
It is also of interest to study the opposite question.  Specifically, suppose we have an (unconditional) basis $(x_k)$ of $E$ and we know the behaviour of $(|\delta_{x_k}|)$ in $\fbp[E]$. Can we deduce from this the behaviour of $(x_k)$? When $(|\delta_{x_k}|)$ is equivalent to the $\ell_2$ basis, the answer is yes:

\begin{thm}\label{t:l2_created_by_c0}
Suppose $E$ is a Banach space with an unconditional basis $(x_k)$, such that the sequence $\big( \bigabs{ \delta_{x_k} } \big)$ in $\fbp[E]$ (for some $1 \leq p < \infty$) is equivalent to the $\ell_2$ basis. Then $(x_k)$ is equivalent to the $c_0$ basis. 
 \end{thm}

The rest of the subsection is devoted to proving this result. First we fix some notation and conventions. 
\\

Since $(|\delta_{x_k}|)$ is equivalent to the $\ell_2$ basis, it is semi-normalized. Renorming if necessary, we can and do assume that the basis $(x_k)$ is normalized and $1$-unconditional in $E$, so that $E$ is a Banach lattice when equipped with the order induced by the basis. In view of \Cref{prop:lowerell2} and \cite[Theorem 5]{ATV}, what we are really assuming is existence of a $C > 0$ so that, for any $a_k \in \Real$,
\begin{equation}
\Big\| \sum_k a_k \bigabs{ \delta_{x_k} } \Big\| \leq C \big( \sum_k a_k^2 \big)^{1/2}.
 \label{eq:equivalence_to_l2}
\end{equation}

We shall denote by $(x_k^*)$ the corresponding biorthogonal functionals in $E^*$, and let $E'$ be the subspace of $E^*$ spanned by them. \\

Denote by $(e_k)$ the canonical basis on a space $\ell_p$. If $Y$ and $Z$ are spaces with semi-normalized unconditional bases $(y_k)$ and $(z_k)$ respectively, and $\alpha = (a_1, a_2, \ldots)\in \ell_p$, we denote by $\Delta_\alpha : Y \to Z$ the diagonal operator which takes $y_k$ to $a_k z_k$.
Further, $\|\alpha\|_p$ refers to the $\ell_p$ norm of $\alpha$ like this.
\\

Abusing the notation slightly, we often identify finitely supported elements of these spaces with their sequence representation. For instance, we identify $f = \sum_k b_k x_k^* \in E'$ with $\sum_k b_k e_k \in \ell_p$, and use $\|f\|_{E'}$ and $\|f\|_p$ as a shorthand for $\|\sum_k b_k x_k^*\|_{E'}$ and $\|\sum_k b_k e_k\|_p$, respectively.
\\

For the rest of this section, we use the notation introduced above ($E, (x_k), C, \ldots$); the proof of \Cref{t:l2_created_by_c0} begins with a  lemma:

\begin{lem}\label{l:l2_to_X'}
 For any $\alpha = (a_k) \in c_{00}$, we have
 $$
 \big( \sum_k a_k^2 \big)^{1/2} \sim \big\| \Delta_\alpha : E \to \ell_2 \big\|.
 $$
\end{lem}

\begin{proof}
 By unconditionality, we can assume that $a_k \geq 0$ for any $k$.
 We clearly have
 $\big( \sum_k a_k^2 \big)^{1/2} \geq \big\| \Delta_\alpha : E \to \ell_2 \big\|$
 (compare with $\Delta_\alpha : c_0 \to \ell_2$), hence we only need to show that
 $$
 \big( \sum_k a_k^2 \big)^{1/2} \lesssim \big\| \Delta_\alpha : E \to \ell_2 \big\|.
 $$
 To establish this, consider the isomorphic embedding $J : \ell_2 \to L_p$ induced by the Rademacher functions $(r_k)$, i.e., $Je_k = r_k$. Then $J \Delta_\alpha$ has a lattice homomorphic extension $T=\widehat{J \Delta_\alpha} : \fbp[E] \to L_p$, with $\|T\| \leq \|J \Delta_\alpha\| \leq \|J\| \|\Delta_\alpha\|$. Clearly $T |\delta_{x_k}| = a_k |r_k| = a_k 1$. Let now $u = \sum_k a_k |\delta_{x_k}|$, then $\|u\| \sim \big( \sum_k a_k^2 \big)^{1/2}$, while $\|Tu\| = \sum_k a_k^2$, which implies $\big( \sum_k a_k^2 \big)^{1/2} \lesssim \|T\|$.
 Then
 $$
 \big\| \Delta_\alpha : E \to \ell_2 \big\| \geq \frac{\|T\|}{\|J\|} \gtrsim  \big( \sum_k a_k^2 \big)^{1/2} .
 \qedhere
 $$
\end{proof}

\begin{proof}[Proof of \Cref{t:l2_created_by_c0}]
By \Cref{l:l2_to_X'}, there exists a constant $C$ so that the inequality $\| \Delta_\alpha : \ell_2 \to E'\| \geq C^{-1} \big(\sum_k |a_k|^2 \big)^{1/2}$ holds for any $\alpha = (a_k) \in c_{00}$.
Suppose, for the sake of contradiction, that the formal identity $\ell_1 \to E' : e_k\mapsto x_k^*$ is not bounded below.
Then there exists a finitely supported $f = (f_k) \in E'_+$ so that $\|f\|_1 = \sum_k f_k = 1$, yet $\|f\|_{E'} < C^{-2}$. 
By \Cref{l:l2_to_X'}, $\Delta_{\sqrt{f}}$ (where $\sqrt{f} = (\sqrt{f_k})$) defines an operator from $\ell_2$ to $E'$, with norm at least $C^{-1}$.
Consequently, we can find a finitely supported norm one $g = (g_k) \in \big( \ell_2 \big)_+$, so that $$
\big\|\Delta_{\sqrt{f}} g\big\|_{E'} = \big\| \big(\sqrt{f_k} g_k\big) \big\|_{E'} \geq C^{-1} .
$$
However, using \cite[Proposition 1.d.2]{LT2} we conclude that
$$
\big\| \big(\sqrt{f_k} g_k\big) \big\|_{E'} \leq
\big\| \big(f_k\big) \big\|_{E'}^{1/2} \big\| \big(g_k^2\big) \big\|_{E'}^{1/2} <
C^{-1} \|g_k^2\|_1^{1/2}= C^{-1} ,
$$
which yields the desired contradiction. Thus, $(x_k^*)$ is equivalent to the $\ell_1$ basis, which implies that $(x_k)$ is equivalent to the $c_0$ basis.
\end{proof}

\begin{rem}
\Cref{p:l2_in_CK} shows that \Cref{t:l2_created_by_c0} fails if $(x_k)$ is assumed to be not a basis, but merely an unconditional basic sequence, in $E$. We do not know whether the unconditionality assumption in \Cref{t:l2_created_by_c0} can be dropped.
\end{rem}

A general question in this direction is:
\begin{question}
Given a basis $(x_k)$, does there exists a basis $(y_k)$ (possibly with some nice additional properties) such that $(x_k)\sim (|\delta_{y_k}|)\subseteq\fbp[\overline{\spn}[y_k]]$? If so, classify/analyse such $(y_k)$.
\end{question}

\section{Bibasic and absolute sequences in free Banach lattices}\label{s:bibasic}
Recall that a sequence of non-zero vectors $(x_k)$ in a Banach lattice is 
\textit{bibasic} if there exists a constant $M\geq 1$ such that for every
$m\in\mathbb N$ and  scalars $a_1,\dots,a_m$, the
following \emph{bibasis inequality} is satisfied:
\begin{equation}\label{bbi}
\bigg \|\bigvee\limits_{n=1}^m\left|\sum\limits_{k=1}^na_kx_k\right|\bigg\|
  \leq M \bigg
  \|\sum_{k=1}^m a_kx_k\bigg\|.
\end{equation}
The least value of the constant $M$ is called the \textit{bibasis constant} of $(x_k)$. Clearly, every bibasic sequence is basic. Indeed, to arrive at the bibasis inequality \eqref{bbi}, one begins with the usual basis inequality
\begin{displaymath}
\bigvee\limits_{n=1}^m\bigg \|\sum\limits_{k=1}^na_kx_k\bigg\|
  \leq K \bigg
  \|\sum_{k=1}^m a_kx_k\bigg\|,
\end{displaymath}
and brings the supremum inside the norm. In general, a basic sequence need not be bibasic; however, this implication does hold in AM-spaces. For further details on bibasic sequences and their equivalent characterizations we refer the reader to~\cite{Gumenchuk:15,TT}. The importance of bibasic sequences stems from  two places. The first is that there are several natural examples, including martingale difference sequences in $L_p(P)$ ($P$ a probability measure and $p>1$), the Walsh basis, unconditional blocks of the Haar in $L_1[0,1]$, and the trigonometric basis. The second important fact is \cite[Theorem 2.1]{TT}, which shows that several forms of convergence are equivalent for bibasic sequences. To set notation, for a basic sequence $(x_k)$, we let $P_n:[x_k]\to[x_k]$ be the n-th canonical basis projection. Here, $[x_k]$ denotes the closed span of $(x_k)$, and for $x=\sum_{k=1}^\infty a_kx_k$ we have $P_nx:=\sum_{k=1}^n a_kx_k$. By definition, $P_nx\xrightarrow{\| \cdot \|} x$.

\begin{thm}[Bibasis Theorem]\label{bibasis theorem}
Let $X$ be a Banach lattice  and $(x_k)$ a basic sequence in $X$. The following statements are equivalent:
\begin{enumerate}
    \item $(x_k)$ is bibasic;
    \item For all $x\in [x_k]$, $P_nx\xrightarrow{u}x$;
    \item For all $x\in [x_k]$, $P_nx\xrightarrow{o}x$;
    \item For all $x\in [x_k]$, $(P_nx)$ is order bounded in $X$;
    \item For all $x\in [x_k]$, $(\bigvee_{n=1}^m|P_nx|)_{m=1}^\infty$ is norm bounded.
\end{enumerate}
\end{thm}

Above, we use several modes of convergence. The norm convergence is denoted by $\xrightarrow{\|\cdot\|}$, while $\xrightarrow{u}$ and $\xrightarrow{o}$ stand for the uniform and order convergence respectively. Specifically, $z_k\xrightarrow{u}0$ if there exists $e \geq 0$ with the property that for every $\varepsilon > 0$ there exists $N$ such that $|z_k| \leq \varepsilon e$ for any $k \geq N$. The condition $z_k\xrightarrow{o}0$ is significantly weaker: There exists a net $(y_\alpha)$, decreasing to $0$, with the property that for every $\alpha$ there exists $N$ such that $|z_k| \leq y_\alpha$ for any $k \geq N$. The reader is referred to e.g.~\cite{TT} for more details.

\subsection{A subspace of a Banach lattice without bibasic sequences}\label{ss:without_bibasic}

In this subsection we show that the unit vector basis of $c_0$ is not bibasic in $\fbp[c_0]$ for finite $p$, and we use this to answer a question from \cite{TT}, by exhibiting a subspace of a Banach lattice without a bibasic sequence. Let us begin with the following observation:

\begin{lem}\label{Everywhere}
  Let $(x_k)$ be a basis of a Banach space $E$. If $(\delta_{x_k})$ is bibasic  in $\fbp[E]$, then $(x_k)$ is also bibasic in any $p$-convex Banach lattice where $E$ linearly isomorphically embeds.
\end{lem}

\begin{proof}
Let $X$ be a $p$-convex Banach lattice and $(y_k)$ a basic sequence in $X$ equivalent to $(x_k)$. Then there is a linear isomorphic embedding $T:E\to X$ with $Tx_k=y_k$. Extend this map to a lattice homomorphism $\widehat{T}:\fbp[E]\rightarrow X$. Then 
$$\bigg\|\bigvee_{n=1}^m\left|\sum_{k=1}^na_ky_k\right|\bigg\|=\bigg\|\widehat{T}\bigvee_{n=1}^m\left|\sum_{k=1}^na_k\delta_{x_k}\right|\bigg\|\leq \|\widehat{T}\| \bigg\|\bigvee_{n=1}^m\left|\sum_{k=1}^na_k\delta_{x_k}\right|\bigg\|\leq$$
$$ \|\widehat{T}\|M
\bigg\|\sum_{k=1}^ma_kx_k\bigg\|\leq \|T\|\|T^{-1}\|M^{(p)}(X)M\bigg\|\sum_{k=1}^m a_ky_k\bigg\|,$$
where $M$ is the bibasis constant of $(\delta_{x_k})$ in $\fbp[E]$. 
\end{proof}

\begin{rem}
For general $p$, it is therefore of interest to know which normalized bases $(x_k)$ of $E$ are such that $(\delta_{x_k})$ is bibasic  in $\fbp[E]$ - when $p=\infty$ this is true for every basis $(x_k)$. 
\end{rem}

\begin{thm}\label{thm:c0nobibasic}
The canonical copy of the $c_0$ basis $(\delta_{e_k})$ in $\fbp[c_0]$ is not bibasic as long as $1\leq p<\infty$.
\end{thm}
\begin{proof}

Suppose, for the sake of contradiction, that $(\delta_{e_k})$
is bibasic in $\fbp[c_0]$ with bibasis constant $M$. Fix $m$. Let $H_m$ be the
$m\times m$ Hilbert matrix defined by
\begin{displaymath}
  H_m=
  \begin{bmatrix}
    \frac{1}{m-1} & \frac{1}{m-2} & \dots & 1 & 0 \\
    \frac{1}{m-2} & \frac{1}{m-3} & \dots & 0 & -1 \\
    \hdotsfor{5} \\
    1 & 0 & \dots & -\frac{1}{m-3} & -\frac{1}{m-2}\\
    0 & -1 & \dots & -\frac{1}{m-2} & -\frac{1}{m-1}\\
  \end{bmatrix}
\end{displaymath}
Here the $(i,j)$-th entry is $\frac{1}{m+1-i-j}$ when $i+j\neq m+1$ and
zero otherwise. We view $H_m$ as an operator from $\ell_\infty^m$ to
$\ell_p^m$. Clearly, $H_m^+$ has the same upper-left quadrant as $H_m$ but
zeros in the lower-right quadrant. In the proof of Proposition~1.2
in~\cite{Kwapien:70}, it is shown that
$\norm{H_m^+}\ge C\ln m\norm{H_m}$, where $C$ is an absolute constant. 
\\

Identifying $\ell_\infty^m$ with $\spn[e_1,\dots,e_m]$ in $c_0$, we
extend $H_m$ to an operator from $c_0$ to $\ell_p^m$ by setting
$H_me_k=0$ whenever $k>m$. Let
$\widehat{H_m}\colon \fbp[c_0]\to\ell_p^m$ be the extension of $H_m$ to
a lattice homomorphism with
$\|\widehat{H_m}\|=\|H_m\|$. Applying the bibasis inequality to
$\delta_{e_k}$'s with $a_1=\dots=a_m=1$ and using the fact
that $H_me_k=\widehat{H_m}\delta_{e_k}$ for all $k$, we get
\begin{multline*}
 \bigg\|\bigvee_{n=1}^m\left|\sum_{k=1}^nH_me_k\right|\bigg\|
   =\bigg\|\bigvee_{n=1}^m\left|
     {\sum_{k=1}^n\widehat{H_m}\delta_{e_k}}\right|\bigg\|
   =\biggnorm{\widehat{H_m}\biggl(\bigvee\limits_{n=1}^m\Bigabs
     {\sum\limits_{k=1}^n\delta_{e_k}}\biggr)}
   \le\norm{\widehat{H_m}}
   \biggnorm{\bigvee\limits_{n=1}^m\Bigabs
     {\sum\limits_{k=1}^n\delta_{e_k}}}\\
   \le M\norm{H_m}\biggnorm{\sum\limits_{k=1}^m\delta_{e_k}}
   =M\norm{H_m}\biggnorm{\sum\limits_{k=1}^m e_k}
   =M\norm{H_m}.
\end{multline*}
Fix $j \leq m$. Then, clearly,
\begin{displaymath}
  \bigvee\limits_{n=1}^m\left|\sum\limits_{k=1}^nH_me_k\right|\geq
  \left|\sum\limits_{k=1}^{m-j}H_me_k\right|.
\end{displaymath}
The $j$-th entry of the vector on the right hand side is
\begin{math}
  1+\frac12+\dots+\frac{1}{m-j}.
\end{math}
This number is also the $j$-th entry of $H_m^+\mathbbm{1}$. It follows that
\begin{math}
  \bigvee_{n=1}^m\left|\sum_{k=1}^nH_me_k\right|\geq H_m^+\mathbbm{1},
\end{math}
so that
\begin{displaymath}
 \biggnorm{\bigvee\limits_{n=1}^m\Bigabs{\sum\limits_{k=1}^nH_me_k}}
  \geq\norm{H_m^+\mathbbm{1}}=\norm{H_m^+}\geq C\ln m\norm{H_m},
\end{displaymath}
which is a contradiction because $m$ is arbitrary.
\end{proof}

It was asked in \cite[Remark 4.4]{TT} whether every subspace of a Banach lattice contains a bibasic sequence. We next provide a negative answer to this question:
\begin{thm}\label{t:c0_no_bibasis}
The subspace $\phi(c_0)$ in $\fbp[c_0]$ does not contain any bibasic sequence as long as $1\leq p<\infty$.
\end{thm}

In particular, for every finite $p$ there exists a $p$-convex Banach lattice containing a subspace without any bibasic sequence. This result is sharp, since, as noted above, any basic sequence in an AM-space is bibasic.

\begin{proof}
  Suppose the contrary, that there is a sequence $(y_k)$ in $c_0$ such that $(\delta_{y_k})$ is bibasic in $\fbp[c_0]$. Let $(x_k)$ be a block sequence of $(y_k)$ which is equivalent to the $c_0$ basis, and complemented in $c_0$ (cf.~\cite[Propositions 1.a.12 and 2.a.2]{LT1}).
 As a blocking of $(\delta_{y_k})$, $(\delta_{x_k})$ is bibasic in $\fbp[c_0]$ by \cite[Corollary 4.1]{TT}. 
\\

Let $P$ be a projection from $c_0$ onto $E=\overline{\spn}[x_k:k \in \Nat]$. By the results of \Cref{s:subspace_problem}, for every finite sequence $(a_k)$, and every $m \in \Nat$, we have
\begin{align*}
\Big\|\bigvee_{n=1}^m\Big|\sum_{k=1}^{n} a_k \delta_{x_k}\Big|\Big\|_{\fbp[c_0]} 
& \leq
\Big\|\bigvee_{n=1}^m\Big|\sum_{k=1}^{n} a_k \delta_{x_k}\Big|\Big\|_{\fbp[E]}
\\ & \leq
\|P\| \Big\|\bigvee_{n=1}^m\Big|\sum_{k=1}^{n} a_k \delta_{x_k}\Big|\Big\|_{\fbp[c_0]} .
\end{align*}

Define now an isomorphism $T: E\rightarrow c_0 : x_k \mapsto e_k$. Then $\overline{T}$ is a lattice isomorphism, and, for $(a_k)$ and $m$ as above,
\begin{align*}
\Big\|\bigvee_{n=1}^m\Big|\sum_{k=1}^{n} a_k \delta_{e_k}\Big|\Big\|_{\fbp[c_0]} 
& =
\Big\| \overline{T} \Big(\bigvee_{n=1}^m \Big|\sum_{k=1}^{n} a_k \delta_{x_k}\Big| \Big) \Big\|_{\fbp[c_0]} 
\\ & \leq
\|T\| \Big\|\bigvee_{n=1}^m\Big|\sum_{k=1}^{n} a_k \delta_{x_k}\Big|\Big\|_{\fbp[E]}
\\
& \leq \|P\|\|T\| \Big\|\bigvee_{n=1}^m\Big|\sum_{k=1}^{n} a_k \delta_{x_k}\Big|\Big\|_{\fbp[c_0]}.
\end{align*}
Now, using that $(\delta_{x_k})$ is bibasic in $\fbp[c_0]$ and equivalent to the $c_0$ basis, it follows easily from the above that $(\delta_{e_k})$ is as well. This contradicts \Cref{thm:c0nobibasic}.
\end{proof}

\begin{rem}\label{r:l_q_no_bibabsis}
The above argument shows that the subspace $\phi(\ell_q)$ of $\fbp[\ell_q]$ does not contain a bibasic sequence, as long as the canonical copy of the $\ell_q$ basis is not bibasic in $\fbp[\ell_q]$. For what values of $p$ and $q$ is the latter condition satisfied? By Theorem~\ref{thm:c0nobibasic}, we know the answer for $c_0$. Furthermore, the canonical copy of the $\ell_2$ basis $(\delta_{e_k})$ is not bibasic in $\fbp[\ell_2]$, for $1\leq p\leq 2$: use \cite[Example 6.2]{TT} to get an orthonormal basis of $L_2[0,1]$ which is not a bibasis, then apply \Cref{Everywhere}. On the other hand, it was shown in \cite{TT} that every copy of the $\ell_1$ basis is bibasic in every Banach lattice; in particular, the canonical copy of the $\ell_1$ basis is bibasic in $\fbp[\ell_1]$. This leads to the following conjecture: Let $(x_k)$ be a normalized (unconditional) basis of a Banach space $E$, and fix $1\leq p<\infty$; if $(\delta_{x_k})$ is bibasic in $\fbp[E]$ then $(x_k)$ is equivalent to the unit vector basis of $\ell_1$.
\\

Being a basis is critical for \Cref{Everywhere} - \Cref{p:l2_in_CK} shows that, for any basic sequence $(x_k)$ in $C(\Omega)$, equivalent to the $\ell_2$ basis, the sequence $(\delta_{x_k})$ is absolute in $\fbp[C(\Omega)]$ (see \Cref{absolute bases} below for a discussion on absoluteness), hence in particular is bibasic.
\\

Note that, as in \Cref{prop:lowerell2}, $c_0$ gives a lower bound on the growth of the left hand side of the bibasis inequality. In other words, the inequality in  the statement of \Cref{prop:lowerell2} admits a natural ``bibasis" analogue, where one places sups inside the norms. \end{rem}

\subsection{A connection with majorizing maps}\label{MM}

Let $E$ be a normed space and $X$ an Archimedean vector lattice. A linear map $T:E\to X$ is called \textit{majorizing} (see \cite[Chapter IV]{Schaefer74}) if for every norm null sequence $(x_k)$ in $E$, the sequence $(Tx_k)$ is order bounded. There are several equivalent characterizations of majorizing operators, which we record in the following proposition. The equivalence \eqref{maj-maj}$\Leftrightarrow$\eqref{maj-conv} is well known, see, for example, \cite[Proposition IV.3.4]{Schaefer74}, but we include the simple proof for the sake of analogy with \Cref{bibasis theorem}.
Our proof also does not require linearity, only positive homogeneity (and, in particular, it works for sublinear mappings, which seem to be underrepresented in the vector lattice literature, but are important in applications):
 
 \begin{prop}\label{seq ob2}
Let $T:E\to X$ be a positively homogeneous map, where $E$ is a normed space and $X$ is an Archimedean vector lattice. The following are equivalent:
\begin{enumerate}
\item\label{maj-maj} $T$ is majorizing;
\item\label{maj-n-u} $x_k\xrightarrow{\|\cdot\|}0$ implies $Tx_k\xrightarrow{u}0$ for all sequences $(x_k)$ in $E$;
\item\label{maj-n-o} $x_k\xrightarrow{\|\cdot\|}0$ implies $Tx_k\xrightarrow{o}0$ for all sequences $(x_k)$ in $E$.
  \newcounter{tmpc}
  \setcounter{tmpc}{\value{enumi}}
  \end{enumerate}
  Moreover, if $X$ is a Banach lattice then these statements are further equivalent to:
  \begin{enumerate}
    \setcounter{enumi}{\value{tmpc}}
    \item\label{maj-sup} $x_k\xrightarrow{\|\cdot\|}0$ implies $(\bigvee_{k=1}^n |Tx_k|)$ is norm bounded for all sequences $(x_k)$ in $E$;
\item\label{maj-conv}  $T$ is $\infty$-convex in the sense of \cite[Definition 1.d.3]{LT2}, i.e., there exists $M\geq 1$ such that for each $x_1,\dots,x_n$ in $E$, $$\bigg\|\bigvee_{k=1}^n|Tx_k|\bigg\|\leq M\bigvee_{k=1}^n \|x_k\|.$$
\end{enumerate}
\end{prop}

\begin{proof}
Clearly, \eqref{maj-n-u}$\Rightarrow$\eqref{maj-n-o}$\Rightarrow$\eqref{maj-maj}; we show \eqref{maj-maj}$\Rightarrow$\eqref{maj-n-u}. Suppose $(x_k)$ is a sequence in $E$ and $x_k\xrightarrow{\|\cdot\|}0$. Then there  exists a sequence $0\leq \lambda_k\uparrow \infty$ in $\mathbb{R}$ such that $\lambda_kx_k\xrightarrow{\|\cdot\|}0$. Hence, $T(\lambda_kx_k)$ is order bounded in $X$, so there exists $0<x\in X$ with $|T(\lambda_kx_k)|\leq x$ for all $k$. It follows that $|Tx_k|\leq \frac{1}{\lambda_k}x$, so that $Tx_k\xrightarrow{u}0$ in $X$.

\eqref{maj-maj}$\Rightarrow$\eqref{maj-sup} is obvious.

\eqref{maj-sup}$\Rightarrow$\eqref{maj-n-u} Assume $x_k\xrightarrow{\|\cdot\|}0$ implies $(\bigvee_{k=1}^n |Tx_k|)$ is norm bounded for all sequences $(x_k)$ in $E$. Then, since $X^{**}$ is monotonically complete, $(Tx_k)$ is order bounded in $X^{**}$. Now, viewing $T$ as a map from $E$ to $X^{**}$, $T$ satisfies \eqref{maj-maj}, and hence \eqref{maj-n-u}. Hence, $Tx_k\xrightarrow{u}0$ in $X^{**}$ and hence in $X$ since $X$ is a closed sublattice of $X^{**}$ (use \cite[Proposition 2.12]{TT}).

\eqref{maj-conv}$\Rightarrow$\eqref{maj-sup}: Suppose $x_k\xrightarrow{\|\cdot\|}0$. Then $(x_k)$ is norm bounded so there exists $K$ such that $\bigvee_{k=1}^n\|x_k\|<K$ for each $n$. Then $ \|\bigvee_{k=1}^n |Tx_k|\|\leq KM$, and we get \eqref{maj-sup}.

\eqref{maj-sup}$\Rightarrow$\eqref{maj-conv}: Suppose \eqref{maj-conv} fails. Then by positive homogeneity of the inequality, for every $m$ there exists $x^m_1,\dots, x^m_{n_m}$ with $\bigvee_{k=1}^{n_m} \|x^m_{k}\|<\frac{1}{2^m}$ and $\|\bigvee_{k=1}^{n_m} |Tx^m_{k}|\|>m$. Then the sequence $x_1^1,\dots, x^1_{n_1}, x^2_1,\dots$ is $\|\cdot\|$-null but after applying $T$, \eqref{maj-sup} fails.
\end{proof}

The infimum over all such numbers $M$ as above is called the majorizing norm of $T$; it is denoted by $\|T\|_m$. 

\begin{prop}\label{majorizing norm}
If, in the notation of \Cref{seq ob2}, $E$ is finite dimensional, and $T$ is linear, then $\vee_{\|z\| \leq 1} |Tz|$ exists, and $\|T\|_m = \| \vee_{\|z\| \leq 1} |Tz| \|$.
\end{prop}

\begin{proof}
First we show that $T (B_E)$ is order bounded. To this end, let $(e_i)_{i=1}^n$ be an Auerbach basis of $E$. Let $x = \sum_{i=1}^n |Te_i|$. Any $z \in B_E$ admits a decomposition $z = \sum_i a_i e_i$, with $\vee_i |a_i| \leq 1$. Then
$$
|Tz| = \big| \sum_i a_i Te_i \big| \leq \sum_i |a_i| |Te_i| \leq x.
$$

Find a nested sequence of finite subsets $S_1 \subseteq S_2 \subseteq \ldots \subseteq B_E$, so that, for every $m$, $S_m$ is a $2^{-m}$-net in $B_E$. Let $x_m = \vee_{z \in S_m} |Tz|$. Then clearly $x_1 \leq x_2 \leq \ldots$. On the other hand, for any $m$, any $z \in B_E$ can be written as $z = u+v$, with $u \in S_m$ and $\|v\| \leq 2^{-m}$. Then $|Tz| \leq |Tu| + |Tv| \leq x_m + 2^{-m} x$. 
In particular, $x_{m+1} \leq x_m + 2^{-m} x$, and therefore, $\|x_{m+1} - x_m\| \leq 2^{-m} \|x\|$. Thus, the sequence $(x_m)$ is increasing to its limit (in the norm topology), call it $x_\infty$. As we have seen, the inequality $|Tz| \leq x_m + 2^{-m} x \leq x_\infty + 2^{-m} x$ holds for any $m$, hence $x_\infty = \vee_{\|z\| \leq 1} |Tz|$. Clearly, $\|T\|_m \leq \|x_\infty\|$. On the other hand, $\|T\|_m \geq \sup_m \|x_m\| = \|x_\infty\|$.
\end{proof}

\begin{rem}
In general, \Cref{majorizing norm} fails for non-linear maps, as the following map $T : \ell_2^2 \to \ell_2$ shows. For $n \in \Nat$ let $T \big( \cos \frac{\pi}{2n}, \sin \frac{\pi}{2n} \big) = e_n/\sqrt{n}$, where $(e_i)$ is the canonical basis for $\ell_2$; let $T (1,0) = 0$.
Extend $T$ to be continuous and homogeneous on $\ell_2^2$. Then $\big\{ |Tz| : z \in B_{\ell_2^2} \big\} \supseteq \{ n^{-1/2} e_n : n \in \Nat\}$, and the latter set is not order bounded.
\\

On the other hand, one can show that the formula $\|T\|_m = \| \vee_{\|z\| \leq 1} |Tz| \|$ holds for various non-linear maps $T$, if $T : E \to X$ takes a finite dimensional $E$ into an AM-space $X$. For this, it suffices to show that, if $S\subseteq X$ is relatively compact, then $x_\infty = \vee_{x \in S} |x|$ exists in $X$, and moreover, for any $\varepsilon > 0$ there exists a finite set $F \subseteq S$ so that $\|\vee_{x \in F} |x| - x_\infty\| < \varepsilon$. Imitating the proof of \Cref{majorizing norm}, find a nested sequence of finite subsets $S_1 \subseteq S_2 \subseteq \ldots \subseteq S$, so that, for every $m$, $S_m$ is a $2^{-m}$-net in $S$. Let $x_m = \vee_{x \in S_m} |x|$. Then clearly $x_1 \leq x_2 \leq \ldots$. Consider a lattice isometric embedding $J : X \to C(K)$, for some Hausdorff compact $K$. For each $m$, $J x_{m+1} \leq J x_m + 2^{-m} 1_K$, hence the sequence $(J x_m)$ converges in norm. As $J$ is isometric, the sequence $(x_m)$ converges to some $x_\infty\in X$, hence no upper bound on $S$ can be strictly less than $x_\infty$.
On the other hand, for any $x \in S$ and $m \in \Nat$, we have
$$
|Jx| \leq J x_m + 2^{-m} 1_K \leq J x_\infty + 2^{-m} 1_K ,
$$
and so $|x| \leq x_\infty$, due to $m$ being arbitrary.
\end{rem}
 
We now specialize to operators of the form $S = S_\olx : \ell_q^n \to X : e_k \mapsto x_k$, where $X$ is a Banach lattice, and $(e_k)_{k=1}^n$ is the canonical basis of $\ell_q^n$.
 Then 
 \begin{equation}
 \|S_\olx\|_m = \Bignorm{ \bigvee_{\sum_k |b_k|^q \leq 1} \sum_k b_k x_k } =
 \Bignorm{ \Big( \sum_k |x_k|^p \Big)^{1/p} } .
 \label{eq:majorizing norm}
 \end{equation}
In particular, given an operator $S=S_\olx\colon\ell_q^n\to E$ $: e_k\mapsto x_k$, where $E$ is a Banach space, we can compose it with $\phi_E$ to get an operator of the above form. Applying~\eqref{eq:majorizing norm} to $\phi_ES_\olx$, we get
\begin{displaymath}
    \Bignorm{\Big( \sum_k\abs{\delta_{x_k}}^p \Big)^{1/p} }
    =\norm{\phi_ES_\olx}_m.
\end{displaymath}
This gives a connection with \eqref{reduce to summing}.
\\

\Cref{seq ob2} presents several equivalent characterizations of operators that map norm convergent sequences to uniformly convergent sequences. In \cite[Propositions 5.3 and 5.4]{TT} the authors study operators that map uniformly convergent sequences to uniformly convergent sequences. We now present an analogue of the equivalence of statements (2) and (5) in \Cref{seq ob2} for such operators:
\begin{prop}\label{Connection to Pedro}
Let $T:E\subseteq X\to Y$ be a positively homogeneous map where $X$ and $Y$ are Banach lattices and $E$ is a subspace of $X$. The following are equivalent:
\begin{enumerate}

\item $T$ is sequentially uniformly continuous; i.e., $x_k\xrightarrow{u} 0$ implies $Tx_k\xrightarrow{u}0$ for all sequences $(x_k)$ in $E$;
\item $T$ is $(\infty,\infty)$-regular in the sense of \cite{SanPe-Tra}; i.e., there exists $M$ such that for any $n$ and any $x_1,\dots,x_n$ in $E$,

$$\bigg\|\bigvee_{k=1}^n |Tx_k|\bigg\|\leq M\bigg\|\bigvee_{k=1}^n |x_k|\bigg\|.$$
\end{enumerate}
\end{prop}
\begin{proof}
(2)$\Rightarrow$(1): Suppose $x_k\xrightarrow{u}0$. Then $(x_k)$ is order bounded so there exists $K$ such that $\|\bigvee_{k=1}^n|x_k|\|<K$ for each $n$. By (2), $ \|\bigvee_{k=1}^n |Tx_k|\|\leq KM$, and we can apply \cite[Proposition 5.4]{TT}. Formally, \cite[Proposition 5.4]{TT} is stated for linear maps, but the proof works for positively homogeneous maps.
\\

(1)$\Rightarrow$(2): Suppose (2) fails. Then by positive homogeneity of the inequality, for every $m$ there exists $x^m_1,\dots, x^m_{n_m}$ with $\|\bigvee_{k=1}^{n_m} |x^m_{k}|\|<\frac{1}{2^m}$ and $\|\bigvee_{k=1}^{n_m} |Tx^m_{k}|\|>m$. Then the sequence $x_1^1,\dots, x^1_{n_1}, x^2_1,\dots$ is $u$-null but after applying $T$ \cite[Proposition 5.4]{TT} fails. 
\end{proof}
\begin{rem}\label{Dales}
Sequentially uniformly continuous operators also appear in the theory of multinormed spaces, and it is known that a bounded operator between Banach lattices is $\infty$-multi-bounded if and only if it is $1$-multi-bounded if and only if it is pre-regular; see Sections~4.2 and~4.5 in \cite{Dales:17}. If $X$ and $Y$ are Banach lattices and $E$ a \emph{subspace} of $X$ then one can also show that a linear map $T:E\subseteq X\to Y$ is sequentially uniformly continuous if and only if it is order bounded (in the sense of \cite{TT}) when viewed as a map $T:E\subseteq X\to Y^{**}$. Here, an operator $T:A\subseteq B\to C$ defined on a subspace $A$ of a vector lattice $B$ and taking values in a vector lattice $C$ is \textit{order bounded} if for any $b \in B_+$, the image of $[-b,b] \cap A$ under $T$ is order bounded in $C$.
\\

Indeed, if $T:E \subseteq X\to Y^{**}$ is order bounded, then it is sequentially uniformly continuous as a map into $Y^{**}$ (see \cite[Proposition 24.1]{Tay}), and hence into $Y$ since uniform convergence of sequences passes between closed sublattices (\cite[Proposition 2.12]{TT}). Conversely, let $B\subseteq E$ be order bounded. Then there exists $x\in X$ with $|b|\leq x$ for each $b\in B$. Now note that for $b_1,\dots, b_n$ in $B$ we have $\||b_1|\vee \cdots \vee |b_n|\|\leq \|x\|$ so that by \Cref{Connection to Pedro} 

$$\bigg\|\bigvee_{k=1}^n|Tb_k|\bigg\|\leq M\|x\|.$$
Now let $\mathcal{F}=\mathcal{P}_f(B)$ be the family of finite subsets of $B$, directed by inclusion. For $F\in \mathcal{F}$ set $y_F=\bigvee \{|Ty| : y\in F\}$. Then $(y_F)$ is an increasing and norm bounded net in $Y$, hence has supremum in $Y^{**}$. Hence, $T(B)$ is order bounded in $Y^{**}$ and we are done. Note the only property of $Y^{**}$ we need is monotonically bounded, so we can replace $Y^{**}$ by any monotonically bounded Banach lattice containing $Y$ as a closed sublattice. 
\\

As a corollary, since  order bounded embeddings map absolute sequences to absolute sequences (this follows directly from \cite[Proposition 7.5]{TT}; see \Cref{absolute bases} below for more information on absolute sequences), so do sequentially uniformly continuous embeddings.
\end{rem}

\smallskip

Given the fact that from every u-null net one can extract a u-null sequence, one may wonder if sequences can be replaced with nets in \Cref{seq ob2}. This is \emph{not} true. Indeed, in \cite{GE83} the authors study strongly majorizing and Carleman operators. A linear map $T:E\to F$ from a normed space $E$ to an Archimedean vector lattice $F$ is \textit{strongly majorizing} if $T$ maps the unit ball of $E$ into an order interval in $F$ - one can easily show that these are exactly the operators which satisfy conditions (2) and (3) in \Cref{seq ob2} when sequences are replaced by nets. Note, for instance, that the identity map on $c_0$ is majorizing, but not strongly so. The reason that there are sequentially uniformly continuous but not uniformly continuous operators (even though uniform convergence is a sequential convergence) stems from the fact that uniform convergence is not topological.
\\

For \textit{Carleman operators} (operators  mapping the unit ball into an order interval in the universal completion of the range), another nice characterization is available: $T$ is Carleman if and only if $T$ maps norm null nets to $uo$-null nets. We refer the reader to \cite{GTX,KT,Tay1} for information on $uo$-convergence and its applications. In summary, many of the operators defined via ``boundedness" in the literature are in fact merely continuous operators, if one finds the right notions of convergence. Moreover, many of the fundamental results hold if the operator is merely defined on a subspace of the lattice.

\subsection{Absolute bases}\label{absolute bases}
Recall that the bibasis inequality \eqref{bbi} arises by commuting the supremum with the norm in the usual basis inequality. If one instead begins with the inequality 
\begin{displaymath}
\bigvee\limits_{\epsilon_k=\pm 1}\bigg \|\sum\limits_{k=1}^m\epsilon_ka_kx_k\bigg\|
  \leq M \bigg
  \|\sum_{k=1}^m a_kx_k\bigg\|,
\end{displaymath}
characterizing unconditional sequences, brings the sup inside the norm, and notes that $\bigvee_{\epsilon_k=\pm 1}\left|\sum_{k=1}^m \epsilon_k a_kx_k\right|=\sum_{k=1}^m|a_kx_k|$, one arrives at the notion of an absolute sequence from \cite{TT}. More formally, we say that a basic sequence $(x_k)$ in a Banach lattice $X$ is \textit{absolute} if there exists a constant $A\geq 1$ such that $$\Big\|\sum_{k=1}^m|a_kx_k|\Big\|\leq A\Big\|\sum_{k=1}^ma_kx_k\Big\|$$ for all $m\in \mathbb{N}$ and scalars $a_1,\dots,a_m$. \cite[Theorem 7.2]{TT} shows that $(x_k)$ is absolute if and only if the convergence of $\sum_{k=1}^\infty a_kx_k$ is equivalent to the convergence of $\sum_{k=1}^\infty |a_kx_k|$.
Note that $\big\|\sum_k a_kx_k \big\|\leq \big\|\sum_k |a_kx_k|\big\|$. 
Consequently, for any absolute basic sequence we have $\big\|\sum_k a_kx_k \big\| \sim \big\|\sum_k |a_kx_k|\big\|$. 
As is easy to see, absolute sequences  must be both unconditional and bibasic.
\\

By \cite[Proposition 7.8]{TT}, any sequence $(x_k)$ in a Banach lattice, equivalent to  the $\ell_1$ basis (on the Banach space level), is absolute. On the other hand, from the above discussion it is clear that in AM-spaces a sequence is absolute if and only if it is unconditional. In this short  subsection, we examine conditions on the sequence $(x_k) \subseteq E$ so that its canonical image $(\delta_{x_k})$ is absolute in $\fbp[E]$.
\\

\begin{prop}\label{r:absolute bases}
If $(x_k)$ is a  normalized basis of $E$, and $(\delta_{x_k})$ is an absolute basic sequence in $\fbp[E]$ for $p<\infty$, then  $(x_k)$ is equivalent to the $\ell_1$ basis.
\end{prop}

\begin{proof}
 As noted above, the sequence $(\delta_{x_k})$ (equivalently, $(x_k)$) is unconditional.
\Cref{p:dominations_abs_values} shows that, for any finite sequence of scalars $(a_k)$,
$$
\Big\| \sum_k a_k x_k \Big\| = \Big\| \sum_k a_k \delta_{x_k} \Big\| \lesssim \Big\| \sum_k a_k \big|\delta_{x_k}\big| \Big\| \leq \Big\| \sum_k |a_k| \big|\delta_{x_k}\big| \Big\| \sim \Big\| \sum_k a_k x_k \Big\| .
$$
Now apply \Cref{h}.
\end{proof}

\begin{rem}
Above we showed that, for a semi-normalized basis $(x_k)$ of $E$, the following two statements are equivalent: (i) $(x_k)$ is equivalent to the $\ell_1$ basis; (ii) for any sequence of scalars $(a_k)$, $\sum_k a_k x_k$ converges if and only if $\sum_k |a_k J x_k|$ converges for any embedding $J : E \to Z$, where $Z$ is a Banach lattice.
This provides a converse to \cite[Proposition 7.8]{TT}, which states that any sequence in a Banach lattice which is equivalent to the $\ell_1$ basis is absolute - we now know that $\ell_1$ is the only normalized basis with this property.
One can also notice the similarity with the well-known fact that, for a normalized basic sequence $(x_k)$, the norm convergence of $\sum_k a_k x_k$ is equivalent to the convergence of $\sum_k \|a_kx_k\|$ if and only if $(x_k)$ is equivalent to the $\ell_1$ basis. 
\end{rem}

For positive sequences, being absolute is the same as being unconditional. Hence, for every unconditional basis $(x_k)$ of $E$, $(|\delta_{x_k}|)$ is absolute in $\fbp[E]$. However, \Cref{Some parts redundant} shows that $(|\delta_{x_k}|)$ need not be absolute if $(x_k)$ is a conditional basis. Is it at least true that $(|\delta_{x_k}|)$ is bibasic? If $p=\infty$ then this is clear, as $\fbl^{(\infty)}[E]$ is an AM-space. However, when $p\in [1,\infty)$ the situation is not as transparent. For example, we do not know whether  $(|\delta_{s_k}|)$ is bibasic in $\fbp[c_0]$ (here $(s_k)$ is the summing basis of $c_0$, and $p\in [1,\infty)$). \Cref{dual to the summing} shows that the dual to the summing basis is not bibasic in $\fbp[\ell_1]$ for any finite $p$, though its modulus is absolute.

\begin{rem}
The proof of \Cref{Everywhere} can be easily adapted to show that if $(x_k)$ is a normalized basis of $E$ such that $(\delta_{x_k})  \subseteq \fbp[E]$ is absolute, then every isomorphic embedding $T$ from $E$ to any $p$-convex Banach lattice $X$ maps $(x_k)$ to an absolute sequence. In fact, if  $(\delta_{x_k}) \subseteq \fbp[E]$ is absolute and $p<\infty$, then $(x_k)$ is equivalent to the unit vector basis of $\ell_1$ by \Cref{r:absolute bases}, hence so is $(Tx_k)$, which implies that $(Tx_k)$ is absolute. On the other hand, for  $p=\infty$, we note  that every unconditional basic sequence in an AM-space is automatically absolute.
\end{rem}

\section{Sublattices of free Banach lattices}\label{s:sublattices of FBL}

In this section we investigate the sublattice structure of $\fbp[E]$. Let us begin with some necessary conditions for a Banach lattice to be a sublattice of $\fbp[E]$ (compare with \cite{AMRT}, where conditions under which $E$ embeds into $\fbl[E]$ as a lattice-complemented sublattice are explored). First, recall that a Banach lattice $X$ satisfies the \emph{$\sigma$-bounded chain condition} (\emph{$\sigma$-bcc}) if there is a countable decomposition $X_+\setminus\{0\}=\bigcup_{n\geq2} \mathcal F_n$ such that for every $n$, every subset $\mathcal G\subseteq \mathcal F_n$ of size $n$ contains a pair of non-disjoint elements. This is stronger than the countable chain condition, meaning that every uncountable family in $X_+$ contains a pair of non-disjoint elements (see \cite{APR}).

\begin{prop}\label{p:sublatticefbp}
If $1 \leq p \leq \infty$, and $F$ is a closed sublattice of $\fbp[E]$ for some Banach space $E$, then:
\begin{enumerate}
\item $F$ is $p$-convex;
\item $F$ satisfies the $\sigma$-bcc;
\item The real-valued lattice homomorphisms separate the points of $F$.
\end{enumerate}
\end{prop}

\begin{proof}
(1) is clear. (2) follows from the fact that $C_{ph}(B_{E^*})$, the lattice of positively homogeneous weak$^*$-continuous functions on $B_{E^*}$, satisfies the $\sigma$-bcc \cite[Theorem 1.3]{APR}, and this is transferred to (not necessarily closed) sublattices. By construction, recall that $\fbp[E]$ can be seen as a (in general, non-closed) sublattice of  $C_{ph}(B_{E^*})$.
\\

(3) is analogous to \cite[Corollary 2.7]{ART}: For every $x^*\in E^*$, the evaluation functional $\widehat{x^*}:\fbp[E]\rightarrow \mathbb R$ given by $\widehat{x^*}(f)=f(x^*)$ is a lattice homomorphism and clearly $f=g$ in $\fbp[E]$ if and only if $\widehat{x^*}(f)=\widehat{x^*}(g)$ for every $x^*\in E^*$. It follows that for every sublattice $F$ of $\fbp[E]$, the real-valued lattice homomorphisms (obtained by restricting $\widehat{x^*}$ to $F$) separate the points of $F$.
\end{proof}

\begin{rem}
As a direct consequence of Proposition \ref{p:sublatticefbp} we see that $L_p(\mu)$ can possibly embed as a sublattice of $\fbp[E]$ only when $\mu$ is purely atomic (with countably many atoms).
\end{rem}

We will actually see next that $\ell_q$ always embeds as a sublattice of $\fbp[\ell_q]$ if $q\geq p$. Specifically, we show:

\begin{thm}\label{p:sublattice}\label{Finite ok}
Suppose $E$ is a Banach lattice, $p$-convex with constant $1$, with the order induced by a $1$-unconditional basis. Then $\fbp[E]$ contains an isometric copy of $E$ as a sublattice. Moreover, there exists a contractive lattice homomorphic projection onto this sublattice.
\end{thm}

This provides an alternative approach to \cite[Theorem 4.1]{AMRT}, valid for arbitrary $p\in[1,\infty]$.
\\

Throughout, we work with a fixed $E$ from \Cref{p:sublattice}.
For this proof, we change our notational conventions slightly and denote the normalized $1$-unconditional basis for $E$ by $(e_i)$.
The corresponding (normalized) biorthogonal functionals shall be denoted by $(e_i^*)_{i \in \nat}$.
For $N \in \nat$ with $N \leq \dim E$, denote by $P_N$ the canonical (contractive) projection from $E$ onto $\spn[e_i : 1 \leq i \leq N]$.
Then $P_N^*$ projects $E^*$ onto $\spn[e_i^* : 1 \leq i \leq N]$, with $P_N^* e_i^* = e_i^*$ if $i \leq N$, $P_N^* e_i^* = 0$ if $i > N$.
We thus observe that $\spn[e_i : 1 \leq i \leq N]$ and $\spn[e_i^* : 1 \leq i \leq N]$ are in duality with each other.
\\

We need to establish a technical lemma:

\begin{lem}\label{l:p-convex norms}
 Denote by ${\mathcal{P}}$ the set of all finite sequences $(\beta_i)$ so that $\sum_i \big| \beta_i \gamma_i\big|^p \leq 1$ whenever $\big\| \sum_i \gamma_i e_i \big\| \leq 1$.
 Then for every choice of scalars $(\alpha_i)$ we have
 $$
 \sup \Big\{ \big( \sum_i \big| \alpha_i \beta_i \big|^p \big)^{1/p} : (\beta_i) \in {\mathcal{P}} \Big\} = \big\| \sum_i \alpha_i e_i \big\| .
 $$
\end{lem}

\begin{proof}
Denote by $E_{(p)}$ the $p$-concavification of $E$, as described in e.g.~\cite[Section 1.d]{LT2}. 
If $(t_i)$ is a finite sequence, then we can view $\sum_i t_i e_i$ as an element of $E_{(p)}$, with the norm 
$$
\big\| \sum_i t_i e_i \big\|_{E_{(p)}} = \big\| \sum_i t_i^{1/p} e_i \big\|_E^p ,
{\textrm{  where   }}
t^{1/p} = {\mathrm{sign}} \, t \cdot |t|^{1/p} .
$$
Such finitely supported sequences are dense in $E_{(p)}$.
 \\
 
Due to the unconditionality of the basis $(e_i)$, we can assume $\alpha_i \geq 0$, and there exists $N \in \Nat$ so that $\alpha_i = 0$ whenever $i > N$. Projecting onto $\spn[e_i : 1 \leq i \leq N]$ and $\spn[e_i^* : 1 \leq i \leq N]$, we can assume that $\beta_i = 0 = \gamma_i$ for $i > N$. By unconditionality, we also assume $\beta_i, \gamma_i \geq 0$.
Now let $t_i = \alpha_i^p$, $s_i = \beta_i^p$, $z_i = \gamma_i^p$. Then $(\beta_i) \in {\mathcal{P}}$ if and only if $\sum_i s_i z_i \leq 1$ whenever
$$
\big\| \sum_i z_i^{1/p} e_i \big\|_E^p = \big\| \sum_i z_i e_i \big\|_{E_{(p)}} \leq 1 . 
$$
By duality, $\big\| \sum_i s_i e_i^* \big\|_{(E_{(p)})^*} \leq 1$ if and only if $(\beta_i) \in {\mathcal{P}}$, and therefore,
\begin{align*}
 \sup \Big\{ \big( \sum_i \big| \alpha_i \beta_i \big|^p \big)^{1/p} : (\beta_i) \in {\mathcal{P}} \Big\}
 &
 = \sup \Big\{ \big( \sum_i t_i s_i \big)^{1/p} : \big\| \sum_i s_i e_i^* \big\|_{(E_{(p)})^*}
 \leq 1 \Big\}
\\ &
 = \big\| \sum_i t_i e_i \big\|_{E_{(p)}}^{1/p} = \big\| \sum_i \alpha_i e_i \big\|_E ,
\end{align*}
 which is what we want. 
\end{proof}

\begin{proof}[Proof of \Cref{p:sublattice}]
We present a proof in the case of infinite dimensional $E$. Only minor adjustments are needed to handle the finite dimensional setting.
\\

 We find a sequence of disjoint functions $f_i \in \fbp[E]_+$, so that, for any finite sequence of positive numbers $(\alpha_i)_i$,
 \begin{equation}
  \big\| \sum_i \alpha_i f_i \big\|_{\fbp[E]} = \big\| \sum_i \alpha_i e_i \big\|_E .
  \label{eq:norms}
 \end{equation}
 Once this is established, we conclude that $E$ is lattice isometric to $F = \spn[f_i : i \in \nat]$.
 \\
 
 For $k \in \nat$, define $f_k \in H[E]_+$:
 $$
 f_k = \Big( \big| \delta_{e_k} \big| - 2^{2k} \big( \sum_{i<k} \big| \delta_{e_i} \big| + \sum_{i>k} 2^{-i} \big| \delta_{e_i} \big| \big) \Big)_+ .
 $$
 As $\sum_i 2^{-i} \big| \delta_{e_i} \big|$ converges in norm, $f_k$ actually belongs to $\fbp[E]$; in fact,
 $$
 f_k = \lim_N \Big( \big| \delta_{e_k} \big| - 2^{2k} \big( \sum_{i<k} \big| \delta_{e_i} \big| + \sum_{i=k+1}^N 2^{-i} \big| \delta_{e_i} \big| \big) \Big)_+ .
 $$
 Moreover, the functions $f_k$ are disjoint. Indeed, suppose $i < j$, and $e^* \in E^*$ is such that both  $f_i(e^*), f_j(e^*) > 0$. We will derive a contradiction. To this end, observe that
 $$
 f_i(e^*) \leq \big| e^*(e_i) \big| - 2^{2i-j} \big| e^*(e_j) \big|
 {\textrm{   and   }}
 f_j(e^*) \leq \big| e^*(e_j) \big| - 2^{2j} \big| e^*(e_i) \big| .
 $$
 Since both $f_i(e^*)$ and $f_j(e^*)$ are strictly positive, we have
 $$
  \big| e^*(e_i) \big| > 2^{2i-j} \big| e^*(e_j) \big| > 2^{2i+j} \big| e^*(e_i) \big| ,
 $$
 which is impossible.
 \\
 
 To establish \eqref{eq:norms}, let $f = \sum_{k=1}^N \alpha_k f_k$ (recall that $\alpha_k \geq 0$). By scaling, assume $\big\| \sum_k \alpha_k e_k \big\|_E = 1$.
 Recall that $\|f\|_{\fbp[E]}$ is the supremum of $\big(\sum_i |f(x_i^*)|^p\big)^{1/p}$, given $\sup_{x\in B_E} \sum_i |x_i^*(x)|^p\leq 1$ (with finite sums).
 \\
 
 We first obtain a lower estimate on $\|f\|$. By \Cref{l:p-convex norms}, we can find $(\beta_i)_{i=1}^N \subseteq [0,\infty)$ so that $\big(\sum_i (\alpha_i \beta_i)^p\big)^{1/p} = \| \sum_i \alpha_i e_i \| = 1$, and $\sum_i |\beta_i \gamma_i|^p \leq 1$ whenever $\| \sum_i \gamma_i e_i \| \leq 1$.
 Let $x_i^* = \beta_i e_i^*$ ($1 \leq i \leq N$). Then clearly $\sup_{x\in B_E} \sum_i |x_i^*(x)|^p\leq 1$, and 
 $$
 \|f\|\geq \big(\sum_{i=1}^n |f(x_i^*)|^p\big)^{1/p} = \big(\sum_i (\alpha_i \beta_i)^p\big)^{1/p} = 1 .
 $$
 
 Next we need to establish an upper bound for $\|f\|_{\fbp[E]}$. Specifically, we suppose $\sup_{x\in B_E} \sum_i |x_i^*(x)|^p\leq 1$, and show that $\big(\sum_{i=1}^n |f(x_i^*)|^p\big)^{1/p} \leq 1$. As $f(x^*) = f(-x^*)$ for every $x^* \in E^*$, we can and do assume that $f(x_j^*) \geq 0$ for any $j$.
 \\
 
  Define an auxiliary function $g : E^* \to [0,\infty)$ (not necessarily continuous) in the following manner. For $x^* \in E^*$ let ${\mathcal{I}}(x^*) = \{k : f_k(x^*) \neq 0\}$. As the functions $f_k$ are disjointly supported, ${\mathcal{I}}(x^*)$ is either empty or a singleton. If ${\mathcal{I}}(x^*) = \emptyset$, let $g(x^*) = 0$. If ${\mathcal{I}}(x^*) = \{i\}$, let $g(x^*) = \alpha_i \big|x^*(e_i)\big|$.
  \\
  
 Now define the disjoint sets $S_i = \{ j : f_i(x_j^*) \neq 0 \}$ (that is, $j \in S_i$ if and only if ${\mathcal{I}}(x_j^*) = \{i\}$). For $j \in S_i$, we have $|f(x_j^*)| \leq \alpha_i \big|x_j^*(e_i)\big| = g(x_j^*)$, hence it suffices to show that
 $$
 \sum_j \big| g(x_j^*) \big|^p \leq 1 .
 $$
 
 Find (positive) scalars $(t_j)$ so that $\sum_j |t_j|^q = 1$ (here, $1/p + 1/q = 1$), and $\big(\sum_{j=1}^n |g(x_j^*)|^p\big)^{1/p} = \sum_j t_j g(x_j^*)$.  For $j\in S_i$, let $\epsilon_j=|x_j^*(e_i)|/x_j^*(e_i)$ if $x_j^*(e_i)\neq 0$ and $\epsilon_j=0$ otherwise, and set
 $$
 y_i^* = \kappa_i^{-1} \sum_{j \in S_i} t_j\epsilon_j x_j^* ,  {\textrm{   where   }} \kappa_i = \big( \sum_{j \in S_i} |t_j|^q \big)^{1/q} .
 $$
 
 Note first that, for $x \in B_E$, $\sum_i |y_i^*(x)|^p \leq 1$. Indeed, find scalars $(s_i)$ so that $\sum_i |s_i|^q =1$, and $\big( \sum_i |y_i^*(x)|^p \big)^{1/p} = \sum_i s_i y_i^*(x)$. Write $u_j = s_i/\kappa_i$ if $j \in S_i$. Then
 $$
 \sum_i s_i y_i^*(x) = \sum_j u_j t_j \epsilon_j x_j^*(x) \leq \big( \sum_j |u_j|^q |t_j|^q \big)^{1/q} \big( \sum_j |x_j^*(x)|^p \big)^{1/p} .
 $$
 We have
 $$
 \sum_j |u_j|^q |t_j|^q = \sum_i |s_i|^q \kappa_i^{-q} \sum_{j \in S_i} |t_j|^q = \sum_i |s_i|^q = 1 ,
 $$
 so $\sum_i s_i y_i^*(x) \leq 1$.
 \\
 
 Next we show that $\sum_i \big| g(y_i^*) \big|^p \geq \sum_j \big| g(x_j^*) \big|^p$. To this end, note first that, by the definition of $f_i$, ${\mathcal{I}}(x^*) = \{i\}$ if and only if
 $$
  \big| x^*(e_i) \big| > 2^{2i} \big( \sum_{k<i} \big| x^*(e_k) \big| + \sum_{k>i} 2^{-k} \big| x^*(e_k) \big| \big) .
 $$
 For any $j \in S_i$, we have
 $$
 |x_j^*(e_i)| > 2^{2i} \big( \sum_{k<i} \big| x_j^*(e_k) \big| + \sum_{k>i} 2^{-k} \big| x_j^*(e_k) \big| \big) .
 $$
 By the convexity of the absolute value, it follows that
 $$
 |y_i^*(e_i)|=\frac{1}{\kappa_i}\sum_{j\in S_i} t_j|x_j^*(e_i)| > 2^{2i} \big( \sum_{k<i} \big| y_i^*(e_k) \big| + \sum_{k>i} 2^{-k} \big| y_i^*(e_k) \big| \big) 
 $$
 as well. Therefore, $g(y_i^*) = \alpha_i |y_i^*(e_i)| = \kappa_i^{-1} \sum_{j \in S_i} t_j g(x_j^*)$. However, $\sum_i \kappa_i^q = 1$, hence
 $$
 \left(\sum_i \big| g(y_i^*) \big|^p\right)^\frac{1}{p} \geq \sum_i \kappa_i g(y_i^*) = 
 \left(\sum_j \big| g(x_j^*) \big|^p\right)^\frac{1}{p} .
 $$

 The reasoning above implies that it suffices to show that 
 \begin{equation}
 \sum_i \alpha_i^p \big| y_i^*(e_i) \big|^p \leq 1  {\textrm{   whenever    }}
 \sup_{x \in B_E} \sum_i \big| y_i^*(x) \big|^p = 1.
 \label{eq:estimate on norm}
 \end{equation}

 By projecting, we can assume that $y_1^*, \ldots, y_N^*$ ``live'' in $E_N^* = \spn[e_i^* : 1 \leq i \leq N]$; this space can be interpreted as the dual of $E_N = \spn[e_i : 1 \leq i \leq N]$.
 Consider the operator 
 $$
 T : E_N \to \ell_p^N : x \mapsto \sum_{i=1}^N y_i^*(x) \sigma_i ,
 $$
 where $(\sigma_i)$ is the canonical basis for $\ell_p^N$.
 The condition $\sup_{x \in B_E} \sum_i \big| y_i^*(x) \big|^p = 1$ is equivalent to $T$ being contractive.
 Tong's argument \cite[Proposition 1.c.8]{LT1} shows that the diagonalization of $T$ -- that is, the operator 
 $$
 T' : E_N \to \ell_p^N : x \mapsto \sum_{i=1}^N y_i^*(e_i) e_i^*(x) \sigma_i ,
 $$
 is contractive as well. Taking $x = \sum_{i=1}^N \alpha_i e_i$, we conclude that
 $$
 \|T'x\| = \Big( \sum_i \alpha_i^p \big| y_i^*(e_i) \big|^p \Big)^{1/p} \leq \|x\| = 1 ,
 $$
 which implies \eqref{eq:estimate on norm}.
 \\
 
 It remains to show that there exists a contractive lattice homomorphic projection from $\fbp[E]$ onto $\spn[f_i : i \in \nat]$.
 To this end, recall that the identity map $I : E \to E$ has a unique lattice homomorphic contractive extension $\widehat{I} : \fbp[E] \to E$. In particular, we have
\begin{align*}
&
 \widehat{I} \Big( \big| \delta_{e_k} \big| - 2^{2k} \big( \sum_{i<k} \big| \delta_{e_i} \big| + \sum_{i=k+1}^p 2^{-i} \big| \delta_{e_i} \big| \big) \Big)_+ 
 \\
 &
 =
 \Big( \big| I {e_k} \big| - 2^{2k} \big( \sum_{i<k} \big| I{e_i} \big| + \sum_{i=k+1}^p 2^{-i} \big| I{e_i} \big| \big) \Big)_+ = e_k ,
\end{align*}
 hence by continuity, $\widehat{I} f_k = e_k$ for any $k$.
 \\
 
 Note that the map $U : E \to \fbp[E] : e_k \mapsto f_k$ is a lattice isometry, and $\widehat{I} U$ is the identity on $E$. Therefore, $U \widehat{I}$ is a contractive projection onto $\spn[f_i : i \in \nat]$, and a lattice homomorphism.
\end{proof}

\begin{rem}
Note that if $E$ is a $p$-convex Banach lattice, then the identity on $E$ extends to a lattice homomorphism $\beta_E:\fbp[E]\rightarrow E$, that is $\beta_E\phi_E=id_E$. This allows us to see $E$ as a complemented subspace of $\fbp[E]$. There is a partial converse to this: Suppose $E$ is a Banach lattice which is isomorphic to a complemented subspace of $\fbp[E]$ for some $1\leq p\leq 2$, then $E$ must be itself $p$-convex (see \cite[Theorem 1.d.7]{LT2} and the remark after it).
\end{rem}

For the next proposition, we need some notation. Let us denote by $\fbl^{(p)n}[E]$ the $n$-fold iterate of $\fbp$'s, i.e., $\fbl^{(p)1}[E]=\fbp[E]$, $\fbl^{(p)2}[E]=\fbp[\fbp[E]],$ etc.  Following the same ideas  in \cite{AMRT} we have:

\begin{prop}\label{p:FBLFBL}
For every Banach space $E$ and $n\geq 1$, there is a lattice isometric embedding $S:\fbp[E]\rightarrow \fbl^{(p)n}[E]$  and a contractive lattice projection onto $S(\fbp[E])$.
\end{prop}

\begin{proof}
By convention, let us set $\fbl^{(p)0}[E]=E$ and for $k\in \mathbb{N}$ let $\phi_k: \fbl^{(p)k-1}[E]\to \fbl^{(p)k}[E]$ be the canonical embedding. Let
$$
T=  \phi_n \circ \cdots\circ \phi_1  :   E\rightarrow \fbl^{(p)n}[E].
$$
Since $\fbl^{(p)n}[E]$ is $p$-convex, there is a lattice homomorphism 
$\widehat T:\fbp[E]\rightarrow\fbl^{(p)n}[E]$ extending $T$, that is, $\widehat T \circ \phi_1=T$, with $\|\widehat T\|=\|T\|=1$. 
\\

Now for $k\in \mathbb{N}$ let $\beta_k := \widehat{I_k} : \fbl^{(p)k+1}[E]\to \fbl^{(p)k}[E]$ be the extension of the identity map $I_k:\fbl^{(p)k}[E]\to \fbl^{(p)k}[E]$, i.e.,  $\beta_k\circ\phi_{k+1}=I_k$ and $\|\beta_k\|=1$. Finally, define $$\beta:=\beta_1\circ\cdots\circ \beta_{n-1}: \fbl^{(p)n}[E]\to \fbp[E].$$

We claim that $\beta\widehat T = I_{\fbp[E]}$. Indeed, given $x\in E$ we have
$$
\beta\widehat T \phi_1(x)=\beta_1\circ\cdots\circ \beta_{n-1}\circ \phi_n \circ \cdots\circ \phi_1(x)=\phi_1(x).
$$
Since $\beta\widehat T$  is a lattice homomorphism and $\fbp[E]$ is lattice-generated by the elements of the form $\phi_1(x)$ with $x\in E$, it follows that $\beta\widehat T = I_{\fbp[E]}$, as claimed.
\end{proof}

\begin{rem}
Proposition \ref{p:FBLFBL} is related to the fact that the pair $(E,\fbp[E])$ has the 1-POE-$p$. Actually, $(E,\fbp[E])$ has the 1-POE-$q$ for every $q\geq p$.
\end{rem}

Finally, we remark that some other results from \cite{AMRT} are valid in the $p$-convex category. For example, if a $p$-convex Banach lattice $P$ is projective for $p$-convex lattices (in the sense of \cite{JLTTT}) then it embeds as a lattice-complemented sublattice of $\fbp[P]$.

\section{Encoding properties of E as properties of $\fbl[E]$}\label{Dictionary}
In this section we begin to build a dictionary between Banach space  properties of $E$ and Banach lattice properties of $\fbp[E]$. There are several results already known in this direction:

\begin{enumerate}
\item $E$ is isomorphic to a complemented subspace of a $p$-convex Banach lattice if and only if $\phi(E)$ is comp\-le\-men\-ted in $\fbp[E]$;
\item $E$ is $C$-linearly projective for $p$-convex lattices if and only if $\fbp[E]$ is  $C$-projective for $p$-convex lattices \cite{Laust-Tra};
\item More generally, we have characterized some relations between an operator $T:F\to E$ and the induced operator $\overline{T}:\fbp[F]\to\fbp[E]$. See \Cref{Injec-Surjec}. 
\end{enumerate}
In this section we significantly expand this list.

\subsection{Finite dimensionality corresponds to strong units and separability to  quasi-interior points}

Recall that when $E$ is finite dimensional, $\fbp[E]$ is lattice isomorphic to the space of continuous functions $C(S_{E^*})$, where $S_{E^*}$ is the unit sphere of $E^*$. In particular, it is $\infty$-convex. The situation is completely different when $E$ is infinite dimensional. Indeed, we now show that when $\dim E=\infty$, $\fbp[E]$ never has a strong unit. Moreover, in \Cref{s:local theory} we show that $\fbl[E]$ can never be more than $2$-convex.

\begin{prop}\label{No unit}
For any $1\leq p\leq \infty$, $\fbp[E]$ has a strong unit if and only if $E$ is finite dimensional.
\end{prop}
\begin{proof}

When $E$ is finite dimensional, all the $\fbp[E]$ are lattice isomorphic to each other, and to a $C(K)$-space, so in particular they have a strong unit. Conversely, suppose there exists $e\in \fbp[E]$ such that $|\delta_x|\leq e$ for all $x\in B_{E}$. Now, for each $x^*\in S_{E^*}$ and $\varepsilon>0$, we can find $x\in B_E$ such that $|x^*(x)|\geq 1-\varepsilon$. Hence, from  $|\delta_x|\leq e$ and the pointwise ordering we get $1-\varepsilon\leq e(x^*)$. Hence, $e$ takes at least the value one on $S_{E^*}$. However, $\fbp[E]$ is the closure of the sublattice generated by $\{\delta_x : x\in E\}$. By \cite[p.~204, Exercise  8(b)]{AB}, a typical element of the (non-closed) sublattice generated by $\{\delta_x : x\in E\}$ is of the form $f=\bigvee_{k=1}^n\delta_{x_k}-\bigvee_{k=1}^n\delta_{y_k}$ with $n\in \mathbb{N}$, $x_1,\dots,x_n,y_1\dots,y_n\in E$. Since $E$ is infinite dimensional, we can find $x^*\in S_{E^*}$ such that $x^*(x_k)=x^*(y_k)=0$ for $k=1,\dots, n$. Then 
$$
\|e-f\|\geq |e(x^*)-f(x^*)|\geq 1.
$$
Hence, $e$ is not in the closure of this sublattice, so is not in $\fbp[E]$.
\end{proof}

\begin{rem}
By contrast, $\fbl^{(\infty)}[E]$ may be linearly isomorphic to a Banach lattice with strong unit: \Cref{p:FBL infty isomorphic to CK} shows that, if $E$ separable, then $\fbl^{(\infty)}[E]$ is isomorphic to $C[0,1]$. Separability is essential here, per \Cref{r:non-sep not C(K)}.
\end{rem}

\begin{rem}
The proof of \Cref{No unit} can be adapted to show that for infinite dimensional $E$, $\fbl^{(\infty)}[E]$ cannot be monotonically bounded (increasing norm bounded nets are order bounded). Indeed, if $\fbl^{(\infty)}[E]$ were monotonically bounded, we could order the finite subsets of $B_{E}$ by inclusion, and consider the net $\{x_1,\dots, x_n\} \mapsto |\delta_{x_1}|\vee \cdots\vee |\delta_{x_n}|$. By definition of the $\fbl^{(\infty)}$-norm, this net is norm bounded, hence order bounded. Hence, there exists $e\in \fbl^{(\infty)}[E]$ such that $|\delta_x|\leq e$ for all $x\in B_{E}$ and we proceed as in the proof of  \Cref{No unit}  to reach a contradiction. 
Note, however, that it \textit{is} possible for $\fbl[E]$ to be monotonically bounded - even strong Nakano (cf.~\cite[Theorem 4.11]{ART}). On the other hand, there are Banach spaces $E$ for which $\fbl[E]$ does not even have the  Fatou property (\cite[Theorem 4.13]{ART}), which is a non-trivial weakening of the strong Nakano property. A characterization of when $\fbl[E]$ has these, or related, properties in terms of properties of $E$ is not currently known.
\end{rem}

We now characterize when $\fbp[E]$ has a quasi-interior point. Recall that an element $e$ of a Banach lattice $X$ is a \textit{quasi-interior point} if the closed ideal generated by $e$ is the whole of $X$. The \textit{center} of $X$, denoted $Z(X)$, is the space of all linear operators $T$ on $X$ for which there is a real number $\lambda>0$ satisfying $|Tx|\leq \lambda |x|$ for all $x\in X$. The center is \textit{trivial} if the only elements of $Z(E)$ are the scalar multiples of the identity operator; the center is called \textit{topologically full} if for each $x,y\in X$ with $0\leq x\leq y$ there is a sequence $(T_n)$ in $Z(X)$ with $T_ny\to x$ in norm.

\begin{prop}\label{QIP}
Let $E$ be a non-zero Banach space and $p\in [1,\infty]$. The following are equivalent:
\begin{enumerate}
\item $E$ is separable;
\item $\fbp[E]$ has a quasi-interior point;
\item $Z(\fbp[E])$ is topologically full;
\item $Z(\fbp[E])$ is non-trivial.
\end{enumerate}
\end{prop}
\begin{proof}
If $E$ is separable then $\fbp[E]$ is separable and, therefore,
has a quasi-interior point. 
Conversely, suppose $\fbp[E]$ has a quasi-interior point, say $e$. Without loss of generality, $e\ge 0$. If $x^*\in B_{E^*}$ satisfies $e(x^*)=0$ then $f(x^*)=0$ for every $f\in I_e$ and, therefore, for all $f\in\fbp[E]$. Thus, $e$ only vanishes at 0. For every $n$, let $U_n=\{x^*\in B_{E^*} : e(x^*)<\frac1n\}$. Then $U_n$ is a weak*-open subset of $B_{E^*}$
and $\bigcap_{n=1}^\infty U_n=\{0\}$. Now the relevant direction of the proof of \cite[Theorem 5.1, p.~134]{Conway} shows that $E$ is separable. This shows $(1) \Leftrightarrow (2)$. \\

The rest of the proof is inspired by \cite[Theorem
8.4]{dePW}. From the proof of \cite[Theorem 8.4]{dePW}, every
Banach lattice with a quasi-interior point has a topologically
full center. Since $E$ is non-zero, the center being
topologically full implies it is non-trivial. For the
implication (4)$\Rightarrow$(1), suppose that the center is
non-trivial. By \cite[Theorem 3.1]{Wickhom}, $\{0\}$ is a
$G_\delta$ set which, as before, implies that $E$
is separable.
\end{proof}

\begin{rem}
Unlike with strong units and quasi-interior points, $\fbp[E]$ always has a weak unit, as was noted in \Cref{p:basicprop}.
\end{rem}

\subsection{Number of generators}\label{generators}

For a Banach lattice $X$, we denote by ${\mathbf{n}}(X)$ the smallest cardinality of a set $S$ which generates $X$ as a Banach lattice. For general information on this parameter, see \cite[Section V.2]{Schaefer74}.

\begin{prop}\label{number of generators}
Suppose $E$ is a Banach space.
\begin{enumerate}
    \item If $\dim E = n \in \Nat$, then ${\mathbf{n}}(\fbp[E]) = n$.
    \item If $\dim E = \infty$, then ${\mathbf{n}}(\fbp[E]) = \infty$.
\end{enumerate}
\end{prop}

\begin{proof}
 If $\dim E = n$, and $e_1, \ldots, e_n \in E$ form a basis of $E$, then $\delta_{e_1}, \ldots, \delta_{e_n}$ also generate $\fbp[E]$ as a Banach lattice. 
 \\
  
 Now consider $f_1, \ldots, f_m \in \fbp[E]$, with $m < \dim E$. By Borsuk-Ulam Theorem, these functions cannot separate points of the sphere of $E^*$: There exist distinct $e_1^*, e_2^* \in S_{E^*}$ such that $f_i(e_1^*) = f_i(e_2^*)$ for $1 \leq i \leq m$. As point evaluations are continuous on $\fbp[E]$, $f(e_1^*) = f(e_2^*)$ for any $f$ in the Banach lattice $L$ generated by $f_1, \ldots, f_m$ inside of $\fbp[E]$.
There exists $e \in E$ so that $e_1^*(e) \neq e_2^*(e)$, or equivalently, $\delta_e(e_1^*) \neq \delta_e(e_2^*)$. Consequently, $L$ is a proper sublattice of $\fbp[E]$.
\end{proof}

\subsection{Weakly compactly generated spaces}\label{WCGsection}
Free Banach lattices  played a fundamental role in solving a problem regarding weak compact generation, raised by J. Diestel in a conference in La Manga (Spain) in 2011 (see \cite{ART}). Recall that a Banach space $E$ is weakly compactly generated (WCG) provided there is a weakly compact set $K\subseteq E$ whose linear span is dense. The Diestel question was first analyzed in \cite{AGLRT}, where the following terminology was introduced: A Banach lattice $X$ is weakly compactly generated as a lattice (LWCG, for short) if there is a weakly compact set $K\subseteq X$ so that the sublattice generated by $K$ is dense in $X$. Diestel asked whether, for Banach lattices, the notions of LWCG and WCG are equivalent. A few years later, this was answered in the negative: \cite{ART} shows that $\fbl[\ell_p(\Gamma)]$ (for $1<p\leq2$) is LWCG but not WCG as long as the index set $\Gamma$ is uncountable.
\\

    The following observation was made in \cite{ART} for $p=1$: If $E$ is a $p$-convex Banach lattice, and $\fbp[E]$ is LWCG, then so is $E$. This is because the identity $I : E \to E$ extends to a surjective lattice homomorphism $\widehat{I} : \fbp[E] \to E$, and a lattice homomorphic image of an LWCG lattice is again LWCG.
\\

Clearly, if $E$ is WCG, then $\fbp[E]$ is LWCG. We can also establish a partial converse:

\begin{prop}\label{p:ordercont}
Suppose $1 \leq p \leq \infty$, and $E$ is either a $p$-convex order continuous Banach lattice, or an AM-space. Then $\fbp[E]$ is LWCG if and only if $E$ is WCG.
\end{prop}

\begin{proof}
As noted in the above paragraph, we only need to show that, if $E$ is a $p$-convex order continuous Banach lattice or an AM-space, and $\fbp[E]$ is LWCG, then $E$ is WCG. The reasoning above shows that, if $\fbp[E]$ is LWCG, then so is $E$. Consequently, $E$ is WCG (apply \cite[Theorem 2.2]{AGLRT} for AM-spaces, and \cite[Theorem 3.1]{AGLRT} in the order continuous case).
\end{proof}

\begin{rem}\label{inside of WCG}
Let $\fbp[E]$ be LWCG and let $K\subseteq \fbp[E]$ be a weakly compact set generating $\fbp[E]$ as a lattice. Note that if $K\subseteq \phi(E)$, then $E$ is WCG. Indeed, let $F$ denote the closed linear span of $K$, and suppose $F\subsetneq \phi(E)$. Let $x\in E$ be such that $\delta_x\notin F$. By Hahn-Banach we can take $x^*\in B_{E^*}$ such that $y(x^*)=x^*(\phi^{-1}(y))=0$ for $y\in K$ and $x^*(x)>0$. As $y(x^*)=0$ for every $y\in K$, and $K$ generates $\fbp[E]$ it follows that $f(x^*)=0$ for every $f\in \fbp[E]$. This is a contradiction with $\delta_x(x^*)>0$. Therefore, $F=\phi(E)$ and $E$ is WCG. However, there seems to be a priori no reason to guarantee that when $\fbp[E]$ is LWCG, there is a generating weakly compact set lying in $\phi(E)$.
\end{rem}

We can characterize when $\fbp[E]$ is LWCG as follows:

\begin{prop}\label{p:equivalence}
Given a Banach space $E$, the following are equivalent:
\begin{enumerate}
\item[(i)] $\fbp[E]$ is LWCG.
\item[(ii)] There exist a WCG Banach space $F$ and a lattice homomorphism $T:\fbp[F]\rightarrow \fbp[E]$ with $\|T\|=1$ and dense range.
\end{enumerate}
\end{prop}

\begin{proof}
Suppose first that $\fbp[E]$ is LWCG. Let $K\subseteq \fbp[E]$ be a weakly compact set whose lattice span is dense. Let $F\subseteq \fbp[E]$ be the closed linear span of $K$. Clearly $F$ is WCG, and the formal inclusion $\iota:F\rightarrow \fbp[E]$ extends to a lattice homomorphism $T=\widehat\iota:\fbp[F]\rightarrow \fbp[E]$ with $\|T\|=1$. Since the lattice span of $F$ is dense in $\fbp[E]$ it follows that $T$ has dense image.
\\

For the converse implication, note that if $F$ is WCG, then $\fbp[F]$ is LWCG. The result follows directly from \cite[Proposition 2.1]{AGLRT}.
\end{proof}

\begin{rem}\label{r:FBLpLWCG}
Note that for $p\leq q\leq \infty$, the formal inclusion $\fbp[E]\hookrightarrow \fbl^{(q)}[E]$ has dense image. Hence, if $\fbp[E]$ is LWCG, then so is $\fbl^{(q)}[E]$ for every $p\leq q\leq \infty$.
\end{rem}

\begin{rem}\label{WCG-subspace}
Recall that a Banach space $E$ is a subspace of a WCG space if and only if its dual unit ball $B_{E^*}$ in the weak$^*$ topology is an Eberlein compact. Moreover, $C(K)$ is WCG if and only if $K$ is Eberlein compact. See, for example, \cite{FMZ} and \cite[Theorem 14.9]{fab-ultimo}.
\end{rem}

Although at this point we do not know whether $E$ must be WCG whenever $\fbp[E]$ is LWCG, we can at least show that if $\fbp[E]$ is LWCG, then $E$ must be a subspace of a WCG space. We will use the following result, which actually proves quite a bit more than what is needed to deduce that $E$ must be a subspace of a WCG space:

\begin{prop}\label{p:LWCG+eberlein}
If $\fbp[E]$ is LWCG then there is a positively homogeneous homeomorphism between $B_{E^*}$ with its weak$^*$ topology and a weakly compact set in a Banach space mapping weakly $p$-summable sequences to weakly $p$-summable sequences.
\end{prop}

\begin{proof}
Suppose $\fbp[E]$ is LWCG. By  \Cref{p:equivalence}, there exist a WCG Banach space $F$ and a lattice homomorphism $T:\fbp[F]\rightarrow \fbp[E]$ with dense image. 
By \Cref{Phi T for specific T} below, it follows that the induced map $\Phi_T:B_{E^*}\rightarrow B_{F^*}$ is in particular injective, and since it is weak$^*$ to weak$^*$ continuous, we deduce that $B_{E^*}$ is homeomorphic to $\Phi_T(B_{E^*})$ both with the weak$^*$ topology. 
\\

As $F$ is WCG, \cite[Theorem 13.20]{fab-ultimo} guarantees the existence of an injective weak$^*$-weak continuous bounded linear operator $S:F^*\rightarrow c_0(\Gamma)$ for some $\Gamma$. 
The composition $S\Phi_T$ is positively homogeneous, injective, and weak$^*$-weak continuous, so it defines a homeomorphism between $B_{E^*}$ and its image (a weakly compact set in $c_0(\Gamma)$). 
Further, by \Cref{Lattice homo}, $\Phi_T$ sends weakly $p$-summable sequences to weakly $p$-summable sequences, hence so does $S\Phi_T$; this completes the proof.
\end{proof}

\begin{cor}\label{c:LWCGeberlein}
If $\fbp[E]$ is LWCG then $E$ is a subspace of a WCG space.
\end{cor}

\begin{proof}
By Proposition \ref{p:LWCG+eberlein}, $B_{E^*}$ in its weak$^*$ topology is homeomorphic to a weakly compact set in a Banach space. Thus, $B_{E^*}$ is Eberlein compact. Now use Remark~\ref{WCG-subspace}.
\end{proof}

\begin{rem}
 We outline an alternative approach to \Cref{c:LWCGeberlein}, which is simpler but less informative than going through \Cref{p:LWCG+eberlein}. Suppose $\fbp[E]$ is LWCG. As $\fbp[E]$ is dense in $\fbl^{(\infty)}[E]$, the latter lattice is also LWCG, hence WCG by \cite[Theorem 2.2]{AGLRT}. 
\end{rem}

Above, we noted that, if $E$ is WCG, then $\fbp[E]$ is LWCG. The WCG assumption on $E$ cannot be weakened to $E$ being LWCG (provided it is a Banach lattice).
Indeed, take $E=\fbl[\ell_2(\Gamma)]$ for any uncountable $\Gamma$. Clearly, $E$ is LWCG; however, it contains a subspace isomorphic to $\ell_1(\Gamma)$ 
by \cite[Theorem 5.4]{ART}.
Therefore, if $\fbp[E]$ were LWCG, \Cref{c:LWCGeberlein} would imply that $\ell_1(\Gamma)$ embeds into a WCG space. This is impossible (see e.g.~\cite[proof of Corollary 5.5]{ART}).
\\

In a similar fashion, one can show that, for $p\in [1,\infty)$, it may happen that $E$ is WCG, while $\fbp[E]$ does not embed into a WCG space. Indeed, take $E=\ell_2(\Gamma)$. Modifying the proof of \cite[Theorem 5.4]{ART} with the help of \Cref{Walsh ex}, we conclude that $\fbp[E]$ contains a copy of $\ell_1(\Gamma)$, hence it cannot embed into a WCG space when $\Gamma$ is uncountable.
\\

Observe that the converse of Corollary \ref{c:LWCGeberlein} does not hold in general. Indeed, it is well-known that $L_1(\mu)$ is WCG, for any $\sigma$-finite measure $\mu$ (indeed, it suffices to consider the case of $\mu$ being a probability measure; then the unit ball of $L_2(\mu)$ is relatively weakly compact, and generates $L_1(\mu)$). \cite{Rosenthal} gives an example of a non-WCG subspace $X_{\mathcal R}$ of $L_1(\mu)$, which has a long unconditional basis. Thus, $X_{\mathcal R}$ is an order continuous Banach lattice, so $\fbl[X_{\mathcal R}]$ is not LWCG, by Proposition \ref{p:ordercont}.
\\

As noted above, for $p\in [1,\infty)$, it is in general false that if $E$ is a subspace of a WCG space, then so is $\fbp[E]$. However, this is true for $p=\infty$:
\begin{prop}
Let $E$ be a Banach space. The following are equivalent:
\begin{enumerate}\label{Eberlein}
    \item $(B_{E^*},w^*)$ is Eberlein compact.
    \item $(B_{\fbl^{(\infty)}[E]^*},w^*)$ is Eberlein compact.
    \item $\fbl^{(\infty)}[E]$ is a sublattice of a WCG Banach lattice.
\end{enumerate}
Furthermore, if $E$ and $F$ are Banach spaces such that $(B_{E^*},w^*)$ is Eberlein compact and $\fbl^{(\infty)}[F]$ is a subspace of $\fbl^{(\infty)}[E]$, then $(B_{F^*},w^*)$ is Eberlein compact. 
\end{prop}

\begin{proof}
We use Remark~\ref{WCG-subspace}.
The implication (3)$\Rightarrow$(2) follows immediately from it. If (2) holds, then $\fbl^{(\infty)}[E]$ is a subspace of a WCG space, and since $E$ is a subspace of $\fbl^{(\infty)}[E]$, (1) follows.
\\

(1)$\Rightarrow$(3): If $(B_{E^*},w^*)$ is Eberlein compact, then  $C(B_{E^*})$ is WCG. Since $\fbl^{(\infty)}[E]$ is a sublattice of $C(B_{E^*})$,  it is a sublattice of a WCG Banach lattice.
\\

To address the ``furthermore'' statement, suppose $(B_{E^*},w^*)$ is Eberlein compact, and $\fbl^{(\infty)}[F]$ embeds as a subspace into  $\fbl^{(\infty)}[E]$. By (1)$\Rightarrow$(3), $\fbl^{(\infty)}[E]$ embeds into a WCG space, hence the same is true for $\fbl^{(\infty)}[F]$. By (2)$\Rightarrow$(1), $(B_{F^*},w^*)$ is Eberlein compact.
\end{proof}

\begin{rem}\label{Eberlein2}
\Cref{Eberlein} implies that for any uncountable set $\Gamma$, $\fbl^{(\infty)}[\ell_1(\Gamma)]$ does not embed (isomorphically) into $\fbl^{(\infty)}[\ell_r(\Gamma')]$ for any $r\in (1,\infty)$ and any set $\Gamma'$. However -- as a special case of the results in the next section -- we will see that $\fbl^{(\infty)}[\ell_1]$ and $\fbl^{(\infty)}[\ell_r]$ are lattice isometric. 
\end{rem}

\subsection{Complemented copies of $\ell_1$}\label{complem l1}
In this subsection, we show that $E$ contains a complemented subspace isomorphic to $\ell_1$ if and only if $\fbl[E]$ contains $\ell_1$ in various ways. In preparation, we need a lemma relating $T$ and $\widehat{T}$, which complements the various relations between $T$ and $\overline{T}$ discussed in \Cref{Section Tbar}.
\begin{lem}\label{l:weaklycompact}
Let $X$ be a Banach lattice not containing $c_0$. Given a Banach space $E$, an operator $T:E\rightarrow X$ is weakly compact if and only if $\widehat T:\fbl[E]\rightarrow X$ is weakly compact.
\end{lem}

\begin{rem}
\Cref{l:weaklycompact} fails if no restrictions on $X$ are assumed. Indeed, suppose $E$ is $2$-dimensional, $X = \fbl[E]$ (which is lattice isomorphic to $C(S^1)$, where $S^1$ is the unit circle), and $T = \phi_E$. Then $\widehat{T}$ is the identity on $\fbl[E]$, which is not weakly compact, since the latter lattice is not reflexive.
\end{rem}

\begin{proof}
We will make use of the Davis-Figiel-Johnson-Pe\l zcy\'nski factorization method, and in particular its version for Banach lattices explained in \cite[Theorems 5.37 \& 5.41]{AB}. 
\\

Suppose $T:E\rightarrow X$ is weakly compact. Let $W$ denote the convex solid hull of $T(B_E)$, which by \cite[Theorems 4.39 \& 4.60]{AB} is a relatively weakly compact set. Let $\Psi$ be the reflexive Banach lattice induced by $W$ as in \cite[Theorem 5.37 \& 5.41]{AB}. This means that we have a commutative diagram 
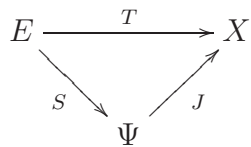
\begin{displaymath}
    \xymatrix{E\ar[dr]_S\ar[rr]^{T}&&X\\
    & \Psi\ar[ru]_J& }
  \end{displaymath}
where $J$ is a lattice homomorphism. Let $\widehat S:\fbl[E]\rightarrow \Psi$ be the lattice homomorphism such that $\widehat S \phi_E=S$. Note that $J\widehat S:\fbl[E]\rightarrow X$ is a lattice homomorphism with the property that $J\widehat S\phi_E=T$. Hence, we must have $\widehat T= J\widehat S$, which implies that $\widehat T$ factors through the reflexive Banach lattice $\Psi$, so $\widehat T$ is weakly compact as claimed. The converse is clear.
\end{proof}

\begin{rem}
The method of proof of \Cref{l:weaklycompact} is quite general. For example, with \Cref{Extend universal property} in mind, a similar argument to \Cref{l:weaklycompact} shows that an operator $T:E\to X$ from a Banach space $E$ to a Banach lattice $X$ is $p$-convex if and only if $\widehat{T}:\fbl[E]\to X$ is $p$-convex.
\end{rem}

\begin{thm}\label{comp ell_1}
For a Banach space $E$, the following are equivalent:
\begin{enumerate}
\item $E$ contains a complemented subspace isomorphic to $\ell_1$.
\item $\fbl[E]$ contains a lattice complemented sublattice isomorphic to $\fbl[\ell_1]$.
\item $\fbl[E]$ contains a lattice complemented sublattice isomorphic to $\ell_1$.
\item $\ell_1$ is a lattice quotient of $\fbl[E]$.
\item $\ell_1$ is a sublattice of $\fbl[E]$.
\item $\ell_1$ is a complemented subspace of $\fbl[E]$.
\end{enumerate}
\end{thm}
\cite[Theorem 4.69]{AB} provides more equivalent characterizations of Banach lattices containing a lattice copy of $\ell_1$.
\begin{proof}
$(1)\Rightarrow (2)$ follows from \cite[Corollary 2.8]{ART} together with the observation that if $P:E\to E$ is a projection onto a subspace isomorphic to $\ell_1$, then $\overline P:\fbl[E]\to\fbl[E]$ is a lattice projection onto the corresponding sublattice isomorphic to $\fbl[\ell_1]$. $(2)\Rightarrow (3)$ follows from \Cref{p:sublattice}. $(3) \Rightarrow (4)$ and $(3) \Rightarrow (5)$ are straightforward. $(5) \Leftrightarrow (6)$ comes from \cite[Theorem 4.69]{AB}.
\\

$(4)\Rightarrow (5)$: Let $P$ be a lattice quotient from $\fbl[E]$ onto $\ell_1$. By \cite[Theorem 11.11]{dePW}, there exists a lattice isomorphic embedding $T:\ell_1\rightarrow \fbl[E]$ such that $PT=id_{\ell_1}$. 
\\

$(5) \Rightarrow (1)$: By \cite[Theorem 2.4.14]{M-N}, $(5)$ holds if and only if $\fbl[E]^*$ is not order continuous, if and only if there is $\varphi_0\in \fbl[E]^*_+$ and $\varepsilon_0>0$ such that for any $f\in \fbl[E]_+$, 
\begin{equation}\label{eq:ballnotequiint}
    B_{\fbl[E]}\nsubseteq [-f,f]+\{g\in\fbl[E]:\varphi_0(|g|)\leq \varepsilon_0\}.
\end{equation}

Let $N_0=\{f\in\fbl[E]:\varphi_0(|f|)=0\}$, which is an ideal in $\fbl[E]$ so that $\varphi_0(|\cdot|)$ defines an AL-norm on the quotient $\fbl[E]/N_0$. By Kakutani representation theorem (cf.~\cite[Theorem 1.b.2]{LT2}), its completion is lattice isometric to $L_1(\mu)$ for some (not necessarily $\sigma$-finite) measure space. Let $Q:\fbl[E]\rightarrow L_1(\mu)$ be the dense range lattice homomorphism induced by the corresponding quotient map.
\\

We claim that $Q\phi_E:E\rightarrow L_1(\mu)$ is not a weakly compact operator. Indeed, if it were, by \Cref{l:weaklycompact}, $Q$ would also be weakly compact. Hence, by \cite[Theorem 5.2.9]{alb-kal} for every $\varepsilon>0$, there is $h\in L_1(\mu)$ such that
$$
Q(B_{\fbl[E]})\subseteq [-h,h]+\varepsilon B_{L_1(\mu)}.
$$
Since $Q$ has dense range, we can find $f'\in \fbl[E]_+$ such that $\|Qf'-h\|_1<\varepsilon$, which implies that
$$
Q(B_{\fbl[E]})\subseteq [-Qf',Qf']+2\varepsilon B_{L_1(\mu)}.
$$
Also, it follows from \cite[Proposition II.2.5]{Schaefer74} and the construction of $Q$ that $Q[-f',f']$ must be dense in $[-Qf',Qf']$. Thus, for every $f\in B_{\fbl[E]}$ there exists $|f''|\leq f'$ and $h\in L_1(\mu)$ such that $\|h\|_1\leq 3\varepsilon$ and
$$
Qf=Qf''+h.
$$
Or equivalently, $\varphi_0(|f-f''|)\leq 3\varepsilon$. This means that
$$
B_{\fbl[E]}\subseteq [-f',f']+\{g\in\fbl[E]:\varphi_0(|g|)\leq 3\varepsilon\},
$$
so taking $\varepsilon<\varepsilon_0/3$ we reach a contradiction with \eqref{eq:ballnotequiint}. 
\\

Therefore, $Q\phi_E$ is not weakly compact as claimed. It follows from \cite[Theorem 5.2.9]{alb-kal} that $Q\phi_E(B_E)$ contains a complemented basic sequence $(h_n)$ equivalent to the canonical basis of $\ell_1$. As a consequence, $E$ must contain a complemented subspace isomorphic to $\ell_1$. Indeed, let $(x_n)\subseteq B_E$ be such that $Q\phi_E(x_n)=h_n$ and let $P:L_1(\mu)\rightarrow L_1(\mu)$ denote a projection onto the span of $(h_n)$; it is straightforward to check that $(x_n)$ must be equivalent to the canonical basis of $\ell_1$, so the linear map $U:[h_n]\rightarrow E$ given by $U(h_n)=x_n$ is bounded, and $UPQ\phi_E$ defines a projection of $E$ onto the span of $(x_n)$.
\end{proof}

\subsection{Upper $p$-estimates and  the local theory of free Banach lattices}\label{s:local theory}

In this section we characterize when $\fbl[E]$ satisfies an upper $p$-estimate. We then use this to study the structure of finite dimensional subspaces and sublattices of free Banach lattices, as well as to find upper $p$-estimate variants of classical theorems on $p$-convexity.
\\

Recall that a Banach lattice $X$ satisfies an \emph{upper $p$-estimate with constant $C$} (resp.~\emph{lower $p$-estimate with constant $C$}) if, for any $x_1, \ldots, x_n \in X$, we have $\bignorm{\bigvee_{k=1}^n \abs{|x_k}} \leq C \bigl(\sum_{k=1}^n \abs{x_k}^p \bigr)^{1/p}$ (resp.~$\bignorm{\sum_{k=1}^n\abs{x_k}} \geq C^{-1}\bigl(\sum_{k=1}^n\abs{x_k}^p \bigr)^{1/p}$).
By \cite[Proposition 1.f.6]{LT2}, it suffices to verify these inequalities when $x_1, \ldots, x_n$ are disjoint. Further, \cite[Proposition 1.f.5]{LT2} shows that $X$ has an upper (resp.~lower) $p$-estimate if and only if $X^*$ has a lower (resp.~upper) $q$-estimate, with the same constant. Here, $\frac{1}{p} +\frac{1}{q} = 1$.
\\

Upper and lower estimates are deeply connected to the convexity and concavity of a Banach lattice, as well as to its type and cotype. In particular, $p$-convexity clearly implies upper $p$-estimates; conversely, upper $p$-estimates imply $r$-convexity for $r < p$ \cite[Theorem 1.f.7]{LT2}. More information about this can be found in  \cite[Section 1.f]{LT2}.
\\

Recall that an operator $T:F\rightarrow E$ is \emph{$(q,p)$-summing} (cf.~\cite[Chapter 10]{DJT}) if there is a $K>0$ such that for every choice of $(x_k)_{k=1}^n\subseteq F$ we have
$$
\Big(\sum_{k=1}^n \|Tx_k\|^q\Big)^{1/q}\leq K \sup_{x^*\in B_{F^*}}\Big(\sum_{k=1}^n |x^*(x_k)|^p\Big)^{1/p}.
$$
We use $\pi_{q,p}(T)$ to denote the least possible constant $K$ in this inequality.

\begin{thm}\label{p:upper est}
Let $E$ be a Banach space and $1\leq p,q\leq\infty$ with $\frac1p+\frac1q=1$. The following are equivalent:
\begin{enumerate}
\item $id_{E^*}$ is $(q,1)$-summing.
\item $\fbl[E]$ satisfies an upper $p$-estimate.
\end{enumerate} 
In this case, the upper $p$-estimate constant of $\fbl[E]$ and $\pi_{q,1}(id_{E^*})$ coincide. 
\end{thm}

\begin{proof}
$(1)\Rightarrow (2)$: We shall show that if $f_1, \ldots, f_n \in \fbl[E]_+$ are disjoint then $\|\sum_{j=1}^n f_j\| \leq  \pi_{q,1}(id_{E^*}) \big( \sum_{j=1}^n \|f_j\|^p \big)^{1/p}$. 
\\

We view elements of $\fbl[E]$ as positively homogeneous functions on the unit ball of $E^*$. Let $g = \sum_{j=1}^n f_j$. Then $\|g\|$ is the supremum of $\sum_k |g(x^*_k)|$, where the finite sequence $(x^*_k) \subseteq E^*$ is such that $\|\sum_k \pm x^*_k\| \leq 1$ for any choice of $\pm$. 
 Fix $(x^*_k)$ as above. For any $j$ let $S_j = \{k : f_j(x^*_k) \neq 0\}$. These sets are disjoint (due to the disjointness of $f_j$'s themselves), and
 $$
 \sum_k |g(x^*_k)| \leq \sum_k \sum_j |f_j(x^*_k)| = \sum_j \sum_{k \in S_j} |f_j(x^*_k)| .
 $$
 For each $j$ let $\alpha_j = \max_{\pm} \|\sum_{k \in S_j} \pm x^*_k\|$, and $y^*_j = {\mathrm{argmax}}_{\pm} \|\sum_{k \in S_j} \pm x^*_k\|$. Then $ \sum_{k \in S_j} |f_j(x^*_k)| \leq \|f_j\| \alpha_j = \|f_j\| \|y^*_j\|$.
 Note that
 $$\max_\pm \| \sum_j \pm y^*_j\| \leq \max_{\pm} \|\sum_k \pm x^*_k\| \leq 1.$$ 
Let $\kappa = \pi_{q,1}(id_{E^*})$, then
 $$
 1 \geq \max_\pm \| \sum_j \pm y^*_j\| \geq \kappa^{-1} \big( \sum_j \|y^*_j\|^q \big)^{1/q} = \kappa^{-1} \big( \sum_j \alpha_j^q \big)^{1/q} .
 $$
 Therefore,
 \begin{align*}
 \sum_k |g(x^*_k)|
 &
 \leq
 \sum_j \sum_{k \in S_j} |f_j(x^*_k)| \leq \sum_j \|f_j\| \alpha_j \leq \big( \sum_j \|f_j\|^p \big)^{1/p} \big( \sum_j \alpha_j^q \big)^{1/q}
 \\
 &
 \leq
 \kappa \big( \sum_j \|f_j\|^p \big)^{1/p} .
 \end{align*}
 We obtain the desired estimate on $\|g\|$ by taking the supremum over all suitable sequences $(x^*_k)$.
 \\

 $(2)\Rightarrow (1)$: 
   Suppose $\fbl[E]$ has an upper  $p$-estimate with constant $C$ (clearly $C \geq 1$). We fix $x_1^*,\dots,x_n^*\in E^*$, and aim to show that
   \begin{equation}
       \Bigl(\sum_{i=1}^n\norm{x_i^*}^q\Bigr)^{\frac1q} \leq C \sup_{x\in B_E}\sum_{i=1}^n\bigabs{x_i^*(x)}.  
       \label{eq:q1 summing}
   \end{equation}
   
   If $\dim E = 1$, then it is easy to see that $\pi_{q,1}(id_\Real) = 1$, hence \eqref{eq:q1 summing} is satisfied (in fact, \cite[Theorem 8.1]{dePW} shows that $\fbl[\Real] = \ell_\infty^2$, hence its upper $p$-estimate constant equals $1$ for all $p$).
   \\
   
   If $\dim E > 1$, then, by a small perturbation argument, we may  assume that no two of the vectors $x_1^*,\dots,x_n^*$ are proportional. Since  $\widehat{x_1^*},\dots,\widehat{x_n^*}$ are lattice homomorphisms in
  $\fbl[E]^*$, they are atoms by \cite[p.~111
  Exercise 5]{AB} and, therefore, disjoint. Fix $\varepsilon>0$. For
  each $i$, pick $f_i\in\fbl[E]_+$ with $\norm{f_i}\le 1$ and
  $\widehat{x_i^*}(f_i)>(1-\varepsilon)\norm{\widehat{x_i^*}}$ or,
  equivalently, $f_i(x_i^*)>(1-\varepsilon)\norm{x_i^*}$. Applying
  Proposition~1.4.13 in~\cite{M-N} to the normalized functionals, we
  may assume that the $f_i$'s are disjoint. We have
\begin{displaymath}
  (1-\varepsilon)\Bigl(\sum_{i=1}^n\norm{x_i^*}^q\Bigr)^{\frac1q}
  \le\Bigl(\sum_{i=1}^nf_i(x_i^*)^q\Bigr)^{\frac1q}
  =\sum_{i=1}^n\lambda_if_i(x_i^*)
\end{displaymath}
for some $\lambda_1,\dots,\lambda_n\in\mathbb R_+$ with
$\sum_{i=1}^n\lambda_i^p=1$.
Put $f=\sum_{i=1}^n\lambda_if_i$.
Since $\fbl[E]$ has the upper $p$-estimate, we get 
\begin{math}
  \norm{f}\le
  C\Bigl(\sum_{i=1}^n\norm{\lambda_if_i}^p\Bigr)^{\frac1p}\le C.
\end{math}
Using the
definition of the $\fbl$ norm, we get
\begin{displaymath}
  \sum_{i=1}^n\lambda_if_i(x_i^*)=\sum_{i=1}^nf(x_i^*)
  \le\norm{f}\sup_{x\in B_E}\sum_{i=1}^n\bigabs{x_i^*(x)}
  \le C\sup_{x\in B_E}\sum_{i=1}^n\bigabs{x_i^*(x)}.  
\end{displaymath}
Since $\varepsilon$ is arbitrary, we conclude that \eqref{eq:q1 summing} holds, thus completing the proof.
\end{proof}

\begin{rem}\label{possible values of q}
If $E$ is infinite dimensional, then, by Dvoretzky Theorem, $id_{E^*}$ cannot be $(q,1)$-summing for $q < 2$. Therefore, $\fbl[E]$ can only have an upper $p$-estimate for $p \leq 2$. In the next section, we shall see that this estimate is sharp, and $\fbl[E]$ can even be $2$-convex (see, for instance, \Cref{p:coincidence type2}). For more general information about possible $q$-convexity of $\fbp[E]$, see \Cref{p:cant be worse than 2-convex}.
On the other hand, recall that if $E$ is finite dimensional, then $\fbp[E]$ is lattice isomorphic to $C(S_{E^*})$, hence it satisfies an upper $r$-estimate for every $r \in [1,\infty]$.
\end{rem}

\begin{rem}
Although \Cref{p:upper est} is stated for $\fbl[E]$, in \Cref{Extrapolation} we will prove an extrapolation result which allows us to characterize when $\fbp[E]$ has non-trivial convexity.
\end{rem}

\begin{cor}\label{c:q1 summing passes to compl subsp}
 Suppose $F$ is a subspace of a Banach space $E$, so that $(F,E)$ has the POE-$1$. Fix $q \in [1,\infty]$. If $id_{E^*}$ is $(q,1)$-summing, then so is $id_{F^*}$.
 \end{cor}
 
 Note that, if $\dim F < \infty$, then $id_{F^*}$ is $(q,1)$-summing for any $q$. If $\dim F = \infty$, then, by \Cref{possible values of q}, we must have $q \in [2,\infty]$.
 
 \begin{proof}
  By \Cref{p:upper est}, $\fbl[E]$ has an upper $p$-estimate, with $1/p + 1/q = 1$. Denote the canonical embedding $F \to E$ by $\iota$. By the POE-$1$, $\overline{\iota} : \fbl[F] \to \fbl[E]$ is a lattice isomorphic embedding, hence $\fbl[F]$ has an upper $p$-estimate as well. Apply \Cref{p:upper est} again to reach the desired conclusion about $id_{F^*}$.
 \end{proof}

We next present the following ``local" version of \Cref{comp ell_1}. We use the shorthand ``$E$ has trivial cotype'' to mean that no non-trivial cotype is present.

\begin{cor}\label{c:criteria for l1}
 For an infinite dimensional Banach space $E$, the following statements are equivalent.
 \begin{enumerate}
 \item $E$ contains uniformly complemented subspaces isomorphic to  $\ell_1^n$.
  \item $\fbl[E]$ contains uniformly lattice-complemented sublattices isomorphic to $\ell_1^n$.
 \item $\fbl[E]$ contains sublattices $\ell_1^n$ uniformly.
 \item $E^*$ has trivial cotype.
 \item $\fbl[E]$ fails to be $p$-convex for any $p>1$.
 \item $\fbl[E]$ contains uniformly complemented subspaces isomorphic to  $\ell_1^n$.
 \item $\fbl[E]^*$ has trivial cotype.
 \end{enumerate}
\end{cor}

Recall that a Banach lattice $X$ is said to contain sublattices $\ell_1^n$ uniformly if there exist lattice isomorphisms $u_n : \ell_1^n \to X_n \subseteq X$ so that $\sup_n \|u_n\| \|u_n^{-1}\| < \infty$.
By Krivine's theorem (see \cite{SchepKr}), the uniform lattice copies of $\ell_1^n$ in $\fbl[E]$ can be taken to be $(1+\varepsilon)$-uniform whenever they exist.
In this case one can select $u_n$'s in such a way that $\lim_n \|u_n\| \|u_n^{-1}\| = 1$.
\\

The following lemma is known, but we include it for the sake of completeness.

\begin{lem}\label{dual no cotype}
 For a Banach space $E$, the following statements are equivalent:
 \begin{enumerate}
     \item $E$ contains uniformly complemented copies of $\ell_1^n$.
     \item $E^*$ contains copies of $\ell_\infty^n$ uniformly.
     \item $E^*$ has trivial cotype.
 \end{enumerate}
\end{lem}

\begin{proof}
 (2)$\Leftrightarrow$(3) is given by \cite[Theorem 14.1]{DJT}, and duality gives us (1)$\Rightarrow$(2). To establish (2)$\Rightarrow$(1), suppose $E^*$ contains copies of $\ell_\infty^n$ uniformly.
 By \cite[Theorem 2.5]{OjP},  we can assume that the said copies of $\ell_\infty^n$ are complemented via weak$^*$ continuous projections, with uniformly bounded norms. Passing to the predual, we conclude that $E$ satisfies (1). 
\end{proof}

\begin{proof}[Proof of \Cref{c:criteria for l1}]
(1)$\Rightarrow$(2) is similar to \Cref{comp ell_1}, where we make use of \Cref{Finite ok} to see that $\fbl[\ell_1^n]$ contains $\ell_1^n$ as a nicely complemented sublattice. The implications (2)$\Rightarrow$(3)$\Rightarrow$(5) and (2)$\Rightarrow$(6) are trivial. (5)$\Rightarrow$(3) is a consequence of the Banach lattice version of Krivine's Theorem \cite{SchepKr}. 
\Cref{dual no cotype} contains (1)$\Leftrightarrow$(4).
\\

(5)$\Rightarrow$(4): If $E^*$ has non-trivial cotype, then by \cite[Theorem 14.1]{DJT} $id_{E^*}$ is $(q,1)$-summing for some $q$. Then, by \Cref{p:upper est}, $\fbl[E]$ has a non-trivial upper estimate, which implies non-trivial convexity \cite[Section 1.f]{LT2}.
\\

(6)$\Leftrightarrow$(7) follows from \Cref{dual no cotype}.
\\

(7)$\Rightarrow$(5): If (7) holds, then by \cite[Section 1.f]{LT2}, $\fbl[E]^*$  cannot be $q$-concave for any finite $q$. By duality, $\fbl[E]$ cannot be $p$-convex for any $p>1$.
\end{proof}

The next remark puts the above results in a broader context:

\begin{rem}\label{Close to a characterization!}
Let $X$ be a Banach lattice.
By \cite[Chapter 16]{DJT} (see also \cite[Section 1.f]{LT2} and \cite{Tal1,Tal2,Tal3}; the relevant results are neatly summarized in \cite{Blasco}), we have the following general implications and no others:
\begin{enumerate}
\item For $2<q<\infty$, $q$-concavity $\Rightarrow$ cotype $q$ $\Leftrightarrow$ $id_X$ is $(q,1)$-summing $\Leftrightarrow$ $X$ has a lower $q$-estimate;
\item For $q=2$, $2$-concavity $\Leftrightarrow$ cotype $2$ $\Rightarrow$ $id_X$ is $(2,1)$-summing $\Rightarrow$ $X$ has a lower $2$-estimate.
\end{enumerate}
By duality \cite[Proposition 1.f.5]{LT2}, if $E$  is a Banach lattice and $1<p<2$ we conclude that $id_{E^*}$ is $(q,1)$-summing ($1/p+1/q=1$) if and only if $E^*$ has a lower $q$-estimate if and only if $E$ has an upper $p$-estimate.   Combining these observations with \Cref{p:upper est} we see that:
\end{rem}

\begin{cor}\label{Char of upper p for Banach lattices}
Suppose $E$ is a Banach lattice and $1<p<2$. The following are equivalent:
\begin{enumerate}
\item $E$ satisfies an upper $p$-estimate;
\item $\fbl[E]$ satisfies an upper $p$-estimate.
\end{enumerate}
\end{cor}

\Cref{p:cant be worse than 2-convex} shows that the above equivalence fails for $p>2$. \\

\Cref{c:q1 summing passes to compl subsp} immediately implies an upper $p$-estimate version of \cite[Theorem 1.d.7]{LT2}:

\begin{cor}\label{upper p version}
Suppose $p \in(1,2)$, $E$ and $F$ are Banach lattices, and $\iota : F \to E$ is a linear isomorphic embedding, so that $\iota(F)$ is complemented in $E$ or, more generally, that $(\iota(F),E)$ has POE-$1$. Then, if $E$ has an upper $p$-estimate, then the same is true for $F$. 
\end{cor}

The existence of a complemented copy of $L_2$ inside of $L_p$ shows that \Cref{upper p version} fails for $2 < p < \infty$. For $p=2$, the proof only shows that, if $id_{E^*}$ is $(2,1)$-summing, then $F$ has an upper $2$-estimate; we do not know if the assumption on $E$ can be relaxed to it merely having an upper $2$-estimate.
In connection to this, we should also mention a ``dual'' analogue of \Cref{upper p version}, discussed on \cite[p.~98-99]{LT2}. Namely, suppose a Banach lattice $F$ embeds isomorphically into a Banach lattice $E$ with a lower $p$-estimate. If $p \in (2,\infty)$, then $F$ has a lower $p$-estimate as well; this is no longer true for $p=2$.
\\

As mentioned previously, if $E$ is finite dimensional then $\fbl[E]$ is lattice isomorphic to a $C(K)$-space, so is in particular $\infty$-convex. The situation is different in the infinite dimensional setting.

\begin{prop}\label{p:cant be worse than 2-convex}
Suppose $E$ is an infinite dimensional Banach space. If $\fbp[E]$ is $q$-convex, then $q \leq \max\{2,p\}$.
\end{prop}

\begin{proof}
Fix $n \in \Nat$. Use Dvoretzky Theorem to find norm $2$ vectors $x^*_1, \ldots, x^*_n \in E^*$, so that the inequality
$$ \big( \sum_j |a_j|^2 \big)^{1/2} \leq \big\| \sum_j a_j x^*_j \big\| \leq 3 \big( \sum_j |a_j|^2 \big)^{1/2} $$
holds for any scalars $a_1, \ldots, a_n$.
Use Local Reflexivity to find $x_1, \ldots, x_n \in E$, of norm not exceeding $1$, and biorthogonal to the $x^*_j$'s. We will establish that
$$
\| f \|_{\fbp[E]} \gtrsim n^{1/r}, {\textrm{   where   }} r = \max\{2,p\} ,  {\textrm{  and  }}
f = \big( \sum_j |\delta_{x_j}|^q \big)^{1/q} .
$$
We shall achieve this by testing $f$ against $x^*_1, \ldots, x^*_n$. Let $F = \spn[x^*_1, \ldots, x^*_n]$. By applying Local Reflexivity, and then passing from $E^{**}$ to $E^{**}/F^\perp \sim F^*$, we obtain
$$
\sup_{x\in B_E} \big( \sum_{j=1}^n |x^*_j(x)|^p \big)^{1/p} = \sup_{x\in B_{F^*}} \big( \sum_{j=1}^n |x^*_j(x)|^p \big)^{1/p} \leq 3 n^\gamma,
$$
with
$$
\gamma = \left\{ \begin{array}{ll} \frac1p - \frac12  &  1 \leq p \leq 2,   \\  0  &  p \geq 2.   \end{array} \right. 
$$
Note that, for any $x^* \in E^*$, $f(x^*) = \big( \sum_j |x^*(x_j)|^q \big)^{1/q}$, and therefore,
$\big( \sum_j |f(x^*_j)|^p \big)^{1/p} = n^{1/p}$. 
By \eqref{eq:ART}, 
$$
\| f \|_{\fbp[E]} \gtrsim \frac{n^{1/p}}{n^\gamma} = n^{1/r} 
$$
(with $r$ as above). On the other hand, if $\fbp[E]$ is $q$-convex, then
$$
\| f \|_{\fbp[E]} \lesssim \big( \sum_j \big\| \delta_{x_j} \big\|_{\fbp[E]}^q \big)^{1/q} \sim n^{1/q} ,
$$
giving the desired estimate for $q$.
\end{proof}

We finish this section with some applications to the local theory.
\\

Recall that for a Banach lattice $E$, the \textit{upper index of $E$} is 
$$S(E)=\sup\{p\geq 1 : E \ \text{satisfies an upper $p$-estimate}\}.$$
By \cite[Section 1.f]{LT2}, ``upper $p$-estimate" can be replaced by ``$p$-convex" in the definition of $S(E)$. $S(E)$ is very important in the local theory of Banach lattices. Indeed, a theorem of Krivine \cite{SchepKr} states that an infinite dimensional Banach lattice $E$ contains, for all integers $n$ and all $\varepsilon>0$, a $(1+\varepsilon)$-lattice copy of $\ell_p^n$ when $p=S(E)$.
\\

If $E$ is a finite dimensional Banach space, then $\fbp[E]$ is lattice isomorphic to an AM-space, hence $S(\fbp[E])=\infty$. 
For infinite dimensional $E$, \Cref{p:cant be worse than 2-convex} shows that $S(\fbp[E]) = p$ for $2 \leq p \leq \infty$, while $p \leq S(\fbp[E]) \leq 2$ for $1 \leq p \leq 2$. In particular, we conclude that for infinite dimensional Banach lattices, the indices are related as follows:

\begin{cor}\label{Upper indices}
Suppose $E$ is an infinite dimensional  Banach lattice. Then $$S(E)\wedge 2 = S(\fbl[E]).$$ 
\end{cor}

\begin{rem}\label{Sub structure}
On the other end of the spectrum, note that $\fbl[E]$ always contains a lattice copy of $c_0$ as long as $\dim E\geq 2$, so in particular contains uniform lattice copies of $\ell_\infty^n$. Further, $\fbl[\ell_1^2]\simeq C(S^1)$ contains isomorphic copies of every separable Banach space, hence so does $\fbl[E]$ for every $E$ with $\dim E\geq 2$. One should note, however, that $\fbl[E]$ being universal is restricted to separable spaces; for other density characters it is an interesting problem to classify the subspaces of $\fbl[E]$ up to isomorphism. For example, $\fbl[E]$ has the same density character as $E$ (\cite[Section 3]{ART}), but, as was shown in \cite{ATV}, when $1\leq p\leq 2$ and $\Gamma$ is uncountable, $\fbl[\ell_p(\Gamma)]$ does not embed into a weakly compactly generated Banach space, and in particular does not embed into $\fbl[\ell_q(\Gamma)]$, $2<q<\infty$, which is WCG. These simple facts will play a role in the next section when we compare $\fbp[E]$ and $\fbl^{(q)}[F]$; in particular, when $E$ and $F$ are separable, we will aim to distinguish these spaces by showing that one does not linearly embed onto a complemented subspace of the other.
\end{rem}

\begin{rem}
We also note that the disjoint sequence structure of $\fbl[E]$ can be very complicated. Indeed, when $E$ is the complementably universal space for unconditional bases (see \cite[Theorem 2.d.10]{LT1} for the construction), then, by \Cref{p:sublattice}, $\fbl[E]$ contains lattice copies of every separable order continuous atomic lattice (i.e., every Banach lattice with lattice structure induced by an unconditional basis). 
\end{rem}

\subsection{Automatic convexity, factorization theory and  isomorphisms between $\fbp[E]$ and $\fbl^{(q)}[F]$}\label{Convexity}

In this section, we characterize when $\fbp[E]$ is $q$-convex via strong factorizations (representing an operator as a composition of two or more, one of which is a lattice homomorphism), and then use $\fbp[E]$ as a tool to study the  classical factorization theory. We also give various situations where we can prove that $\fbp[E]$ and $\fbl^{(q)}[F]$ are lattice isomorphic, and other situations where we can prove that one of these spaces does not even linearly embed as a complemented subspace of the other.  \\

We begin with some preparation:

\begin{prop}\label{p:isomorphic characterization of fbp}
 Suppose $E$ is a Banach space, $Z$ is a $p$-convex Banach lattice with constant $1$, and $\iota : E \to Z$ is an isometric embedding with the following properties:
 \begin{enumerate}
  \item $Z$ is generated by $\iota(E)$ as a Banach lattice.
  \item There exists a constant $C$ so that for every contraction $T : E \to L_p(\mu)$ there is a lattice homomorphism $T' : Z \to L_p(\mu)$ with $T' \iota=T$ and $\|T'\| \leq C$. 
 \end{enumerate}

Then $Z$ is $C$-lattice isomorphic to $\fbp[E]$. More precisely, the canonical extension $\widehat{\iota} : \fbp[E] \to Z$ is invertible and $\|\widehat{\iota}^{-1}\| \leq C$.
\end{prop}

\begin{proof}
From the definition of $\FBLp[E]$, there exists a unique lattice homomorphism 
 $\widehat\iota\colon\FBLp[E]\to Z$ such that $\widehat\iota\phi_E=\iota$ and
$\norm{\widehat\iota}=1$. Observe that $\widehat\iota$ has dense range. Indeed,
fix $z\in Z$ and $\varepsilon>0$. By assumption, there exists $u$ in
the sublattice generated by $\iota(E)$ such that
$\norm{z-u}<\varepsilon$. We can write $u$ as a lattice-linear
expression $u=F(\iota x_1,\dots,\iota x_n)$ for some $x_1,\dots,x_n\in
E$. Then $u=\widehat\iota
F(\delta_{x_1},\dots,\delta_{x_n})\in\Range\widehat\iota$.
\\

Let $f\in\FBLp[E]$ with $\norm{f}>1$. By the definition of the
$\FBLp$-norm, there exists $n\in\mathbb N$ and a contractive operator
$T\colon E\to\ell_p^n$ such that $\norm{\widehat{T}f}>1$, where
$\widehat{T}$ is the unique lattice homomorphism
$\widehat{T}\colon\FBLp[E]\to\ell_p^n$ such that
$\widehat{T}\phi_E=T$. By assumption, there exists a lattice
homomorphism $T'\colon Z\to\ell_p^n$ such that $T'\iota=T$ and
$\norm{T'}\le C$. We have $T'\widehat\iota\phi_E(x)=T'\iota x=Tx=\widehat{T}\phi_E(x)$ for
every $x\in E$. It follows that $T'\widehat\iota$ agrees with $\widehat{T}$ on $\phi(E)$ and,
therefore, $T'\widehat\iota=\widehat{T}$. We now have
\begin{math}
  1<\norm{\widehat{T}f}=\norm{T'\widehat\iota f}\le C\norm{\widehat\iota f}.
\end{math}
It follows that $\widehat\iota$ is bounded below. In particular, it is
invertible and $\norm{\widehat\iota^{-1}}\le C$.
\end{proof}

A standard direct sum argument implies:

\begin{cor}\label{p:isomorphic characterization of fbp 2}
 Suppose $E$ is a Banach space, $Z$ is a $p$-convex Banach lattice with constant $1$, and $\iota: E \to Z$ is an isometric embedding with the following properties:
 \begin{enumerate}
  \item $Z$ is generated by $\iota(E)$ as a Banach lattice.
  \item Every contraction $T : E \to L_p(\mu)$ extends to a lattice homomorphism $T' : Z \to L_p(\mu)$.
 \end{enumerate}
Then $Z$ is lattice isomorphic to $\fbp[E]$. More precisely, the canonical extension $\widehat{\iota} : \fbp[E] \to Z$ is a surjective isomorphism.
\end{cor}

\begin{proof} 
By \Cref{p:isomorphic characterization of fbp}, it is enough to show that there is a uniform constant $C$ such that every contraction $T:E\rightarrow L_p(\mu)$ extends to a lattice homomorphism $T':Z\rightarrow L_p(\mu)$ with $\|T'\|\leq C$. Suppose this is not the case, and let $T_n:E\rightarrow L_p(\mu_n)$ be such that $\|T_n\|=1$, but any lattice homomorphism $S:Z\rightarrow L_p(\mu_n)$ extending $T_n$ has $\|S\|\geq 2^{n/p}n$.

Consider $L_p(\nu)$ to be the infinite $\ell_p$ sum of the spaces $L_p(\mu_n)$ and let $T:E\rightarrow \ell_p(L_p(\mu_n))=L_p(\nu)$ be given by $Tx=(\frac{T_nx}{2^{n/p}})_{n=1}^\infty$. Note that
$$
\|Tx\|=\Big(\sum_{n=1}^\infty \frac{\|T_nx\|^p}{2^n}\Big)^{\frac1p}\leq \|x\|.
$$
Let $T':Z\rightarrow L_p(\nu)$ be a lattice homomorphism extending $T$. Note that if $\pi_n:\ell_p(L_p(\mu_n))\rightarrow L_p(\mu_n)$ denotes the canonical band projection, we have that the operator $T'_n=2^{n/p}\pi_n T':Z\rightarrow L_p(\mu_n)$ is a lattice homomorphism extending $T_n$. Hence, $2^{n/p}n\leq\|2^{n/p}\pi_n T'\|$, which yields $n\leq \|T'\|$. As this holds for every $n\in\mathbb N$, we get a contradiction with the fact that $T'$ is bounded.
\end{proof}
 
\Cref{p:isomorphic characterization of fbp} has a natural analogue for free Banach lattices satisfying an upper $p$-estimate. We first recall some facts on weak $L_p$-spaces and $(p,\infty)$-convex operators:
\\

For $f\in L_0(\mu)$ and $0<p<\infty$, let $$\|f\|_{p,\infty}=\{\sup_{t>0} t^p\mu(\{|f|>t\})\}^{1/p}.$$
The space $L_{p, \infty}(\mu)$ is the set of all $f\in L_0(\mu)$ such that $\|f\|_{p,\infty}<\infty.$ It is well-known that when $\mu$ is $\sigma$-finite and $0<r<p$ the expression 
$$\vertiii{f}_{p,\infty,[r]}:=\sup_{0<\mu(E)<\infty} \mu(E)^{-\frac{1}{r}+\frac{1}{p}}\left(\int_E|f|^rd\mu\right)^{\frac{1}{r}}$$
satisfies $$\|f\|_{p,\infty}\leq \vertiii{f}_{p,\infty,[r]}\leq \left(\frac{p}{p-r}\right)^\frac{1}{r}\|f\|_{p,\infty}$$
(see, for example, \cite[Exercise 1.1.12]{Grafakos}).

If $(X,\mu)$ is a measure space with $\mu$ finite, $0<q<p$ and $f\in L_{p,\infty}(\mu)$ then 
\begin{equation}
\int_X|f(x)|^qd\mu(x)\leq \frac{p}{p-q}\mu(X)^{1-\frac{q}{p}}\|f\|_{p,\infty}^q,
\label{eq:q versus p,inf norm}
\end{equation}
i.e., $L_{p,\infty}(\mu)$ continuously injects into $L_q(\mu)$ with control of the constants (see \cite[Exercise 1.1.11]{Grafakos}). This will be used in the proof of \Cref{p:isomorphic characterization of fbl upper p} below to justify a certain multiplication operator being bounded by universal constants.
\\

Below, we concern ourselves with $p \in (1,\infty)$. Equip $L_{p,\infty}(\mu)$ with the equivalent norm $\vertiii{ \cdot }_{p,\infty,[1]}$, or, for short, $\vertiii{ \cdot }_{p,\infty}$. This turns $L_{p,\infty}$ into a Banach lattice. Moreover,   $(L_{p,\infty}, \vertiii{ \cdot }_{p,\infty})$ has the upper $p$-estimate with constant $1$. To establish the latter fact, we show that the inequality $\vertiii{ \vee_{i=1}^n |f_i| }_{p,\infty} \leq ( \sum_{i=1}^n \vertiii{ f_i }_{p,\infty}^p \big)^{1/p}$ holds for any $f_1, \ldots, f_n \in L_{p,\infty}(\Omega,\mu)$. In other words, we show that, for any $E \subseteq \Omega$, we have
 $$ \sup_{E \subseteq \Omega} \mu(E)^{1/p - 1} \int_E \vee_i |f_i| \leq \big( \sum_i \vertiii{f_i}_{p,\infty}^p \big)^{1/p} . $$
 Represent $E$ as a union of disjoint sets $E_j$ ($1 \leq j \leq n$), so that $\vee_i |f_i| = |f_j|$ on $E_j$. For the sake of convenience write $p' = p/(p-1)$ (so $1/p + 1/p' = 1$), $a_i = \int_{E_i} |f_i|$, and $b_i = \mu(E_i)^{1/p'}$ (by getting rid of ``redundant'' $f_i$'s, we can assume that $b_i > 0$ for any $i$). Then $ \vertiii{ f_i }_{p,\infty} \geq b_i^{-1} a_i$; therefore, it suffices to show that
 $$
 \Big( \sum_i ( b_i^{-1} a_i )^p \Big)^{1/p} \geq \big( \sum_i b_i^{p'} \big)^{-1/p'} \sum_i a_i .
 $$
 The last inequality is equivalent to
 $$
 \sum_i a_i \leq \Big( \sum_i ( b_i^{-1} a_i )^p \Big)^{1/p} \big( \sum_i b_i^{p'} \big)^{1/p'} ,
 $$
 which is an easy consequence of H\"older's Inequality.
 \\

Let $X$ be a Banach lattice and $E$ a Banach space. Recall that an operator $T:X\to E$  is $(q,p)$-concave if there is a constant $C$ such that, for any $x_1,\dots,x_n\in X$ we have 
$$\left(\sum_{k=1}^n\|Tx_k\|^q\right)^{1/q}\leq C\bigg\|\left(\sum_{k=1}^n|x_k|^p\right)^{1/p}\bigg\|.$$
The least constant that works is denoted $K_{q,p}(T)$. It is easy to see that if $p>q$ then the only $(q,p)$-concave operator is the zero operator. Moreover, $(p,p)$-concave operators are exactly the $p$-concave operators, and for $1\leq p< q<\infty$  an operator is $(q,p)$-concave  if and only if it is $(q,1)$-concave (see \cite[Corollary 16.6]{DJT}). An operator $S: E\to X$ is $(p,q)$-convex if there is a constant $C$ such that for each $x_1,\dots,x_n$ in $E$ we have 
$$\bigg\|\left(\sum_{k=1}^n|Sx_k|^q\right)^{1/q}\bigg\|\leq C\left(\sum_{k=1}^n\|x_k\|^p\right)^{1/p}.$$
There is a natural duality between $(p,q)$-convexity and $(p',q')$-concavity ($1/p + 1/p' = 1 = 1/q + 1/q'$); see \cite[Theorem 16.21]{DJT}.
\\

Following \cite{JLTTT}, we denote by $\fbl_K^{\uparrow p}[E]$ the free Banach lattice satisfying an upper $p$-estimate with constant $K$ over $E$. This is the (necessarily unique) Banach lattice $Z$ so that (i) $Z$ satisfies an upper $p$-estimate with constant $K$; (ii) there is an isometric embedding $\psi : E \to Z$, generating $Z$ as a lattice; (iii) for any linear operator $T : E \to X$, where $X$ is a Banach lattice satisfying an upper $p$-estimate with constant $K$, there exists a lattice homomorphism $\widehat{T} : Z \to X$, with $\widehat{T} \psi = T$, and $\|\widehat{T}\| = \|T\|$. We write $\fbl^{\uparrow p}[E]$ for $\fbl_1^{\uparrow p}[E]$.
\\

The existence and uniqueness of $\fbl_K^{\uparrow p}[E]$ was established in \cite{JLTTT}. Moreover, the lattices $\fbl^{\uparrow p}_K[E]$ for different values of $K$ are canonically lattice isomorphic: by \cite[Remark 1.5 and its proof]{Pisier}, a Banach lattice satisfying an upper $p$-estimate with constant $K$ can be $K$-renormed to satisfy an upper $p$-estimate with constant one.
\\

Many aspects of $\fbl^{\uparrow p}[E]$  remain mysterious. For instance, no functional representation of this lattice, and no explicit norm arising from it, are known (compare and contrast with \Cref{construction}). However, we have the following result:

\begin{prop}\label{p:isomorphic characterization of fbl upper p}
 Suppose $E$ is a Banach space, $Z$ is a Banach lattice, $1<p<\infty$, and $i : E \to Z$ is an isometric embedding with the following properties:
 \begin{enumerate}
  \item $Z$ is generated by $i(E)$ as a Banach lattice.
  \item There exists a constant $C$ so that every operator $T : E \to L_{p,\infty}(\mu)$ extends to a lattice homomorphism $T' : Z \to L_{p,\infty}(\mu)$ with $\|T'\| \leq C\|T\|$ $(\mu$ is a finite measure$)$.
 \end{enumerate}
Then for any Banach lattice $X$ and any $(p,\infty)$-convex operator $S:E\to X$ there exists a (necessarily unique) lattice homomorphism $S' : Z \to X$ satisfying $S' i=S$. Moreover, $\|S'\| \leq \gamma C K^{(p,\infty)}(S)$, where $K^{(p,\infty)}(S)$ is the $(p,\infty)$-convexity constant of $S$, and the constant $\gamma$ depends on $p$ only.
\end{prop}

\begin{proof}
For brevity, we write $K = K^{(p,\infty)}(S)$.
By the universality of $\fbl[E]$, we have a lattice homomorphism $\widehat i:\fbl[E]\rightarrow Z$ extending $i:E\rightarrow Z$. As $i(E)$ generates $Z$, it follows that $\widehat i$ has dense range. Also, let $\widehat S:\fbl[E]\rightarrow X$ be the lattice homomorphism such that $\widehat S\phi=S$. Consider the following:
\\

\textbf{Claim:} There is a constant $\gamma>0$ (depending only on $p$) such that
\begin{equation}\label{eq:XZ}
\|\widehat S f\|_X\leq \gamma CK\|\widehat i f\|_Z\quad\forall f\in \fbl[E].
\end{equation}

\begin{proof}[Proof of claim]
Given $f\in \fbl[E]$, choose $x^*\in X^*_+$ with $\norm{x^*}=1$ and $x^*\bigl(\abs{\widehat{S}f}\bigr)=\norm{\widehat{S}f}_X$. Let $N_{x^*}$ denote the null ideal generated by~$x^*$, that is, $N_{x^*}=\bigl\{x\in X : x^*\bigl(\abs{x}\bigr)=0\bigr\}$, and let~$Y$ be the completion of the quotient lattice~$X/N_{x^*}$ with respect to the norm $\norm{x+N_{x^*}}:=x^*\bigl(\abs{x}\bigr)$. Since this is an abstract $L_1$-norm, $Y$ is lattice isometric to~$L_1(\Omega,\Sigma,\mu)$ for some measure space $(\Omega,\Sigma,\mu)$ (see, e.g., \cite[Theorem~1.b.2]{LT2}). The canonical quotient map of $X$ onto $X/N_{x^*}$ induces a lattice homomorphism $Q\colon X\rightarrow L_1(\Omega,\Sigma,\mu)$ with $\norm{Q}=1$. For our purposes, we may without loss of generality assume that $(\Omega,\Sigma,\mu)$ is $\sigma$-finite, passing for instance to the band generated by $Q(\widehat{S} f)$.
\\

Since $Q$ is a lattice homomorphism and $S$ is $(p,\infty)$-convex  with constant $K$, we have
\begin{displaymath}
    \Big\|\bigvee_{k=1}^n\bigabs{QS(x_k)}\Big\|_{L_1(\mu)} \leq  \Big\|\bigvee_{k=1}^n\bigabs{S(x_k)}\Big\|_X \leq K\Bigl(\sum_{k=1}^n\|x_k\|_E^p\Bigr)^{\frac{1}{p}}
\end{displaymath}
for every finite sequence $(x_k)$ in $E$.  Hence, by \cite[Theorem 1.2]{Pisier}, there exists $h\in L_1(\mu)_+$ with $\int_\Omega h\,d\mu\leq 1$, yielding a factorization
 \begin{displaymath}
    \xymatrix{E\ar[dr]_T\ar[rr]^{QS}&&L_1(\mu)\\
    & L_{p,\infty}(h\,d\mu)\ar[ru]_R& } ,
  \end{displaymath}
  with $R$ being a lattice homomorphism implemented by multiplication by $h$. \eqref{eq:q versus p,inf norm} gives
\begin{align*}
    \|Rf\|_{L_1(\mu)}&=\|hf\|_{L_1(\mu)}=\|f\|_{L_1(hd\mu)}\\
    & \leq \frac{p}{p-1}\left(\int_\Omega hd\mu\right)^{1-\frac{1}{p}} \cdot \|f\|_{L_{p,\infty}(hd\mu)} \leq  \frac{p}{p-1}\|f\|_{L_{p,\infty}(hd\mu)} ,
\end{align*}
hence $\|R\|\leq \frac{p}{p-1}$.
\\

Moreover, in the above factorization $h$ can be chosen in such a way that $\|T\|\leq \gamma_0 K$, where $\gamma_0$ depends only on $p$. To see this, we follow the proof of \cite[Theorem 1.2]{Pisier}. In \cite[Theorem 1.1]{Pisier}, let us take $r=1$, and choose our subset of $L_1(\mu)$ to be  $\{QSx : \|x\|_E\leq 1\}$. We claim that statement (iii) of this theorem holds with $C$ being $K$. Indeed,
$$\bigg\|\bigvee_{k=1}^n |\alpha_kQSx_k|\bigg\|_{L_1(\mu)}\leq K\left(\sum_{k=1}^n\|\alpha_kx_k\|_E^p\right)^\frac{1}{p}\leq K\left(\sum_{k=1}^n|\alpha_k|^p\right)^\frac{1}{p}.$$
Thus, tracing through the proof of \cite[Theorem 1.1]{Pisier}, statement (ii) holds with $K''=K\left(1-\frac{1}{p}\right)^{\frac{1}{p}-1}$. This tells us (with a bit of a clash of notation - what one should do is avoid the appeal to Theorem 1.2, only appeal to Theorem 1.1, and use Theorem 1.1 to prove 1.2 with control of the constants) that \cite[Theorem 1.2(iii)]{Pisier} holds, which is just a restatement of \cite[Theorem 1.2(iv)]{Pisier}. In other words, $\|T\|\leq \gamma_0 K$, where  $\gamma_0 = \left(1-\frac{1}{p}\right)^{\frac{1}{p}-1}$.
\\

  By hypothesis, there is a lattice homomorphism $T':Z\rightarrow L_{p,\infty}(h\,d\mu)$ with $T'i=T$ and $\|T'\|\leq C\|T\|$. Let us consider the composition $RT'\widehat i:\fbl[E]\rightarrow L_1(\mu)$. Note this is a lattice homomorphism which for $x\in E$ satisfies
$$
RT'\widehat i \phi_E(x)= RT'i(x)=RT(x)=QS(x).
$$
It follows from the universality of $\fbl[E]$ that $RT'\widehat i=Q\widehat S$. In particular,
\begin{align*}
\|\widehat S f\|_X & = \|Q\widehat S f\|_{L_1}=\|RT' \widehat i f\|_{L_1} \\ & \leq \frac{p}{p-1}C \|T\| \| \widehat i f\|_Z \leq \gamma C K \| \widehat i f\|_Z,    \, \, {\textrm{where}} \, \, \gamma = \frac{p}{p-1} \gamma_0 ,
\end{align*}
as we wanted to show.
\end{proof}

Having proven the claim, for $f\in\fbl[E]$, put $S'(\widehat{i}f):=\widehat{S}f$. By~\eqref{eq:XZ}, $S'$ is well-defined and bounded on $\Range\widehat{i}$; it is easy to see that it is a lattice homomorphism. Since  $\Range\widehat{i}$ is dense, $S'$ extends to a lattice homomorphism on $Z$. We clearly have  $S'i=S$, and $\|S'\| \leq \gamma C K$.
\end{proof}

Generally speaking, $p$-convexity and $p$-concavity are much better understood than upper and lower $p$-estimates. However, using free Banach lattice technology we can find upper $p$-estimate versions of classical theorems on $p$-convexity. Indeed, in \Cref{upper p version} we were able to extend \cite[Theorem 1.d.7]{LT2}; we now show that   \cite[Theorem 3]{RT} has a natural analogue for upper $p$-estimates:
\begin{cor}\label{Factoring}
Suppose $p \in (1,\infty)$, $E$ is a Banach space, $X$ is a Banach lattice and $T : E \to X$ is any operator. The following statements are equivalent:
\begin{enumerate}
    \item $T$ is $(p,\infty)$-convex;
    \item There exists a Banach lattice $Y$ with an upper $p$-estimate, and a factorization $T = S \phi$, where $\phi : E \to Y$ is bounded, and $S : Y \to X$ is a lattice homomorphism.
\end{enumerate}
Moreover, in (2) we can take $\phi$ to be the isometric embedding of $E$ into $\fbl^{\uparrow p}[E]$, $S=\widehat{T}$, and $\|S\| \leq \kappa K^{(p,\infty)}(S)$, with $\kappa$ depending only on $p$. 
\end{cor}

\begin{proof}
 For $(2)\Rightarrow (1)$, suppose $T$ factors through $Y$ as above.
 As function calculus intertwines with lattice homomorphisms, for any $x_1, \ldots, x_n \in E$ we have
\begin{align*}
    \bigg\|\bigvee_{k=1}^n |Tx_k|\bigg\|_X & \leq \|S\| \bigg\|\bigvee_{k=1}^n |\phi(x_k)|\bigg\|_Y 
    \\ & \leq M \|S\| \left(\sum_{k=1}^n \|\phi x_k\|_Y^ p\right)^{1/p}\leq M \|S\| \|\phi\| \left(\sum_{k=1}^n \|x_k\|_E^p\right)^{1/p}
\end{align*}
($M$ is the upper $p$-estimate constant of $Y$), showing that $T$ is $(p,\infty)$-convex.
\\

For $(1)\Rightarrow (2)$, by \Cref{p:isomorphic characterization of fbl upper p}, it suffices to extend an operator
$T:E\to (L_{p,\infty}(\mu),\|\cdot\|_{p,\infty})$ to a lattice homomorphism from $\fbl^{\uparrow p}[E]$ to $L_{p,\infty}(\mu)$, with norm of the extension controlled.
Let $S=IT$, where $I$ is the identity $(L_{p,\infty}(\mu),\|\cdot\|_{p,\infty})\to  (L_{p,\infty}(\mu),\vertiii{\cdot}_{p,\infty})$. Then $\|S\|\leq C_p\|T\|$.
Extend $S$ to $\widehat{S}: \fbl^{\uparrow p} [E] \to (L_{p,\infty}(\mu),\vertiii{\cdot}_{p,\infty})$ with $\|\widehat S\|=\|S\|$. Now the map $T':=I^{-1}\widehat{S}:Z\to (L_{p,\infty}(\mu),\|\cdot\|_{p,\infty})$ is a lattice homomorphism extending $T$ and satisfying $\|T'\|\leq C_p\|T\|$.
\end{proof}

To characterize the spaces $E$ for which $\fbl[E]$ and $\fbp[E]$ are lattice isomorphic, we need two definitions. Suppose $E$ is a Banach space, $C \geq 0$, and $Z$, $X$ are Banach lattices. We say that $T:E \to Z$ \emph{$C$-strongly factors through $X$} if we can write $T = US$, where $S : E \to X$ is a contraction, and $U : X \to Z$ is a lattice homomorphism, with $\|U\| \leq C$. If ${\mathcal{X}}$ is a class of Banach lattices, we say that $T : E \to Z$ \emph{$C$-strongly factors through ${\mathcal{X}}$} if it $C$-strongly factors through some $X \in {\mathcal{X}}$. If, in the preceding setting, $X$ and $Z$ are both spaces of functions on the same space, we say that $T$ \emph{$C$-multiplicatively factors through $X$} if $U$ as above is implemented by a multiplication operator. We say that $T$ factors \emph{strongly} (or \emph{multiplicatively}) if such factorization exists for some $C$. Obviously, \Cref{Factoring} can be restated in this language.

\begin{prop}\label{P_p} 
Let $E$ be a Banach space, $p>q\geq 1$, and $C\geq 1$. The following are equivalent:
\begin{enumerate}
\item $\fbp[E]$ is lattice $C$-isomorphic to $\fbl^{(q)}[E]$;
\item $\fbp[E]$ is canonically lattice $C$-isomorphic to $\fbl^{(q)}[E]$, that is, the map taking $\delta_x$ ($x \in E$) to itself generates a lattice $C$-isomorphism between $\fbp[E]$ and $\fbl^{(q)}[E]$;
\item Every contraction $T:E\to L_q(\mu)$ $C$-strongly factors through a $p$-convex Banach lattice with $p$-convexity constant $1$;
\item Every contraction $T:E\to L_q(\mu)$ $C$-multiplicatively factors through $L_p(\mu)$;
\item Every contraction $T:E\to L_q(\mu)$ is $p$-convex with constant $C$, i.e., for  all finite sequences $(x_k)$ in $E$ we have 
$$\bigg\|\left(\sum_{k=1}^n|T(x_k)|^p\right)^\frac{1}{p}\bigg\|_{L_q(\mu)}\leq C\left(\sum_{k=1}^n\|x_k\|^p\right)^\frac{1}{p}.$$
\end{enumerate}
\end{prop}

Corollaries \ref{p:coincidence type} and \ref{p:coincidence type2} below provide examples of Banach spaces $E$ which possess the equivalent properties described here.

\begin{proof}
(2)$\Rightarrow$(1) is straightforward.
\\

(1)$\Rightarrow$(3) Suppose that there is a lattice isomorphism $V\colon\fbl^{(q)}[E]\to\fbp[E]$ such that $\norm{V}=1$ and $\norm{V^{-1}}\leq C$. Let $T\colon E\to L_q(\mu)$ be a contraction. Consider $\widehat{T}\colon\fbl^{(q)}[E]\to L_q(\mu)$. Then $T=(\widehat{T}V^{-1})(V\phi_E)$ is a required factorization.
\\

 (3)$\Rightarrow$(2) We will use \Cref{p:isomorphic characterization of fbp} with $p$ replaced with $q$, $Z=\fbp[E]$, and $\iota=\phi_E\colon E\to \fbp[E]$. Let $T\colon E\to L_q(\mu)$ be a contraction. By assumption, we can factor $T$ through a $p$-convex Banach lattice $X$ with constant 1, $T\colon E\xrightarrow{S}X\xrightarrow{U}L_q(\mu)$ such that $\norm{S}\le 1$, $\norm{U}\le C$, and $U$ is a lattice homomorphism. Then $T':=U\widehat{S}\colon\fbp[E]\to L_q(\mu)$ extends $T$, is a lattice homomorphism, and $\norm{T'}\le C$. By \Cref{p:isomorphic characterization of fbp}, $\phi_E$ extends to a lattice $C$-isomorphism from $\fbl^{(q)}[E]$ to $\fbp[E]$.
 \\
 
To prove (3)$\Rightarrow$(5), we use the strong factorization $T = US$, with $\|U\| \leq C$ and $\|S\| \leq 1$. Then
$$\bigg\|\left(\sum_{k=1}^n|T(x_k)|^p\right)^\frac{1}{p}\bigg\|_{L_q(\mu)}=\bigg\|\left(\sum_{k=1}^n|US(x_k)|^p\right)^\frac{1}{p}\bigg\|_{L_q(\mu)}\leq$$
$$ C\bigg\|\left(\sum_{k=1}^n|S(x_k)|^p\right)^\frac{1}{p}\bigg\|_X \leq C\left(\sum_{k=1}^n\|Sx_k\|^p\right)^\frac{1}{p}\leq C\left(\sum_{k=1}^n\|x_k\|^p\right)^\frac{1}{p}.$$
The first (in)equality is the factorization, the second is since $U$ is a lattice homomorphism of norm at most $C$, the third by $p$-convexity of $X$, and the last since $S$ is a contraction. 
\\

Clearly (4)$\Rightarrow$(3). The equivalence between (4) and (5) is essentially \cite[p.~264]{Wojt}. 
\end{proof}

We can also state an upper $p$-estimate variant of \Cref{P_p}. 

\begin{prop}\label{Abstract upper p}
Let $E$ be a Banach space and $p>q\geq 1$. The following are equivalent:
\begin{enumerate}
\item $\fbl^{\uparrow p}[E]$ is lattice isomorphic to $\fbl^{(q)}[E]$;
\item $\fbl^{\uparrow p}[E]$ is canonically lattice isomorphic to $\fbl^{(q)}[E]$;
\item There exists $C\geq 1$ such that for every Banach lattice $Y$ with $M^{(q)}(Y) = 1$, every contraction $T:E\to Y$ $C$-strongly factors through a Banach lattice $X$ which has an upper $p$-estimate with constant $1$;
\item There exists $C\geq 1$ such that every contraction $T:E\to L_q(\mu)$ $C$-strongly factors through a Banach lattice $X$ which has an upper $p$-estimate with constant $1$;
\item There exists $C\geq 1$ such that every contraction $T:E\to L_q(\mu)$ $C$-multiplicatively factors through $L_{p,\infty}(\mu)$;
\item There exists $C\geq 1$ such that every contraction $T:E\to L_q(\mu)$  is $(p,\infty)$-convex with constant $C$.
\end{enumerate}
\end{prop}

\begin{proof}
The implications $(2)\Rightarrow (1)\Rightarrow (3)\Rightarrow (4)\Rightarrow (2)$ are similar to \Cref{P_p}. (4)$\Rightarrow$(6) follows from the same factorization argument used in the proofs of  \Cref{Factoring} and the implication (3)$\Rightarrow$(5) in \Cref{P_p}. (5)$\Leftrightarrow$(6) is \cite[Theorem 1.2]{Pisier}, and (5)$\Rightarrow$(4) follows  because $L_{p,\infty}(\mu)$ (with an appropriate norm) satisfies an upper $p$-estimate with constant $1$.
 \end{proof}

In statement (4) of \Cref{P_p} we require that \emph{every} contraction $T:E\to L_q(\mu)$ factor multiplicatively through $L_p(\mu)$; similarly, in statement (5) we require that \emph{every} operator verify a certain inequality. This makes statements (4) and (5) properties of the Banach space $E$. However, as was evident from the proof, statements (4) and (5) hold on an operator-by-operator basis. More precisely, a contraction $T:E\to L_q(\mu)$ factors multiplicatively through $L_p(\mu)$ if and only if it verifies the inequality in statement (5) of \Cref{P_p}. Analogous reasoning (using \cite{Pisier}) shows that  similar results hold true when $L_p(\mu)$ is replaced by $L_{p,\infty}(\mu)$. As we will now see, the fact that we quantify over all operators gives some interesting relations between the roles of $L_p(\mu)$, $L_q(\mu)$ and $L_{p,\infty}(\mu)$ in the above statements. More precisely, we have the following extrapolation theorem:

\begin{thm}\label{Extrapolation}
Suppose $\fbp[E]$ is $q$-convex for some $1\leq p<q$. Then $\fbl^{(r)}[E]$ is $q$-convex for all $1\leq r\leq \infty.$
\end{thm}
\begin{proof}
By \Cref{P_p}, there exists a constant $C$ so that any contraction $u : E \to L_p(\mu)$ $C$-strongly factors through $L_q(\mu)$. As suggested on \cite[p.~42]{M}, consider the dual pair $(E^*,E)$, where $E^*$ is equipped with its weak$^*$ topology $\sigma(E^*,E)$; this turns $E^*$ into a locally convex Hausdorff space, or ``elcs'' (espace localement convexe s{\'e}par{\'e}) in the French language of \cite{M}. The dual space is then $E$, with its norm topology.
\\

Applying \cite[Th{\'e}or{\`e}me 23, (c) $\Rightarrow$ (a)]{M} to this dual pair (or, alternatively, using Exercise 2 on p.~286 of \cite{Wojt}, and its  solution on p.~336), we conclude that the inequality $\pi_p(T)\leq C\pi_q(T)$ holds for every $T:E^*\to \ell_q$ (we also used the fact that the $(p,{\mathrm{weak}})$ summing norms of $n$-tuples in $E^*$ can be computed using either $B_E$ or $B_{E^{**}}$, cf.~\eqref{eq:ARTbidual}). 
By \cite[3.17 Extrapolation Theorem]{DJT}, for any Banach space $F$ we have $\Pi_q(E^*,F)=\Pi_1(E^*,F)$. In particular, any $q$-nuclear operator from $E^*$ into $\ell_q$ is $1$-summing. By \cite[Th{\'e}or{\`e}me 23, (b) $\Rightarrow$ (c)]{M} (or  invoking \cite[p.~270]{Wojt}), any $u : E \to L_1(\mu)$ strongly factors through $L_q(\mu)$. By \Cref{P_p}, we conclude that $\fbl[E]$ is lattice isomorphic to $\fbl^{(q)}[E]$. Thus, $\fbl[E]$ is $q$-convex, and therefore, for any $r\in [1,q]$, $\fbl^{(r)}[E]$ is lattice isomorphic to $\fbl^{(q)}[E]$, hence $q$-convex. For $r > q$, the $q$-convexity is automatic.
\end{proof}

 We now use the preceding result to establish a few facts regarding factorizable and $p$-summing operators.
 
 \begin{cor}\label{no factor}
 Suppose $1\leq q$, $p>\max\{2,q\}$, and $E$ is an infinite dimensional Banach space. There exist $T, S \in B(E, L_q(\mu))$ so that $T$ (respectively, $S$) does not strongly factor through $L_p(\mu)$ (respectively, $L_{p,\infty}(\mu)$).
 \end{cor}
 
 \begin{proof}
 (i) If any operator in $B(E, L_q(\mu))$ strongly factors through $L_p(\mu)$, then, by \Cref{P_p}, $\fbl^{(q)}[E]$ is $p$-convex. This, however, contradicts \Cref{p:cant be worse than 2-convex}.
 \\
 \\
 (ii) If any operator in $B(E, L_q(\mu))$ strongly factors through $L_{p,\infty}(\mu)$, then, by \Cref{Abstract upper p}, $\fbl^{(q)}[E]$ has an upper $p$-estimate. Consequently, $\fbl^{(q)}[E]$ is $s$-convex for any $s < p$ \cite[Section 1.f]{LT2}, which contradicts \Cref{p:cant be worse than 2-convex} (one can take $s \in (\max\{q,2\},p)$).
 \end{proof}
 
 The following result indicates the limits of extrapolation of summing maps.
 
 \begin{cor}
  Suppose $1 \leq r < p \leq \infty$, and $\Pi_p(E,\ell_p) = \Pi_r(E,\ell_p)$, for some infinite dimensional $E$. Then $p \leq 2$.
 \end{cor}
 
 The restriction $p \leq 2$ is sharp. For instance, $\Pi_r(H,\ell_2) = \Pi_2(H,\ell_2)$, for any Hilbert space $H$ and $r \in [1,\infty)$.
 
 \begin{proof}\label{no extra}
 Suppose, for the sake of contradiction, that $E$ is infinite dimensional, $p \in (2,\infty]$, and $\Pi_p(E,\ell_p) = \Pi_r(E,\ell_p)$ for some $r \in [1,p)$.
 \\
 
 (i) $p = \infty$. If $\Pi_r(E,\ell_\infty) = B(E,\ell_\infty)$, then $\Pi_r(E,\ell_\infty(I)) = B(E,\ell_\infty(I))$ for any index $I$. Now find $I$ so large that $E$ embeds into $\ell_\infty(I)$. Then $id_E$ is $r$-summing, which is impossible.
 \\
 
  (ii) $p < \infty$. If $\Pi_p(E,\ell_p) = \Pi_r(E,\ell_p)$, then, by Extrapolation Theorem \cite[3.17]{DJT}, $\Pi_p(E,\ell_p) = \Pi_1(E,\ell_p)$.
  Imitating the reasoning from the proof of \Cref{Extrapolation}, we apply \cite[Th{\'e}or{\`e}me 23]{M} to the dual pair $(E,E^*)$ ($E$ is equipped with its norm topology). We then conclude that any operator from $E^*$ to $L_1(\mu)$ strongly factors through $L_p(\mu)$, which implies $p\le 2$ by \Cref{no factor}.
 \end{proof}
 
 Returning to free Banach lattices, we prove:
 \begin{cor}\label{convexity vs summing}
 If $E$ is a Banach space, and $1 \leq r < p \leq \infty$, then the following statements are equivalent:
 \begin{enumerate}
     \item $\fbl^{(r)}[E]$ has an upper $p$-estimate.
     \item $\fbl[E]$ has an upper $p$-estimate.
     \item $id_{E^*}$ is $(q,1)$-summing, with $1/p + 1/q = 1$.
 \end{enumerate}
 \end{cor}
 
 \begin{proof}
 $(2) \Leftrightarrow (3)$ has been established in \Cref{p:upper est}. To handle $(1) \Leftrightarrow (2)$, pick $s \in (r,p)$. From \cite[Section 1.f]{LT2}, we know that an upper $p$-estimate implies $s$-convexity. If one of the lattices involved -- either $\fbl^{(r)}[E]$ or $\fbl[E]$ -- is $s$-convex, then the two coincide, by \Cref{Extrapolation}.
 \end{proof}

  Note that \Cref{p:upper est} identifies the upper $p$-estimate constant of $\fbl[E]$ as the $(q,1)$-summing norm of $id_{E^*}$; we make no claim that the upper $p$-estimate constant of $\fbl^{(r)}[E]$ agrees with that of $\fbl[E]$.
\\

 Combining \Cref{convexity vs summing} with \Cref{P_p}, we obtain:
\begin{cor}
Suppose $id_{E^*}$ is $(q,1)$-summing, $1/p+1/q=1$, and $1 \leq r < s < p$. Then any operator from $E$ to $L_r(\mu)$ multiplicatively factors through $L_s(\mu)$.
\end{cor}

We now examine conditions guaranteeing, or precluding, lattice isomorphism between $\fbp[E]$ and $\fbl^{(q)}[E]$.
\begin{cor}\label{p:coincidence type}
 Suppose a Banach space $E$ has type $s \in (1,2)$. Then, for $1 \leq p < q < s$, $\fbp[E]$ and $\fbl^{(q)}[E]$ are canonically lattice isomorphic.
\end{cor}

\begin{proof}
By \Cref{P_p}, we need to show that there exists a constant $C$ so that any contraction $T : E \to L_p(\mu)$ has a lattice homomorphic extension $T' : \fbl^{(q)}[E] \to L_p(\mu)$, with $\|T'\| \leq C \|T\|$. Emulating the proof of \Cref{p:isomorphic characterization of fbp 2}, we see that it actually suffices to establish the existence of {\it some} extension $T'$; the norm will be controlled automatically.
\\

To obtain the desired extension, we use Maurey-Nikishin Extension Theorem \cite[III.H.12]{Wojt}: $T$ can be factored through $L_q(\mu)$ as $T = u S$, where $u : L_q(\mu) \to L_p(\mu)$ is a lattice homomorphism. Then $S$ has a lattice homomorphic extension $\widehat{S} : \fbl^{(q)}[E] \to L_q(\mu)$. Then $T' = u \widehat{S}$ is the extension we want.
\end{proof}
\begin{rem}
 An alternative argument for \Cref{p:coincidence type} is to note that if $E$ has type $s$ then the dual has cotype $s'$ (for $\frac{1}{s}+\frac{1}{s'}=1$), hence $id_{E^*}$ is $(s',1)$-summing, which characterizes upper $s$-estimates of $\fbl[E]$.
\end{rem}

Using \cite[III.H.16]{Wojt} instead of \cite[III.H.12]{Wojt}, we obtain:

\begin{cor}\label{p:coincidence type2}
 For $2 \leq r \leq \infty$ and $1 \leq p \leq 2$, $\fbp[L_r(\mu)]$ and $\fbl^{(2)}[L_r(\mu)]$ are canonically lattice isomorphic.
\end{cor}

\begin{rem}
 Suppose $r\geq 2$ and $1\leq p\leq 2$. \Cref{p:coincidence type2} implies that the moduli of the $\ell_r$ basis in $\fbp[\ell_r]$ and in $\fbl[\ell_r]$ are equivalent; by \cite{ATV}, both are equivalent to the $\ell_s$ basis, with $1/s = 1/r + 1/2$.
 We do not know what the span of these moduli is for $r,p\in (2,\infty)$. 
 \end{rem}

On the other hand, it follows immediately from \Cref{p:cant be worse than 2-convex} that:

\begin{cor}\label{c:we are all different}
 Suppose $E, F$ are infinite dimensional Banach spaces, $p \in [1,\infty]$, $q \in (2, \infty]$, and $p \neq q$. Then $\fbp[E]$ is not lattice isomorphic to $\fbl^{(q)}[F]$.
\end{cor}
\begin{rem}
By \Cref{c:criteria for l1}, if $E^*$ has finite cotype, then $\fbl[E]$ is $p$-convex for some $p>1$. Using the fact that the $r$-convexification of a $s$-convex space is $sr$-convex, one can easily show that for such $E$,  $\fbp[E]$ is not lattice isomorphic to the $p$-convexification of $\fbl[E]$.
\end{rem}

We finish this section with a simple observation precluding $\fbl^{(q)}[F]$ from being isomorphic (in the Banach space sense) to a complemented subspace of $\fbp[E]$. Indeed, by combining \Cref{Char of upper p for Banach lattices} with \Cref{upper p version}, we improve \cite[Theorem 9]{ATV}:

\begin{cor}\label{Not banach iso}
Let $1\leq p<\min\{2,q\}\leq \infty$. Then $\fbl[\ell_p]$ is not linearly isomorphic to a complemented subspace of $\fbl[\ell_q]$.
\end{cor}

\begin{rem}\label{Not banach iso -refined}
A related result follows from \cite[Theorem 1.d.7 and the remark after]{LT2}: if $p\in (1,2]$ and $\fbl^{(q)}[F]$ is not $p$-convex, then it does not linearly embed complementably into $\fbp[E]$ for any $E$. Here, complementation is key as $\fbl[E]$ contains isomorphic copies of every separable Banach space as long as $\dim E\geq 2$; see \Cref{Sub structure}, which also discusses the non-separable setting.
\end{rem}

\begin{rem}
In this section, we focused on  {\it strong} factorizations via lattices which are $p$-convex, or have upper $p$-estimates. Related factorizations (which were not assumed to involve lattice homomorphisms) are considered in \cite{Byrd} (positive factorizations via lattices with upper or lower estimates) and \cite{Reisner} (factorizations using operators with given convexity and concavity). 
\end{rem}

\section{Isomorphism of free Banach lattices}\label{s:isomorphism}

In this section we give a partial resolution to the question of whether $\fbp[E]$ and $\fbp[F]$ can be lattice isomorphic (some negative results can be extracted from the $r$-convexity and $r$-upper estimate criterion presented in  the previous section).

\subsection{Representation of lattice homomorphisms}\label{ss:lattice homs}

In this subsection, we represent lattice homomorphisms on free lattices as composition operators, and gather some consequences of this representation. The following proposition is similar to some results of \cite{Laust-Tra}.

\begin{prop}\label{l:latticehomocomp}\label{Lattice homo}
Given Banach spaces $E$, $F$, $p\in[1,\infty]$ and a lattice homomorphism $T:\fbp[F] \rightarrow \fbp[E]$,
there exists a mapping $\Phi_T:E^*\rightarrow F^*$ so that $Tf=f\circ \Phi_T$ for every $f\in \fbp[F]$. Moreover, $\Phi_T$ satisfies the following properties:
\begin{enumerate}
\item For any $x^* \in E^*$ and $y \in F$, $\Phi_T x^* (y)=(T\delta_y)(x^*)$,
\item $\Phi_T$ is positively homogeneous,
\item $\Phi_T$ is weak$^*$ to weak$^*$ continuous on bounded sets,
\item For $y^* \in E^*$, we have $\|\Phi_T y^*\| \leq \|T\| \|y^*\|$. If $p<\infty$, then for every $(y_k^*)_{k=1}^m\subseteq E^*$ we have
$$
\sup_{x\in B_F}\Big(\sum_{k=1}^m  |[\Phi_T y_k^*](x)|^p\Big)^{\frac1p}\leq \|T\| \sup_{y\in B_E}\Big(\sum_{k=1}^m  |y_k^*(y)|^p\Big)^{\frac1p}\ .
$$

\end{enumerate}
\end{prop}

\begin{proof}
First recall that the atoms of $\fbp[E]^*$ are precisely the linear functionals which act on $\fbp[E]$ as lattice homomorphisms \cite[p.~111]{AB}, and these correspond to point evaluations (\cite[Corollary 2.7]{ART} establishes this for $p=1$, but the proof for other values of $p$ works in the same way).
For $x^* \in E^*$, denote the corresponding evaluation functional on $H[E]$ (and therefore, on $\fbp[E]$) by $\widehat{x^*}$.
One can check that $\|\widehat{x^*}\|_{\fbp[E]^*} = \|x^*\|_{E^*}$, for every $p$, and, as $H[E]$ consists of positively homogeneous functions, we have $\widehat{\alpha x^*} = \alpha \widehat{x^*}$ for $\alpha \geq 0$.
\\

If $T : \fbp[F] \to \fbp[E]$ is a lattice homomorphism, then $T^*$ is interval preserving, and, in particular, maps atoms to atoms. Using the description of atoms given in the previous paragraph, we conclude that $T^*$ induces a positively homogeneous map $\Phi_T : E^* \to F^*$, via $\Phi_T x^* = T^* \widehat{x^*}\circ \phi_F$ (that is, $\widehat{\Phi_T x^*} = T^* \widehat{x^*}$).
\\

By construction, for every $f\in \fbp[F]$ we have $Tf=f\circ \Phi_T$. Indeed, for $x^*\in E^*$ let $y^* = \Phi_T x^*$. Then
$$(f\circ \Phi_T)(x^*)=f(y^*) = \widehat{y^*}(f) = [T^* \widehat{x^*}](f) =Tf(x^*).$$

Plugging in $f = \delta_y$, we obtain (1). This, in turn, implies (2): for $\lambda\geq0$, $x^*\in E^*$ and $y\in F$,
$$
\Phi_T (\lambda x^*) (y)=(T\delta_{y})(\lambda x^*)=\lambda (T\delta_{y})(x^*)=\lambda \Phi_T x^* (y).
$$

To establish (3), note that, if $y^*_\alpha\overset{w^*}{\rightarrow} y^*$ is a bounded net in $E^*$, then for every $x\in F$ we have
$$
[\Phi_T y^*_\alpha] (x) = [T\delta_{x}](y^*_\alpha)\longrightarrow [T\delta_{x}](y^*)=[\Phi_T y^*] (x),
$$
as $T\delta_{x}\in \fbp[E]$ is weak$^*$ continuous on bounded sets.
\\

To handle (4), let $(y^*_k)_{k=1}^{m}\subseteq E^*$. We have
\begin{eqnarray*}
\sup_{x\in B_F}\sum_{k=1}^{m}|\Phi_T y^*_k (x)|^p&=&\sup_{x\in B_F}\sum_{k=1}^{m}|T\delta_x (y^*_k)|^p\\
&\leq &\sup_{x\in B_F}\|T\delta_x\|^p_{\fbp[E]}\sup_{y\in B_E}\sum_{k=1}^{m}|y^*_k (y)|^p\\
&\leq& \|T\|^p\sup_{y\in B_E}\sum_{k=1}^{m}|y^*_k (y)|^p.
\qedhere
\end{eqnarray*}
\end{proof}

In certain cases, more can be said about the map $\Phi_T$. The proof of the following proposition is straightforward.

\begin{prop}\label{Phi T for specific T}
In the notation of \Cref{Lattice homo}, we have:
\begin{enumerate}
\item
Suppose $T$ is surjective, so that, by Open Mapping Theorem,   there exists $c>0$ so that for every $g \in \fbp[E]$ there exists $f \in \fbp[F]$ with $Tf = g$ and $\|f\| \leq c^{-1} \|g\|$. Then $c \|x^*\| \leq \|\Phi_T x^*\|$ for every $x^* \in E^*$.
\item
If $T$ has dense range, then $\Phi_T$ is injective.
\item
If $T$ is a lattice isomorphism, then $\Phi_T$ is bijective, and $\Phi_{T^{-1}} = \Phi_T^{-1}$.
\item
If $T$ is a lattice isometry, then $\|\Phi_T x^*\| = \|x^*\|$ for any $x^* \in E^*$.
\end{enumerate}
\end{prop}

\begin{rem}\label{r:lattice isometries between fbp}
Suppose $T : \fbp[F] \to \fbp[E]$ is a lattice isometry, and $F^*$ has the \emph{weak$^*$} (or \emph{dual}) \emph{Kadec-Klee Property}, investigated in \cite{DilKu} and \cite{HaTo}. That is, if $(x_n^*)$ is a sequence in $F^*$ weak$^*$-converging to $x^* \in F^*$, and such that $\|x_n^*\| \to \|x^*\|$, then $\|x_n^* - x^*\| \to 0$. Then we can further deduce that $\Phi_T$ is norm to norm continuous.
\end{rem}
\begin{rem}
\Cref{l:overlineT composition} shows that for $T:F\to E$, the induced map $\overline{T}:\fbp[F]\to\fbp[E]$ satisfies $\Phi_{\overline{T}}=T^*$.
\end{rem}

 \Cref{p:structure of fbl infty} immediately implies that the converse of \Cref{l:latticehomocomp} is  valid for $p=\infty$.

\begin{cor}\label{c:maps between fbl infty}
 Suppose $E$ and $F$ are Banach spaces, and $\Phi : E^* \to F^*$ is a positively homogeneous map, weak$^*$ to weak$^*$ continuous on bounded sets, so that $C:=\displaystyle \sup_{y^* \in E^* \backslash \{0\}} \frac{\|\Phi y^*\|}{\|y^*\|} < \infty$. Then there exists a lattice homomorphism $T:\fbl^{(\infty)}[F] \to \fbl^{(\infty)}[E]$ so that $\|T\|=C$, and $\Phi = \Phi_T$.
\end{cor}

\begin{rem}\label{converse fails}
In contrast, the converse of \Cref{l:latticehomocomp} fails for $p=1$. Below we present a map $\Phi$, satisfying \Cref{Lattice homo}(2,3,4) for $p=1$, but not implementing a lattice homomorphism of $\fbl[\ell_1]$ to itself. Specifically, define 
$$
\Phi \big( (a_i)_{i=1}^\infty \big) = \Big( |a_1| \wedge \big( \vee_{i \geq 2} \frac{|a_i|}i \big), 0, 0, \ldots \Big).
$$
Clearly $\Phi$ is positively homogeneous and weak$^*$ continuous (relative to the canonical identification $\ell_\infty = \ell_1^*$) on bounded sets, so (2) and (3) of \Cref{Lattice homo} hold. To establish (4), consider a finite collection $(x_k) \subseteq \ell_\infty$, with $\max_{\pm} \|\sum_k \pm x_k\| \leq 1$. Write $x_k = (a_{ki})_{i=1}^\infty$. Then $\vee_i \sum_k |a_{ki}| \leq 1$. Consequently, 
$$\max_{\pm} \|\sum_k \pm \Phi x_k\| \leq \sum_k |a_{k1}| \leq 1 . $$

Let $e = (1, 0, 0, \ldots) \in \ell_1$. Then $f = |\delta_e| : (a_i) \mapsto |a_1|$ belongs to $\fbl[\ell_1]$. Now consider $g : \ell_\infty \to \Real : x^* \mapsto f(\Phi x^*)$ -- that is, $$g \big( (a_i) \big) = |a_1| \wedge \big( \vee_{i \geq 2} \frac{|a_i|}i \big).$$ By \cite[Example 2.11]{ART}, $g \notin \fbl[\ell_1]$. This shows that the composition operator defined by $\Phi$ does not map $\fbl[\ell_1]$ to itself, as claimed.
\end{rem}

The  following statement is reminiscent of the notion of ``dependence on finitely many coordinates'' in \cite{Norm-attaining}.

\begin{lem}\label{l:almost fin many}
Suppose $1 \leq p \leq \infty$, and $T : \fbp[F] \to \fbp[E]$ is a lattice homomorphism. Then for any $y \in F$ and $\varepsilon > 0$ there exist $N = N[y] \in \Nat$, $(x_i[y])_{i=1}^{N[y]} \subseteq E$, and a ${\bf F}[y] : \Real^N \to \Real$, represented by finitely many linear and lattice operations, so that
 $$\big| [\Phi_T x^*](y) - {\bf F}[y]( (x^*(x_i[y]))_{i=1}^{N[y]} ) \big| \leq \varepsilon \|x^*\| {\textrm{  for  any  }} x^* \in E^* .$$
\end{lem}

\begin{proof}
 The function $T \delta_y : E^* \to \Real : x^* \mapsto [\Phi_T x^*](y)$ belongs to $\fbp[E]$, hence it is the limit (in the $\fbp[E]$ norm, and, consequently, in the $\sup$ norm on $B_{E^*}$) of elements of $\FVL[E]$.
 Now recall that elements of $\FVL[E]$ can be written as $f(\delta_{x_1}, \ldots, \delta_{x_N})$, where $f$ is a composition of finitely many linear and lattice operations.
\end{proof}

For future use (addressing the same setting), we state the following:

\begin{cor}\label{c:almost fin many for subspace}
Suppose $1 \leq p \leq \infty$, and $T : \fbp[F] \to \fbp[E]$ is a lattice homomorphism. Let $G$ be a finite dimensional subspace of $F$, and $\varepsilon > 0$. Then there exist $N \in \Nat$, and $x_1, \ldots, x_N \in E$, so that if $x^* \in E^*$, $\|x^*\| \leq 1$, and $x^*(x_i) = 0$ for $1 \leq i \leq N$, then $| [\Phi_T x^*](y) | \leq \varepsilon \|y\|$ for any $y \in G$. 
\end{cor}

\begin{proof}
By scaling, assume $\|T\| \leq 1$. Let $(y_j)_{j=1}^M$ be an $\varepsilon/2$-net in the unit ball of $G$. By \Cref{l:almost fin many}, there exist $x_1, \ldots, x_N \in E$, so that if $x^* \in E^*$, $\|x^*\| \leq 1$, and $x^*(x_i) = 0$ for $1 \leq i \leq N$, then $| [\Phi_T x^*](y_j) | \leq \varepsilon/2$ for $1 \leq j \leq M$.
For an arbitrary $y$ in the unit ball of $G$, find $j$ so that $\|y - y_j\| < \varepsilon/2$. Then 
\begin{align*}
    \big| [\Phi_T x^*](y) \big| & = \big| [T \delta_y](x^*) \big| \leq \big| [T \delta_{y_j}](x^*) \big| + \|y - y_j\|
    \\ & = \big| [\Phi_T x^*](y_j) \big| + \|y - y_j\| \leq \varepsilon . \qedhere
\end{align*}
\end{proof}

\Cref{Lattice homo} also allows us to study lattice transitivity of $\fbp$ in the following sense. We say that a Banach lattice $X$ is \emph{lattice almost transitive} if, for any norm one $x, y \in X_+$, and $\varepsilon > 0$, there exists a surjective lattice isometry $T : X \to X$ so that $\|Tx - y\| < \varepsilon$ (note that $T^{-1}$ is a lattice isometry as well). The spaces $L_p(0,1)$ ($1 \leq p < \infty$) are known to be lattice almost transitive (see e.g.~the proof of \cite[Theorem 12.4.3]{FlemingJamison2}, or \cite[Proposition 3.5]{FLAMT}). Another example is the ``Gurarij AM-space'', recently constructed in \cite{FLAMT}. Despite the fact that $\fbp$ lattices possess a large number of lattice homomorphisms, we will now show that such lattices  fail to be lattice almost transitive.

\begin{prop}\label{p:non transitive_fbp}
 For any non-trivial Banach space $E$, and any $p \in [1,\infty]$, the space $\fbp[E]$ is not lattice almost transitive.
\end{prop}

\begin{proof}
Fix a norm one $e \in E$, and let $f = \big[ \delta_e \big]_+$, $g = \big| \delta_e \big|$. Note that $\|f\|_\infty \leq \|f\| \leq \|e\|$, hence $\|f\| = 1$. Similarly, $\|g\| = 1$. We shall show that $\|Tf - g\| \geq 1/3$ whenever $T$ is a surjective lattice isometry on $\fbp[E]$.
\\

Suppose, for the sake of contradiction, that $\gamma := \|Tf - g\| < 1/3$. By the preceding discussion, $T$ is implemented by a positively homogeneous map $\Phi = \Phi_T : B_{E^*} \to B_{E^*}$, weak$^*$ continuous on bounded sets, which preserves norms; $\Phi^{-1}$ has the same properties, since it implements $T^{-1}$. Then, for any $x^* \in B_{E^*}$, we have
\begin{equation}
\big| |x^*(e)| - [\Phi x^*(e)]_+ \big| \leq \gamma .
\label{eq:small difference}
\end{equation}

Let now 
\begin{align*}
    & U_+ = \{x^* \in B_{E^*} : x^*(e) \geq 1/3\},\, U_- = \{x^* \in B_{E^*} : x^*(e) \leq -1/3\} ,  \\ &  U = U_+ \cup U_-, \,V = \{x^* \in B_{E^*} : x^*(e) \geq 2/3\}.
\end{align*}
If $\Phi x^* \in V$, then, by \eqref{eq:small difference}, $|x^*(e)| \geq 2/3 - \gamma > 1/3$, hence $x^* \in U$. In other words, $V \subseteq \Phi U = \Phi U_+ \cup \Phi U_-$.
\\

The sets $U_+$ and $U_-$ are closed (in the relative weak$^*$ topology of $B_{E^*}$), hence the same is true of their images. Since $V$ is a convex set, in particular it is path connected, hence there exists $\eta \in \{-1,+1\}$ so that $\Phi U_\eta \cap V = \emptyset$. Now take a norm one $x^*$ so that $x^*(e) = \eta$. Then $|x^*(e)| = 1$, while $\Phi x^*(e) < 2/3 < 1 - \gamma$, contradicting \eqref{eq:small difference}.
\end{proof}

\subsection{For $1 \leq p < \infty$, $\fbp$ lattices are often distinct}\label{ss:fbp distinct}
In this subsection, we establish that, for $p < \infty$, in certain cases $\fbp[E]$ and $\fbp[F]$ cannot be lattice isomorphic. 
As a tool, we need the ``weak $p$'' norms (see e.g.~\cite{DJT}). Recall that, for $(z_i)_{i=1}^N \subseteq Z$,
$$
\|(z_i)\|_{p, {\mathrm{weak}}} = \sup_{z^* \in B_{Z^*}} \big( \sum_i |z^*(z_i)|^p \big)^{1/p} = 
\sup \big\{ \| \sum_i \alpha_i z_i \| : \sum_i |\alpha_i|^q \leq 1 \big\} ,
$$
where $\frac1p + \frac1q = 1$.
For $(z_i^*)_{i=1}^N \subseteq Z^*$, moreover, \eqref{eq:ARTbidual} yields:
$$
\|(z_i^*)\|_{p, {\mathrm{weak}}} = \sup_{z^{**} \in B_{Z^{**}}} \big( \sum_i |z^{**}(z_i^*)|^p \big)^{1/p} =  \sup_{z \in B_Z} \big( \sum_i |z_i^*(z)|^p \big)^{1/p} .
$$
By duality, $\|(z_i^*)\|_{p, {\mathrm{weak}}}$ coincides with the norm of the operator $\ell_q^N \to Z^* : e_i \mapsto z_i^*$, where $(e_i)$ is the canonical basis of $\ell_q^N$, and $\frac1p + \frac1q = 1$.
\\

Suppose $E$ and $F$ are Banach spaces, and fix $C > 0$ and $p \in [1,\infty)$. Define a $(C,p)$-game between two players as follows:
\\

At the start of the $n$-th round, we have finite dimensional subspaces $F_1, \ldots, F_{n-1} \subseteq F$, $E_1 , \ldots, E_{n-1} \subseteq E$, and norm one $y_i^* \in F_i^\perp$, $x_i^* \in E_i^\perp$ for $1 \leq i \leq n-1$ (here, for a subspace $G\subseteq F$, we denote $G^\perp=\{x^*\in F^*:x^*(x)=0,\, \forall x\in G\}$).
\\

Round $n$, step 1: Player 1 selects a finite dimensional $E_n \subseteq E$, then Player 2 picks a finite dimensional $F_n \subseteq F$.
\\

Round $n$, step 2:
Player 1 chooses $y_n^* \in S_{F_n^\perp}$ (the unit sphere of $F_n^\perp$), then Player 2 picks $x_n^* \in S_{E_n^\perp}$. \\

Player 1 wins the $(C,p)$-game after $N$ rounds if there exist $\alpha_1, \ldots, \alpha_N \geq 0$ so that 
$\|(\alpha_i x_i^*)_{i=1}^N\|_{p, {\mathrm{weak}}} > C \|(\alpha_i y_i^*)_{i=1}^N\|_{p, {\mathrm{weak}}}$
(we say that $(\alpha_i, x_i^*, y_i^*)_{i=1}^N$ witnesses the win of Player 1).
 \\

We shall say that $E^*$ \emph{$p$-dominates} $F^*$ (relative to preduals $E$ and $F$, which we will omit if the duality is canonical) if Player 1 has a winning strategy for the $(C,p)$ game for any $C>0$ (that is, Player 1 can win, no matter what Player 2 does).
Note that we can always assume that $E_1 \subseteq E_2 \subseteq \ldots$, and $F_1 \subseteq F_2 \subseteq \ldots$.
\\

We need a simple observation combining duality with small perturbations.

\begin{lem}\label{l:small pert}
Suppose $Z$ is a Banach space, and $\varepsilon > 0$. 
\begin{enumerate}
 \item
 Suppose $G$ is a subspace of $Z$. Then for any $z^* \in Z^*$,
 $ {\mathrm{dist}} (z^*, G^\perp) = \sup \big\{ |z^*(z)| : z \in G, \|z\| \leq 1 \}$. Further, 
 $$
 {\mathrm{dist}} (z^*, G^\perp) \geq \frac12 \inf \big\{ \|z^* - w^* \| : w^* \in G^\perp, \|w^*\| = \|z^*\| \} .
 $$

 \item
 Suppose $G$ and $G_0$ are subspaces of $Z$, so that for every $z \in G \backslash \{0\}$ there exists $z_0 \in G_0$ so that $\|z_0\| = \|z\|$ and $\|z - z_0\| < \varepsilon \|z\|$.
 Then any $z^* \in G_0^\perp$ satisfies ${\mathrm{dist}} (z^*, G^\perp) < \varepsilon \|z^*\|$.
\end{enumerate}
 \end{lem}

Below, we apply this lemma for finite dimensional $G$ and $G_0$. In this case, the statement of (2) can be strengthened slightly:  if for every $z \in G$ there exists $z_0 \in G_0$ so that $\|z_0\| = \|z\|$ and $\|z - z_0\| \leq \varepsilon \|z\|$, then any $z^* \in G_0^\perp$ satisfies ${\mathrm{dist}} (z^*, G^\perp) \leq \varepsilon \|z^*\|$.

\begin{proof}
(1) The equality $${\mathrm{dist}} (z^*, G^\perp) = \sup \big\{ |z^*(z)| : z \in G, \|z\| \leq 1 \}$$ follows from the canonical identification between $G^*$ and $Z^*/G^\perp$. To establish the ``further'' statement, if suffices to show that, if $\|z^*\| = 1$, and ${\mathrm{dist}} (z^*, G^\perp) < c$, then there exists a norm one $w^* \in G^\perp$ with $\|z^* - w^*\| < 2c$. 
To this end, find $u^* \in G^\perp$ so that $\|z^* - u^*\| < c$. By the triangle inequality, $\big| \|u^*\| - 1 \big| < c$. Let $w^* = u^*/\|u^*\|$, so $u^* = \|u^*\| w^*$, and therefore, $\|u^* - w^*\| = \big| \|u^*\| - 1 \big| < c$. Consequently, $$\|z^* - w^*\| \leq \|z^* - u^*\| + \|u^* - w^*\| < 2c.$$

\medskip

(2) Pick a norm one $z^* \in G_0^\perp$. By (1), ${\mathrm{dist}} (z^*, G^\perp) = \sup \big\{ |z^*(z)| : z \in G, \|z\| =1 \}$. For any $z$ as in the right hand side,  
find $z_0 \in B_{G_0}$ so that $\|z - z_0\| < \varepsilon$.
Then  $|z^*(z)| \leq |z^*(z_0)| + \|z-z_0\| < \varepsilon$. 
\end{proof}

\begin{prop}\label{p:example of domination}
 Suppose $\infty \geq u > \max\{v,p\} \geq v \geq 1$, 
 $E = (\sum_i E_i)_u$ $(E_1, E_2, \ldots$ are finite dimensional; for $u=\infty$, consider the $c_0$-sum$)$, and $F^*$ contains a copy of $\ell_{v'}$, with $1/v + 1/v' = 1$.
 Then $E^*$ $p$-dominates $F^*$.
\end{prop}

\begin{proof}
Assume $F^*$ contains a normalized basic sequence, $K$-equivalent to the canonical basis of $\ell_{v'}$.
Fix $C > 0$. Let $u' = u/(u-1)$ (so $1/u + 1/u' = 1$). In the course of a $(C,p)$-game, Player 1 can arrange $(y_i^*)_{i=1}^N$ to be $2K$-equivalent to the unit vector basis of $\ell_{v'}^N$, and force Player 2 to make $(x_i^*)_{i=1}^N$ to be $2$-equivalent to the unit vector basis of $\ell_{u'}^N$ (this follows from a ``gliding hump'' argument, permitted by \Cref{l:small pert}(2)).
Then 
$$ \frac1{2K} \|(y_i^*)\|_{p, {\mathrm{weak}}} \leq \|id : \ell_q^N \to \ell_{v'}^N\| =
 \left\{ \begin{array}{ll} N^{1/p-1/v}  &  v > p  \\  1  &  v \leq p   \end{array}  \right. $$
(here $id$ stands for the formal identity and $\frac1p+\frac1q=1$). Similarly, $2 \|(x_i^*)\|_{p, {\mathrm{weak}}} \geq N^{1/p-1/u}$ (since $u > p$). Thus, $\|(x_i^*)\|_{p, {\mathrm{weak}}} > C \|(y_i^*)\|_{p, {\mathrm{weak}}}$, for $N$ large enough.
\end{proof}

Above, we defined $p$-domination, and established examples when it occurs. Next, we use it to show that certain free Banach lattices cannot be lattice isomorphic.

\begin{prop}\label{p:domination}
 Suppose $1 \leq p < \infty$, $E$ and $F$ are Banach spaces, and $E^*$ $p$-dominates $F^*$.
Then $\fbp[F]$ is not lattice isomorphic to a lattice quotient of $\fbp[E]$.
\end{prop}

The proof requires an auxiliary result:

\begin{lem}\label{l:sums}
Let $\frac1p+\frac1q=1$. For any Banach space $Z$, and any $z_1^* , \ldots, z_n^* \in Z^*$, we have
$$ \|(z_i^*)\|_{p,{\mathrm{weak}}} = \sup \big\{ \|\sum_{i=1}^n \gamma_i \widehat{z_i^*} \|_{\fbp[Z]^*} : \sum_{i=1}^n |\gamma_i|^q \leq 1 \big\} . $$
\end{lem}

\begin{proof}
 Let $T:Z\rightarrow \ell_p^n$ be given by $z\mapsto(z_i^*(z))_{i=1}^n$ and consider its canonical extension $\widehat T:\fbp[Z]\rightarrow \ell_p^n$. Note that $(\widehat T)^*:\ell_q^n\rightarrow\fbp[Z]^*$ maps the unit vector basis to $\widehat{z_i^*}$. Hence,
 \begin{align*}
\|(z_i^*)\|_{p,{\mathrm{weak}}} & = \|T\|=\|\widehat T\|=\|\widehat T^*\| \\
&=\sup \big\{ \|\sum_{i=1}^n \gamma_i \widehat{z_i^*} \|_{\fbp[Z]^*} : \sum_{i=1}^n |\gamma_i|^q \leq 1 \big\}.     
 \end{align*}
\end{proof}

\begin{proof}[Proof of \Cref{p:domination}]
Henceforth, suppose $T : \fbp[E] \to \fbp[F]$ is a surjective lattice homomorphism (by scaling, we can assume it is contractive).
There exists $c > 0$ so that for any $g \in \fbp[F]$ there exists $f \in \fbp[E]$ so that $Tf = g$, $\|f\| \leq c^{-1} \|g\|$.
We keep the earlier notation $\Phi_T$. By \Cref{Phi T for specific T}, the inequality $c \|y^*\| \leq \|\Phi_T y^*\| \leq \|y^*\|$ holds for any $y^* \in F^*$. 
\\

Fix $\varepsilon \in (0,1/4)$ and $C > 1$.
Find  $K > (C+\eps) c^{-1}$.
Now let us start a $(K,p)$-game.
\\

Suppose $n-1$ rounds have been played; we have $E_1 \subseteq \ldots \subseteq E_{n-1} \subseteq E$, $F_1\subseteq \ldots \subseteq F_{n-1} \subseteq F$; norm one $y_i^* \in F_i^\perp$ and $x_i^* \in E_i^\perp$, for $1 \leq i < n$, so that $\|t_i x_i^* - \Phi_T y_i^*\| < 4^{-i} \varepsilon$, for some $t_i \in [c,1]$; these have been chosen in such a way that Player 1 can still win the $(K,p)$-game if they keep playing.
\\

On the first step of the $n$-th round, Player 1 picks a finite dimensional $E_n \subseteq E$ which contains $E_{n-1}$, and
permits winning.
Then Player 2 chooses $F_n\subseteq F$, $F_n \supseteq F_{n-1}$ so that, for any norm one $y^* \in F_n^\perp$, and any $x \in E_n$, we have $|[\Phi_T y^*] (x)| \leq 4^{-1-n} \varepsilon \|x\|$ (this is possible, by \Cref{c:almost fin many for subspace}).
\\

On the second step, Player 1 selects a norm one $y_n^* \in F_n^\perp$ consistent with victory. By \Cref{l:small pert}(1), we have that 
$$
\inf\big\{ \|\Phi_T y_n^* - w^* \| : w^* \in E_n^\perp, \|w^*\| = \|\Phi_T y_n^*\| \}\leq2\mathrm{dist}(\Phi_T y_n^*,E_n^\perp)< 4^{-n}\varepsilon.
$$
Hence, Player 2 can find $x_n^* \in E_n^\perp$ with $\norm{x_n^*}=1$,
for which there exists $t_n \in [c,1]$ so that $\|\Phi_T y_n^* - t_n x_n^*\| < 4^{-n} \varepsilon$.
\\

Continue until we obtain $(y_i^*)_{i=1}^N$ and $(x_i^*)_{i=1}^N$ witnessing the victory of Player 1. That is, we can find $\alpha_1, \ldots, \alpha_N \geq 0$ so that
$$
\|(\alpha_i x_i^*)\|_{p, {\mathrm{weak}}} > K \|(\alpha_i y_i^*)\|_{p, {\mathrm{weak}}}.
$$
By scaling, we can assume $\max_i \alpha_i = 1$. Denote $\|(\alpha_i y_i^*)\|_{p, {\mathrm{weak}}}$ by $M$. Then clearly $M \geq 1$. By convexity,
$$
\|(\alpha_i t_i x_i^*)\|_{p, {\mathrm{weak}}} \geq c \|(\alpha_i x_i^*)\|_{p, {\mathrm{weak}}} > KcM .
$$
Then
\begin{align*}
 \big\| \big( \alpha_i \Phi_T y_i^* \big) \big\|_{p, {\mathrm{weak}}}
 &
 \geq
\big\| \big( \alpha_i t_i x_i^* \big) \big\|_{p, {\mathrm{weak}}}  - \sum_i \alpha_i \|\Phi_T y_i^* - t_i x_i^*\| 
\\
&
> KcM - \sum_i 4^{-i} \varepsilon > (Kc-\varepsilon) M > CM.
\end{align*}
By \Cref{l:sums},
$$
M = \sup \big\{ \| \sum_i \gamma_i \alpha_i \widehat{y_i^*} \|_{\fbp[F]^*} : \sum_i \gamma_i^q \leq 1 \big\}
$$
and 
$$
\big\| \big( \alpha_i \Phi_T y_i^* \big) \big\|_{p, {\mathrm{weak}}} = \sup \big\{ \| \sum_i \gamma_i \alpha_i T^* \widehat{y_i^*} \|_{\fbp[E]^*} : \sum_i \gamma_i^q \leq 1 \big\} .
$$
Thus, $\|T^*\| > C$. This contradicts the assumption that $\|T\| \leq 1$.
\end{proof}

We also have a ``local'' criterion for free lattices being ``different''.

 \begin{prop}\label{p:compare cotypes}
 Fix $u,v \in [2,\infty]$, $p \in [1,\infty]$, and $u < \min\{v,p'\}$,  where $1/p + 1/p' = 1$. 
 Suppose $E^*$ has cotype $u$, and $F^*$ does not have cotype less than $v$. Then $\fbp[F]$ is not lattice isomorphic to a lattice quotient of $\fbp[E]$.
 \end{prop}
 
 \begin{proof}
 Find $q \in (u , \min\{v,p'\})$. By \cite[Chapter 14]{DJT}, there exists $C > 0$ such that for any $n$ we can find $y_1^*, \ldots, y_n^* \in F^*$ with the property that, for any $(\alpha_i)$, we have 
 $$\max_i |\alpha_i| \leq \big\|\sum_i \alpha_i y_i^*\big\| \leq C \big(\sum_i |\alpha_i|^q\big)^{1/q}.$$ 
 Consequently, $\min_i \|y_i^*\| \geq 1$, and $\|(y_i^*)\|_{q',{\mathrm{weak}}} \leq C$.
 \\
 
 Suppose, for the sake of contradiction, that $T : \fbp[E] \to \fbp[F]$ is a surjective lattice homomorphism (without loss of generality, $T$ is contractive). Then, by \Cref{Lattice homo}, for $(y_i^*)$ as above we have $\|(\Phi_T y_i^*)\|_{q',{\mathrm{weak}}} \leq C$.
 \\
 
On the other hand, $T^*$ is bounded below by some $c > 0$, hence by \Cref{Phi T for specific T} the inequality $\|\Phi_T y^*\| \geq c \|y^*\|$ holds for any $y^* \in F^*$. 
 By cotype $u$, $\max_{\pm} \|\sum_i \pm \Phi_T y_i^*\| \geq K c n^{1/u}$ ($K$ is the cotype constant), so 
 $$\|(\Phi_T y_i^*)\|_{q',{\mathrm{weak}}} \geq \max_{\pm} \|\sum_i \pm n^{-1/q} \Phi_T y_i^*\| \geq Kc n^{1/u-1/q};$$ the latter exceeds $C$ for large $n$. This is the desired contradiction.
 \end{proof}

 \begin{cor}
  Suppose $r \in [1,2)$, and $s \in (r, \infty]$. Then $\fbl[L_r]$ is not a lattice quotient of $\fbl[L_s]$.
 \end{cor}

This corollary generalizes the classical result that, for $r$ and $s$ as above, $L_r$ is not a quotient of $L_s$.

 \begin{proof}
 Following the usual convention, we assume $1/r + 1/r' = 1 = 1/s + 1/s'$.  Let $E = L_s$, $F = L_r$, and note that $E^*$ has cotype $\max\{2,s'\}$, while $F^*$ has cotype $r' > \max\{2,s'\}$, but no smaller. Apply \Cref{p:compare cotypes} with $E$, $F$ as above, and $p = 1$.
 \end{proof}

The above results leads one to ask:

\begin{question}\label{q:properties preserved}
Suppose $\fbp[E]$ is lattice isomorphic to $\fbp[F]$. What properties do the spaces $E$ and $F$ necessarily share?
\end{question}

The results of \Cref{Dictionary} provide positive answers for certain properties (such as containing a complemented copy of $\ell_1$, or $\ell_1^n$, see \Cref{comp ell_1}, respectively~\Cref{c:criteria for l1}). Some other properties are covered by the following partial result.

\begin{prop}\label{p:superrefl}
Suppose $\fbl[E]$ is lattice isomorphic to $\fbl[F]$, and $E$ is a separable space which has $c_0$ as a quotient. Then:
\begin{enumerate}
    \item If $F$ is reflexive, it cannot be $K$-convex.
    \item $F$ is not super-reflexive.
\end{enumerate}
\end{prop}

\begin{proof}
 (1) Suppose $F$ is reflexive, and $T : \fbl[F] \to \fbl[E]$ is a lattice isomorphism. Let  $\Phi_T : E^* \to F^*$ be the corresponding map given by \Cref{l:latticehomocomp}. By the proof of \cite[Proposition 2.e.9]{LT1}, $E^*$ contains a weak$^*$ null sequence $(e_i^*)$, equivalent to the $\ell_1$ basis. The sequence $(\Phi_T e_i^*)$ is semi-normalized, and weakly null in $F^*$, hence, by \cite{DOSZ}, we can find $i_1 < i_2 < \ldots$ so that $(\Phi_T e_{i_k}^*)$ is Schreier unconditional. We have
 $$\max_\pm \|\sum_k \pm \alpha_k \Phi_T e_{i_k}^*\| \sim \max_\pm \|\sum_k \pm \alpha_k e_{i_k}^*\| \sim \sum_k |\alpha_k| ,$$
 hence for any $n$, and any choice of signs $\pm$,
 $$\|\sum_{k=n+1}^{2n} \pm \alpha_k \Phi_T e_{i_k}^*\| \sim \max_\pm \|\sum_{k=n+1}^{2n} \pm \alpha_k \Phi_T e_{i_k}^*\| \sim \sum_{k=n+1}^{2n} |\alpha_k| ,$$
 which shows that $F^*$ contains copies of $\ell_1^n$ uniformly. This is equivalent to the lack of $K$-convexity for $F^*$, hence also for $F$ \cite[Chapter 13]{DJT}.
\\
 
 (2) is a consequence of (1). Indeed, if $F$ is super-reflexive, then it is necessarily reflexive. Also, it cannot contains copies of $\ell_1^n$ uniformly, which implies $K$-convexity.
 \end{proof}
 
 We do not know whether, under the hypotheses of \Cref{p:superrefl}, $F$ necessarily has a $c_0$ quotient. One major obstacle is that a weakly null sequence may not have an unconditional subsequence \cite{MR} (see also \cite{JMS}).
\\

Note that \Cref{q:properties preserved} can be interpreted as inquiring which properties of Banach spaces are preserved under positively homogeneous bijections which are weak$^*$ to weak$^*$ continuous on bounded sets in both directions. For the discussion on Banach space properties preserved by other types of non-linear isomorphisms, see e.g.~\cite[Chapter 14]{alb-kal}. For instance, there it is shown that Lipschitz isomorphisms preserve super-reflexivity (\Cref{p:superrefl} above can be viewed as a weaker version of that).

\subsection{Isometries between $\fbp$'s for finite $p$}\label{ss:fbp isometries}

To examine the existence of lattice isometries between lattices of the form $\fbp[E]$ and $\fbp[F]$, recall that a Banach space $Z$ is called \emph{smooth} if, for every point $z$ on its unit sphere, there exists a unique support functional, which we call $f_z$ (that is $f_z(z)=\|f_z\|=1$).
For more information on smoothness, and on the related topic of strict convexity, we refer to \cite[Ch.~2]{DiGBS}.
\\

Recall that if $E$ and $F$ are linearly isometric, then $\fbp[E]$ and $\fbp[F]$ are lattice isometric. A converse to this is the main result of this section, which can be considered as a Banach-Stone type theorem for free Banach lattices: 

\begin{thm}\label{p:isometric smooth}
Suppose $1 \leq p < \infty$, and $E, F$ are Banach spaces so that $E^*, F^*$ are smooth. Then $T : \fbp[E] \to \fbp[F]$ is a surjective lattice isometry if and only if $T = \overline{U}$, for some surjective isometry $U : E \to F$. Consequently, $E$ and $F$ are isometric if and only if $\fbp[E]$ is lattice isometric to $\fbp[F]$.
\end{thm}
It is known that $Z$ is \emph{strictly convex} (that is, the equality $\|z_1 + z_2\| = 2$ holds for $z_1, z_2 \in S_Z$ if and only if $z_1=z_2$) whenever $Z^*$ is smooth. For reflexive spaces, the converse implication holds as well.
\\

Before proving \Cref{p:isometric smooth}, we recall some facts related to the geometry of the norm of a Banach space, and use them to describe the behavior of $\|(x,ty)\|_{p,{\mathrm{weak}}}$ for $t \approx 0$.
\\

Suppose $x$ is a point on the unit sphere of a Banach space $Z$. Denote by ${\mathcal{F}}(x)$ the set of support functionals for $x$ -- that is, of functionals $x^*$ for which $\|x^*\| = 1 = x^*(x)$ (note that this set is weak$^*$ closed, hence weak$^*$ compact).
Now suppose $y \in Z$, $\|y\| = 1$, and $\lambda \in \Real$. It is known (see e.g. \cite[Section 6]{JL}) that there exists $x^* \in {\mathcal{F}}(x)$ so that $x^*(y) = \lambda$ if and only if
\begin{equation}
\lim_{t \to 0^-} \frac{\|x+ty\| - 1}{t} \leq \lambda \leq \lim_{t \to 0^+} \frac{\|x+ty\| - 1}{t} .
\label{eq:values of functional}
\end{equation}
In particular, if ${\mathcal{F}}(x) = \{x^*\}$ (in this case, $x^* = f_x$), then 
$$ 
\lim_{t \to 0} \frac{\|x+ty\| - 1}{t} = x^*(y) .
$$ 
We begin the proof of \Cref{p:isometric smooth} with a lemma.

\begin{lem}\label{l:derivative}
 Suppose $x, y$ are elements of the unit sphere of $Z$, and $1 \leq p < \infty$. Let $\kappa = \sup_{x^* \in {\mathcal{F}}(x)} |x^*(y)|$. Then, for $t \to 0$, $$\|(x,ty)\|_{p,{\rm{weak}}} = 1 + \frac{\kappa^p}p |t|^p + o(|t|^p).$$
\end{lem}

Note that, in the definition of $\kappa$, $\sup$ can be replaced by $\max$.

\begin{proof}
Replacing $y$ by $-y$ if necessary, we assume (see \eqref{eq:values of functional}) that
$$
\kappa = \lim_{t \to 0+} \frac{\|x+ty\|-1}{t} ,
$$
hence 
$$
    \|x+ty\| = 1 + \kappa t + o(t) \, {\textrm{  for  }} \, t \to 0^+ .$$
Further, set
$$
\kappa' = \lim_{t \to 0^-} \frac{\|x+ty\|-1}{t} , 
$$
hence $\|x-ty\| = 1 - \kappa' t + o(t)$ for $t \to 0^+$. By our assumption, $|\kappa'| \leq \kappa$, hence 
\begin{equation}
   \max_\pm \|x \pm ty \| = 1 + \kappa |t| + o(t) .
   \label{eq:1 weak} 
   \end{equation}
To complete the proof for $p=1$, recall that $\big\| (x, ty ) \big\|_{1, {\textrm{weak}}} = \max_\pm \|x \pm ty \|$.
\\

Now consider $p \in (1,\infty)$. To estimate $\|(x,ty)\|_{p,{\textrm{weak}}}$ from below, find $x^* \in {\mathcal{F}}(x)$ so that $x^*(y) = \kappa$ (this is possible, due to the weak$^*$ compactness of ${\mathcal{F}}(x)$). Then
$$
\|(x,ty)\|_{p,{\textrm{weak}}} \geq \big( |x^*(x)|^p + |t x^*(y)|^p \big)^{1/p} = \big( 1 + \kappa^p |t|^p \big)^{1/p}. 
$$
Taylor expansion gives
$$  \big( 1 + \kappa^p |t|^p \big)^{1/p} = 1 + \frac{\kappa^p}{p} |t|^p + o(|t|^p). $$

The rest of the proof is devoted to estimating
$$ 
\|(x,ty)\|_{p,{\textrm{weak}}} = \max \big\{ \|\alpha x + \beta t y\| : |\alpha|^q + |\beta|^q \leq 1 \big\}
$$ 
from above (here $q = p/(p-1)$, so $1/p + 1/q = 1$). 
First, we show that, for any $\varepsilon > 0$, there exists $t_0 > 0$ so that
\begin{equation}
    \|\alpha x + \beta t y\| \leq 1 + \frac{\kappa^p}{p} |t|^p + \varepsilon |t|^p
    \label{eq:specific upper est}
\end{equation}
whenever $\alpha \geq 0$, $\alpha^q + |\beta|^q = 1$, and $|t| \leq t_0$. 
\\

The case of $\alpha \leq 1/2$ is easy: for $|t| \leq 1/2$, $\|\alpha x + \beta t y\| \leq \alpha + |t| \leq 1$. The remainder of the proof deals with $\alpha > 1/2$. Then $\alpha = (1 - |\beta|^q)^{1/q}$; by \eqref{eq:1 weak},
\begin{equation}
\begin{split}
\|\alpha x + \beta t y\| 
&
= \alpha \Big\| x + \frac{\beta}{\alpha} t y \Big\| 
\\
&
\leq
(1 - |\beta|^q)^{1/q} + |\beta| |t| \kappa + \alpha \phi \Big(\frac\beta\alpha t\Big) ,
\end{split}
\label{eq:toward p weak norm}
\end{equation}
where $\phi(s) = o(s)$ near $0$. To analyze the supremum of the above expression, we show the existence of $B > 0$ (depending solely on $p$) so that
\begin{equation}
(1-|\beta|^q)^{1/q} \leq 1 - 2 |\beta| |t| \, {\textrm{  for  }} \, |\beta| \geq B |t|^{1/(q-1)} .
    \label{eq:calculus estimate}
\end{equation}
Taking this inequality for granted, combine \eqref{eq:toward p weak norm} with \eqref{eq:calculus estimate}: for $|\beta| \geq B |t|^{1/(q-1)}$ and $\alpha \geq 1/2$, 
$$
\|\alpha x + \beta t y\| \leq 1 - |\beta t| + \alpha \phi \Big( \frac\beta\alpha t\Big) .
$$ 
Find $s_0$ so that $|\phi(s)| \leq |s|/2$ whenever $|s| \leq s_0$. Then for $|t| \leq s_0/2$, we have
$$\Big| \phi \Big( \frac\beta\alpha t\Big) \Big| \leq \frac12 \big| \frac\beta\alpha t\big| \leq |\beta| |t|, $$ and therefore, for such $t$,
\begin{equation}
    \label{eq:large beta}
\max \big\{ \|\alpha x + \beta t y\|  : |\alpha|^q + |\beta|^q \leq 1 , |\beta| \geq B |t|^{1/(q-1)} \big\} \leq 1 .
\end{equation}

Finally, for $\varepsilon > 0$, find $s_1 > 0$ so that $|\phi(s)| \leq B^{-1} \varepsilon |s|$ for $|s| \leq s_1$. Then, for $|t| \leq s_1/2$ and $|\beta| \leq B |t|^{1/(q-1)}$,
$$
\alpha \Big| \phi \Big( \frac\beta\alpha t\Big) \Big| \leq \alpha B^{-1} \varepsilon \big| \frac\beta\alpha t\big| \leq B^{-1} \varepsilon \beta |t| \leq \varepsilon |t|^p 
$$
(since $p = 1 + 1/(q-1)$). Therefore, for such $t$,
\begin{align*}
&
\max \big\{ \|\alpha x + \beta t y\|  : |\alpha|^q + |\beta|^q \leq 1 , |\beta| \leq B |t|^{1/(q-1)} \big\} 
\\
&
\leq \max \big\{ (1 - |\beta|^q)^{1/q} + \beta |t| \kappa : \beta \in [-1,1] \big\}  + \varepsilon |t|^p .
\end{align*}
By H\"older's Inequality,
\begin{align*}
    &   \max \big\{ (1 - |\beta|^q)^{1/q} + \beta |t| \kappa : \beta \in [-1,1] \big\} 
    \\ &    = \max \big\{ \alpha \cdot 1 + \beta \cdot \kappa |t| : |\alpha|^q + |\beta|^q \leq 1 \big\} 
    \\ &    = \|(1, \kappa |t|)\|_p = \big( 1 + (\kappa |t|)^p \big)^{1/p} \leq 1+ \frac{\kappa^p |t|^p}p ,
\end{align*}
hence
\begin{equation}
\begin{split}
    & \max \big\{ \|\alpha x + \beta t y\|  : |\alpha|^q + |\beta|^q \leq 1 , |\beta| \leq B |t|^{1/(q-1)} \big\}
    \\ & \leq 1 + \frac{\kappa^p}{p} |t|^p + \varepsilon |t|^p.
\end{split}
    \label{eq:small beta}
\end{equation}
Together, \eqref{eq:large beta} and \eqref{eq:small beta} establish \eqref{eq:specific upper est}, with $t_0 = \min\{s_0,s_1\}/2$.
\\

It remains to establish \eqref{eq:calculus estimate}. For convenience we shall only deal with non-negative values of $\beta$ and $t$. That is, we have to show that $$(1-2\beta t)^q \geq 1 - \beta^q$$ for $\beta \geq B t^{1/(q-1)}$. By Bernoulli's Inequality, $$(1-2\beta t)^q \geq 1 - 2 q \beta t,$$ hence it suffices to select $B$ to guarantee that $1 - 2 q \beta t \geq 1 - \beta^q$ holds for $\beta \geq B t^{1/(q-1)}$. Clearly $B = (2q)^{1/(q-1)}$ works.
\end{proof}

\begin{proof}[Proof of \Cref{p:isometric smooth}]
Following \cite{IT}, we define the following semi-inner product on $E^*$: for $x^*, y^* \in E^*$, 
$$
[y^*,x^*] = \left\{\begin{array}{cc}
    0 & \text{ if }x^*=0, \\
    f_{x^*}(y^*) &  \text{ if }x^*\neq0,
\end{array}\right.
$$ 
where $f_{x^*} \in E^{**}$ is the unique support functional at $x^*$ -- that is, $\|f_{x^*}\| = \|x^*\| = \sqrt{f_{x^*}(x^*)}$. A semi-inner product on $F^*$ is defined in a similar fashion.
\\

By \Cref{l:latticehomocomp}, $T$ is implemented by a surjective positively homogeneous map $\Phi_T : F^* \to E^*$, weak$^*$ to weak$^*$ continuous on bounded sets, which preserves the $(p, {\textrm{weak}})$-norms of tuples; $\Phi_T^{-1}$ also has all these properties.
By \Cref{l:derivative}, $\Phi_T$ preserves  absolute value of the semi-inner product defined above. By \cite{IT}, there exist a linear surjective isometry $V : F^* \to E^*$ and a function $\sigma : F^* \to \{-1,1\}$ so that $\Phi_T f^* = \sigma(f^*) V f^*$ for any $f^* \in F^*$.
Due to the positive homogeneity of $\Phi_T$, $\sigma$ is constant on rays -- that is, $\sigma(t f^*) = \sigma(f^*)$ for any $f^* \neq 0$, and $t > 0$.
\\

We claim that $\sigma$ is a constant on the sphere of $F^*$. Indeed, otherwise, up to a sign change, we can assume that there exists a sequence $(f_k^*)$ on the unit sphere of $F^*$, converging to $f^*$ in norm, so that $\sigma(f_k^*) = 1$ for any $k$, and $\sigma(f^*) = -1$ (we make use of the connectedness of the unit sphere). Then $(\Phi_T f_k^*)$ converges in norm, and hence also weak$^*$, to $V f^*$. On the other hand, $\Phi_T f_k^* \to \Phi_T f^* = - V f^*$ weak$^*$, which is a contradiction.
\\

By changing sign if necessary, we can assume $\sigma = 1$ everywhere, hence $\Phi_T f^* = V f^*$ for any $f^* \in F^*$. The linear isometry $V$, and its inverse, are weak$^*$ to weak$^*$ continuous on bounded sets. 
It remains to show that $V$ is an adjoint operator -- that is, $V = U^*$, with some $U \in B(E,F)$ (such a $U$ is automatically a surjective isometry). To this end, consider $V^*  : E^{**} \to F^{**}$. By \cite[Corollary 4.46]{FHHMPZ}, $e^{**} \in E^{**}$ comes from $\kappa_E(E)$ (where $\kappa_E$ denotes the canonical embedding into the bidual) if and only if $\ker e^{**} \cap B_{E^*}$ is weak$^*$ closed; the same is true regarding $f^{**} \in F^{**}$. As $V$ is a surjective isometry, we have 
$$
\ker (V^* e^{**}) \cap B_{F^*} = V^{-1}(\ker e^{**} \cap B_{E^*}).
$$ 
Since $V$ is weak$^*$ to weak$^*$ continuous on bounded sets, it follows that $\ker (V^* e^{**}) \cap B_{F^*}$ is weak$^*$ closed whenever $\ker e^{**} \cap B_{E^*}$ is.  In other words, $V^*$ maps $\kappa_E(E)$ into $\kappa_F(F)$. Consequently, $V = U^*$, where $U = \kappa_F^{-1} V^* \kappa_E \in B(E,F)$.
\end{proof}

The smoothness assumption is essential for the preceding proof. Without smoothness, we can obtain some partial results. 

\begin{prop}\label{p:preserve smoothness and convexity}
Suppose $1 \leq p < \infty$, and $\fbp[E]$ is lattice isometric to $\fbp[F]$. 
\begin{enumerate}
    \item If $E^*$ is strictly convex, then so is $F^*$.
    \item If both $E$ and $F$ are reflexive, and $E^*$ is smooth, then $F^*$ is smooth as well.
\end{enumerate}
\end{prop}

For the proof, we need a particular case of \Cref{l:derivative}:

\begin{cor}\label{c:descition of flats}
 Suppose $p \in [1,\infty)$, and $z, y \in Z$ with $\|z\| = 1 = \|y\|$. Then $\max_\pm \|z \pm y\| = 2$ if and only if
 \begin{equation}
 \lim_{t \to 0} p t^{-p} \Big( \big\| (z, ty ) \big\|_{p, {\textrm{weak}}} - 1 \Big) = 1 .
 \label{eq:chcarterize norm 2}
 \end{equation}
\end{cor}
 
 \begin{proof}
 If $\|z+y\|=2$ or $\|z-y\| =2$, find $z^* \in {\mathcal{F}}(z)$ so that $|z^*(y)| = 1$. Apply \Cref{l:derivative}. 
 Conversely, if \eqref{eq:chcarterize norm 2} holds, then there exists $z^* \in {\mathcal{F}}(z)$ with $|z^*(y)| = 1$. Then $\max_\pm \|z \pm y\| = 2$.
 \end{proof}

\begin{proof}[Proof of \Cref{p:preserve smoothness and convexity}(1)]
Suppose, for the sake of contradiction, that $E^*$ is strictly convex, but $F^*$ is not. Find norm one $y_0^*, y_1^* \in F^*$ so that $y_0^* \neq y_1^*$, and $\|y_1^* + y_0^*\| = 2$. For $s \in (0,1)$ let $y_s^* = (1-s) y_0^* + s y_1^*$. It is easy to see that $\|y_s^*\| = 1$, and $\|y_0^* + y_s^*\| = 2$, for any $s\in(0,1)$.
Consequently, by \Cref{c:descition of flats},
$$
\lim_{t \to 0} p t^{-p} \Big( \big\| (y_0^*, t y_s^*) \big\|_{p, {\textrm{weak}}} - 1 \Big) = 1 .
$$

Find a positively homogeneous map $\Phi : F^* \to E^*$, weak$^*$ to weak$^*$ continuous on bounded sets, which implements a surjective lattice isometry $\fbp[E] \to \fbp[F]$. Let $x_s^* = \Phi y_s^*$. Then for any $s\in(0,1)$, $\|x_s^*\| = 1$. Moreover, $\Phi$ preserves $(p, {\textrm{weak}})$-norms of tuples, so
$$
\lim_{t \to 0} p t^{-p} \Big( \big\| (x_0^*, t x_s^*) \big\|_{p, {\textrm{weak}}} - 1 \Big) = 1 .
$$
Consequently, by \Cref{c:descition of flats}, $\max_\pm \|x_0^* \pm x_s^*\| = 2$. By the strict convexity of $E^*$,  $x_s^* \in \{x_0^*, -x_0^*\}$ for any $s$. However, all $x_s^*$'s must be distinct, which gives a contradiction.
\end{proof}

The following topological result is likely  known to experts.

\begin{lem}\label{l:connectedness}
Consider a Banach space $Z$, equipped with the weak topology. Let $G$ be a closed subspace of $Z$. Then $Z \backslash G$ is path connected if $\dim Z/G \geq 2$, and is disconnected if $\dim Z/G = 1$.
\end{lem}

\begin{proof}
If $\dim Z/G = 1$, find $z^* \in Z^*$ so that $G = \ker z^*$. Then $Z \backslash G$ is a union of two open sets -- $\{z \in Z : z^*(z) > 0 \}$ and $\{z \in Z : z^*(z) < 0 \}$ -- hence disconnected.
\\

Now suppose $\dim Z/G > 1$. For $z_0, z_1 \in Z \backslash G$, we need to find a path connecting these two points. By replacing $Z$ by its subspace, we can and do assume that $\dim Z/G = 2$. Represent $Z$ as $G \oplus H$, with $\dim H = 2$. Represent $z_i = g_i + h_i$, with $g_i \in G$ and $h_i \in H \backslash \{0\}$.
Find norm-continuous functions $g : [0,1] \to G$ and $h : [0,1] \to H \backslash \{0\}$ so that $g(0) = g_0$, $g(1) = g_1$, $h(0) = h_0$, and $h(1) = h_1$. Then $t \mapsto g(t) + h(t)$ is the desired path.
\end{proof}

\begin{proof}[Proof of \Cref{p:preserve smoothness and convexity}(2)]
Suppose, for the sake of contradiction, that $E^*$ is smooth, but $F^*$ is not.
Find a positively homogeneous map $\Phi : F^* \to E^*$, weak$^*$ to weak$^*$ continuous on bounded sets, which implements a surjective lattice isometry $\fbp[E] \to \fbp[F]$. 
Find a non-smooth point $f^*$ on the unit sphere of $F^*$. 
Let $e^* = \Phi f^*$. Let $e = f_{e^*}$. It follows that $x^* \in E^*$ satisfies
$$
\lim_{t \to 0} \frac1{t^p} \big( \|(e^*, t x^*)\|_{p,{\textrm{weak}}} - 1 \big) = 0 
$$
if and only if $x^* \in \ker e =: A$. 
\\

Now let $B$ be the set of all $y^* \in F^*$ for which
$$
\lim_{t \to 0} \frac1{t^p} \big( \|(f^*, t y^*)\|_{p,{\textrm{weak}}} - 1 \big) = 0 .
$$
Then $B = \cap \big\{ \ker f : f \in {\mathcal{F}}(f^*) \big\}$. Further, $A = \Phi(B)$, and $E^* \backslash A = \Phi(F^* \backslash B)$. By \Cref{l:connectedness}, $F^* \backslash B$ is connected (since $\dim F^*/B \geq 2$), and $E^* \backslash A$ is not (since $\dim E^*/A = 1$). However, a disconnected set cannot be a continuous image of a connected set.
\end{proof}

\subsection{Isomorphism between $\fbl^{(\infty)}$ lattices}\label{ss:fbl infty isomorphisms}

It turns out that
$\fbl^{(\infty)}[X]$ and $\fbl^{(\infty)}[Y]$ may be lattice isomorphic, or even isometric, even when $X$ and $Y$ are non-isomorphic. For motivation, we recall a result from \cite{JLTTT}:
\begin{prop}
Let $E$ be a Banach space. Then $C(B_{E^*})$ is the free unital AM-space over $E$. More specifically, for any compact Hausdorff space $K$ and any norm one operator $T:E\to C(K)$ there exists a unique unital lattice homomorphism $\widehat T: C(B_{E^*})\to C(K)$ such that $\widehat T\circ \phi_E=T$. Moreover, $\|\widehat T\|=1.$
\end{prop}
From this result we deduce that the free \emph{unital} AM-spaces over $E$ and $F$ are (isometrically) lattice isomorphic if and only if $(B_{E^*},w^*)$ and $(B_{F^*},w^*)$ are homeomorphic. In particular, since the dual ball of any separable infinite dimensional Banach space is weak$^*$ homeomorphic to the Hilbert cube $[0,1]^\omega$ by Keller's Theorem (see \cite[Section 12.3]{fab-ultimo}), the free unital AM-space cannot distinguish between separable Banach spaces. We now prove a similar result for $\fbl^{(\infty)}$. This has the added difficulty that one needs to build positive homogeneity into the homeomorphism.
\\

We say that $(Z_i)$ is a \emph{finite dimensional decomposition} (\emph{FDD} for short) in a Banach space $Z$ if the spaces $Z_i \subseteq Z$ are finite dimensional, $Z = \overline{\textrm{span}}[Z_i : i \in \Nat]$, and there exist projections $P_i$ from $Z$ onto $Z_i$ so that $P_i P_j = 0$ whenever $i \neq j$, and $\sup_n \|\overline{P}_n\| < \infty$ (here $\overline{P}_n = P_1 + \ldots + P_n$), and, for any $z \in Z$, $\overline{P}_n z \to z$ in norm (equivalently, weakly \cite{McA}).
Then $\overline{P}_n^*$ converges to $I_{Z^*}$ in the point-weak$^*$ topology.
We say that an FDD is \emph{monotone} if $\overline{P}_n$ is contractive for every $n$. A classical renorming procedure makes an FDD monotone.

\begin{thm}\label{t:monot_FDD}
 Suppose a Banach space $X$ has a monotone FDD. Then $\fbl^{(\infty)}[X]$ is lattice isometric to $\fbl^{(\infty)}[c_0]$.
\end{thm}

\begin{rem}\label{finite vs infinite p}
This result contrasts sharply with those of \Cref{ss:fbp distinct}. For instance, whereas the above yields that $\fbl^{(\infty)}[c_0]$ and $\fbl^{(\infty)}[\ell_1]$ are lattice isometric, combining \Cref{p:example of domination} with \Cref{p:domination}, we conclude that $\fbp[c_0]$ and $\fbp[\ell_1]$ are not lattice isomorphic for any $p \in [1,\infty)$. In fact, we know no examples of non-isomorphic Banach spaces $E, F$ for which $\fbp[E]$ and $\fbp[F]$ ($1 \leq p < \infty$) are lattice isomorphic.
\end{rem}

The following easy observation will be used throughout the proof of \Cref{t:monot_FDD}.

\begin{lem}\label{l:simple criterion of weak* convergence}
Suppose $X$ is a Banach space with an FDD implemented by projections $(P_n)$. Then a norm bounded net $(x_\alpha^*) \subseteq X^*$ weak$^*$ converges if and only if the net $(\overline{P}_n^* x_\alpha^*)_\alpha$ is norm convergent for every $n$. Further, $(x_\alpha^*)$ weak$^*$ converges to $x^* \in X^*$ if and only if $\overline{P}_n^* x_\alpha^* \to \overline{P}_n^* x^*$ in norm for any $n$.
\end{lem}

\begin{proof}[Proof of \Cref{t:monot_FDD}]
Let $(P_i)$ be the FDD projections in $X$, and let $E_i = \ran P_i^* \hookrightarrow X^*$.
Then $\overline{P}_n^* = P_1^* + \ldots + P_n^*$ is a projection onto $\overline{E}_n = E_1 \oplus \ldots \oplus E_n$.
 Let $F_i = \ell_1^{n_i}$, with $n_i = \dim E_i$. Find a bi-continuous bijection $\rho_i : E_i \to F_i$, so that $\|\rho_i y^*\| = \|y^*\|$ and $\rho_i(t y^*) = t \rho_i(y^*)$ for any $y^*$ and $t \in \Real$. Finally, we identify $c_0$ with $Y = (\sum_i F_i^*)_{c_0}$. Denote the coordinate projections from $Y$ onto $F_i^*$ by $Q_i$, and let $\overline{Q}_n = Q_1 + \ldots + Q_n$. Then $\overline{Q}_n^*$ is a projection from $Y^*$ onto $\overline{F}_n=F_1 \oplus \ldots \oplus F_n$. Note also that $Y^* = (\sum_i F_i)_1 \sim \ell_1$.
 \\
 
 We recursively define continuous norm-preserving positively homogeneous bijections $\Psi_n : \overline{E}_n \to \overline{F}_n$ (henceforth we say that a map $\Psi$ is \emph{norm-preserving} if $\|\Psi z\| = \|z\|$ on the domain of $\Psi$).
 \\
 
 To begin, let $\Psi_1 x^* = \rho_1 x^*$.
 Now suppose $\Psi_{n-1}$ with the desired properties has been defined and let us describe $\Psi_n$.
 Any $x^* \in \overline{E}_n$ can be written, in a unique way, as $x^* = x_0^* + x_n^*$, where $x_0^* \in \overline{E}_{n-1}$, and $x_n^* \in E_n$. If $x_0^* = 0$, let  $\Psi_n x^* = \rho_n x_n^*$, while if $x_n^* = 0$, let $\Psi_n x^* = \Psi_{n-1} x^*$. Otherwise, let 
 $$
 \Psi_n(x_0^* + x_n^*) = \kappa_n ( \|x_n^*\|/\|x_0^*\| ) \Psi_{n-1} (x_0^*) + t \rho_n x_n^* ,
 $$
 where $ \kappa_n (s) = \big( 1 + \min\{ 4^{-n}, s\} \big)^{-1}$, and
$$
t = \frac{\|x_0^* + x_n^*\|}{\|x_n^*\|} - \kappa_n \Big( \frac{\|x_n^*\|}{\|x_0^*\|} \Big) \frac{\|x_0^*\|}{\|x_n^*\|}  > 0 , 
$$
which guarantees $\|\Psi_n(x_0^* + x_n^*)\| = \|x_0^* + x_n^*\|$.
 Note that
 \begin{equation}
 \|\overline{Q}_{n-1}^* \big(\Psi_{n-1} x_0^* - \Psi_n(x_0^* + x_n^*)\big)\| \leq \frac{4^{-n}}{1+4^{-n}} \|x_0^*\| .
 \label{eq:small_compression}
 \end{equation}
From this description, $\Psi_n$ is bijective -- in fact, it maps $\overline{E}_{n-1} + \Real x_n^*$ onto $\overline{F}_{n-1} + \Real \rho_n x_n^*$. One can also observe that $\Psi_n$ is continuous.
 For $x^* \in X^*$, we define $\Psi x^* = {\textrm{weak}}^*-\lim_n \Psi_n \overline{P}_n^* x^*$.
\\

 We shall show that $\Psi$ is well defined (the limit above indeed exists, even in the norm topology), is positively homogeneous, norm preserving, and weak$^*$ continuous.
To begin, note that $\|x^*\| = \lim_n \|\overline{P}_n^* x^*\|$ holds for any $x^* \in X^*$. Indeed, $\|x^*\| \geq \|\overline{P}_n^* x^*\|$ for any $n$, by monotonicity. On the other hand, $x^* = {\textrm{weak}}^*-\lim_n \overline{P}_n^* x^*$, hence $\|x^*\| \leq \liminf \|\overline{P}_n^* x^*\|$.
 Recall that, for $j \geq n$, we have $\overline{Q}_n^* \overline{Q}_j^* = \overline{Q}_n^*$.
As $\overline{Q}_n^*$ is contractive,  \eqref{eq:small_compression} implies that
 \begin{equation}
\begin{split}
&
 \|\overline{Q}_n^* \Psi_j \overline{P}_j^* x^* - \overline{Q}_n^* \Psi_{j+1} \overline{P}_{j+1}^* x^*\| 
 \\ &
 \leq \|\Psi_j \overline{P}_j^* x^* - \overline{Q}_j^* \Psi_{j+1} \overline{P}_{j+1}^* x^*\|  \leq 4^{-j} \|x^*\| .
 \end{split}
 \label{eq:small increase}
 \end{equation}
 Consequently, for every $n$, the sequence $\big(\overline{Q}_n^* \Psi_j \overline{P}_j^* x^*\big)_j$ is Cauchy, hence convergent (in norm). 
 Note that $\|\Psi_j \overline{P}_j^* x^*\| \nearrow \|x^*\|$, hence, in particular, the sequence $(\Psi_j \overline{P}_j^* x^*)_j$ is norm bounded. Therefore, \Cref{l:simple criterion of weak* convergence} (applied with the sequence $(\Psi_j \overline{P}_j^* x^*)_j$ instead of the net $(x_\alpha^*)_\alpha$, and with projections $Q_n$ instead of $P_n$) shows that the weak$^*$ limit of $(\Psi_j \overline{P}_j^* x^*)$ exists (and therefore, $\Psi$ is well defined), with $\|\Psi x^*\big\| \leq \|x^*\|$.
 \\

 On the other hand, for any $n$, $\|\Psi x^*\big\| \geq \limsup_m \|\overline{Q}_n^* \Psi_m \overline{P}_m^* x^* \|$.
 For $m>n$, we can use \eqref{eq:small increase} to write a telescopic sum:
\begin{equation}
    \begin{split}
 \|\Psi_n \overline{P}_n^* x^* - \overline{Q}_n^* \Psi_m \overline{P}_m^* x^*\|
 & \leq
 \sum_{j=n}^{m-1} \|\overline{Q}_n^* \Psi_j \overline{P}_j^* x^* - \overline{Q}_n^* \Psi_{j+1} \overline{P}_{j+1}^* x^*\| 
 \\
 &
 \leq
 \sum_{j=n}^{m-1} 4^{-j} \|x^*\|  \leq 2 \cdot 4^{-n} \|x^*\| ,
    \end{split}
    \label{eq:psi-n minus Q-n}
\end{equation}
 hence
 $ \|\overline{Q}_n^* \Psi_m \overline{P}_m^* x^* \| \geq \|\Psi_n \overline{P}_n^* x^* \| - 2 \cdot 4^{-n} \|x^*\|$,  and therefore,
 \begin{align*}
     \|\Psi x^*\| 
     &
     \geq \limsup_n \big( \| \Psi_n \overline{P}_n^* x^* \|  - 2 \cdot 4^{-n} \|x^*\| \big)
     \\ &
     = \limsup_n \| \Psi_n \overline{P}_n^* x^* \| =\|x^*\| .
 \end{align*}
 Thus, $\Psi$ is norm-preserving.  As $\Psi_n$ is positively homogeneous for any $n$, so is $\Psi$.
 \\
 
  We observe that the sequence $(\Psi_n \overline{P}_n^* x^*)$ is not merely weak$^*$-convergent, but also Cauchy, hence convergent in norm. To this end, consider $m > n$, and recall several relevant facts. 
  \\
  
  (i) \eqref{eq:psi-n minus Q-n} shows that $\|\Psi_n \overline{P}_n^* x^* - \overline{Q}_n^* \Psi_m \overline{P}_m^* x^*\|  \leq 2 \cdot 4^{-n} \|x^*\|$.
  \\
  
  (ii) There exists $t \in [0,1]$ so that $t \Psi_n \overline{P}_n^* x^* = \overline{Q}_n^* \Psi_m \overline{P}_m^* x^*$; consequently, $\|\Psi_n \overline{P}_n^* x^* - \overline{Q}_n^* \Psi_m \overline{P}_m^* x^*\| = \|\Psi_n \overline{P}_n^* x^*\| - \|\overline{Q}_n^* \Psi_m \overline{P}_m^* x^*\|$.
  \\
  
  (iii)  Further,
  \begin{align*}
  & \|(I-\overline{Q}_n^*) \Psi_m \overline{P}_m^* x^*\|
  = \|\Psi_m \overline{P}_m^* x^*\| - \|\overline{Q}_n^* \Psi_m \overline{P}_m^* x^*\| 
  \\ &
  = \big( \|\Psi_m \overline{P}_m^* x^*\| - \|\Psi_n \overline{P}_n^* x^*\| \big) + \big( \|\Psi_n \overline{P}_n^* x^*\| - \|\overline{Q}_n^* \Psi_m \overline{P}_m^* x^*\| \big) .
  \end{align*}
  
  (iv) Finally, recall that $\Psi_n$ and $\psi_m$ are norm-preserving.
  \\
  
   In light of the above,
 \begin{align*}
 &
  \|\Psi_n \overline{P}_n^* x^* - \Psi_m \overline{P}_m^* x^*\|
  \\ &
  = \|\Psi_n \overline{P}_n^* x^* - \overline{Q}_n^* \Psi_m \overline{P}_m^* x^*\| + \|(I-\overline{Q}_n^*) \Psi_m \overline{P}_m^* x^*\| 
  \\
  &
  =
  2 \|\Psi_n \overline{P}_n^* x^* - \overline{Q}_n^* \Psi_m \overline{P}_m^* x^*\| + \big( \|\Psi_m \overline{P}_m^* x^*\| - \|\Psi_n \overline{P}_n^* x^*\| \big) 
  \\
  &
  \leq 
  4 \cdot 4^{-n} \|x^*\| + \big( \|\overline{P}_m^* x^*\| - \|\overline{P}_n^* x^*\| \big) .
 \end{align*}
 As $\|\overline{P}_n^* x^*\| \nearrow \|x^*\|$, we conclude that $(\Psi_n \overline{P}_n^* x^*)$ is a Cauchy sequence.
 \\
 
 Next we establish that $\Psi$ is weak$^*$ continuous on bounded sets. Suppose a net $(x_\alpha^*) \subseteq B_{X^*}$ converges weak$^*$ to $x^*$ (hence $x^* \in B_{X^*}$ as well). We shall show that $(\Psi x_\alpha^*)$ weak$^*$ converges to $\Psi x^*$. In light of \Cref{l:simple criterion of weak* convergence}, we have to prove that, for any $n \in \Nat$ and $\varepsilon > 0$, the inequality $\|\overline{Q}_n^* \Psi x_\alpha^* - \overline{Q}_n^* \Psi x^*\| < \varepsilon$ holds for $\alpha$ large enough.
 \\
 
  Fix $m \geq n$ so that $4^{1-m} < \varepsilon/2$. As $\overline{Q}_n^* \Psi x^* = \lim_j \overline{Q}_n^* \Psi_j \overline{P}_j^* x^*$, \eqref{eq:small increase} implies
  $$
  \|\overline{Q}_n^* \Psi x^* - \overline{Q}_n^* \Psi_m \overline{P}_m^* x^*\| \leq \sum_{j=m}^\infty \|\overline{Q}_n^* \Psi_j \overline{P}_j^* x^* - \overline{Q}_n^* \Psi_{j+1} \overline{P}_{j+1}^* x^*\| < 2 \cdot 4^{-m} ,
  $$
  and likewise, $\|\overline{Q}_n^* \Psi x_\alpha^* - \overline{Q}_n^* \Psi_m \overline{P}_m^* x_\alpha^*\| < 2 \cdot 4^{-m}$ for any $\alpha$.
  \\
  
  For $m$ as above, we have $\lim_\alpha \overline{P}_m^* x_\alpha^* = \overline{P}_m^* x^*$ (as $\overline{P}_m$ has finite rank, weak$^*$ and norm convergence coincide). By the continuity of $\Psi_m$, the equality $\lim_\alpha \Psi_m \overline{P}_m^* x_\alpha^* = \Psi_m \overline{P}_m^* x^*$ holds as well. In particular, $\|\Psi_m \overline{P}_m^* x_\alpha^* - \Psi_m \overline{P}_m^* x^*\| < \varepsilon/2$ for $\alpha$ large enough. For such $\alpha$,
  \begin{align*}
      & \|\overline{Q}_n^* \Psi x_\alpha^* - \overline{Q}_n^* \Psi x^*\| \leq \|\Psi_m \overline{P}_m^* x_\alpha^* - \Psi_m \overline{P}_m^* x^*\| +
      \\ &
      \|\overline{Q}_n^* \Psi x^* - \overline{Q}_n^* \Psi_m \overline{P}_m^* x^*\| + \|\overline{Q}_n^* \Psi x_\alpha^* - \overline{Q}_n^* \Psi_m \overline{P}_m^* x_\alpha^*\|
      \\ & < \frac\varepsilon 2 + 2 \cdot 2 \cdot 4^{-m} < \varepsilon ,
  \end{align*}
 as desired.
 \\
 
 Next we define a map $\Phi$, and prove that it is the inverse of $\Psi$, possessing the desired properties.
First let $\Phi_n = \Psi_n^{-1} : \overline{F}_n \to \overline{E}_n$. This map is clearly positively homogeneous and norm-preserving; it is also continuous, by Inverse Function Theorem.
We shall show that, for any $y^* \in \ell_1$, the sequence $( \Phi_n \overline{Q}_n^* y^*)$ is weak$^*$ convergent. Once this is done, we let $\Phi y^* = {\textrm{weak}}^*-\lim \Phi_n \overline{Q}_n^* y^*$.
 \\
 
 First fix $n$, $y_0^* \in \overline{F}_{n-1}$, and $y_n^* \in F_n$. Let $x_0^* = \overline{P}_{n-1}^* \Phi_n (y^*_0 + y_n^*)$ and $x_n^* = P_n^* \Phi_n (y^*_0 + y_n^*)$. Then
 $$
 y_0^* = \overline{Q}_{n-1}^* (y_0^* + y_n^*) = \kappa_n( \|x_n^*\|/\|x_0^*\|) \Psi_{n-1} x_0^* .
 $$
 Apply $\Phi_{n-1}$ to both sides to obtain
 $$
  \kappa_n( \|x_n^*\|/\|x_0^*\|)^{-1} \Phi_{n-1} \overline{Q}_{n-1}^* (y^*_0 + y_n^*) = \overline{P}_{n-1}^* \Phi_n (y_0^* + y_n^*) ,
 $$
 and therefore,
 \begin{align*}
 &
 \big\| \Phi_{n-1} \overline{Q}_{n-1}^* (y^*_0 + y_n^*) -  \overline{P}_{n-1}^* \Phi_n (y^*_0 + y_n^*) \big\| 
 \\
 &
 \leq
 \big( \kappa_n( \|x_n^*\|/\|x_0^*\|)^{-1} - 1 \big) \|y_0^* + y_n^*\|
 \leq 
 4^{-n} \|y_0^* + y_n^*\| . 
 \end{align*}
 Consequently, for $m > n$ and $y^* \in \ell_1$,
 \begin{equation}
 \| \overline{P}_n^* \Phi_m \overline{Q}_m^* y^* - \Phi_n \overline{Q}_n^* y^*\| \leq 2 \cdot 4^{-n} \|y^*\| .
 \label{eq:small_perturbations2}
 \end{equation}
For $n < m < k$ we therefore have:
\begin{align*}
&
\| \overline{P}_n^* \Phi_k \overline{Q}_k^* y^* - \overline{P}_n^* \Phi_m \overline{Q}_m^* y^*\| = \| \overline{P}_n^* \overline{P}_m^* \Phi_k \overline{Q}_k^* y^* - \overline{P}_n^* \Phi_m \overline{Q}_m^* y^*\| 
\\ & \leq
\| \overline{P}_m^* \Phi_k \overline{Q}_k^* y^* - \Phi_m \overline{Q}_m^* y^*\| \leq 2 \cdot 4^{-m} \|y^*\| .
\end{align*}
 Thus, for any $n$, the sequence $(\overline{P}_n^* \Phi_m \overline{Q}_m^* y^*)_m$ is Cauchy, hence convergent in norm.
 Additionally, $\|\Phi_m \overline{Q}_m^* y^*\| = \|\overline{Q}_m^* y^*\| \leq \|y^*\|$, hence, by \Cref{l:simple criterion of weak* convergence}, $(\Phi_m \overline{Q}_m^* y^*)$ has a weak$^*$ limit, say $x^*$.
 Note that $\|x^*\| = \|y^*\|$. Indeed,
 $ \|x^*\| \leq \liminf \|\Phi_m \overline{Q}_m^* y^*\| = \|y^*\|$. On the other hand,
 $\|x^*\| \geq \limsup_m \| \overline{P}_n^* \Phi_m \overline{Q}_m^* y^*\|$ for any $n$.
 Combining \eqref{eq:small_perturbations2} with the fact that
 $\|y^*\| = \lim_n \|\Phi_n \overline{Q}_n^* y^*\|$, we obtain the opposite inequality.
 \\
 
 Thus, the map $\Phi$ is well-defined, positively homogeneous, and norm preserving. The weak$^*$ continuity of $\Phi$ is established in the same manner as that of $\Psi$.
 \\
 
 It remains to show that $\Phi \Psi = I_{X^*}$, and $\Psi \Phi = I_{Y^*}$. We shall only establish the first of these identities, as the second one can be treated similarly.
 Fix $x^* \in X^*$, and let $y^* = \Psi x^*$. Further, for $n \in \Nat$ let $x_n^* = \overline{P}_n^* x^*$ and $y_n^* = \overline{Q}_n^* y^*$.
 As we observed before, $x^* = {\textrm{weak}}^*-\lim x_n^*$, and $y^* = {\textrm{weak}}^*-\lim y_n^*$. 
By the definition of $\Psi$, $y_n^*$ is a scalar multiple of $\Psi_n x_n^*$, with $\|\Psi_n x_n^* - y_n^*\| \leq 4^{-n} \|x^*\|$. Consequently, $\lim_n \|\Phi_n y_n^* - x_n^*\| = 0$, and so, 
$$
\Phi y^* = {\textrm{weak}}^*-\lim \Phi_n y_n^* = {\textrm{weak}}^*-\lim x_n^* = x^* ,
$$
which gives $\Phi \Psi x^* = x^*$.
\\

To summarize: we have defined an invertible map $\Psi : X^* \to Y^*$ so that $\Psi$ itself, and its inverse, are positively homogeneous, norm preserving, and weak$^*$ continuous on bounded sets. Composition with $\Phi = \Psi^{-1}$ induces a lattice homomorphism $T : \fbl^{(\infty)}[X] \to \fbl^{(\infty)}[c_0]$ so that $\Phi = \Phi_T$; then $\Psi = \Phi_{T^{-1}}$. Consequently, $\fbl^{(\infty)}[X]$ and $\fbl^{(\infty)}[c_0]$ are lattice isometric.
\end{proof}

\begin{rem}
Note in particular, that the map $\Psi$ given in the proof of \Cref{t:monot_FDD} is positively homogeneous and provides a weak$^*$ homeomorphism between $rB_{X^*}$ and $rB_{\ell_1}$ for every $r>0$. This provides an improvement to \cite{DvM} where homeomorphisms between $rB_{\ell_p}$ and $rB_{\ell_q}$ were constructed using some heavy machinery from topology.
\end{rem}

For brevity, we shall use the notation ${\mathcal{U}} = \fbl^{(\infty)}[c_0]$. Above, we have shown that $\fbl^{(\infty)}[X]$ is lattice isometric to ${\mathcal{U}}$ whenever $X$ has a monotone FDD.
One can ask whether this lattice isomorphism (or even lattice isometry) holds for any separable $X$. While we cannot answer this question, below we list some partial results.
\\

For future use, we describe a class of spaces possessing a monotone basis (and, consequently, a monotone FDD).
The results below may be known to experts, but we have not been able to find them in the literature.

\begin{prop}\label{p:basis in AM}
 Any separable ${\mathcal{L}}_{\infty,1+}$ space has a monotone basis. Therefore, any separable AM-space has a monotone basis.
\end{prop}

\begin{proof}
 (1) By \cite{LL}, any separable ${\mathcal{L}}_{\infty,1+}$  space $X$ can be written as $\overline{\cup_i E_i}$, where $E_1 \subseteq E_2 \subseteq \ldots$, and $E_n$ is isometric to $\ell_\infty^n$. For each $n$, there exists a contractive projection $R_n$ from $E_n$ to $E_{n-1}$ (for convenience set $R_1 = 0$). For each $i$, find a norm one $x_i$ in $\ker R_i$; clearly $(x_i)$ is a monotone basis.
 \\
 
 (2) If $X$ is an AM-space, then (see \cite[Section 1.b]{LT2}) $X^{**}$ is an $L_\infty$ space, which is a ${\mathcal{L}}_{\infty,1+}$ space. Then $X$ is an ${\mathcal{L}}_{\infty,1+}$ space as well, by Local Reflexivity (see \cite{JL}).
\end{proof}

\begin{cor}\label{c:fbl infty general} For every separable Banach space $X$ we have:
\begin{enumerate}
 \item If $X$ has an FDD, then $\fbl^{(\infty)}[X]$ is lattice isomorphic to ${\mathcal{U}}$.
 \item If $X$ has the Bounded Approximation Property, then $\fbl^{(\infty)}[X]$ is lattice isomorphic to a lattice complemented sublattice of ${\mathcal{U}}$.
 \item There exist a linear isometry $J : \fbl^{(\infty)}[X] \to {\mathcal{U}}$ and a contractive lattice homomorphism $P : {\mathcal{U}} \to\fbl^{(\infty)}[X]$ so that $PJ = id_{\fbl^{(\infty)}[X]}$.
\end{enumerate}
\end{cor}

\begin{proof}
 (1) is obtained by combining \Cref{t:monot_FDD} with the fact that any Banach space with an FDD can be renormed to make this FDD monotone.
 (2) follows from (1), since any separable Banach space with the BAP embeds complementably into a Banach space with a basis \cite[Theorem 1.e.13]{LT1}.
\\

 (3): By \Cref{p:basis in AM}, $Y = \fbl^{(\infty)}[X]$ has a monotone basis, hence there exists a lattice isometry $T : \fbl^{(\infty)}[Y] \to {\mathcal{U}}$. The formal identity $id : Y \to Y$ extends to a contractive surjective lattice homomorphism $\widehat{id} : \fbl^{(\infty)}[Y] \to Y$. Then $J = T \phi_Y$ and $P = \widehat{id} T^{-1}$ have the desired property.
\end{proof}

\begin{rem}\label{topological approaches}\label{p:not the same}
 We do not know whether $\fbl^{(\infty)}[E]$ and $\fbl^{(\infty)}[F]$ must be lattice isometric whenever $E$ and $F$ are separable Banach spaces. However, \Cref{Eberlein} and \Cref{Eberlein2} use WCG arguments to show that there exists non-separable Banach spaces $E$ and $F$ of the same density character, for which $\fbl^{(\infty)}[E]$ and $\fbl^{(\infty)}[F]$ are not lattice isomorphic. A different approach to distinguishing $\fbl^{(\infty)}[E]$ from $\fbl^{(\infty)}[F]$ exploits topological properties of $E^*$ and $F^*$. If $\fbl^{(\infty)}[E]$ and $\fbl^{(\infty)}[F]$ are lattice isomorphic, then there exists a positively homogeneous map $\Phi : F^* \to E^*$, so that both $\Phi$ and its inverse are weak$^*$ continuous on bounded sets, and there exists $C  \geq 1$ so that $C^{-1} \|\Phi f^*\| \leq \|f^*\| \leq C \|\Phi f^*\|$ for any $f^* \in F^*$. Then $B_{F^*}$ is weak$^*$ sequentially compact if and only if $B_{E^*}$ is (see \cite[Chapter XIII]{DiSeq} for general facts about weak$^*$ sequential compactness). Now suppose $\kappa$ is a cardinal, greater or equal than the continuum. Let $E = \ell_1(\kappa)$ and $F = \ell_2(\kappa)$. Both spaces have density character $\kappa$. Then $B_{F^*}$ is weak$^*$ sequentially compact (see \cite[Chapter XIII, Theorem 4]{DiSeq} for a more general fact), while $B_{E^*}$ is not (see \cite[p.~226]{DiSeq}).
\end{rem}

\begin{rem}
To our knowledge, weak$^*$ continuous non-linear maps have not been deeply studied in the Banach space literature. One application has to do with extensions of operators into $C(K)$ spaces \cite[p.~490]{DS1}: if $F$ is a subspace of $E$, then any operator $T : F \to C(K)$ admits an extension $\widetilde{T} : F \to C(K)$ with $\|\widetilde{T}\| \leq C \|T\|$ if and only if there exists a weak$^*$-continuous map $\Phi : B_{F^*} \to C \cdot B_{E^*}$ so that $\Phi f^*|_F = f^*$ for any $f^* \in F^*$ ($\Phi$ implements a ``Hahn-Banach extension''). In \cite{Zippin_ext} such extensions are used to show that, if $F$ is a subspace of $c_0$, then any operator $T : F \to C(K)$ has an extension to $c_0$.
\\

Note that, by Bartle-Graves Theorem  (see e.g. \cite{Michael} for its generalizations), a norm continuous $\Phi$ like this exists for any $C > 1$; by \cite[Corollary 7.4]{BL-nonlin}, we cannot in general make $\Phi$ uniformly continuous. On the other hand, we can ask for which pairs $F \hookrightarrow E$ there exists a positively homogeneous weak$^*$-continuous map $\Phi : B_{F^*} \to C \cdot B_{E^*}$ so that $\Phi f^*|_F = f^*$ for any $f^* \in F^*$.
\\

Another instance where non-linear weak$^*$ continuous maps between Banach spaces have been considered is in \cite{Hein:86, MV}. In these works, among other things it is shown that if two dual Banach spaces $E^*$ and $F^*$ are uniformly homeomorphic with respect to the weak$^*$ topologies, then necessarily $E$ and $F$ must be linearly isomorphic. 
\end{rem}

Finally we investigate Banach space properties of $\fbl^{(\infty)}[E]$.

\begin{prop}\label{p:FBL infty isomorphic to CK}
If $X$ is a separable Banach space, then $\fbl^{(\infty)}[X]$ is isomorphic to $C[0,1]$ as a Banach space.
\end{prop}

\begin{rem}\label{r:non-sep not C(K)}
The situation is different in the non-separable setting. For instance, \cite{B77} gives an example of a non-separable AM-space $X$, which is not isomorphic to a complemented subspace of any $C(K)$. Clearly $X$ is complemented in $\fbl^{(\infty)}[X]$, hence $\fbl^{(\infty)}[X]$ cannot be isomorphic to a complemented subspace of a $C(K)$ space either.
\end{rem}

\begin{rem}
As noted above, if $E$ is an AM-space, then it is complemented in $\fbl^{(\infty)}[E]$. This need not be true if $E$ is merely an $\ell_1$ predual. Indeed, by \cite{BenLin}, there exists a (necessarily separable) Banach space $E$, so that $E^* = \ell_1$ isometrically, and $E$ is not isomorphic to a complemented subspace of any $C(K)$. By \Cref{p:FBL infty isomorphic to CK}, $\fbl^{(\infty)}[E]$ is isomorphic to $C[0,1]$, hence $E$ cannot be complemented in $\fbl^{(\infty)}[E]$.
\end{rem}

To prove \Cref{p:FBL infty isomorphic to CK}, we introduce some notation. Suppose $K$ is a compact Hausdorff space, and $B\subseteq K$ is a closed subset.
Let $C_B(K) = \{ f \in C(K) : f|_B = 0\}$. The following lemma may be known to experts.

\begin{lem}\label{l:isomorphic spaces}
 Suppose $K$ is a compact metrizable space, and $B\subseteq K$ is a closed subset for which $K \backslash B$ is uncountable. Then $C_B(K)$ is linearly isomorphic to $C[0,1]$.
\end{lem}

\begin{proof}
 Throughout the proof, we rely heavily on Milutin's Theorem (see e.g. \cite[III.D.19]{Wojt}), which states that $C(S)$ is isomorphic to $C[0,1]$ whenever $S$ is compact, metrizable, and uncountable. In particular, this is true for $S = K$.
 \\
 
 Note first that $C_B(K)$ is a complemented subspace of $C(K)\sim C[0,1]$. Indeed, consider the restriction operator $v : C(K) \to C(B) : f \mapsto f|_B$. By \cite[III.D.17]{Wojt}, $v$ has a right inverse $u$: specifically, $u$ is a ``linear extension'' operator $u : C(B) \to C(K)$, which satisfies $vu = I_{C(B)}$. Then $uv$ is a projection on $C(K)$, whose kernel is $C_B(K)$.
 \\
 
 Next show that, conversely, $C[0,1]$ embeds complementably into $C_B(K)$.
 Once this is established, invoke the fact that $C[0,1]$ is isomorphic to $c_0(C[0,1])$ \cite[II.B.24]{Wojt}, and use Pe\l czy\'nski decomposition (cf.~\cite[Theorem 2.2.3]{alb-kal}) to complete the proof.
 \\
 
 Pick $\delta > 0$ small enough so that the closed set $V = \{k \in K : {\textrm{dist}}(k,B) \geq \delta\}$ is uncountable. By \cite[III.D.16]{Wojt}, there exists a linear extension operator $u : C(V) \to C_B(K)$. As before, denote by $v : C_B(K) \to C(V)$ the restriction operator; then $vu$ is a projection from $C_B(K)$ onto $C(V)$. Therefore, $C(V) \sim C[0,1]$ embeds complementably into $C_B(K)$.
\end{proof}

\begin{proof}[Proof of \Cref{p:FBL infty isomorphic to CK}]
If $X$ is finite dimensional, then $\fbl^{(\infty)}[X]$ is lattice isomorphic to $C(S_X)$ ($S_X$ being the unit sphere of $X$), hence linearly isomorphic to $C[0,1]$, by Milutin's Theorem.
To handle the case of separable infinite dimensional $X$, below we briefly review the construction from \cite{B73}, recently re-examined in \cite{OiTu}.
\\

For brevity, denote the unit ball of $X^*$, equipped with its weak$^*$ topology, by $A$.
Fix a dense sequence $(h_k)$ in the unit ball of $\fbl^{(\infty)}[X]$, and let $h = \sum_{k=1}^\infty 2^{-k} |h_k|$.
For $n \in \Nat$ let $$A_n = \{x^* \in A : 2^{-n} \leq h(x^*) \leq 2^{1-n} \}$$ (as $h(0) = 0$, the set $A_n$ is non-empty for $n$ large enough).
Let $\overline{A}$ be the one-point compactification of the ``formal'' disjoint union $\sqcup_n A_n$ (achieved by adding a point we call $\infty$).
Define the map $T : \fbl^{(\infty)}[X] \to C(\overline{A})$ as follows: for $f\in\fbl^{(\infty)}[X]$, 
$$[T f](a) = \left\{\begin{array}{cc}
  2^{-n} f(a)/h(a)   & \text{if }a \in A_n, \\
  0  & \text{if }a=\infty. 
\end{array}\right.$$  
It is easy to check that $T$ is a (non-surjective) lattice isomorphism.
\\

Fix $n$ for a moment. For $a, b \in A_n$, the equality $[Tf](a) = [Tf](b)$ holds for any $f \in \fbl^{(\infty)}[X]$ if and only if 
the ratio $f(a)/f(b)$ is independent of $f$. Hahn-Banach Theorem shows that this happens if and only if $a$ is a positive scalar multiple of $b$. Denote this relation on $A_n$ by $\sim$. In the notation of  \cite{B73}, 
let $K_n = A_n/\sim$.
Then $T$ gives rise to a lattice isomorphism from $ \fbl^{(\infty)}[X]$ onto $Z \subseteq C(K)$, where $K = K_1 \sqcup K_2 \sqcup \ldots \sqcup \{\infty\}$ is the one-point compactification of $K_1 \sqcup K_2 \sqcup \ldots$.
Note that any element of the unit sphere of $X^*$ gives rise, via the evaluation map, to at most one point of $K_n$ for each $n$, hence $K$ is uncountable.
\\

Now suppose $t \in K_m$, and $s \in K_n$. If there exists $\lambda \in \Real$ so that $z(t) = \lambda z(s)$ for every $z \in Z$, then, as shown in \cite{B73}, $n \neq m$, and $\lambda = 2^{n-m}$.
For each $m$, denote by $B_m$ the set of all $t \in K_m$ for which there is $s \in K_n$, $n < m$, so that $z(t) = 2^{n-m} z(s)$ holds for every $z \in Z$. Note that such an $s \in K_n$, if it exists, must be unique. 
As shown in \cite{OiTu}, the sets $B_m$ are closed.
\\

Let $B = \overline{\bigcup_m B_m}$. By \cite{B73}, $Z$ (hence also  $\fbl^{(\infty)}[X]$) is isomorphic to $C_B(K)$. To show that $K \backslash B$ is uncountable, pick the smallest $m$ for which $K_m$ is uncountable. As $B_m$ is countable, $K_m \backslash B_m$ is uncountable. Note that $K_m \cap (K \backslash B) = K_m \backslash B_m$, hence $K \backslash B$ is uncountable. The result now follows from \Cref{l:isomorphic spaces}.
\end{proof}

\section*{Acknowledgments}

In the process of preparing this manuscript we benefited from discussions with many colleagues. Special thanks go to: Antonio Avil\'es, David de Hevia, Enrique Garc\'ia-S\'anchez, Alexander Helemskii, Bill Johnson, Niels Laustsen, Gonzalo Mart\'inez-Cervantes, and Gilles Pisier.


\begin{thebibliography}{10}
\bibitem{Abbadini}
M.~Abbadini, \textit{Operations that preserve integrability, and truncated
Riesz spaces.}  Forum Math. \textbf{32} (2020), no. 6, 1487--1513. 

\bibitem{AA}
Y. A. Abramovich and C. D. Aliprantis, \textit{An invitation to operator theory}, Graduate Studies in Mathematics, 50. American Mathematical Society, Providence, RI, 2002.

 \bibitem{alb-kal}
 F.~Albiac and N.~J. Kalton, \textit{Topics in {B}anach space theory}, Graduate
   Texts in Mathematics, vol. 233, Springer, New York, 2006. 
   
   \bibitem{Locally solid}
   C.D.~Aliprantis and O.~Burkinshaw, \textit{Locally solid Riesz spaces with applications to economics}, 2nd ed., AMS, Providence, RI, 2003.

 \bibitem{AB}
 C.D.~Aliprantis and O.~Burkinshaw, \textit{Positive operators}, Springer,
   Dordrecht, 2006, Reprint of the 1985 original. 

\bibitem{Arbib_Manes} M. Arbib and E. Manes,
\textit{Arrows, structures, and functors. The categorical imperative}, Academic Press, 
New York-London, 1975. 

\bibitem{AGLRT}
A.~Avil\'es, A. J. Guirao,  S. Lajara, J.~Rodr\'{\i}guez, and P.~Tradacete, \textit{Weakly compactly generated Banach lattices.} Studia Math. \textbf{234} (2016), no. 2, 165--183.

\bibitem{AMR}
A.~Avil\'es, G.~Mart\'{i}nez-Cervantes, and J.D.~Rodr\'iguez-Abell\'an, \textit{On the Banach lattice $c_0$.}  Rev. Mat. Complut. \textbf{34} (2021), no. 1, 203--213.

\bibitem{AMR2} A.~Avil\'es, G.~Mart\'{i}nez-Cervantes, and J.D.~Rodr\'iguez-Abell\'an, \textit{On projective Banach lattices of the form $C(K)$ and $\fbl[E]$.}  J. Math. Anal. Appl. \textbf{489} (2020), no. 1, 124129, 11 pp.

\bibitem{AMR3}
A.~Avil\'es, G.~Mart\'{i}nez-Cervantes, and A.~ Rueda Zoca, \textit{Local complementation under free constructions.} Preprint. \url{https://arxiv.org/abs/2107.11339}

\bibitem{AMRR}
A.~Avil\'es, G.~Mart\'{i}nez-Cervantes, J.D.~Rodr\'iguez-Abell\'an, and A.~ Rueda Zoca, \textit{Free Banach lattices generated by a lattice and projectivity.}  Proc. Amer. Math. Soc. \textbf{150} (2022), 2071--2082.

\bibitem{AMRR2}
A.~Avil\'es, G.~Mart\'{i}nez-Cervantes, J.D.~Rodr\'iguez-Abell\'an, and A.~ Rueda Zoca, \textit{Lattice embeddings in free Banach lattices over lattices.}  Math. Inequal. Appl. \textbf{25} (2022), no. 2, 495--509.


\bibitem{AMRT}
A.~Avil\'es, G.~Mart\'{i}nez-Cervantes, J.~Rodr\'{\i}guez, and P.~Tradacete,
  \textit{A Godefroy-Kalton principle for free Banach lattices.} Israel J. Math. \textbf{247} (2022), 433--458.
  
  \bibitem{AMZT}
  A.~Avil\'es, G.~Mart\'{i}nez-Cervantes, A.~ Rueda Zoca, and P.~Tradacete, \textit{Linear versus lattice embeddings between Banach lattices.}  Adv. Math. \textbf{406} (2022), Paper No. 108574, 14 pp.
  
  
 \bibitem{APR}
A.~Avil\'es, G.~Plebanek, and J.~D.~Rodr\'iguez-Abell\'an, \textit{Chain conditions in free Banach lattices.} J. Math. Anal. Appl. \textbf{465} (2018), no. 2, 1223--1229.

\bibitem{ART}
A.~Avil\'es, J.~Rodr\'{\i}guez, and P.~Tradacete,
  \textit{The free Banach lattice generated by a Banach space.} J. Funct. Anal. \textbf{274}, No. 10, 2955--2977 (2018).

\bibitem{AR1}
  A.~Avil\'es and J.D.~Rodr\'iguez-Abell\'an, \textit{The free Banach lattice generated by a lattice.}  Positivity \textbf{23} (2019), no. 3, 581--597.
  
 \bibitem{AR}
  A.~Avil\'es and J.D.~Rodr\'iguez-Abell\'an, \textit{Projectivity of the free Banach lattice generated by a lattice.}  Arch. Math. (Basel) \textbf{113} (2019), no. 5, 515--524.
   
\bibitem{AT} A.~Avil\'es and P.~Tradacete,
\textit{Amalgamation and injectivity in Banach lattices.} International Mathematics Research Notices, 2021, rnab285, \url{https://doi.org/10.1093/imrn/rnab285}.
   
\bibitem{ATV}
A.~Avil\'es, P.~Tradacete, and I.~Villanueva,
\textit{The free Banach lattices generated by $\ell_p$ and $c_0$.} Rev. Mat. Complut. \textbf{32}, No. 2, 353--364 (2019).

 \bibitem{Baker}
 K.~A. Baker, \textit{Free vector lattices.} Canad. J. Math. \textbf{20} (1968),
   58--66. 
   
   \bibitem{Beauzamy} 
   B. Beauzamy,
   \textit{Introduction to Banach spaces and their geometry}, Second edition. 
   North-Holland Publishing Co., Amsterdam, 1985.
   
   \bibitem{B73}
   Y.~Benyamini, \textit{Separable $G$ spaces are isomorphic to $C(K)$ spaces.}  Israel J. Math. \textbf{14} (1973), 287--293. 
   
   \bibitem{B77}
   Y.~Benyamini, \textit{An M-space which is not isomorphic to a $C(K)$ space.}  Israel J. Math. \textbf{28} (1977), 98--102. 
   
   
   \bibitem{BenLin}  
   Y.~Benyamini and J. Lindenstrauss,
   \textit{A predual of $\ell_1$ which is not isomorphic to a $C(K)$ space.}
   Israel J. Math. \textbf{13} (1972), 246--254. 
   
   \bibitem{BL-nonlin}  Y.~Benyamini and J. Lindenstrauss,
   \textit{Geometric nonlinear functional analysis}, Vol. 1. 
American Mathematical Society, Providence, RI, 2000. 
   
   \bibitem{Bern} S.J.~Bernau,
   \textit{A unified approach to the principle of local reflexivity}, Notes in Banach spaces, pp.~427--439, Univ. Texas Press, Austin, Tex., 1980.
   
 \bibitem{BP}
 C. Bessaga and A. Pe\l czy\'nski,
\textit{On bases and unconditional convergence of series in Banach spaces.}
Studia Math. \textbf{17} (1958), 151--164. 
  
   \bibitem{Blasco}
   O.~Blasco and T.~Signes, \textit{$q$-concavity and $q$-Orlicz property in symmetric sequence spaces.}  Taiwanese J. Math. \textbf{5} (2001), no. 2, 331--352.

 \bibitem{Bleier}
 R.~D. Bleier, \textit{Free vector lattices.} Trans. Amer. Math. Soc. \textbf{176}
   (1973), 73--87. 
   
   \bibitem{BS}
G.~Buskes and C.~Schwanke, \textit{Functional completions of Archimedean vector lattices.} Algebra Univers. \textbf{76} (2016) 53--69.
   
   \bibitem{BrOz} N. Brown and N. Ozawa, 
   \textit{$C^*$-algebras and finite-dimensional approximations},
American Mathematical Society, Providence, RI, 2008. 
   
   \bibitem{Byrd}
   S.~Byrd, \textit{Factoring operators satisfying $p$-estimates.} Trans. Amer. Math. Soc. \textbf{310}, 2, (1988), 567--582.

\bibitem{CL} D.~Cartwright and H.~Lotz,
\textit{Some characterizations of AM- and AL-spaces.}
Math. Z. \textbf{142} (1975), 97--103. 

   \bibitem{CN}
   P.~Casazza and N.~Nielsen, \textit{Embeddings of Banach spaces into Banach lattices and the Gordon-Lewis property.} Positivity \textbf{5} (2001), 297--321, 
   
   \bibitem{CN03}
   P.~Casazza and N.~Nielsen, \textit{The Maurey extension property for Banach spaces with the Gordon-Lewis property and related structures.} Studia Math. \textbf{155} (2003), 1--21.
   
   \bibitem{Conway}
   J.B.~Conway, \textit{A course in Functional Analysis}, Springer, second edition, New York, 1990.

\bibitem{Dales:17}     H.G. Dales, N.J. Laustsen, T. Oikhberg, and V.G. Troitsky,    \textit{Multi-norms and Banach lattices},   Dissertationes Mathematica, \textbf{524}, 2017, 1--115.
   
   \bibitem{Norm-attaining}
   S.~Dantas, G.~Mart\'inez-Cervantes, J.~Rodr\'iguez Abell\'an, and A.~Rueda Zoca,
   \textit{Norm-attaining lattice homomorphisms.} Rev. Mat. Iberoam. \textbf{38} (2022), no. 3, 981--1002.

   \bibitem{DMRR}  S.~Dantas, G.~Mart\'inez-Cervantes, J.~Rodr\'iguez Abell\'an, and A.~Rueda Zoca, \textit{Octahedral norms in free Banach lattices.}  Rev. R. Acad. Cienc. Exactas F\'is. Nat. Ser. A Mat. RACSAM \textbf{115} (2021), no. 1, Paper No. 6, 20 pp.
   
   \bibitem{Davidson} K. Davidson, \textit{$C^*$-algebras by example}, 
   American Mathematical Society, Providence, RI, 1996. 
   
   \bibitem{Day} M.M.~Day, \textit{Normed Linear Spaces}, third edition, Springer-Verlag, Berlin/New York, 1973.
   
   \bibitem{dean}    D.~Dean, 
\textit{The equation $L(E,X^{**})=L(E,X)^{**}$ and the principle of local reflexivity.}
Proc. Amer. Math. Soc. \textbf{40} (1973), 146--148. 

\bibitem{dHT} D.~de~Hevia and P.~Tradacete, \textit{Free complex Banach lattices.} Preprint, \url{https://arxiv.org/abs/2207.08090}.

 \bibitem{dePW}
 B.~de~Pagter and A.~W. Wickstead, \textit{Free and projective {B}anach lattices.}
   Proc. Roy. Soc. Edinburgh Sect. A \textbf{145} (2015), no.~1, 105--143.
   
   \bibitem{DGZ} R. Deville, G. Godefroy, and V. Zizler,
\textit{Smoothness and renormings in Banach spaces},
Longman Scientific \& Technical, Harlow,
1993. 

 \bibitem{DiGBS}
 J.~Diestel, \textit{Geometry of Banach spaces -- selected topics},
 Springer-Verlag, Berlin-New York, 1975. 
   
 \bibitem{DiSeq}
 J.~Diestel, \textit{Sequences and series in Banach spaces},
 Springer-Verlag, Berlin-New York, 1984. 
   
 \bibitem{DJT}
 J.~Diestel, H.~Jarchow, and A.~Tonge, \textit{Absolutely summing operators},
   Cambridge University
   Press, Cambridge, 1995. 
   
   \bibitem{DilKu} S.~Dilworth and D.~Kutzarova,
   \textit{Kadec-Klee properties for ${\mathcal{L}}(\ell_p,\ell_q)$}, Function spaces (Edwardsville, IL, 1994), 71--83,
Lecture Notes in Pure and Appl. Math., 172, Dekker, New York, 1995. 

   \bibitem{DOSZ} S. Dilworth, E. Odell, T. Schlumprecht, and A. Zsak,
   \textit{Partial unconditionality.}
Houston J. Math. \textbf{35} (2009), 1251--1311. 

\bibitem{DS1} N.~Dunford and J.~Schwartz,
\textit{Linear operators. Part I},
John Wiley \& Sons, Inc., New York, 1988. 

\bibitem{DvM}
J. J. Dijkstra and J. van Mill, \textit{Topological equivalence of discontinuous norms.} Israel J. Math. \textbf{128} (2002), 177--196.

\bibitem{EG} E.~Emelyanov and S.~Gorokhova,
\textit{Free uniformly complete vector lattices.} Preprint, \url{https://arxiv.org/abs/2109.03895}


\bibitem{FMZ}
M.~Fabian, V.~Montesinos, and V.~Zizler, \textit{A characterization of subspace of weakly compactly generated Banach spaces.} J. London Math. Soc. (2) \textbf{69} (2004) 457-464.


 \bibitem{FHHMPZ}
 M.~Fabian, P.~Habala, P.~H{\'a}jek, V.~Montesinos, J. Pelant, and V.~Zizler, \textit{Functional analysis and infinite-dimensional geometry},
CMS Books in Mathematics/Ouvrages de Math\'ematiques de la SMC, 8. Springer-Verlag, New York, 2001.
   
 \bibitem{fab-ultimo}
 M.~Fabian, P.~Habala, P.~H{\'a}jek, V.~Montesinos, and V.~Zizler, \textit{Banach
   space theory. The basis for linear and nonlinear analysis}, CMS Books in Mathematics/Ouvrages de Math\'ematiques de la
   SMC, Springer, New York, 2011.
   
   \bibitem{FLAMT} V.~Ferenczi, J.~Lopez-Abad, B.~Mbombo, and S.~Todorcevic,
\textit{Amalgamation and Ramsey properties of $L_p$ spaces.}
Adv. Math. \textbf{369} (2020), 107190, 76 pp. 
   
   \bibitem{FJS1} T.~Figiel, W.B.~Johnson, and G.~Schechtman,
\textit{Factorizations of natural embeddings of $\ell^n_p$ into $L_r$, I.}
Studia Math. \textbf{89} (1988), 79--103. 
   
   \bibitem{FJS2} T.~Figiel, W.B.~Johnson, and G.~Schechtman,
\textit{Factorizations of natural embeddings of $\ell^n_p$ into $L_r$, II.}
 Pacific J. Math. \textbf{150} (1991), 261--277. 
   
   \bibitem{FlemingJamison2} R.~Fleming and J.~Jamison,
   \textit{Isometries on Banach spaces. Vol. 2. Vector-valued function spaces},
   Chapman \& Hall/CRC, Boca Raton, FL, 2008. 
   
   \bibitem{Fo} S.~Foguel,
   \textit{On a theorem by A. E. Taylor.}
Proc. Amer. Math. Soc. \textbf{9} (1958), 325. 


  \bibitem{F04} D.H. Fremlin, \textit{Measure Theory, vol. 3}, Torres Fremlin, Colchester, 2004, Broad foundations, Corrected second printing of the 2002 original.

\bibitem{GTX} N. Gao, V.G. Troitsky, and F. Xanthos, \textit{Uo-convergence and its applications to Ces\`aro Means in Banach Lattices.} Israel J. Math. \textbf{220} (2017), no. 2, 649--689.
   
   \bibitem{Garling74}
   D.J.H.~Garling, \textit{Diagonal mappings between sequence spaces.} Studia Math. \textbf{51} (1974), 129--138, 
   
   \bibitem{Grafakos}
   L.~Grafakos, \textit{Classical Fourier Analysis}, Third edition. Graduate Texts in Mathematics, 249. Springer, New York, 2014.
   
\bibitem{Godefroy}
   G.~Godefroy, \textit{A survey on Lipschitz-free Banach spaces.} Comment. Math. \textbf{55} (2015), no. 2, 89--118. 
   
      \bibitem{GK}
   G.~Godefroy and N.J.~Kalton,  \textit{Lipschitz-free Banach spaces.} Studia Math. \textbf{159} (2003), 121--141.
   
   \bibitem{GKS} G. Godefroy, N.J. Kalton, and P. Saphar, \textit{Unconditional ideals in Banach spaces.} Studia Math. \textbf{104} (1993), 13--59.
   
   \bibitem{GE83}
J.J.~Grobler and P.~van Eldik, \textit{Carleman operators in Riesz spaces.} Nederl. Akad. Wetensch. Indag. Math. \textbf{45} (1983), no. 4, 421--433.
   
   
   \bibitem{Gumenchuk:15}
  A.~Gumenchuk, O.~Karlova, and M.~Popov,
   \textit{Order Schauder bases in Banach lattices.}
  J.\ Funct.\ Anal. \textbf{269} (2015), no. 2, 536--550. 
  
  \bibitem{HaTo} P.~H{\'a}jek and J.~Talponen, 
\textit{Note on Kadets Klee property and Asplund spaces.}
Proc. Amer. Math. Soc. \textbf{142} (2014), 3933--3939. 
  
  \bibitem{Hein:86}
  S. Heinrich, \textit{The uniform classification of boundedly compact locally convex spaces.} Czechoslovak Math. J. \textbf{36} (111) (1986), no. 1, 68--71.
  
  \bibitem{HeiMan}
 S. Heinrich and P. Mankiewicz, \textit{Applications of ultrapowers to the uniform and Lipschitz classification of Banach spaces.} Studia Math. \textbf{73} (1982), no. 3, 225--251. 
  
  
 \bibitem{Hel-projectivity}  A.~Helemskii,
 \textit{Metric freeness and projectivity for classical and quantum normed modules.}
Sb. Math. \textbf{204} (2013), 1056--1083.
 
 \bibitem{Hel13}  A.~Helemskii,
 \textit{Extreme version of projectivity for normed modules over sequence algebras.}
Canad. J. Math. \textbf{65} (2013), 559--574. 
 
 \bibitem{Hel14}  A.~Helemskii,
 \textit{Projectivity for operator modules: approach based on freedom.} Rev. Roumaine Math. Pures Appl. \textbf{59} (2014), 219--236. 
 
 \bibitem{Hel20}  A.~Helemskii,
 \textit{Projective and free matricially normed spaces}, Banach algebras and applications, 151--164,
De Gruyter, Berlin, 2020. 
 
 \bibitem{HO} A.~Helemskii and T.~Oikhberg,
 \textit{Free and projective generalized multinormed spaces.} J. Math. Anal. Appl. \textbf{518} (2023), no. 1, Paper No. 126660.
  
  \bibitem{IT} D. Ilisevic and A. Turnsek,
\textit{On Wigner's theorem in smooth normed spaces.}
Aequationes Math. \textbf{94} (2020), 1257--1267. 

\bibitem{James72} R.C. James,
\textit{Super-reflexive Banach spaces.}
Canadian J. Math. \textbf{24} (1972), 896--904. 
  
  \bibitem{Jarchow} H.~Jarchow, \textit{On Hilbert-Schmidt spaces}, Proceedings of the 10th Winter School on Abstract Analysis. Circolo Matematico di Palermo, Palermo, 1982. Rendiconti del Circolo Matematico di Palermo, Serie II, Supplemento No. 2. pp. [153]--160.
  
    \bibitem{JLTTT} 
H.~Jard\'on-S\'anchez, N.J.~Laustsen, M.A.~Taylor, P.~Tradacete, and V.G.~Troitsky, \textit{Free Banach lattices under convexity conditions.}  Rev. R. Acad. Cienc. Exactas F\'is. Nat. Ser. A Mat. RACSAM \textbf{116} (2022), no. 1, Paper No. 15.

\bibitem{MIP-RE} Z. Ji, A. Natarajan, T. Vidick, J. Wright, and H. Yuen, \textit{MIP$^*$=RE.} Preprint, 
\url{https://arxiv.org/abs/2001.04383}

    \bibitem{JL} W. B. Johnson and J. Lindenstrauss,
  \textit{Basic concepts in the geometry of Banach spaces},
  Handbook of the geometry of Banach spaces, Vol. I, 1--84, North-Holland, Amsterdam, 2001. 
  
  \bibitem{JMS}   W. B. Johnson, B. Maurey, and G. Schechtman,
  \textit{Weakly null sequences in $L_1$.} J. Amer. Math. Soc. \textbf{20} (2007), 25--36. 
   
  \bibitem{JMST}   W. B. Johnson, B. Maurey, G. Schechtman, and L. Tzafriri,
\textit{Symmetric structures in Banach spaces.}
Mem. Amer. Math. Soc. \textbf{19} (1979), no. 217
  
\bibitem{Ka} S. Kakutani,
\textit{Concrete representation of abstract (M)-spaces (A characterization of the space of continuous functions).} Ann. of Math. (2) \textbf{42} (1941), 994--1024. 

  \bibitem{Kalton_memoir}
   N. J. Kalton, \textit{Lattice structures on Banach spaces.} Mem. Amer. Math. Soc. \textbf{103} (1993), no. 493.
   
  \bibitem{Kalton84}
   N. J. Kalton, \textit{Locally complemented subspaces and $\mathscr{L}_p$-spaces for $0 < p < 1$.} Math. Nachr. \textbf{115} (1984), 71--97.
   
   \bibitem{KT}
   M.~Kandi\'c and M.A.~Taylor, \textit{Metrizability of minimal and unbounded topologies.} J.~Math.~Anal.~Appl. \textbf{466} (2018), no.1, 144-159.
   
   \bibitem{KV}
   A.V.~Koldunov and A.I.~Veksler, \textit{On normed lattices and their Banach completions.} Positivity \textbf{9} (2005), no. 3, 415--435.
   
    \bibitem{K77}
  V.M.~Kopytov, \textit{Lattice-ordered Lie algebras.} Sibirsk. Mat. Z. \textbf{18} (1977), no. 3, 595--607, 718.
  
    \bibitem{KM94}
  V.M.~Kopytov and N.~Ya.~Medvedev, \textit{The theory of lattice-ordered groups,} Mathematics and its Applications, 307. Kluwer Academic Publishers Group, Dordrecht, 1994.
   
  \bibitem{Kwapien:70}
  S.~Kwapie\'{n} and A.~Pe\l czy\'{n}ski,
  \textit{The main triangle projection in matrix spaces and its applications.}
  Studia Math. \textbf{34} (1970), 43--68. 

  \bibitem{LatOl}
  R.~Latala and K.~Oleszkiewicz,
\textit{On the best constant in the Khinchin-Kahane inequality.}
Studia Math. \textbf{109} (1994), 101--104.   

\bibitem{LL} A.~Lazar and J.~Lindenstrauss,
\textit{Banach spaces whose duals are $L_1$ spaces and their representing matrices.}
Acta Math. \textbf{126} (1971), 165--193. 

\bibitem{Laust-Tra}
N. J. Laustsen and P. Tradacete, \textit{Isomorphisms between free Banach lattices.} (Work in progress).  

\bibitem{Laust-Tro}
N. J. Laustsen and V. G. Troitsky, \textit{Vector lattices admitting a positively homogeneous continuous function calculus.} Q. J. Math. \textbf{71} (2020), no. 1, 281--294.


  \bibitem{Levy} M.~Levy,
  \textit{Prolongement d'un op\'erateur d'un sous-espace de $L_1(\mu)$ dans $L_1(\nu)$}, 
  Seminar on Functional Analysis, 1979-1980, 
  Exp. No. 5, 5 pp., \'Ecole Polytech., Palaiseau, 1980. 
  
   \bibitem{lind_memoir} J. Lindenstrauss,
\textit{Extension of compact operators.}
Mem. Amer. Math. Soc. \textbf{48} (1964). 

 \bibitem{lind_ros} J. Lindenstrauss and H. Rosenthal,
  \textit{The ${\mathcal{L}}_p$ spaces.}
   Israel J. Math. \textbf{7} (1969), 325--349. 
  
\bibitem{LT1}
J. Lindenstrauss and L. Tzafriri, \textit{Classical Banach Spaces. I},
Springer Verlag, 1977.

\bibitem{LT2} J. Lindenstrauss and L. Tzafriri, \textit{Classical Banach spaces. II}, Springer-Verlag, Berlin, 1979.

\bibitem{Lotz} H. P. Lotz,
\textit{Extensions and liftings of positive linear mappings on Banach lattices.}
Trans. Amer. Math. Soc. \textbf{211} (1975), 85--100. 


\bibitem{Macula:92}
A.J.~Macula, \textit{Free $\alpha$-extensions of an Archimedean vector lattice and their topological duals.} Trans. Amer. Math. Soc.,
\textbf{332} (1992), 437--448.
  
  \bibitem{MV}
 P. Mankiewicz and J. Vil{\'i}movsk{\'y}, \textit{A remark on uniform classification of boundedly compact linear topological spaces.} Rocky Mountain J. Math. \textbf{10} (1980), no. 1, 59--64.

\bibitem{M} B.~Maurey, \textit{Th\'eoremes de factorisation pour les operat\`eurs lin\'eaires \`a valeurs dans les espaces $L_p$.} Ast\'erisque \textbf{11}, Soci\'et\'e Math.~de France.
 
\bibitem{MR} B. Maurey and H. P. Rosenthal, \textit{Normalized weakly null sequence with no unconditional subsequence.} Studia Math. \textbf{61} (1977), no. 1, 77--98. 
  
\bibitem{McA} C.W.~McArthur,
\textit{The weak basis theorem.} Colloq. Math. \textbf{17} (1967), 71--76. 

\bibitem{M-N} P. Meyer-Nieberg,
\textit{Banach lattices}, Springer, Berlin, 1991.

\bibitem{Michael} E. Michael,
\textit{Continuous selections. I.} Ann. of Math. (2) \textbf{63} (1956), 361--382. 


\bibitem{OiTu} T.~Oikhberg and M.A.~Tursi,
\textit{Renorming AM-spaces.} Proc. Amer. Math. Soc., to appear.

\bibitem{OjP} E. Oja and M. Poldvere,
\textit{Principle of local reflexivity revisited.}
Proc. Amer. Math. Soc. \textbf{135} (2007), 1081--1088. 

  \bibitem{P93}
  V.G.~Pestov, \textit{Universal arrows to forgetful functors from categories of topological algebra.}  Bull. Austral. Math. Soc. \textbf{48} (1993), no. 2, 209--249. 


\bibitem{Pisier_FACT} G. Pisier,
\textit{Factorization of linear operators and geometry of Banach spaces},
Amer. Math. Soc., Providence, RI, 1986.

\bibitem{Pisier_Volume} G. Pisier,
\textit{The volume of convex bodies and Banach space geometry},
Cambridge University Press, Cambridge, 1989.

\bibitem{Pisier}
G.~Pisier, \textit{Factorization of operators through $L_{p\infty}$ or $L_{p1}$ and non-commutative generalizations.} Math. Ann. \textbf{276} (1986), no. 1, 105--136.

\bibitem{Pisier_REG94} G. Pisier,
\textit{Complex interpolation and regular operators between Banach lattices.}
Arch. Math. (Basel) \textbf{62} (1994), 261--269. 

\bibitem{Pisier_REG95} G. Pisier,
\textit{Regular operators between non-commutative $L_p$-spaces.}
Bull. Sci. Math. \textbf{119} (1995), 95--118. 

\bibitem{Pisier-TP} G. Pisier,
\textit{Tensor products of $C^*$-algebras and operator spaces -- the Connes-Kirchberg problem}, 
Cambridge University Press, Cambridge, 2020. 

\bibitem{Raja} M. Raja,
\textit{Weak$^*$ locally uniformly rotund norms and descriptive compact spaces.}
J. Funct. Anal. \textbf{197} (2003), 1--13. 

\bibitem{Rauch} P. Rauch,
\textit{Pseudocompl{\'e}mentation dans les espaces de Banach.}
Studia Math. \textbf{100} (1991), 251--282. 

\bibitem{RT}
Y.~Raynaud and P.~Tradacete, \textit{Interpolation of Banach lattices and factorization of $p$-convex and $q$-concave operators.} Integral Equations Operator Theory \textbf{66} (2010), no. 1, 79--112.

\bibitem{Reisner}
S.~Reisner, \textit{Operators which factor through convex Banach lattices.} Canad. J. Math. \textbf{32} (1980), 1482-1500.

\bibitem{Rosenthal:70}
H. P. Rosenthal, \textit{On relatively disjoint families of measures, with some applications to Banach space theory.} Studia Math. \textbf{37} (1970), 13--36.

\bibitem{Rosenthal}
H. P. Rosenthal,
\textit{The heredity problem for weakly compactly generated Banach spaces.}
Compositio Math. \textbf{28} (1974), 83--111. 

\bibitem{Rud}
M. Rudelson,
\textit{Characterization of 2-trivial Banach spaces with an unconditional basis.} 
J. Soviet Math. \textbf{44} (1989), 800--808. 

\bibitem{SanPe-Tra}
E. A. S\'anchez-P\'erez and P. Tradacete, \textit{$(p,q)$-regular operators between Banach lattices.}
Monatsh. Math. \textbf{188} (2019), no. 2, 321--350. 

\bibitem{Schaefer74} H.H. Schaefer, \textit{Banach lattices and positive operators.}
Springer-Verlag, New York-Heidelberg, 1974.

\bibitem{SchepKr} A. Schep,
\textit{Krivine's theorem and the indices of a Banach lattice.}
Acta Appl. Math. \textbf{27} (1992), 111--121. 

\bibitem{Schep} A.~Schep, \textit{Products and factors of Banach function spaces.}
Positivity \textbf{14} (2010), 301--319. 

\bibitem{SimsYost}
B. Sims and D. Yost, \textit{Linear Hahn-Banach extension operators.} Proc. Edinburgh Math. Soc. (2) \textbf{32} (1989), no. 1, 53--57.

\bibitem{Singer}
I. Singer, \textit{Bases in Banach spaces I}, Springer, New York, 1970.

\bibitem{Tal1} M.~Talagrand, \textit{Cotype of operators from $C(K)$.} Invent. Math. \textbf{107} (1992), no. 1, 1-40. 

\bibitem{Tal2} M.~Talagrand, \textit{Cotype and $(q,1)$-summing norm in a Banach space.} Invent. Math. \textbf{110} (1992), no. 3, 545-556.

\bibitem{Tal3} M.~Talagrand, \textit{Orlicz property and cotype in symmetric sequence spaces.} Israel J. Math. \textbf{87} (1994), no. 1-3, 181-192. 


\bibitem{Tay1}
M.A.~Taylor, \textit{Unbounded topologies and uo-convergence in locally solid vector lattices.} J.~Math.~Anal.~Appl. \textbf{472} (2019), no.~1, 981-1000.

\bibitem{Tay}
M.A.~Taylor, \textit{Unbounded convergences in vector lattices}, Thesis, 2018.

\bibitem{TT}
  M.A.~Taylor and V.G.~Troitsky,
 \textit{Bibasic sequences in Banach lattices.}
  J.\ Funct.\ Anal. \textbf{278} (2020), no. 10, 108448, 33 pp. 
  
    \bibitem{T65}
  D.M.~Topping, \textit{Some homological pathology in vector lattices.} Canadian J. Math. \textbf{17} (1965), 411--428.

  
 \bibitem{Tro}
 V.G.~Troitsky,  \textit{Simple constructions of $FBL(A)$ and $FBL[E]$.}
   Positivity \textbf{23} (2019), no. 5, 1173--1178.   
   
\bibitem{vanWaaij:13}
J.~van Waaij, \textit{Tensor products in Riesz space theory}, thesis, Mathematical Institute, University of Leiden, 2013.


\bibitem{Wickhom}
A.W.~Wickstead, \textit{Banach lattices with trivial centre.} Proc. Roy. Irish Acad. Sect. A \textbf{88} (1988), no. 1, 71--83. 
  
  \bibitem{Wojt}
  P.~Wojtaszczyk, 
\textit{Banach spaces for analysts},
Cambridge University Press, Cambridge, 1991. 
  
  \bibitem{Zippin_ext} M. Zippin,
  \textit{Applications of Michael's continuous selection theorem to operator extension problems.}
Proc. Amer. Math. Soc. \textbf{127} (1999), 1371--1378. 
  
  \bibitem{Zippin_survey} M. Zippin,
  \textit{Extension of bounded linear operators}, Handbook of the geometry of Banach spaces, Vol. 2, 1703--1741, North-Holland, Amsterdam, 2003.
  
 
 \end{thebibliography}
\end{document}